\documentclass[11pt]{amsart}
\usepackage[T1]{fontenc}
\usepackage[utf8]{inputenc}
\usepackage{lmodern}
\usepackage[a4paper,margin=1in]{geometry}
\usepackage{amsmath,amssymb,amsthm,mathtools,mathrsfs}
\usepackage{array,booktabs,longtable}
\usepackage{enumitem}
\usepackage{microtype}
\usepackage{xcolor}
\usepackage[colorlinks=true,linkcolor=blue!45!black,citecolor=blue!45!black,urlcolor=blue!45!black]{hyperref}
\usepackage[nameinlink,noabbrev]{cleveref}

\newtheorem{theorem}{Theorem}[section]
\newtheorem{proposition}[theorem]{Proposition}
\newtheorem{lemma}[theorem]{Lemma}
\newtheorem{corollary}[theorem]{Corollary}
\theoremstyle{definition}

\theoremstyle{remark}
\newtheorem{remark}[theorem]{Remark}

\newcommand{\C}{\mathbf C}

\newcommand{\Q}{\mathbf Q}
\newcommand{\F}{\mathbf F}

\newcommand{\Tr}{\operatorname{Tr}}
\newcommand{\Nm}{\operatorname{N}}
\newcommand{\Gal}{\operatorname{Gal}}
\newcommand{\ph}{\operatorname{ph}}

\newcommand{\Res}{\operatorname{Res}}

\newcommand{\HerbPsi}{\psi^{\mathrm H}}

\newcommand{\ceil}[1]{\left\lceil #1\right\rceil}
\newcommand{\floor}[1]{\left\lfloor #1\right\rfloor}

\newcommand{\OO}{\mathcal O}\newcommand{\pp}{\mathfrak p}

\newcommand{\CC}{\mathbf C}\newcommand{\ind}{\mathbf1}
\newcommand{\pr}{\operatorname{pr}_0}
\newcommand{\Cart}{\mathcal C}
\title[The Second Main Lemma for Local Epsilon Factors]
{A Local Proof of Langlands's Second Main Lemma\\
for Local Epsilon Factors over Nonarchimedean Local Fields}
\author{Fukuhiro Ueda}
\address{Research Institute for Mathematical Sciences, Kyoto University, Kyoto 606--8502, Japan}
\email{fueda@kurims.kyoto-u.ac.jp}
\date{September 2026}
\hypersetup{
  pdftitle={A Local Proof of Langlands's Second Main Lemma for Local Epsilon Factors over Nonarchimedean Local Fields},
  pdfauthor={Fukuhiro Ueda},
  pdfsubject={A local proof of the Second Main Lemma in mixed and equal characteristic},
  pdfkeywords={local epsilon factors, Artin L-functions, Gauss sums, ramification, Second Main Lemma},
  bookmarksdepth=2
}
\begin{document}
\begin{abstract}
We give a local proof of Langlands's Second Main Lemma for local epsilon
factors over nonarchimedean local fields. The lemma compares the induction expressions attached to compatible characters of distinct intermediate fields of prime degree in a bicyclic extension of prime-square degree.
 It is one of the main identities used in Langlands's
construction of epsilon factors of local Weil representations.  
 
The proof  builds on the work of Dwork,
Lakkis, and Langlands.   The main new point in the present paper is the wild dyadic case, i.e. biquadratic extensions over a nonarchimedean local field of residue characteristic two. This completes the proof of  the Second Main Lemma, including the equal-characteristic case. Finally, together
with the local proof of the First Main Lemma \cite{Ueda}  and Langlands's remaining reductions, we obtain  a purely
local construction of local epsilon factors envisioned by Langlands.

\end{abstract}
\maketitle
\clearpage
\begingroup
\raggedbottom
\makeatletter
\renewcommand{\@pnumwidth}{3em}
\makeatother
\tableofcontents
\clearpage
\endgroup

\section{Introduction}\label{sec:introduction}
\subsection{Langlands's Second Main Lemma}
Langlands's Main Lemmas are identities between local
constants associated to characters over  different  nonarchimedean local fields. They are used to define epsilon factors of local Weil representations by induction from one-dimensional representations~\cite{Langlands}.
The Second Main Lemma compares the induction expressions attached to characters of distinct intermediate
fields that have the same norm pullback. 

Let $F$ be a nonarchimedean local field with residue characteristic $p$.
Let $K/F$ be a Galois extension with bicyclic Galois group $C_\ell\times C_\ell$, where $\ell$ is
prime. For a degree-$\ell$ intermediate field $L$, put
\[
 S(L/F)=\{\mu\in\operatorname{Hom}_{\mathrm{cont}}(F^\times,\CC^\times):
             \mu\circ \Nm_{L/F}=1\}.
\]
Here $\operatorname{Hom}_{\mathrm{cont}}$ denotes continuous group
homomorphisms. The elements of $S(L/F)$ are called norm characters.
Fix a nontrivial (continuous) additive character $\psi_F$. For every
finite extension $E/F$ occurring below, put
$\psi_E=\psi_F\circ\Tr_{E/F}$.

Let $\Theta:K^\times\to\CC^\times$ be a (continuous) quasi-character
invariant under $\Gal(K/F)$. For every degree-$\ell$ intermediate field
$L$, choose a quasi-character $\theta_L$ of $L^\times$ such
that \[\theta_L\circ \Nm_{K/L}=\Theta.\] We call $(\theta_L)$ a
\emph{compatible family}. It is called \emph{primitive} if there is no
 quasi-character $\theta_F$ of $F^\times$ with
$\Theta=\theta_F\circ \Nm_{K/F}$.

\begin{theorem}[Second Main Lemma]\label{U:main}
For every primitive compatible family $(\theta_L)$, the quantity
\[
 \mathcal A_L(\theta_L)=\Delta_L(\theta_L,\psi_L)
                 \prod_{\mu\in S(L/F)}\Delta_F(\mu,\psi_F)
\]
is independent of the degree-$\ell$ intermediate field $L$.
This holds in mixed and equal characteristic, for every residue
characteristic $p$ and every prime $\ell$.
\end{theorem}
Here $\Delta_E(\theta,\psi)$ is the local constant defined in Section \ref{U:statement}, which is essentially a normalized Gauss sum.

The characters $\theta_L$ exist for every invariant $\Theta$.
Indeed,   Lemma~\ref{C:hilbert90}
describes the kernel of $\Nm_{K/L}$, and then by Galois-invariance  $\Theta$ gives a character of
$\Nm_{K/L}(K^\times)\subset L^\times$. Norm filtration shows that this
character is trivial on a sufficiently deep unit group of $L$, and
Lemma~\ref{C:character-extension} extends it to $L^\times$.
Conversely, if characters of two distinct degree-$\ell$ intermediate fields $L_i$ have the same norm
pullback, that pullback is invariant under both subgroups
$\Gal(K/L_i)$, which generate $\Gal(K/F)$. 

The primitivity assumption implies that
$\theta_L^\sigma/\theta_L$ is nontrivial for every
$1\ne\sigma\in\Gal(L/F)$, hence generates
$S(K/L)$. Indeed, if $\theta_L$ were
$\Gal(L/F)$-invariant, the same consideration as above would give
$\theta_L=\lambda\circ \Nm_{L/F}$ for some quasi-character
$\lambda$ of $F^\times$, and hence
$\Theta=\lambda\circ \Nm_{K/F}$.  

   Conceptually, primitivity means that  the corresponding representation $\widehat{\Theta}$ of the Weil group $W_K$  does not
extend to a character of $W_F$. Consequently, for every degree-$\ell$
intermediate field $L$ and every extension $\widehat{\theta}_L$ of
$\widehat{\Theta}$ to $W_L$, the representation
$\operatorname{Ind}_{W_L}^{W_F}\widehat{\theta}_L$ is irreducible; these induced
representations are isomorphic for all $L$. The Second Main Lemma
asserts that the local constant obtained from these different induction
presentations is independent of $L$.

\subsection{Historical background and contribution of the present paper}

Langlands attributed the Second Main Lemma to Dwork in~\cite{Langlands}.
He later stated that Dwork had proved the lemma, but that
Dwork never published the proof.    Thus the existence of a complete proof by Dwork is
known from Langlands's account, not from a surviving published proof
of Dwork.  We therefore do not claim that any particular calculation
below was absent from Dwork's unpublished proof.

Lakkis's 1966 paper gives detailed
local calculations for the Second Main Lemma~\cite{Lakkis}. His proof of
the Second Main Lemma is exact in odd prime degree, but leaves the dyadic comparison only up to
sign~\cite[pp.~226--227, 273--351]{Lakkis}. He acknowledges that
Dwork's dissertation and other notes greatly helped him prove
S\"atze~1 and~2~\cite[p.~184]{Lakkis}. Both Dwork and Lakkis worked
in the mixed-characteristic setting.

Langlands formulated the Second Main Lemma in both mixed and equal characteristic.  The surviving Chapter~12 of \cite{Langlands} states the lemma and
proves preliminary conductor reductions.  When one degree-$\ell$ intermediate field is unramified, Chapter~12 proves the conductor-one case and the case in which the
character on the unramified intermediate field has conductor at least
$2T$, where $T$ is the common conductor of the nontrivial norm characters
of a ramified degree-$\ell$ intermediate field; see Sections~\ref{sec:tame-comparison}, \ref{sec:odd-ur-stable}, and~\ref{sec:dyadic-ur-stable} below for the corresponding comparisons in the present paper.  Langlands later explained that, after Deligne found a shorter proof using global functional equations, he stopped preparing a complete version of the local argument.

The main new point in the present paper is the exact treatment of the
biquadratic dyadic case.  For a biquadratic extension, the First Main Lemma gives
\[
\left(
\frac{\mathcal A_{L_1}(\theta_{L_1})}
     {\mathcal A_{L_2}(\theta_{L_2})}
\right)^2=1,
\]
so it leaves a possible sign.  Lakkis's dyadic calculation leaves the
same sign undetermined.  We determine it by comparing the actual finite
sums in Lamprecht's formula.  This proves
\[
\mathcal A_{L_1}(\theta_{L_1})
=
\mathcal A_{L_2}(\theta_{L_2})
\]
in residue characteristic two for every ramification type and every
character conductor.  Together with the odd-prime and tame
calculations, this gives the Second Main Lemma in full.

The Lean formalization also suggested several simplifications of the arguments below.

\subsection{Outline of the proof}

Sections~\ref{sec:local-notation}--\ref{sec:stationary} collect the
results used throughout the proof.  Section~\ref{sec:local-notation}
defines the local constants and recalls the First Main Lemma.
Section~\ref{sec:ramification} records the trace, norm, conductor, and
ramification formulas used later.  Section~\ref{sec:stationary} gives
the form of Lamprecht's formula used in the proof and establishes the
comparison results for the resulting finite sums.

Local constants of conductor zero or one are evaluated directly.  If
the conductor is greater than one, Lamprecht's formula expresses the
local constant as the product of a character value, an additive
character value, and a finite-sum factor; see
Lemma~\ref{U:stationary-factor}.  The proof compares these factors for
two degree-$\ell$ intermediate fields.

If $\ell\ne p$, the extension is tame.  There is a unique unramified
degree-$\ell$ intermediate field.  Theorem~\ref{O:F:tame} compares its
normalized induction expression with that of every ramified
degree-$\ell$ intermediate field.   This includes
$\ell=2$ when $p$ is odd.  The proof uses finite-field Gauss sums at
conductor one and stable twisting at higher conductor.

Now suppose that $\ell=p>2$.  Corollary~\ref{U:odd-power} gives
\[
 \left(\frac{\mathcal A_{L_1}(\theta_{L_1})}
 {\mathcal A_{L_2}(\theta_{L_2})}\right)^p=1.
\]
Let $I$ be the inertia subgroup of $\Gal(K/F)$.   
If $|I|=p$, the comparison is carried out in
Section~\ref{sec:odd-ur}.  If $|I|=p^2$, it is carried out in
Sections~\ref{sec:odd-estimates}--\ref{sec:odd-calculation}.  In both
cases the calculations show that the same quotient is a fourth root of
unity.  Since $p$ is odd, the quotient is $1$.  When $p=3$, the additional boundary cancellation is proved in the
proof of Proposition~\ref{O:I:cubicidentity}. The same formulas are used
in mixed and equal characteristic. The additional terms with coefficients
divisible by $p$ vanish in equal characteristic; in mixed characteristic,
valuation estimates make their additive-character values trivial.

It remains to consider $\ell=p=2$.  Proposition~\ref{U:quadratic-square}
gives
\[
 \frac{\mathcal A_{L_1}(\theta_{L_1})}
 {\mathcal A_{L_2}(\theta_{L_2})}\in\{1,-1\},
\]
so equality of the squares does not determine the sign.  The proof
therefore compares the critical functions in Lamprecht's formula.
Lemmas~\ref{U:e2}--\ref{U:common-function} compare these functions when
their stationary coefficients are norms of the same element of
$K$.  If the induced norm map on the required unit quotient has image
of index two, Proposition~\ref{U:missing-coset} compares the two finite sums;
its hypotheses are proved using the third quadratic intermediate
field.

If $|I|=2$, this argument is carried out in
Section~\ref{sec:dyadic-ur}.  Suppose instead that $|I|=4$, so that
$K/F$ is totally ramified.  In characteristic two,
Section~\ref{sec:dyadic-equal} uses Artin--Schreier equations and local
residues. In mixed characteristic, the three lower breaks are either
\[
 (2a-1,2r-1,2r-1),\qquad 1\le a\le r\le v_F(2),
\]
or
\[
 (2a-1,2v_F(2),2v_F(2)),\qquad 1\le a\le v_F(2).
\]
These two cases are treated in Sections~\ref{sec:dyadic-nonmaximal} and
\ref{sec:dyadic-maximal}, respectively.

Thus the proof is divided as follows:
\[
\begin{array}{c|l}
\text{hypotheses} & \text{section}\\ \hline
\ell\ne p
  & \ref{sec:tame}\\
\ell=p>2,\ |I|=p
  & \ref{sec:odd-ur}\\
\ell=p>2,\ |I|=p^2
  & \ref{sec:odd-estimates}\text{--}\ref{sec:odd-calculation}\\
p=\ell=2,\ |I|=2
  & \ref{sec:dyadic-ur}\\
p=\ell=2,\ |I|=4,\ \operatorname{char}F=2
  & \ref{sec:dyadic-equal}\\
p=\ell=2,\ |I|=4,\ (2a-1,2r-1,2r-1)
  & \ref{sec:dyadic-nonmaximal}\\
p=\ell=2,\ |I|=4,\ (2a-1,2v_F(2),2v_F(2))
  & \ref{sec:dyadic-maximal}.
\end{array}
\]

Finally, Section~\ref{sec:completion} proves that these cases exhaust all
$C_\ell\times C_\ell$ extensions and all conductors occurring in
Theorem~\ref{U:main}, and then proves the theorem.  

\medskip
The results from \cite{Ueda} used here are proved without local
class field theory. Some results in the current paper are also proved
to avoid invoking that theory. For example,
Proposition~\ref{D:norm-separation} and Lemma~\ref{D:quadratic-norms}
establish the distinctness of the norm subgroups required in
Proposition~\ref{U:conjugacy}, rather than deducing it from the local class
field correspondence. Thus the proof of Theorem~\ref{U:main} does
not use local class field theory; the earlier Weil-representation interpretation
is only explanatory.

\section{Local constants and the First Main Lemma}\label{sec:local-notation}\label{U:statement}
\subsection{Characters, conductors, and local constants}
Let $E$ be a nonarchimedean local field.  Write
\[
 \OO_E=\{x\in E:v_E(x)\ge0\},
 \qquad
 \pp_E=\{x\in E:v_E(x)\ge1\},
\]
and let $k_E=\OO_E/\pp_E$.  We normalize $v_E(E^\times)=\mathbf Z$.
Put
\[
 U_E^0=\OO_E^\times,
 \qquad
 U_E^r=1+\pp_E^r\quad(r\ge1).
\]

All quasi-characters are continuous homomorphisms
$E^\times\to\CC^\times$, and all additive characters are continuous
homomorphisms from the additive group of $E$ to $\CC^\times$.

For a quasi-character $\theta$ of $E^\times$, the conductor exponent
$m_E(\theta)$ is the least $m\ge0$ for which $\theta$ is trivial on
$U_E^m$.  Thus $m_E(\theta)=0$ precisely when $\theta$ is unramified.
For a nontrivial additive character $\psi$ of $E$, the integer
$n_E(\psi)$ is characterized by the condition that
$\pp_E^{-n_E(\psi)}$ is the largest fractional ideal on which $\psi$
is trivial.

We recall Langlands's local constant from \cite{Ueda}.
For $z\in\CC$, put
\[
 \ph(z)=
 \begin{cases}
 z/|z|,&z\ne0,\\
 1,&z=0.
 \end{cases}
\]
Let $m=m_E(\theta)$ and $n=n_E(\psi)$, and choose
$\Gamma\in E^\times$ with
\[
 \Gamma\OO_E=\pp_E^{m+n}.
\]
Such an element $\Gamma$ is called \emph{admissible}.  For a positive
Haar measure $du$ on $U_E^0$, define
\begin{equation*}
 \Delta_E(\theta,\psi)=\theta(\Gamma)\,
 \ph\!\left(\int_{U_E^0}\psi(u/\Gamma)\theta(u)^{-1}\,du\right).
\end{equation*}
The integral is nonzero by \cite[Corollary~2.5]{Ueda}.
The value of $\Delta_E(\theta,\psi)$ is independent of the admissible element
$\Gamma$ and the positive normalization of $du$ by
\cite[Lemma~2.7]{Ueda}.
If $\theta$ is unramified and $\varpi_E$ is a uniformizer, then
\begin{equation}\label{U:unramified-local-constant}
 \Delta_E(\theta,\psi)
   =\theta(\varpi_E)^{n_E(\psi)}.
\end{equation}
This is \cite[(10)]{Ueda}.

\subsection{Langlands's First Main Lemma and elementary identities}\label{U:interface}
For a cyclic extension $E/F$ of prime degree and a quasi-character $\chi$
of $F^\times$, the First Main Lemma is
\begin{equation}\label{U:FML}
 \Delta_E(\chi\circ \Nm_{E/F},\psi_F\circ\Tr_{E/F})
       \prod_{\mu\in S(E/F)}\Delta_F(\mu,\psi_F)
   =\prod_{\mu\in S(E/F)}\Delta_F(\chi\mu,\psi_F).
\end{equation}
This is \cite[Theorem~1.1]{Ueda}.

We also have
\begin{equation*}
 \Delta_E(1,\psi)=1,
\end{equation*}
\begin{equation}\label{U:additive-scaling}
 \Delta_E(\theta,\psi(c\,\cdot))=\theta(c)\Delta_E(\theta,\psi)
 \qquad(c\in E^\times),
\end{equation}
and, for a finite-order character $\mu$,
\begin{equation*}
 \Delta_E(\mu,\psi)\Delta_E(\mu^{-1},\psi)=\mu(-1).
\end{equation*}
These are \cite[Lemmas~2.10, 2.12, 2.13]{Ueda}.
If $E/F$ is cyclic of odd prime degree, then
\begin{equation}\label{U:odd-norm-factor}
 \prod_{\mu\in S(E/F)}\Delta_F(\mu,\psi_F)=1.
\end{equation}
This is \cite[Corollary~2.14]{Ueda}.

For $\sigma\in\operatorname{Aut}(E)$, write
$\theta^\sigma(x)=\theta(\sigma^{-1}(x))$. If
$\psi_E\circ\sigma=\psi_E$, then
\begin{equation}\label{U:automorphism-invariance}
 \Delta_E(\theta^\sigma,\psi_E)=\Delta_E(\theta,\psi_E).
\end{equation}
This follows from the definition of $\Delta_E$ by changing variables
by $\sigma$ and using the admissible element $\sigma(\Gamma)$.

\section{Ramification and norm characters}\label{sec:ramification}
\subsection{Trace ideals and additive conductors}

Let $E/F$ be a finite separable extension of nonarchimedean local fields.
Write $e_{E/F}$ for its ramification index, $\mathfrak D_{E/F}$ for
its different ideal, and $D_{E/F}=v_E(\mathfrak D_{E/F})$.
For $j,r\in\mathbf Z$ and a nontrivial additive character $\psi_F$ of $F$,
\begin{equation}\label{U:trace-ideal}
 \begin{aligned}
 \Tr_{E/F}(\pp_E^j)
   &=\pp_F^{\lfloor(j+D_{E/F})/e_{E/F}\rfloor},\\
 n_E(\psi_F\circ\Tr_{E/F})
   &=e_{E/F}n_F(\psi_F)+D_{E/F},
 \end{aligned}
\end{equation}
In particular,
\[
 n_E(\psi_F\circ\Tr_{E/F})=n_F(\psi_F)
 \quad\text{if $E/F$ is unramified},
\]
and, if $E/F$ is tamely and totally ramified of prime degree $\ell$,
\[
 n_E(\psi_F\circ\Tr_{E/F})
   =\ell n_F(\psi_F)+\ell-1.
\]
\begin{equation}\label{U:trace-dual}
 \{x\in E:\Tr_{E/F}(x\pp_E^r)\subset\OO_F\}
   =\pp_E^{-r-D_{E/F}}.
\end{equation}
See \cite[proof of Theorem~3.1(a) and Corollary~3.7]{Ueda}.

\subsection{Norms, norm characters, and conductor pullback}\label{U:norm-characters}\label{U:norm-filtration}
Write $\HerbPsi_{E/F}$ for the inverse Herbrand function. For an
unramified extension $E/F$, one has $\HerbPsi_{E/F}(s)=s$ for $s\ge0$.
Suppose $E/F$ is ramified cyclic of prime degree $\ell$. It is totally
ramified; write $t$ for its lower ramification break. Then
\[
 D_{E/F}=(\ell-1)(t+1),\qquad
 \HerbPsi_{E/F}(s)=
 \begin{cases}
 s,&0\le s\le t,\\
 t+\ell(s-t),&s\ge t.
 \end{cases}
\]
For each integer $r\ge0$, the norm in the ramified case induces
\begin{equation}\label{U:graded-norm}
 U_E^{\HerbPsi_{E/F}(r)}/U_E^{\HerbPsi_{E/F}(r)+1}
 \longrightarrow U_F^r/U_F^{r+1}.
\end{equation}
This map is an isomorphism for $r\ne t$; at $r=t$, its kernel and
cokernel have order $\ell$, by \cite[Theorem~3.1]{Ueda}.
Moreover, \cite[Theorem~3.1(d) and Lemma~3.10(c)]{Ueda} give
\begin{equation}\label{U:unit-norm-surjective}
 \Nm_{E/F}(U_E^{\HerbPsi_{E/F}(r)})=U_F^r
 \quad
 \begin{cases}
 r\ge0,&E/F\text{ unramified},\\
 r>t,&E/F\text{ ramified}.
 \end{cases}
\end{equation}
For ramified $E/F$, let $h\in\mathbf Z$, let $0\le s\le t$ be an integer,
and let $c\in F^\times$ have $v_F(c)=h$.
By \cite[Lemma~6.6]{Ueda}, there is $\widetilde{c}\in E^\times$ such that
\begin{equation}\label{U:norm-approximation}
 v_E(\widetilde{c})=h,\qquad c-\Nm_{E/F}(\widetilde{c})\in\pp_F^{h+s}.
\end{equation}

For a cyclic extension $E/F$ of prime degree $\ell$,
\cite[Corollary~3.11]{Ueda} gives
\begin{equation}\label{U:norm-character-group}
 F^\times/\Nm_{E/F}(E^\times)\simeq C_\ell,\qquad |S(E/F)|=\ell.
\end{equation}
Let $\chi$ be a quasi-character of $F^\times$, and put
\[
 m=m_F(\chi),\qquad m'=m_E(\chi\circ \Nm_{E/F}).
\]
If $E/F$ is unramified, \cite[Lemma~3.10(d)]{Ueda} gives
\begin{equation}\label{U:unramified-pullback-conductor}
 m'=m.
\end{equation}
Suppose that $E/F$ is ramified, with lower break $t$.
By \cite[Corollary~3.8]{Ueda},
\begin{equation}\label{U:norm-character-conductor}
 m_F(\mu)=t+1\qquad(1\ne\mu\in S(E/F)).
\end{equation}
Proposition~3.9 of \cite{Ueda} gives
\begin{equation}\label{U:norm-pullback-conductor}
 m'=\begin{cases}
 m,&m\le t,\\
 1+\HerbPsi_{E/F}(m-1),&m>t+1.
 \end{cases}
\end{equation}
If $m=t+1$, then
\begin{equation}\label{U:break-pullback-conductor}
 m'\le m,\qquad
 m'<m\ \Longleftrightarrow\ m_F(\mu\chi)<m
 \text{ for some }\mu\in S(E/F).
\end{equation}

\subsection{Symmetric coefficients in wild prime degree}
For a finite Galois extension $L/M$ and $x\in L$, write
$e_i^{L/M}(x)$ for the $i$th elementary symmetric function of the
$L/M$-conjugates of $x$; when the extension is clear, write $e_i(x)$.
Suppose now that $E/F$ is wildly ramified cyclic of prime degree
$p=\operatorname{char}k_F$, with lower break $t$. By \cite[Lemma~3.3]{Ueda},
\begin{equation}\label{U:symmetric}
 v_F(e_i^{E/F}(x))\ge
 \left\lfloor\frac{i\,v_E(x)+(p-1)(t+1)}p\right\rfloor
 \qquad(1\le i<p).
\end{equation}

\begin{lemma}[Powers and norms]\label{O:A:lem:power}
Let $L/M$ be totally ramified cyclic of degree $p$, with break $u$.
For every $y\in L$,
\[
 v_L(y^p-\Nm_{L/M}(y))\ge p v_L(y)+(p-1)u.
\]
\end{lemma}
\begin{proof}
For $p=2$, let $\sigma$ be the nontrivial automorphism of $L/M$.
The identity $y^2-\Nm_{L/M}(y)=y(y-\sigma(y))$ and
Lemma~\ref{C:leading-ramification} give the assertion.
For $p>2$ and $1\le i<p$, \eqref{U:symmetric} implies
$pv_M(e_i(y))\ge i v_L(y)+(p-1)u$.
Every intermediate term $e_i(y)y^{p-i}$ of the characteristic polynomial
therefore has valuation at least $p v_L(y)+(p-1)u$. Its constant term is $-\Nm_{L/M}(y)$
because $p$ is odd.
\end{proof}

\section{Lamprecht's formula and comparison of local constants}\label{sec:stationary}
\subsection{Stationary representatives, Lamprecht's formula, and stable twists}\label{U:stationary-setup}
Let $E$ be a nonarchimedean local field, $\theta$ a quasi-character
of $E^\times$ of conductor $m>1$, and $\psi_E$ a nontrivial additive
character of $E$. Choose an admissible element $\Gamma$.
By \cite[Lemma~4.2]{Ueda}, there is a unique class
$\beta\in\OO_E/\pp_E^{\lfloor m/2\rfloor}$ such that every
representative satisfies
\begin{equation}\label{U:stationary-linearization}
 \theta(1+z)=\psi_E(\beta z/\Gamma)
 \qquad(z\in\pp_E^{\lceil m/2\rceil}).
\end{equation}
Every representative is a unit; we call it a
\emph{stationary representative}.
For any stationary representative $\beta$ and any quasi-character
$\nu$ of $E^\times$, \cite[Proposition~4.7]{Ueda} gives
\begin{equation}\label{U:stable}
 \Delta_E(\theta\nu,\psi_E)=\nu(\Gamma/\beta)\Delta_E(\theta,\psi_E)
 \qquad\bigl(m_E(\nu)\le\lfloor m/2\rfloor\bigr).
\end{equation}

\begin{lemma}[Lamprecht's formula with a scaled additive character]
\label{U:stationary-factor}
Let $E$ be a nonarchimedean local field, let
$\theta:E^\times\to\CC^\times$ be a quasi-character of conductor $m\ge2$,
and let $\psi_E$ be a nontrivial additive character of $E$.
For $\alpha\in E^\times$, put $\Psi_E(x)=\psi_E(\alpha x)$ and let
$\pp_E^J$ be its largest trivial ideal. Let $Z\in E^\times$ satisfy
$v_E(Z)=J-m$ and
$\theta(1-z)=\Psi_E(Zz)$ for $z\in\pp_E^{\ceil{m/2}}$.
Then
\begin{equation}\label{U:normalized-factor}
 \Delta_E(\theta,\psi_E)=
 \theta(( -\alpha Z)^{-1})\Psi_E(-Z)g_E.
\end{equation}
For even $m$, $g_E=1$. For odd $m$, choose a \emph{critical element}
$\delta\in E$ with $v_E(\delta)=\floor{m/2}$.
For $\bar x\in k_E$, let $x\in\OO_E$ be a lift, and set
\begin{equation}\label{U:normalized-residual}
 H_E(\bar x)=\Psi_E(-Z\delta x)\theta(1+\delta x)^{-1},\qquad
 g_E=|k_E|^{-1/2}\sum_{x\in k_E}H_E(x).
\end{equation}
We call $H_E$ the \emph{critical function}. It is well defined and unimodular.
The value $g_E$ is independent of the chosen critical element $\delta$.
We write $g_E(\theta,Z)$ when specifying the character and coefficient.

Put $\psi_0(\bar z)=\Psi_E(-Z\delta^2z)$ for any lift $z$ of
$\bar z\in k_E$. The polar pairing of $H_E$ is
\begin{equation}\label{U:critical-polar}
 \frac{H_E(x+y)}{H_E(x)H_E(y)}
       =\psi_0(xy)\qquad(x,y\in k_E).
\end{equation}
This pairing is nondegenerate, and
\begin{equation}\label{U:critical-sum-size}
 \left|\sum_{x\in k_E}H_E(x)\right|=|k_E|^{1/2},\qquad |g_E|=1.
\end{equation}

\end{lemma}
\begin{proof}
Apply \cite[Theorem~4.5 and Lemma~4.6]{Ueda} with
$\beta=-\alpha Z\Gamma$.
\end{proof}

\subsection{Character identities for compatible families}\label{U:common-principles}

\begin{proposition}[Conjugate characters and restriction to $F^\times$]
\label{U:determinant}\label{U:conjugacy}
Let $\ell$ be prime, let $K/F$ be finite Galois with Galois group $C_\ell^2$,
and let $L_1,L_2$ be distinct degree-$\ell$ intermediate fields
with distinct subgroups $\Nm_{L_i/F}(L_i^\times)$ of $F^\times$. Put
$\epsilon_i=\prod_{\omega\in S(L_i/F)}\omega$.
Let $\theta_i$ be primitive compatible characters of $L_i^\times$,
with common norm pullback $\Theta$.

For a generator $\sigma_i$ of $\Gal(L_i/F)$, the quotient
$\theta_i^{\sigma_i}/\theta_i$ generates the group $S(K/L_i)$.
Every character in the group is fixed by $\Gal(L_i/F)$, thus the conjugates of $\theta_i$
are precisely its twists by $S(K/L_i)$. 

Moreover,
\begin{equation}\label{U:det-character}
 D_\theta=\theta_i|_{F^\times}\epsilon_i
\end{equation}
is independent of $i$. For odd $\ell$, $\epsilon_i=1$; for $\ell=2$,
$\epsilon_i=\omega_i$ is the nontrivial member of $S(L_i/F)$.
\end{proposition}
\begin{proof}
Put
$\mu_i=\theta_i^{\sigma_i}/\theta_i$.
The character $\Theta$ is invariant under both $\Gal(K/L_i)$,
which generate $\Gal(K/F)$. Let $\widetilde\sigma_i\in\Gal(K/F)$
restrict to $\sigma_i^{-1}$ on $L_i$. Then for $x\in K^\times$,
\[
 \theta_i^{\sigma_i}(\Nm_{K/L_i}(x))
 =\Theta(\widetilde\sigma_i x)
 =\Theta(x)
 =\theta_i(\Nm_{K/L_i}(x)).
\]
Thus $\mu_i$ is trivial on $\Nm_{K/L_i}(K^\times)$. If $\mu_i=1$, Lemmas~\ref{C:hilbert90}
and~\ref{C:character-extension} would give a quasi-character $\lambda$
of $F^\times$ with $\theta_i=\lambda\circ \Nm_{L_i/F}$. This would imply
$\Theta=\lambda\circ \Nm_{K/F}$, contrary to primitivity. Since
$L_i^\times/\Nm_{K/L_i}(K^\times)$ has prime order, $\mu_i$ generates
$S(K/L_i)$. By Lemma~\ref{D:crossed-norm}, every character in $S(K/L_i)$
is of the form $\omega\circ \Nm_{L_i/F}$ with $\omega\in S(L_j/F)$. Since
$\Nm_{L_i/F}\circ\sigma_i=\Nm_{L_i/F}$, every such character is fixed by $\sigma_i$.
Thus, from $\theta_i^{\sigma_i}=\theta_i\mu_i$, induction gives
\[
 \theta_i^{\sigma_i^j}=\theta_i\mu_i^j\qquad(0\le j<\ell).
\]
As $\mu_i$ generates $S(K/L_i)$, the conjugates of $\theta_i$ are
precisely its twists by $S(K/L_i)$.

For $x\in L_j^\times$, compatibility gives
\[
 \theta_i(\Nm_{L_j/F}(x))=\Theta(x)=\theta_j(x)^\ell.
\]
Since $\theta_j(\sigma_j^r(x))=\theta_j^{\sigma_j^{-r}}(x)
=\theta_j(x)\mu_j(x)^{-r}$, multiplication over $r$ gives
\[
 \theta_j(\Nm_{L_j/F}(x))=\theta_j(x)^\ell
       \mu_j(x)^{-\ell(\ell-1)/2}.
\]
For odd $\ell$ the last factor is one, and the characters in each
$S(L_i/F)$ pair with their inverses, so $\epsilon_i=1$.
For $\ell=2$, that factor is
$\mu_j(x)=\omega_i(\Nm_{L_j/F}(x))$ and $\omega_j(\Nm_{L_j/F}(x))=1$.
In either case $\theta_i|_{F^\times}\epsilon_i$ and
$\theta_j|_{F^\times}\epsilon_j$ agree on $\Nm_{L_j/F}(L_j^\times)$,
and by symmetry on $\Nm_{L_i/F}(L_i^\times)$. These distinct index-$\ell$
subgroups generate $F^\times$, proving the result.
\end{proof}

For a primitive compatible family, two inducing choices on $L_i^\times$
with the same norm pullback differ by an element of $S(K/L_i)$.
They are therefore conjugate and have equal local constants by
\eqref{U:automorphism-invariance}.

\begin{corollary}[Odd-degree powers]\label{U:odd-power}
Under the hypotheses of Proposition~\ref{U:conjugacy}, suppose $\ell$ is odd. Then
\[
 \Delta_{L_i}(\theta_i,\psi_{L_i})^\ell=\Delta_K(\Theta,\psi_K).
\]
In particular,
\[
 \left(\frac{\mathcal A_{L_1}(\theta_1)}
              {\mathcal A_{L_2}(\theta_2)}\right)^\ell=1.
\]
\end{corollary}
\begin{proof}
The norm-character products for $L_i/F$ and $K/L_i$ are $1$ by
\eqref{U:odd-norm-factor}.

Apply \eqref{U:FML} to $K/L_i$. By Proposition~\ref{U:conjugacy}, the twists
of $\theta_i$ are conjugate, and their local constants agree by
\eqref{U:automorphism-invariance}. This gives the power relation.
Use \eqref{U:odd-norm-factor} for $L_i/F$ to obtain the ratio identity.
\end{proof}

\begin{proposition}[Squares in the biquadratic case]\label{U:quadratic-square}
Let $K/F$ be biquadratic with quadratic intermediate fields
$L_1,L_2,L_3$, and let $\omega_i$ be the nontrivial character in
$S(L_i/F)$. For a primitive compatible family $(\theta_i)$,
\[
 \mathcal A_{L_i}(\theta_i)^2
 =\Delta_K(\Theta)\prod_{j=1}^3\Delta_F(\omega_j)
 \qquad(i=1,2,3).
\]
Hence, for $i\ne j$,
\[
 \frac{\mathcal A_{L_i}(\theta_i)}{\mathcal A_{L_j}(\theta_j)}
 \in\{1,-1\}.
\]
\end{proposition}
\begin{proof}
Lemma~\ref{D:quadratic-norms} gives pairwise distinct norm subgroups,
so Proposition~\ref{U:conjugacy} applies to any pair of intermediate fields. Thus, by that proposition
and \eqref{U:FML} for $K/L_i$,
$\Delta_{L_i}(\theta_i)^2=\Delta_K(\Theta)\Delta_{L_i}(\omega_j\circ \Nm_{L_i/F})$
for $j\ne i$.

Equation \eqref{U:FML} for $L_i/F$, applied to $\omega_j$, gives
$\Delta_{L_i}(\omega_j\circ \Nm_{L_i/F})\Delta_F(\omega_i)
 =\Delta_F(\omega_j)\Delta_F(\omega_i\omega_j)$.
Multiply this with the preceding equality.
\end{proof}

\subsection{Quadratic norm representatives and critical sums}

\begin{lemma}[Quadratic trace--norm conversion]\label{U:quadratic-conversion}
Let $E/F$ be wildly ramified quadratic with lower break $t$, put $T=t+1$,
and let $\omega$ be the nontrivial character in $S(E/F)$.
There is a nontrivial additive character $\Psi_F$ of $F$ with largest
trivial ideal $\pp_F^T$ such that
\[
 \omega(1-z)=\Psi_F(z)\qquad(z\in\pp_F^{\ceil{T/2}}).
\]
For any such character, put $\Psi_E=\Psi_F\circ\Tr_{E/F}$. Then
\begin{equation}\label{D:UR:phase}
 \Psi_E(z)=\Psi_F(\Nm_{E/F}(z))
 \qquad(z\in\pp_E^{\ceil{T/2}}).
\end{equation}
\end{lemma}
\begin{proof}
The character $\Psi_F$ exists by \eqref{U:stationary-linearization}.
For $z\in\pp_E^{\ceil{T/2}}$, the trace-ideal formula
\eqref{U:trace-ideal} and total ramification put both
$\Tr_{E/F}(z)$ and $\Nm_{E/F}(z)$ in $\pp_F^{\ceil{T/2}}$.
The quadratic norm identity and triviality of $\omega$ on norms give
\[
 1=\omega(\Nm_{E/F}(1-z))
  =\Psi_F(\Tr_{E/F}(z)-\Nm_{E/F}(z)),
\]
which proves \eqref{D:UR:phase}.
\end{proof}

\begin{lemma}[The second elementary symmetric function]\label{U:e2}
For a biquadratic extension $K/F$, $C\in K$, and any quadratic
intermediate field $L$,
\[
 e_2(C)=\Tr_{L/F}(\Nm_{K/L}(C))+\Nm_{L/F}(\Tr_{K/L}(C)),
\]
where $e_2(C)$ is the second elementary symmetric function of the four
conjugates of $C$.
\end{lemma}
\begin{proof}
Write the four conjugates, paired by the involution fixing $L$, as
$c_1,c_2$ and $c_3,c_4$. The two terms on the right are
$c_1c_2+c_3c_4$ and $(c_1+c_2)(c_3+c_4)$.
Adding gives the six terms defining $e_2(C)$.
\end{proof}

For the next two lemmas, let $K/F$ be biquadratic with quadratic
intermediate fields $L_i$. Let $\theta_i$ be compatible characters of
$L_i^\times$ with common norm pullback $\Theta$, and let $\omega_i$ be
the nontrivial character in $S(L_i/F)$. Fix $\alpha\in F^\times$ and put
$\Psi_E(x)=\psi_E(\alpha x)$.

\begin{lemma}[Norms as stationary coefficients]\label{U:common-origin}
Assume in addition that the compatible family $(\theta_i)$ is primitive.
By Proposition~\ref{U:determinant} and Lemma~\ref{D:quadratic-norms},
$D_\theta=\theta_i|_{F^\times}\omega_i$ is independent of $i$.
Put $m_i=m_{L_i}(\theta_i)$ and $m_i^\omega=m_F(\omega_i)$, and assume
$m_i,m_i^\omega>1$. Let $C\in K^\times$ satisfy $\Tr_{K/L_i}(C)\ne0$,
and suppose
\[
 \begin{aligned}
 \theta_i(1-z)&=\Psi_{L_i}(\Nm_{K/L_i}(C)\,z)
 &&(z\in\pp_{L_i}^{\lceil m_i/2\rceil}),\\
 \omega_i(1-z)&=\Psi_F\bigl(\Nm_{L_i/F}(\Tr_{K/L_i}(C))z\bigr)
 &&(z\in\pp_F^{\lceil m_i^\omega/2\rceil}).
 \end{aligned}
\]
Let $g_{\theta_i}$ and $g_{\omega_i}$ be the normalized sums in
\eqref{U:normalized-residual} for $\theta_i$ and $\omega_i$ with these
coefficients, respectively, with value one in even conductor. Then
\[
 \mathcal A_{L_i}(\theta_i)
   =D_\theta(-\alpha^{-1})\Theta(C)^{-1}\Psi_F(-e_2(C))\,g_{\theta_i}g_{\omega_i}.
\]
\end{lemma}
\begin{proof}
Multiply \eqref{U:normalized-factor} for $\theta_i$ and $\omega_i$.
Compatibility, Proposition~\ref{U:determinant}, and Lemma~\ref{U:e2} give the formula.
\end{proof}

\begin{lemma}[A common character product]\label{U:common-function}
Retain the standing notation preceding Lemma~\ref{U:common-origin}.
Let $C,C'\in K^\times$ have nonzero traces
to every $L_i$, and put
\[
 \begin{aligned}
 1+z_i&=\frac{\Nm_{K/L_i}(C')}{\Nm_{K/L_i}(C)},&
 1+w_i&=\frac{\Nm_{L_i/F}(\Tr_{K/L_i}(C'))}
                 {\Nm_{L_i/F}(\Tr_{K/L_i}(C))}.
 \end{aligned}
\]
Then the products
\[
 \Psi_{L_i}\bigl(-\Nm_{K/L_i}(C)z_i\bigr)\theta_i(1+z_i)^{-1}
 \Psi_F\bigl(-\Nm_{L_i/F}(\Tr_{K/L_i}(C))w_i\bigr)\omega_i(1+w_i)^{-1}
\]
are independent of $i$. Their common value is
\[
 \Theta(C'/C)^{-1}\Psi_F\bigl(-e_2(C')+e_2(C)\bigr).
\]
\end{lemma}
\begin{proof}
Compatibility gives $\theta_i(1+z_i)=\Theta(C'/C)$. Also
$\omega_i(1+w_i)=1$.
The additive exponent is
\[
 -\Tr_{L_i/F}\bigl(\Nm_{K/L_i}(C')-\Nm_{K/L_i}(C)\bigr)
 -\Nm_{L_i/F}(\Tr_{K/L_i}(C'))
 +\Nm_{L_i/F}(\Tr_{K/L_i}(C)).
\]
which is $-e_2(C')+e_2(C)$ by Lemma~\ref{U:e2}.
\end{proof}

\begin{proposition}[Cancellation from a common pullback]\label{U:missing-coset}
Let $I$ be an index set with at least two elements. For $i\in I$,
let $V_i$ be finite abelian groups of the same order,
and let $H_i:V_i\to\mathbf C^\times$ satisfy $H_i(0)=1$. Suppose
\[
 (x,y)\longmapsto\frac{H_i(x+y)}{H_i(x)H_i(y)}
\]
is a nondegenerate bicharacter. Let $Z$ be a finite abelian group and let
$p_i:Z\to V_i$ be group homomorphisms whose kernels have order two and
whose images have index two. Suppose
\[
 H_i\circ p_i=Q
\]
for the same function $Q$ on $Z$, and suppose the kernels of the $p_i$
are pairwise distinct. Then, for every $i$,
\[
 \sum_{x\in V_i}H_i(x)=\frac12\sum_{z\in Z}Q(z).
\]
In particular the sums divided by $|V_i|^{1/2}$ are equal.
\end{proposition}
\begin{proof}
Write $B_i$ for the polar pairing of $H_i$.
Put $I_i=p_i(Z)$. Fix $i$ and choose $j\ne i$. Let $a$ be the nonzero element of
$\ker p_j$. Since the kernels are distinct and both have order two,
$a\notin\ker p_i$. Put $w_i=p_i(a)$. Then $w_i\ne0$ and
$H_i(w_i)=Q(a)=H_j(0)=1$.

If $x=p_i(z)\in I_i$, then
\[
 H_i(w_i+x)=Q(a+z)=H_j(p_j(a+z))=H_j(p_j(z))=Q(z)=H_i(x).
\]
Hence $B_i(w_i,x)=1$ for every $x\in I_i$. Since $w_i\ne0$ and the
pairing is nondegenerate, $B_i(w_i,\cdot)$ is nontrivial on $V_i$.
It is therefore the nontrivial character of $V_i/I_i$.

Since $w_i\in I_i$, translation $y\mapsto y+w_i$ preserves the
nontrivial coset $V_i\setminus I_i$. For $y\in V_i\setminus I_i$,
one has $B_i(w_i,y)=-1$ and $H_i(w_i)=1$, hence
$H_i(y+w_i)=-H_i(y)$. Therefore the sum of $H_i$ over
$V_i\setminus I_i$ is zero. On $I_i$,
every element has exactly two preimages under $p_i$. Summing
$H_i p_i=Q$ gives the stated identity.

In fact, the proof uses only that, for every $w\in V_i$, the map
$x\mapsto B_i(w,x)$ is a character of $V_i$ and is nontrivial when
$w\ne0$; multiplicativity in the first variable is not used.
\end{proof}

\section{Finite-field identities and the tame comparison}\label{O:sec:tame}\label{sec:tame}
\subsection{Finite-field Gauss-sum identities}
For a finite field $k$, a nontrivial additive character
$\psi_k:k\to\C^\times$, and a multiplicative character
$\chi:k^\times\to\C^\times$, put
\[
 \tau_k(\chi)=-\sum_{x\in k^\times}\chi(x)^{-1}\psi_k(x).
\]
Thus $\tau_k(1)=1$. For every finite extension $\kappa/k$, take
$\psi_\kappa=\psi_k\circ\Tr_{\kappa/k}$. The Hasse--Davenport lifting
formula \cite[Theorem~5.2]{Ueda} gives
\begin{equation*}
 \tau_\kappa(\chi\circ \Nm_{\kappa/k})=\tau_k(\chi)^{[\kappa:k]}.
\end{equation*}

Let $s$ be a positive divisor of $|k^\times|$, and let
$\eta:k^\times\to\C^\times$ be a multiplicative character of order $s$.
The Hasse--Davenport multiplication formula \cite[Theorem~5.3]{Ueda} is
\begin{equation*}
 \chi(s^s)\tau_k(\chi^s)\prod_{j=1}^{s-1}\tau_k(\eta^j)
 =\prod_{j=0}^{s-1}\tau_k(\chi\eta^j).
\end{equation*}

\begin{proposition}[Gauss-sum comparison in prime degree]\label{O:F:finite}
Let $\ell$ be prime, $\ell\mid q-1$, and let $\kappa/k$ have degree
$\ell$, where $|k|=q$. Let $\chi_\kappa$ and $\chi_0$ be multiplicative
characters of $\kappa^\times$ and $k^\times$, respectively, satisfying
\begin{equation*}
 \chi_\kappa^\ell=\chi_0\circ\Nm_{\kappa/k},
 \qquad \chi_0|_{\mu_\ell}\ne1.
\end{equation*}
Let $\mu:k^\times\to\C^\times$ be any multiplicative character of order $\ell$.
Choose a nontrivial additive character $\psi_k:k\to\C^\times$, and put
$\psi_\kappa=\psi_k\circ\Tr_{\kappa/k}$. Then
\begin{equation}\label{O:F:finiteidentity}
 \tau_\kappa(\chi_\kappa)
 =\chi_0(\ell)\tau_k(\chi_0)
       \prod_{j=1}^{\ell-1}\tau_k(\mu^j).
\end{equation}
\end{proposition}
\begin{proof}
Put $q=p^f$, where $p=\operatorname{char}k$. Use the coefficient-field construction of
Section~\ref{A:coefficient-field} with residue field $\kappa$,
and let $\zeta$ be the chosen primitive $p$th root of unity and
$\varpi=\zeta-1$. First take $\psi_k(x)=\zeta^{\Tr_{k/\F_p}(x)}$.
Let $\omega_\kappa$ be the Teichm\"uller character of $\kappa^\times$,
and put $\omega_k=\omega_\kappa|_{k^\times}$. Put
$S=(q^\ell-1)/(q-1)$ and $h=(q-1)/\ell$.
Write $\chi_0=\omega_k^a$, $1\le a\le q-2$, $\ell\nmid a$.
The exponent of $\chi_\kappa$ is $aS/\ell$ modulo
$(q^\ell-1)/\ell$. Multiplication by $q$ changes it by
$a(q^\ell-1)/\ell$; hence all $\ell$ possibilities are Frobenius
conjugates. Their Gauss sums are equal, so we may take the exponent
exactly $a_*=aS/\ell$.

Set
\[
 L=\frac{\tau_\kappa(\chi_\kappa)}
 {\chi_0(\ell)\tau_k(\chi_0)\prod_{j=1}^{\ell-1}\tau_k(\mu^j)}.
\]
The order-$\ell$ characters over $\kappa$ are the norm lifts of
$\mu^j$. Their twists of $\chi_\kappa$ are the Frobenius conjugates
just considered. The Hasse--Davenport lifting and multiplication formulas above, together with
$\chi_\kappa(\ell^\ell)=\chi_0(\ell)^\ell$, give
$L^\ell=1$.

To determine $L$, we use Corollary~\ref{A:grouped-gauss}.
Write $a=\ell m+r_0$, $1\le r_0<\ell$.
The unordered base-$q$ digits of $aS/\ell$ are
\begin{equation*}
 \gamma_j=m+jh+\mathbf1_{j\ge\ell-r_0},\qquad0\le j\le\ell-1.
\end{equation*}
Here ``unordered'' is sufficient for their sum and factorial product.
Multiplication of $r_0S/\ell$ by $\ell$, using
$q=\ell h+1$, gives successive residues
$i_j=((\ell-1-j)r_0\bmod\ell)$ and carries
$c_j=\lfloor(\ell-1-j)r_0/\ell\rfloor$.
The zeroth digit has the additional $r_0-c_0=1$; at a later digit the
additional $c_{j-1}-c_j$ is one exactly when $i_j\ge\ell-r_0$.
The $i_j$ permute $0,\ldots,\ell-1$, proving the equation above.
All digits lie between zero and $q-1$, since $m\le h-1$.
In particular
$\sum\gamma_j=a+\sum_{j=1}^{\ell-1}jh$.

We claim the multiplicative factorial congruence
\begin{equation*}
 \ell^a\prod_{j=0}^{\ell-1}\gamma_j!
 \equiv_\times a!\prod_{j=1}^{\ell-1}(jh)!.
\end{equation*}
In fact
\[
 \frac{\prod\gamma_j!}{\prod(jh)!}
 =\prod_{j=0}^{\ell-1}\prod_{n=1}^{m}(n+jh)
      \prod_{j=\ell-r_0}^{\ell-1}(m+1+jh).
\]
For the first product,
$\ell(n+jh)=\ell n-j+jq\equiv_\times\ell n-j$:
the positive integer $\ell n-j$ is less than $q=p^f$, so its
$p$-valuation is less than $f$. The numbers $\ell n-j$ exhaust
$1,\ldots,\ell m$. For the second product the same argument applies
to $\ell(m+1)-j$, which exhausts $\ell m+1,\ldots,a$, again all
less than $q$. Multiplication proves the claimed factorial congruence, even when
some of its factorials are divisible by $p$.

Apply Corollary~\ref{A:grouped-gauss} with $r=\ell$ to the numerator
of $L$, and with $r=1$ to its Gauss-sum factors in the denominator. The Teichm\"uller value $\chi_0(\ell)$ reduces to
$\ell^a$. The digit-sum identity and factorial congruence above give
$L\equiv_\times1$. Reduction is injective on the group of
$\ell$th roots of unity because $\ell\ne p$.
Together with $L^\ell=1$, this proves $L=1$.

For an arbitrary additive character of $k$, change the canonical one
by multiplication by $c\in k^\times$. The two sides acquire factors
$\chi_\kappa(c)$ and $\chi_0(c)\prod_{j=1}^{\ell-1}\mu^j(c)$.
These agree. Indeed the restriction exponent of $\chi_\kappa$ is
$aS/\ell$ modulo $q-1$: it is $a$ for odd $\ell$, and $a+(q-1)/2$
for $\ell=2$ because $a$ is odd. This is exactly the exponent of
$\chi_0\prod\mu^j$. This proves \eqref{O:F:finiteidentity} for every nontrivial additive character.
\end{proof}

\subsection{Tame comparison}\label{sec:tame-comparison}
Let $F$ have residue characteristic $p$, and let $K/F$ be a tame Galois
extension with Galois group $C_\ell^2$, where $\ell\ne p$. Let $U/F$ be its unique
unramified degree-$\ell$ subextension, and let $E$ be any other
degree-$\ell$ intermediate field. Then $E/F$ is totally tamely ramified, so
its cyclic inertia gives $\ell\mid |k_F|-1$; $K/E$ is unramified and
$K/U$ is totally ramified.
Let $\theta_U$ and $\theta_E$ be characters of $U^\times$ and
$E^\times$ with a common invariant norm pullback that does not descend
from $F$. With $S(L/F)$ as defined in
Section~\ref{sec:introduction}, put
\[
 \epsilon_L=\prod_{\mu\in S(L/F)}\mu,\quad
 \Lambda_L=\prod_{\mu\in S(L/F)}\Delta_F(\mu,\psi_F).
\]
For odd $\ell$, $\epsilon_L=1$; for $\ell=2$ it is the nontrivial
quadratic norm character.

\begin{lemma}[Stationary coefficients under tame or unramified norm pullback]\label{U:tame-covector}
Let $L/F$ be cyclic of prime degree $\ell$, and let
$\lambda:F^\times\to\C^\times$ be a quasi-character of conductor
$m>1$. Suppose $b_F\in F^\times$ satisfies
\[
 \lambda(1+z)=\psi_F(b_F z)\qquad
 (z\in\pp_F^{\lceil m/2\rceil}).
\]
If $L/F$ is unramified, then
\[
 (\lambda\circ\Nm_{L/F})(1+z)=\psi_L(b_F z)
 \qquad(z\in\pp_L^{\lceil m/2\rceil}).
\]
If $L/F$ is tamely and totally ramified, then the same identity holds for
\[
 z\in\pp_L^{\ceil{(1+\ell(m-1))/2}}.
\]
\end{lemma}
\begin{proof}
Choose an admissible $\Gamma$ for $\lambda$; then $\Gamma b_F$ is a
stationary representative. Proposition~6.8(a) of \cite{Ueda} gives
the unramified assertion. In the tamely ramified case, Proposition~6.8(b)
applies when $\ell=2$, and Proposition~6.8(c) applies when $\ell$ is
odd (there $T=1$ and $\lfloor m/2\rfloor\ge1$).
\end{proof}

\begin{theorem}[Tame comparison]\label{O:F:tame}
Theorem~\ref{U:main} holds when $\ell\ne p$.
\end{theorem}
\begin{proof}
Write $\Theta=\theta_U\circ \Nm_{K/U}=\theta_E\circ \Nm_{K/E}$, and let
$\sigma$ generate $\Gal(U/F)$. The norm subgroups for $U/F$ and
$E/F$ are distinct by Proposition~\ref{D:norm-separation}.
Proposition~\ref{U:conjugacy} therefore shows that
$\nu=\theta_U^\sigma/\theta_U$ generates $S(K/U)$ and that every
twist of $\theta_U$ by $S(K/U)$ has conductor
$m=m_U(\theta_U)$. The break of $K/U$ is zero, so
\eqref{U:norm-character-conductor} gives $m_U(\nu)=1$.
Then $m\ge1$ and
$\theta_U^\sigma|_{U_U^1}=\theta_U|_{U_U^1}$.

The conductor-pullback statement of
Section~\ref{U:norm-characters}, applied to $K/U$, gives
$m_K(\Theta)=1+\ell(m-1)$.
Indeed, for $m>1$ this is \eqref{U:norm-pullback-conductor} with
break zero, and for $m=1$ the endpoint clause gives $m_K(\Theta)=1$
because all twists of $\theta_U$ by $S(K/U)$ have conductor one.
Since $K/E$ is unramified, Section~\ref{U:norm-characters} gives
$m_E(\theta_E)=m_K(\Theta)=1+\ell(m-1)$.

Suppose $m>1$. We first construct a quasi-character of $F^\times$ whose
norm pullback agrees with $\theta_U$ on $U_U^1$.
If $u\in U_U^1$ has $\Nm_{U/F}(u)=1$, Lemma~\ref{C:hilbert90} gives
$u=z/\sigma(z)$. Multiplying $z$ by a power of a uniformizer of $F$
makes $z$ a unit. Since $u\equiv1\pmod{\pp_U}$, the residue of $z$
is fixed by $\sigma$ and lies in $k_F^\times$. Dividing $z$ by a lift
of this residue in $\OO_F^\times$ then makes $z\in U_U^1$. The $\sigma$-invariance above gives
$\theta_U(u)=1$.

By \eqref{U:unit-norm-surjective}, $\Nm_{U/F}(U_U^r)=U_F^r$ for
$r\ge1$. Hence
\[
 \lambda_1:U_F^1\longrightarrow\C^\times,\qquad
 \lambda_1(\Nm_{U/F}(u))=\theta_U(u)\quad(u\in U_U^1)
\]
is well defined, trivial on $U_F^m$, and nontrivial on $U_F^{m-1}$.
Lemma~\ref{C:character-extension} extends it to a quasi-character
$\lambda:F^\times\to\C^\times$ of conductor $m$.

Define characters $\chi_L:L^\times\to\C^\times$ by
\[
 \chi_L=\theta_L/(\lambda\circ \Nm_{L/F})\qquad(L=U,E).
\]
Norm transitivity gives
\[
 \chi_U\circ \Nm_{K/U}=\chi_E\circ \Nm_{K/E}
       =\Theta/(\lambda\circ \Nm_{K/F}).
\]
This character is $\Gal(K/F)$-invariant. If it descended from $F$,
then so would $\Theta$, so $(\chi_U,\chi_E)$ is primitive and compatible.

By construction $\chi_U|_{U_U^1}=1$, while
$\chi_U^\sigma/\chi_U=\nu\ne1$. An unramified character is
$\sigma$-invariant, so $m_U(\chi_U)=1$. Applying the conductor
relation above to this pair gives $m_E(\chi_E)=1$.

Choose an admissible $\Gamma$ for $(\lambda,\psi_F)$ and a stationary
representative $\beta$ as in \eqref{U:stationary-linearization}, and put
$b_F=\beta/\Gamma$. The quasi-characters $\lambda\circ \Nm_{U/F}$ and
$\lambda\circ \Nm_{E/F}$ have conductors $m$ and $1+\ell(m-1)$,
respectively, both at least two. Lemma~\ref{U:tame-covector} gives
stationary coefficient $b_F$ for each. Since $m_L(\chi_L)=1$,
\eqref{U:stable} gives, for $L=U,E$,
\begin{equation*}
 \Delta_L(\theta_L)=\chi_L(b_F^{-1})\Delta_L(\lambda\circ\Nm_{L/F}).
\end{equation*}
Every $\mu\in S(L/F)$ has conductor at most
$1\le\lfloor m/2\rfloor$. Equations \eqref{U:FML} and \eqref{U:stable}
therefore give
\begin{equation*}
 \Delta_L(\lambda\circ\Nm_{L/F})\Lambda_L
 =\prod_{\mu\in S(L/F)}\Delta_F(\mu\lambda)
 =\epsilon_L(b_F^{-1})\Delta_F(\lambda)^\ell.
\end{equation*}
Proposition~\ref{U:determinant}, applied to $(\chi_U,\chi_E)$, identifies
$\chi_U|_{F^\times}\epsilon_U$ with $\chi_E|_{F^\times}\epsilon_E$.
Together with the preceding two local-constant identities, this gives
\begin{equation}\label{O:F:localidentity}
 \Delta_U(\theta_U)\Lambda_U=\Delta_E(\theta_E)\Lambda_E.
\end{equation}

Suppose now that $m=1$, so $m_E(\theta_E)=1$ as well.
Put $k=k_F=k_E$ and $\kappa=k_U=k_K$. Let
$\chi_\kappa:\kappa^\times\to\C^\times$ and
$\chi_0:k^\times\to\C^\times$ be induced by the restrictions of
$\theta_U$ and $\theta_E$ to $\OO_U^\times$ and $\OO_E^\times$,
respectively. Reducing the norm-compatibility identity gives
$\chi_\kappa^\ell=\chi_0\circ \Nm_{\kappa/k}$.
The residue character induced by $\nu$ is nontrivial of order $\ell$.
Put $q=|k|$, choose a Teichm\"uller character $\omega_\kappa$ of
$\kappa^\times$, and put $\omega_k=\omega_\kappa|_{k^\times}$.
Write $\chi_0=\omega_k^a$ and $S=(q^\ell-1)/(q-1)$. If
$\chi_\kappa=\omega_\kappa^b$, then
$\ell b\equiv aS\pmod{q^\ell-1}$, so
$(q-1)b\equiv a(q^\ell-1)/\ell$. Thus the Frobenius quotient of
$\chi_\kappa$ is nontrivial exactly when $\ell\nmid a$, equivalently
$\chi_0|_{\mu_\ell}\ne1$.

Write $n=n_F(\psi_F)$, choose a uniformizer $\varpi_F$ of $F$, and put
$\Gamma=\varpi_F^{n+1}$. Equation~\eqref{U:trace-ideal} gives
\[
 n_U(\psi_U)=n,\qquad
 n_E(\psi_E)=\ell n+\ell-1.
\]
Now $U/F$ is unramified, so
$v_U(\Gamma)=v_F(\Gamma)=n+1$. Since $E/F$ is totally ramified of
degree $\ell$,
\[
 v_E(\Gamma)=\ell v_F(\Gamma)=\ell(n+1)
             =1+n_E(\psi_E).
\]
Thus $\Gamma$ is admissible for both local constants. Define $\psi_k(\bar x)=\psi_F(x/\Gamma)$ using a lift
$x\in\OO_F$, and put $\psi_\kappa=\psi_k\circ\Tr_{\kappa/k}$.
The finite-sum formula \cite[Proposition~2.4]{Ueda} gives
\begin{align*}
 \Delta_U(\theta_U)&=-\theta_U(\Gamma)\ph\!\left(\tau_\kappa(\chi_\kappa)\right),\\
 \Delta_E(\theta_E)&=-\theta_E(\Gamma)\chi_0(\ell)
                                   \ph\!\left(\tau_k(\chi_0)\right).
\end{align*}
The factor $\chi_0(\ell)$ occurs because the residue additive character
for $E$ is $x\mapsto\psi_k(\ell x)$.
Choose a character $\mu:k^\times\to\C^\times$ of order $\ell$.
The characters in $S(U/F)$ are unramified and
$\epsilon_U(\varpi_F)=(-1)^{\ell-1}$. Thus
\eqref{U:unramified-local-constant} gives
\[
 \Lambda_U=\epsilon_U(\varpi_F)^n=(-1)^{(\ell-1)n}.
\]
The nontrivial characters in $S(E/F)$ have conductor one, with residue
restrictions $\mu,\ldots,\mu^{\ell-1}$, and hence
\[
 \Lambda_E=(-1)^{\ell-1}\epsilon_E(\Gamma)
                 \prod_{j=1}^{\ell-1}\ph\!\left(\tau_k(\mu^j)\right).
\]
By Proposition~\ref{U:determinant},
\[
 \frac{\theta_U(\Gamma)}{\theta_E(\Gamma)}
 =\frac{\epsilon_E(\Gamma)}{\epsilon_U(\Gamma)},\qquad
 \epsilon_U(\Gamma)=(-1)^{(\ell-1)(n+1)}.
\]
Substituting these identities gives
\[
 \frac{\Delta_U(\theta_U)\Lambda_U}
      {\Delta_E(\theta_E)\Lambda_E}
 =\ph\!\left(
 \frac{\tau_\kappa(\chi_\kappa)}
 {\chi_0(\ell)\tau_k(\chi_0)\prod_{j=1}^{\ell-1}\tau_k(\mu^j)}
 \right)=1
\]
by Proposition~\ref{O:F:finite}.

Equation~\eqref{O:F:localidentity} compares $U$ with every other
degree-$\ell$ intermediate field. Comparing each of two fields with
$U$ proves Theorem~\ref{U:main}.
\end{proof}

\section{Wild odd degree: the unramified--ramified case}\label{O:sec:inertia}\label{sec:odd-ur}
 
Let $F$ be a nonarchimedean local field of residue characteristic $p>2$.
Let $K/F$ be Galois with group $C_p\times C_p$, and suppose its inertia
subgroup has order $p$. Let $U$ be the unramified degree-$p$
intermediate field and $E$ any ramified degree-$p$ intermediate field. Then $K=EU$, $K/E$ is unramified,
and $E/F$ and $K/U$ are totally ramified cyclic of degree $p$, with the same
break $t\ge1$. Put $T=t+1$.

Fix a nontrivial additive character $\psi_F$ and set
$\psi_L=\psi_F\circ\Tr_{L/F}$ for all four fields. We shall prove the following

\begin{theorem}[Odd-degree comparison with an unramified intermediate field]\label{O:I:main}
Let $(\theta_U,\theta_E)$ be a primitive compatible pair of characters
of $U^\times$ and $E^\times$. Then
\[
 \Delta_U(\theta_U,\psi_U)=\Delta_E(\theta_E,\psi_E).
\]
Consequently $\mathcal A_U(\theta_U)=\mathcal A_E(\theta_E)$.
The assertion holds in mixed and equal characteristic.
\end{theorem}

\subsection{Truncated logarithms and norm characters}
Put
\[
 P(Z)=\sum_{j=1}^{p-1}\frac{Z^j}{j}.
\]
Its coefficients lie in $\mathbf Z_{(p)}$. Our sign convention is that
$P(z)$ is the truncation of $-\log(1-z)$.

\begin{lemma}[Truncated-logarithm coefficients]\label{O:P:coefficients}
Let $L$ be a nonarchimedean local field of residue characteristic $p>2$,
let $\psi_L$ be a nontrivial additive character of $L$, and let
$\theta:L^\times\to\CC^\times$ be a quasi-character of conductor $n\ge2$.
There exists $\gamma\in L^\times$ such that
\begin{equation}\label{O:P:fullP}
 \theta(1-z)=\psi_L(\gamma P(z))
 \quad\bigl(z\in\pp_L^{\ceil{n/p}}\bigr),\qquad
 v_L(\gamma)=-n_L(\psi_L)-n.
\end{equation}
The coefficient is unique modulo
$\pp_L^{-n_L(\psi_L)-\ceil{n/p}}$.
\end{lemma}
\begin{proof}
Apply Lemma~\ref{O:I:logchart} with
$r=\lceil n/p\rceil$ and $M=n$, followed by
\cite[Lemma~4.1]{Ueda} with denominator $1$. The resulting coefficient
is unique modulo
$\pp_L^{-n_L(\psi_L)-\lceil n/p\rceil}$ and may be chosen with
$v_L(\gamma)\ge-n_L(\psi_L)-n$.
On $\pp_L^{\ceil{n/2}}$ one has $P(z)\equiv z\pmod{\pp_L^n}$,
so the same coefficient gives $\theta(1-z)=\psi_L(\gamma z)$ there.
The exact-conductor argument of \cite[Lemma~4.2]{Ueda}, with the sign
corresponding to $1-z$, gives the stated valuation of $\gamma$.
\end{proof}

\begin{lemma}[Truncated logarithms of norms]\label{O:I:normlog}
Let $L/M$ be a totally ramified cyclic extension of degree $p$ and break $t$.
For $h\ge0$ and $z\in\pp_L^{h+\ceil{T/p}}$,
\begin{equation}\label{O:I:normlogeq}
 P(1-\Nm_{L/M}(1-z))\equiv
 \Tr_{L/M}(P(z))+\Nm_{L/M}(P(z))\pmod{\pp_M^{T+h}}.
\end{equation}
For an unramified degree-$p$ extension $L/M$ and $z\in\pp_L^{\ceil{T/p}}$, one has
\begin{equation}\label{O:I:unramlog}
 P(1-\Nm_{L/M}(1-z))\equiv\Tr_{L/M}(P(z))\pmod{\pp_M^T}.
\end{equation}
\end{lemma}
\begin{proof}
Put $e_0=\floor{(p-1)T/p}=T-\ceil{T/p}$.
Let $z_1,\ldots,z_p$ be indeterminates and $e_i$
their elementary symmetric functions. Give $e_i$ weight $i$ and put
\[
 x=1-\prod_{a=1}^p(1-z_a)=e_1-e_2+\cdots+e_p,
\quad
 \mathscr D=P(x)-\sum_aP(z_a)-\prod_aP(z_a).
\]
The polynomial $\mathscr D$ belongs to
$\mathbf Z_{(p)}[e_1,\ldots,e_p]$. Every monomial has weight at least $p$
and at least two elementary-symmetric factors. Indeed the formal logarithm identity
cancels all terms of weight less than $p$, and the weight-$p$ term before the final product
is $(\sum_a z_a^p-e_1^p)/p$. Newton's identity gives coefficient $1$ for
$e_p$ in this polynomial; the final product cancels it. There is no possible
one-factor monomial of larger weight.

Every coefficient of a pure power $e_p^j$ is divisible by $p$:
modulo $p$, specialize
$e_1=\cdots=e_{p-1}=0$, adjoin $w$ with $w^p=e_p$, and note that
$\sum_aP(z_a)=0$ and $\prod_aP(z_a)=P(w)^p=P(e_p)$.
The injectivity of $\F_p[e_p]\to\F_p[w]$ justifies this coefficient test.

Write $q=h+\ceil{T/p}$ and specialize to the conjugates of $z$. Then
$v_M(e_p)\ge q$ and, for $i<p$,
\[
 v_M(e_i)\ge\floor{(iq+(p-1)T)/p}
 \ge \frac{iq+(p-1)t}{p},\qquad v_M(e_i)\ge e_0.
\]
A monomial containing both $e_p$ and a smaller-index factor has valuation
at least $q+e_0=T+h$. A monomial consisting of $r\ge2$ smaller-index
factors and having total weight at least $p$ has valuation at least
$q+\frac{2(p-1)t}{p}\ge T+h$,
because $\ceil{T/p}+(p-2)t/p-1\ge0$.
In mixed characteristic $e=v_M(p)\ge e_0$, by
$p=\Tr_{L/M}(1)$ and \eqref{U:trace-ideal}. Every term $c_je_p^j$
of $\mathscr D$ with $j\ge2$ has $c_j\in p\mathbf Z_{(p)}$, hence
$v_M(c_je_p^j)\ge e+jq\ge e+2q\ge T+h$.
When $\operatorname{char}M=p$, the images of the coefficients
$c_j\in p\mathbf Z_{(p)}$ are zero, so these terms are absent.
This proves \eqref{O:I:normlogeq}.

For an unramified extension, every monomial of weight at least $p$
has valuation at least $p\ceil{T/p}\ge T$, and the norm of
$P(z)$ has valuation at least $p\ceil{T/p}$. This gives \eqref{O:I:unramlog}.
\end{proof}

Let $\phi$ be arithmetic Frobenius in $\Gal(U/F)$. Proposition~\ref{U:conjugacy} and Lemma~\ref{D:crossed-norm} determine a unique nontrivial
$\tau_F\in S(E/F)$ by
\[
 \frac{\theta_U^\phi}{\theta_U}=\tau_F\circ\Nm_{U/F}.
\]
Put $\tau_U=\tau_F\circ\Nm_{U/F}$. By \eqref{U:norm-character-conductor}, $\tau_F$ has conductor $T$.
Equation~\eqref{O:P:fullP}, applied to $\tau_F$, gives $\alpha\in F^\times$ such that
\begin{equation}\label{O:I:tauchart}
 \tau_F(1-z)=\psi_F(\alpha P(z))\quad(z\in\pp_F^{\ceil{T/p}}),
 \qquad v_F(\alpha)=-n_F(\psi_F)-T.
\end{equation}
Set
\[
 \Psi_L(z)=\psi_L(\alpha z).
\]
The largest ideal on which $\Psi_L$ is trivial is $\pp_L^T$ for
\emph{each} of $F,U,E,K$. This follows from
\eqref{U:trace-ideal} for $E/F$ and $K/U$, and from the unramified case
of the same formula for $U/F$ and $K/E$.
Equation~\eqref{O:I:unramlog} gives
\begin{equation}\label{O:I:tauUchart}
 \tau_U(1-z)=\Psi_U(P(z))\qquad(z\in\pp_U^{\ceil{T/p}}).
\end{equation}

\begin{corollary}[Trace--norm conversion]\label{O:I:conversion}
For $L/M=E/F$ or $K/U$ and every
$z\in\pp_L^{\ceil{T/p}}$,
\begin{equation}\label{O:I:conversioneq}
 \Psi_M(\Tr_{L/M}(z)+\Nm_{L/M}(z))=1,
 \qquad
 \Psi_L(z)=\Psi_M(-\Nm_{L/M}(z)).
\end{equation}
\end{corollary}
\begin{proof}
Write $z=P(x)$ with $x\in\pp_L^{\ceil{T/p}}$, using Lemma~\ref{O:I:logchart} as above.
The norm of $1-x$ belongs to $U_M^{\ceil{T/p}}$: the intermediate symmetric
terms have valuation at least $\ceil{T/p}$ by \eqref{U:symmetric}, while total
ramification gives the same bound for the norm term. The norm character is trivial on $\Nm_{L/M}(1-x)$.
Equations \eqref{O:I:tauchart} and \eqref{O:I:tauUchart},
together with Lemma~\ref{O:I:normlog} at $h=0$, prove
\eqref{O:I:conversioneq}.
\end{proof}

\subsection{Centered stationary coefficients and their changes}
\begin{proposition}[Centered stationary coefficient and its changes]
\label{O:I:enhanced}
Let $L$ be a nonarchimedean local field of odd residue characteristic $p$,
and let $\theta:L^\times\to\C^\times$ be a quasi-character of conductor
$m\ge2$. For a nontrivial additive character $\psi_L$ and
$\alpha\in L^\times$, let $\Psi_L=\psi_L(\alpha\,\cdot)$ have largest
trivial ideal $\pp_L^J$. 
There is $B\in L$ such that
\[
 v_L(B)=J-m,\qquad
 \theta(1-z)=\Psi_L(BP(z))\quad(z\in\pp_L^{\floor{m/2}}).
\]
For this choice of $B$,
\[
 \Delta_L(\theta,\psi_L)
 =\theta((-\alpha B)^{-1})\Psi_L(-B)g_L,
 \qquad g_L\in\mu_4,
\]
and $g_L=1$ when $m$ is even.

If $C=B+\eta$ with $v_L(\eta)\ge J-\ceil{m/2}$, then
$\theta(1-z)=\Psi_L(Cz)$ for $z\in\pp_L^{\ceil{m/2}}$, and
\begin{equation}\label{O:I:completeformula}
 \Delta_L(\theta,\psi_L)=
 \theta((-\alpha C)^{-1})\Psi_L(-C)
 \Psi_L\!\left(\frac{\eta^2}{2B}\right)g_L.
\end{equation}
The correction factor $\Psi_L\!\left(\frac{\eta^2}{2B}\right)$ is $1$ when $m$ is even, and also whenever
$v_L(\eta)\ge J-\floor{m/2}$.
\end{proposition}
\begin{proof}
Put $d=\floor{m/2}$ and $s=\ceil{m/2}$.
Since $pd\ge m$, Lemma~\ref{O:I:logchart} and
\cite[Lemma~4.1]{Ueda}, with
$M=J$, $q=m$, $r=d$, $\Gamma=1$, and additive character $\Psi_L$,
give $B$ modulo $\pp_L^{J-d}$.
The argument of \cite[Lemma~4.2]{Ueda} gives $v_L(B)=J-m$.

For $z\in\pp_L^d$,
\[
 P(z)-z-z^2/2\in\pp_L^{3d}\subseteq\pp_L^m,
\]
so
$\theta(1+z)=\Psi_L(-Bz+Bz^2/2)$.
On $\pp_L^s$ one has $B(P(z)-z)\in\pp_L^J$, so $B$ is also a
stationary linear coefficient. If $m=2d+1$, choose
$\delta\OO_L=\pp_L^d$. The centered critical function is
\[
 \bar z\longmapsto\Psi_L(-B\delta^2z^2/2).
\]
Its normalized sum $g_L$ is a quadratic Gauss sum. By
\cite[Lemma~5.6]{Ueda}, $g_L^2=\nu(-1)$, with $\nu$ the quadratic
character of $k_L^\times$, so $g_L\in\mu_4$.
For even $m$, Lemma~\ref{U:stationary-factor} gives $g_L=1$.
That lemma, with $Z=B$, gives the centered formula above.

Now let $C=B+\eta$. For $z\in\pp_L^s$, both $\eta z$ and
$B(P(z)-z)$ lie in $\pp_L^J$, so
$\theta(1-z)=\Psi_L(Cz)$. In odd conductor the critical function is
\[
 \bar z\longmapsto
 \Psi_L(-\eta\delta z-B\delta^2z^2/2).
\]
Translation by $\eta/(B\delta)$ contributes
$\Psi_L(\eta^2/(2B))$ and leaves the centered quadratic Gauss sum
$g_L$ unchanged. Lemma~\ref{U:stationary-factor} now gives
\eqref{O:I:completeformula}. Finally,
$v_L(\eta^2/B)\ge J$ in even conductor, and also when
$v_L(\eta)\ge J-d$.
\end{proof}

\subsection{Compatible characters of minimal conductor}
Let $k$ be the residue field of $F,E$ and $\kappa$ that of $U,K$.
Write $q=|k|$.
Choose $\bar d\in\kappa$ with $\bar d^q-\bar d=1$.
Then $\bar c=\bar d^p-\bar d\in k$ has absolute trace $1$.
Choose any lift $c\in\OO_F$ of $\bar c$. The polynomial
$X^p-X-c$ is irreducible and unramified over $F$, so we can choose
\begin{equation*}
 d\in U,\qquad d^p-d=c,\qquad d\bmod\pp_U=\bar d.
\end{equation*}
In mixed characteristic, Hensel's lemma at $d+1$ gives
$\phi(d)-(d+1)\in p\OO_U$. In equal characteristic this difference is zero.

\begin{proposition}[Construction of compatible characters of minimal conductor]
\label{O:I:canonical}
With the normalization \eqref{O:I:tauchart}, there are characters $\chi_U^0$
and $\chi_E^0$ of conductor $T$, with a common invariant primitive norm pullback,
such that
\begin{align}
 (\chi_U^0)^\phi/\chi_U^0&=\tau_U,\label{O:I:modelcomm}\\
 \chi_U^0(1-z)&=\Psi_U(d^pP(z))
 &&(z\in\pp_U^{\ceil{T/p}}),\label{O:I:modelU}\\
 \chi_E^0(1-z)&=\Psi_E(cP(z))
 &&(z\in\pp_E^{\ceil{T/p}}).\label{O:I:modelE}
\end{align}
\end{proposition}
\begin{proof}
By Lemma~\ref{O:I:logchart}, the formula
\[
 \xi(1-z)=\Psi_U(d^pP(z))\qquad(z\in\pp_U^{\ceil{T/p}})
\]
defines a character of $U_U^{\ceil{T/p}}$.
Since $d$ is a unit, \eqref{O:I:tauchart} shows that its conductor is $T$. Moreover,
\[
 \frac{\xi^\phi}{\xi}(1-z)
 =\Psi_U\!\left((\phi(d^p)-d^p)P(z)\right)
 =\Psi_U(P(z))
 =\tau_U(1-z).
\]
Indeed, $\phi(d^p)-d^p-1\in p\OO_U$ in mixed characteristic, while
it is zero in equal characteristic; in the mixed-characteristic case,
\eqref{U:trace-ideal} for $K/U$ gives
$v_U(p)\ge T-\ceil{T/p}$, so the error term lies in $\pp_U^T$.

Put $V=U_U^1$ and $I=\{\phi^{-1}(v)v^{-1}:v\in V\}$.
The rule $\phi^{-1}(v)v^{-1}\mapsto\tau_U(v)$ defines a character
on $I$: the kernel of $v\mapsto\phi^{-1}(v)v^{-1}$ is $U_F^1$,
and $\tau_U(f)=\tau_F(f^p)=1$ for $f\in U_F^1$.
The prescribed character agrees with $\xi$ on
$I\cap U_U^{\ceil{T/p}}$. Indeed, if
$\phi^{-1}(v)v^{-1}\in U_U^{\ceil{T/p}}$, then the class of $v$ in
$U_U^1/U_U^{\ceil{T/p}}$ is fixed by $\phi$. By the Teichm\"uller expansion
for the unramified extension $U/F$, one can write
$v=fh$ with $f\in U_F^1$ and $h\in U_U^{\ceil{T/p}}$. Hence
\[
 \xi(\phi^{-1}(v)v^{-1})
 =\xi(\phi^{-1}(h)h^{-1})
 =\tau_U(h)=\tau_U(v).
\]
The two characters therefore define a character on $I U_U^{\ceil{T/p}}$,
trivial on $U_U^T$. By Lemma~\ref{C:character-extension}, extend it
across $V/U_U^T$ and denote the extension by $\chi_U^0$.
Its prescribed values on $I$ give \eqref{O:I:modelcomm} on $V$.
Extend $\chi_U^0$ to $U^\times$ by taking it to be $1$ on the
Teichm\"uller subgroup and on a uniformizer of $F$. The character
$\tau_U$ has order $p$, so it is trivial on the Teichm\"uller subgroup,
and $\tau_U(\varpi_F)=\tau_F(\varpi_F^p)=1$. Hence
\eqref{O:I:modelcomm} holds on all of $U^\times$, while
\eqref{O:I:modelU} and the exact conductor $T$ are unchanged.

Put
$\Theta^0=\chi_U^0\circ\Nm_{K/U}$.
It is invariant under $\Gal(K/U)$. If
$\sigma\in\Gal(K/E)$ restricts to $\phi$ on $U$, then
\[
 \frac{(\Theta^0)^\sigma}{\Theta^0}
 =\left(\frac{(\chi_U^0)^\phi}{\chi_U^0}\right)
       \circ\Nm_{K/U}
 =\tau_U\circ\Nm_{K/U}=1.
\]
Thus $\Theta^0$ is $\Gal(K/F)$-invariant. By
Lemma~\ref{C:hilbert90}, its invariance under $\Gal(K/E)$ implies
that it is trivial on $\ker\Nm_{K/E}$. Therefore
\[
 \widetilde\chi_E\!\left(\Nm_{K/E}(x)\right)=\Theta^0(x)
 \qquad(x\in K^\times)
\]
defines a character of $\Nm_{K/E}(K^\times)$.
Lemma~\ref{C:character-extension} extends it to a character
$\chi_E^0$ of $E^\times$.

For $z\in\pp_K^{\ceil{T/p}}$, put $w=P(z)$. Equations
\eqref{O:I:modelU}, \eqref{O:I:normlogeq}, and
\eqref{O:I:conversioneq}, together with
$\Nm_{K/U}(dw)=d^p\Nm_{K/U}(w)$, give
\[
 \begin{aligned}
 \Theta^0(1-z)
 &=\Psi_K(d^pw)\,
   \Psi_U\!\left(\Nm_{K/U}(dw)\right)\\
 &=\Psi_K((d^p-d)w)
 =\Psi_K(cw).
 \end{aligned}
\]
For $y=1-\Nm_{K/E}(1-z)\in\pp_E^{\ceil{T/p}}$,
the unramified congruence \eqref{O:I:unramlog} gives
\[
 \Psi_E(cP(y))
 =\Psi_E\!\left(c\Tr_{K/E}(P(z))\right)
 =\Psi_K(cP(z))
 =\Theta^0(1-z).
\]
Since
$\Nm_{K/E}(U_K^{\ceil{T/p}})=U_E^{\ceil{T/p}}$ by
\eqref{U:unit-norm-surjective}, this proves \eqref{O:I:modelE}.
Because $c$ is a unit and $\Psi_E$ has largest trivial ideal
$\pp_E^T$,  $\chi_E^0$ has exact
conductor $T$.

Finally, suppose that $\Theta^0=\lambda\circ\Nm_{K/F}$ for a
quasi-character $\lambda$ of $F^\times$. Then
\[
 \frac{\chi_U^0}{\lambda\circ\Nm_{U/F}}\in S(K/U).
\]
Both $\lambda\circ\Nm_{U/F}$ and every member of $S(K/U)$ are fixed
by $\phi$, the latter by Lemma~\ref{D:crossed-norm}. This would make
$\chi_U^0$ fixed by $\phi$, contradicting
\eqref{O:I:modelcomm} because $\tau_U$ is nontrivial. Hence the common
pullback $\Theta^0$ is primitive.
\end{proof}

\begin{proposition}[Reduction to a common twist]\label{O:I:actualtwist}
Let $\chi_U^0,\chi_E^0$ be the characters of
Proposition~\ref{O:I:canonical}. There exist a quasi-character $\lambda$
of $F^\times$, a character $\mu_E\in S(K/E)$, and an integer $h\ge0$
such that, putting $\chi_E=\chi_E^0\mu_E$,
\begin{equation}\label{O:I:twistpair}
 \begin{aligned}
  \theta_U&=\chi_U^0(\lambda\circ\Nm_{U/F}),&
  m_U(\theta_U)&=T+h,\\
  \theta_E&=\chi_E(\lambda\circ\Nm_{E/F}),&
  m_E(\theta_E)&=T+ph.
 \end{aligned}
\end{equation}
For $h>0$, $m_F(\lambda)=T+h$; for $h=0$, $m_F(\lambda)\le T$.
The character $\chi_E$ has conductor $T$, satisfies
\eqref{O:I:modelE}, and has the same norm pullback to $K$ as $\chi_U^0$.
\end{proposition}
\begin{proof}
By the choice of $\tau_U$ in \eqref{O:I:tauUchart} and \eqref{O:I:modelcomm},
\[
 \frac{\theta_U^\phi}{\theta_U}
 =\frac{(\chi_U^0)^\phi}{\chi_U^0}
 =\tau_U.
\]
Hence $\theta_U/\chi_U^0$ is $\phi$-invariant.
Lemmas~\ref{C:hilbert90} and~\ref{C:character-extension} therefore
give $\theta_U/\chi_U^0=\lambda\circ \Nm_{U/F}$ for a quasi-character
$\lambda$ of $F^\times$. Put
\[
 \mu_E=
 \frac{\theta_E}{\chi_E^0(\lambda\circ\Nm_{E/F})}.
\]
Compatibility and the common norm pullback of $\chi_U^0,\chi_E^0$
give $\mu_E\circ\Nm_{K/E}=1$, so $\mu_E\in S(K/E)$.
Since $K/E$ is unramified, $\mu_E$ is unramified. Thus
$\chi_E=\chi_E^0\mu_E$ has conductor $T$, still satisfies
\eqref{O:I:modelE}, and has the same norm pullback as $\chi_E^0$.

Since $\tau_U$ has conductor $T$, the displayed equality gives
$m_U(\theta_U)\ge T$. If it exceeds $T$, then $\lambda\circ \Nm_{U/F}$ has the same conductor as
$\theta_U$, because $\chi_U^0$ has conductor $T$. Since $U/F$ is
unramified, \eqref{U:unramified-pullback-conductor} gives
$m_U(\lambda\circ \Nm_{U/F})=m_F(\lambda)$.
All twists of $\theta_U$ by $S(K/U)$ are conjugate and have the same
conductor. Equations~\eqref{U:break-pullback-conductor} and
\eqref{U:norm-pullback-conductor} give
$m_K(\theta_U\circ \Nm_{K/U})=T+ph$.
Since $K/E$ is unramified, \eqref{U:unramified-pullback-conductor} gives
$m_E(\theta_E)=m_K(\theta_U\circ \Nm_{K/U})$.
\end{proof}

\subsection{Norm representatives and cancellation of the correction factors}\label{O:I:comparison-coefficients}
Use Proposition~\ref{O:I:actualtwist} and suppose $0\le h\le t$. Set
\[
 m_U(\theta_U)=T+h,\quad m_E(\theta_E)=T+ph,\quad
 d_U=\floor{m_U(\theta_U)/2},\quad d_E=\floor{m_E(\theta_E)/2}.
\]
In particular $d_U,d_E\ge \ceil{T/p}$.
By Lemma~\ref{O:I:logchart} and \cite[Lemma~4.1]{Ueda}, choose
$\gamma\in F$ so that
\begin{equation}\label{O:I:lambdachart}
 \lambda(1-z)=\psi_F(\gamma P(z))\qquad(z\in\pp_F^{d_U}).
\end{equation}
Since $P$ preserves the filtration on $\pp_F^{d_U}$ and the largest
ideal on which $\Psi_F$ is trivial is $\pp_F^T$, the exact conductor
of $\lambda$ gives $v_F(\gamma/\alpha)=-h$ for $h>0$. For $h=0$, one may choose
$\gamma/\alpha\in\OO_F$. In all cases $\gamma/\alpha$ is determined modulo
$\pp_F^{T-d_U}$.

\begin{lemma}[Norm representatives and the boundary correction]\label{O:I:chooseA}
There is $x\in E$ with $A=\Nm_{E/F}(x)$ having the following properties.
If $0\le h<t$, one may take $\gamma=\alpha A$ in \eqref{O:I:lambdachart}.
If $h=t$, one has instead
\begin{equation}\label{O:I:epsilon}
 \gamma/\alpha=A+\epsilon,\qquad\epsilon\in\OO_F.
\end{equation}
For $h>0$, $v_E(x)=v_F(A)=-h$. For $h=0$, $x$ is integral.
\end{lemma}
\begin{proof}
Suppose first that $0\le h<t$. If $\gamma/\alpha\in\pp_F^{T-d_U}$, then
$h=0$ and $\lambda|_{U_F^{d_U}}=1$; take $x=A=0$.
Otherwise, put $a=v_F(\gamma/\alpha)$ and $s=T-d_U-a$.
For $h>0$, one has $a=-h$ and
$1\le s=T+h-d_U=\ceil{(T+h)/2}\le t$,
because $T+h\le2t$. For $h=0$, one has $0\le a<T-d_U$ and
$1\le s\le T-d_U=\ceil{T/2}\le t$.
Equation~\eqref{U:norm-approximation}, applied to $c=\gamma/\alpha$, therefore
provides $x\in E^\times$ such that
\[
 v_E(x)=a,\qquad
 \gamma/\alpha-\Nm_{E/F}(x)\in\pp_F^{a+s}=\pp_F^{T-d_U}.
\]
Thus \eqref{O:I:lambdachart} holds with $\gamma=\alpha A$, where $A=\Nm_{E/F}(x)$.

For $h=t$, apply \eqref{U:norm-approximation} to $c=\gamma/\alpha$ with $s=t$.
Since $v_F(\gamma/\alpha)=-t$, this gives $x\in E^\times$ satisfying
\[
 v_E(x)=-t,\qquad
 \gamma/\alpha-\Nm_{E/F}(x)\in\pp_F^{-t+t}=\OO_F.
\]
Taking $A=\Nm_{E/F}(x)$ and $\epsilon=\gamma/\alpha-A$ gives
\eqref{O:I:epsilon}. Finally, total ramification gives
$v_F(A)=v_E(x)$ whenever $x\ne0$.
\end{proof}
Choose $x,A,\epsilon$ as in Lemma~\ref{O:I:chooseA}, set
$\epsilon=0$ for $h<t$, and put
\begin{equation*}
 B_U=A+d^p+\epsilon\in U,\qquad B_E=A-x+c\in E.
\end{equation*}

\begin{proposition}[Coefficients for $\theta_U$ and $\theta_E$]\label{O:I:actualcoeff}
For the characters in \eqref{O:I:twistpair},
\begin{equation}\label{O:I:actualP}
 \theta_U(1-z)=\Psi_U(B_UP(z))\quad(z\in\pp_U^{d_U}),
 \qquad
 \theta_E(1-z)=\Psi_E(B_EP(z))\quad(z\in\pp_E^{d_E}).
\end{equation}
Their coefficient valuations are $v_U(B_U)=-h$ and $v_E(B_E)=-ph$.
\end{proposition}
\begin{proof}
Apply the unramified assertion of Lemma~\ref{O:I:normlog} with
$T$ replaced by $m_U(\theta_U)$. Since $pd_U\ge m_U(\theta_U)$, it applies at depth
$d_U$; by \eqref{O:I:lambdachart}, the pullback of $\lambda$ to $U$ has
coefficient $\gamma/\alpha$. Combine it with
\eqref{O:I:modelU} to obtain the first formula.

For $z\in\pp_E^{d_E}$ put $y=1-\Nm_{E/F}(1-z)$.
Equation~\eqref{U:symmetric} gives
\[
 v_F(y)\ge\min\{d_E,\floor{(d_E+(p-1)T)/p}\}\ge d_U.
\]
The last inequality follows from $d_E\ge d_U$ and
$d_E+(p-1)T-pd_U=(p-1)(T+\varepsilon)/2\ge0$, where
$\varepsilon\in\{0,1\}$ is the common conductor parity.
Also $d_E-h\ge \ceil{T/p}$, since
$d_E-h=\floor{(T+(p-2)h)/2}\ge\floor{T/2}\ge \ceil{T/p}$.
For $h=0$ the same inequalities hold with $x$ integral.
Lemma~\ref{O:I:normlog} therefore gives
\[
 P(y)\equiv\Tr_{E/F}(P(z))+\Nm_{E/F}(P(z))\pmod{\pp_F^{T+h}}.
\]
Multiply by $A$, whose valuation is $-h$ for $h>0$ and nonnegative for
$h=0$. Equation \eqref{O:I:conversioneq} applies to $xP(z)$, and the
\emph{exact} equality $\Nm_{E/F}(x)=A$ gives
\[
 \Psi_F(A P(y))
 =\Psi_E(A P(z))\Psi_F(\Nm_{E/F}(xP(z)))
 =\Psi_E((A-x)P(z)).
\]

At $h=t$, $d_E=(p+1)t/2\ge T$. Equation~\eqref{U:symmetric}, applied to the norm expansion, gives
$y\in\pp_F^T$. Hence the extra coefficient $\epsilon\in\OO_F$ in
\eqref{O:I:epsilon} contributes $\Psi_F(\epsilon P(y))=1$.
Thus the displayed formula for $\lambda\circ \Nm_{E/F}$ also holds when $h=t$. Multiply by \eqref{O:I:modelE} to obtain the second formula.

Since $P$ preserves the filtration and $\Psi_U,\Psi_E$ have largest
trivial ideals $\pp_U^T,\pp_E^T$, the formulas just proved and the
exact conductors in Proposition~\ref{O:I:actualtwist} give
$v_U(B_U)=T-m_U(\theta_U)=-h$ and
$v_E(B_E)=T-m_E(\theta_E)=-ph$.
\end{proof}

Put
\[
 \begin{gathered}
 C=x+d\in K^\times,\qquad Z_U=\Nm_{K/U}(C),\qquad Z_E=\Nm_{K/E}(C),\\
 \eta_U=Z_U-B_U,\qquad\eta_E=Z_E-B_E.
 \end{gathered}
\]

\begin{lemma}[Norm coefficients and coefficient differences]\label{O:I:normdepths}
One has
\begin{align}
 Z_U&=d^p+\sum_{i=1}^{p-1}e_i^{E/F}(x)d^{p-i}+A,\label{O:I:ZU}\\
 Z_E&=x^p-x+c,\label{O:I:ZE}
\end{align}
and
\[
 \eta_U=\sum_{i=1}^{p-1}e_i^{E/F}(x)d^{p-i}-\epsilon,
 \qquad \eta_E=x^p-A.
\]
For $0\le h<t$, one has $v_U(\eta_U)\ge T-d_U$ and
$v_E(\eta_E)\ge T-d_E$. Hence both correction factors in
\eqref{O:I:completeformula} equal one.
For $h=t$, one has
$v_U(\eta_U)\ge T-\lceil m_U(\theta_U)/2\rceil$ and
$v_E(\eta_E)\ge T-\lceil m_E(\theta_E)/2\rceil$.
If also $p>3$, then $v_E(\eta_E)\ge T-d_E$.
For $h=t$, one further has
\begin{equation*}
 \eta_U\equiv e_{p-1}^{E/F}(x)d-\epsilon\pmod{\pp_U},\qquad e_{p-1}^{E/F}(x)\in\OO_F.
\end{equation*}
\end{lemma}
\begin{proof}
The norm expansion gives \eqref{O:I:ZU}; evaluating $X^p-X-c$ at
$-x$ gives \eqref{O:I:ZE}. Subtracting $B_U,B_E$ gives the two
coefficient differences.

Put $e_0=\floor{(p-1)T/p}=T-\ceil{T/p}$.
For integral $x$, \eqref{U:symmetric} puts all the $e_i^{E/F}(x)$ in $\pp_F^{e_0}$.
Lemma~\ref{O:A:lem:power}, applied to $E/F$ and $x$, gives
$v_E(x^p-A)\ge p v_E(x)+(p-1)t$. Thus for $h=0$,
$v_U(\eta_U)\ge e_0$ and $v_E(\eta_E)\ge(p-1)t$;
both lower bounds are at least $T-\floor{T/2}$.
If $x=0$ the assertions are immediate.

For $h>0$, \eqref{U:symmetric} and Lemma~\ref{O:A:lem:power} give
\begin{equation*}
 v_U\!\left(\sum_{i=1}^{p-1}e_i^{E/F}(x) d^{p-i}\right)
 \ge\floor{(p-1)(T-h)/p},\qquad
 v_E(\eta_E)\ge(p-1)t-ph.
\end{equation*}
If $h<t$, put $r=T-h\ge2$. The first bound is
$r-\ceil{r/p}\ge\ceil{r/2}=T-d_U$.
Here $\ceil{r/p}\le\floor{r/2}$ for $r\ge2$, $p\ge3$.
For the second, $h\le t-1$ gives
\[
 2((p-1)t-ph)-(T-ph)
 =(2p-3)t-ph-1\ge(p-3)t+p-1\ge0;
\]
hence $(p-1)t-ph\ge T-d_E$. This proves the bounds for $h<t$.

At $h=t$, $m_U(\theta_U)=2t+1$, $m_E(\theta_E)=(p+1)t+1$, and the required bounds
are $0$ over $U$ and $-(p-1)t/2$ over $E$.
The two bounds $v_U(\eta_U)\ge0$ and $v_E(\eta_E)\ge-t$ suffice.
If $p>3$, also $-t\ge1-(p-1)t/2=T-d_E$, which gives the stronger bound for $\eta_E$. Finally, at this boundary $e_{p-1}^{E/F}(x)$ is integral and all
$e_i^{E/F}(x)$ with $i\le p-2$ have valuation at least
$\floor{(t+p-1)/p}=\ceil{t/p}\ge1$. This proves the asserted residue formula for $\eta_U$.
\end{proof}

\begin{proposition}[The quotient of the two local constants]\label{O:I:scalarphase}
Let
$S=\Tr_{U/F}(Z_U)-\Tr_{E/F}(Z_E)$.
The local constants satisfy
\begin{equation}\label{O:I:ratioformula}
 \frac{\Delta_U(\theta_U)}{\Delta_E(\theta_E)}
 \in\mu_4\,\Psi_F(-S)\,
 \frac{\Psi_U(\eta_U^2/(2B_U))}
      {\Psi_E(\eta_E^2/(2B_E))}.
\end{equation}
Furthermore one has the identity
\begin{equation}\label{O:I:Sexact}
 S=p A-\Tr_{E/F}(x^p)+p\Tr_{E/F}(x).
\end{equation}
\end{proposition}
\begin{proof}
Apply Proposition~\ref{O:I:enhanced} to the coefficient pairs
$(B_U,Z_U)$ and $(B_E,Z_E)$, respectively, using
Lemma~\ref{O:I:normdepths}.
The factors at $-\alpha\in F^\times$ cancel by \eqref{U:det-character}.
Compatibility evaluated at $C$ gives
$\theta_U(Z_U)=\theta_E(Z_E)$.
Substitution in \eqref{O:I:completeformula} gives
\eqref{O:I:ratioformula}.

Newton identities for $X^p-X-c$ give
\[
 \Tr_{U/F}(d^j)=0\ (1\le j<p-1),\quad
 \Tr_{U/F}(d^{p-1})=p-1,\quad
 \Tr_{U/F}(d^p)=pc.
\]
Take traces in \eqref{O:I:ZU} and \eqref{O:I:ZE} and subtract. This gives
\eqref{O:I:Sexact}.
\end{proof}

\begin{lemma}[Valuation of the trace difference]\label{O:I:Sdepth}
One has $S\in\pp_F^T$ if $h=0$, if $1\le h<t$, or if $p>3$ and
$h=t$.
\end{lemma}
\begin{proof}
Use the integral universal identity
\[
 \Tr_{E/F}(x^p)-pA
 =e_1^{E/F}(x)^p+
   p\,Q(e_1^{E/F}(x),\ldots,e_{p-1}^{E/F}(x)),
\]
where every monomial of $Q$ has weight $p$ and at least two factors.
Indeed the power sum is congruent to $e_1^{E/F}(x)^p$ modulo $p$ as an
integral symmetric polynomial, its $e_p^{E/F}(x)$ coefficient is $p$,
and there is no one-factor monomial of weight $p$ among
$e_1^{E/F}(x),\ldots,e_{p-1}^{E/F}(x)$.

In mixed characteristic, \eqref{U:symmetric} and
\eqref{U:trace-ideal}, with $v_E(x)\ge-h$, give
\[
 v_F(e_i^{E/F}(x))\ge\frac{-ih+(p-1)t}{p},\qquad
 v_F(p)\ge\frac{(p-1)t}{p}.
\]
Hence the three terms $p e_1^{E/F}(x)$, $e_1^{E/F}(x)^p$, and every
monomial of $pQ$ have respective lower bounds
\[
 \frac{2(p-1)t-h}{p},\qquad
 (p-1)t-h,\qquad
 \frac{3(p-1)t}{p}-h.
\]
These are strictly greater than $t$ when $p>3$ and $0\le h\le t$,
or when $p=3$ and $0\le h<t$, and hence are at least $T$.
Together with \eqref{O:I:Sexact}, this proves the assertion.

When $\operatorname{char}F=p$, the terms divisible by $p$ vanish and
the identity reduces to
$\Tr_{E/F}(x^p)=e_1^{E/F}(x)^p$.
The same bound
$v_F(e_1^{E/F}(x)^p)\ge(p-1)t-h\ge T$
proves the result in the stated ranges.
\end{proof}

\begin{proposition}[Fourth-root bound below the stable range]\label{O:I:cubicidentity}
If $0\le h\le t$, then
\[
 \frac{\Delta_U(\theta_U)}{\Delta_E(\theta_E)}\in\mu_4.
\]
\end{proposition}
\begin{proof}
For $h<t$, both corrections in \eqref{O:I:ratioformula} are $1$ by
Lemma~\ref{O:I:normdepths}, so Lemma~\ref{O:I:Sdepth} already puts the quotient
in $\mu_4$.
At $h=t$ one has $B_U=A+d^p+\epsilon$ with $v_F(A)=-t$ and
$\eta_U\equiv e_{p-1}^{E/F}(x) d-\epsilon\pmod{\pp_U}$. Therefore
\begin{equation*}
 \Psi_U\!\left(\frac{\eta_U^2}{2B_U}\right)
 =\Psi_U\!\left(\frac{(e_{p-1}^{E/F}(x)d-\epsilon)^2}{2A}\right).
\end{equation*}
Indeed,
\[
 \begin{aligned}
 v_U\!\left(\frac{\eta_U^2-(e_{p-1}^{E/F}(x)d-\epsilon)^2}{2B_U}\right)&\ge t+1=T,\\
 v_U\!\left(\frac{(e_{p-1}^{E/F}(x)d-\epsilon)^2}{2}
                \left(\frac1{B_U}-\frac1A\right)\right)&\ge2t\ge T.
 \end{aligned}
\]
Thus both differences belong to $\pp_U^T$. If $p>3$, both $\Tr_{U/F}(d)$ and $\Tr_{U/F}(d^2)$ vanish, while the trace
of $1$ is $p$. The correction over $U$ is therefore
$\Psi_F(p\epsilon^2/(2A))$.  In mixed characteristic its argument has
valuation at least $v_F(p)+t\ge T$, so this value is $1$.  When
$\operatorname{char}F=p$, the scalar $p$ is zero in $F$, so the
argument itself is zero and the value is again $1$.
The correction over $E$ is one by Lemma~\ref{O:I:normdepths}, and
Lemma~\ref{O:I:Sdepth} gives $\Psi_F(-S)=1$. Thus all three factors
multiplying $\mu_4$ in \eqref{O:I:ratioformula} are $1$, proving this case.

It remains to treat $p=3$ and $h=t$, without any restriction on divisibility
of $t$. Write
\[
 a=\Tr_{E/F}(x),\qquad b=e_2^{E/F}(x),\qquad A=\Nm_{E/F}(x).
\]
Equation~\eqref{U:symmetric} gives
\begin{equation*}
 v_F(a)\ge\ceil{t/3}\ge1,\quad v_F(b)\ge0,
 \quad v_F(A)=-t,\quad v_E(x)=-t.
\end{equation*}
The coefficients in \eqref{O:I:actualP} are $B_U=A+d^3+\epsilon$ and
$B_E=A-x+c$.
Because $\Tr_{U/F}(d^2)=2$ and $\Tr_{U/F}(d)=0$,
the same correction formula over $U$ gives
\begin{equation*}
 \Psi_U\!\left(\frac{\eta_U^2}{2B_U}\right)=\Psi_F(b^2/A).
\end{equation*}
The remaining term $3\epsilon^2/(2A)$ belongs to $\pp_F^T$.
Thus $\epsilon$ does not affect this correction value.

The characteristic polynomial of $x$ gives the identity
$\eta_E=x^3-A=a x^2-bx$.
One has $v_E(\eta_E)\ge-t$, and
$v_E(1/B_E-1/A)\ge5t$. Consequently
\begin{equation*}
 \Psi_E\!\left(\frac{\eta_E^2}{2B_E}\right)
 =\Psi_F\!\left(\frac{\Tr_{E/F}(\eta_E^2)}{2A}\right),
\end{equation*}
since the omitted $E$-argument has valuation at least $3t\ge T$.
The trace maps $\pp_E^T$ into $\pp_F^T$ by \eqref{U:trace-ideal}.

With $S=3a-a^3+3ab$ from \eqref{O:I:Sexact}, substitution of these two correction values into \eqref{O:I:ratioformula} gives
$\Delta_U(\theta_U)/\Delta_E(\theta_E)\in\mu_4\Psi_F(-\Lambda)$, where
\[
 \Lambda=3a-a^3+3ab-\frac{b^2}{A}
              +\frac{\Tr_{E/F}(\eta_E^2)}{2A}.
\]
Newton's identities give
\[
 \Tr_{E/F}(x^2)=a^2-2b,\quad
 \Tr_{E/F}(x^3)=a^3-3ab+3A,\quad
 \Tr_{E/F}(x^4)=a^4-4a^2b+2b^2+4aA.
\]
Using $\eta_E=ax^2-bx$, one obtains
\[
 \Tr_{E/F}(\eta_E^2)
 =a^6-6a^4b+9a^2b^2-2b^3+4a^3A-6abA.
\]
Substitution gives
\[
 \Lambda=3a+a^3-\frac{b^2+b^3}{A}
          +\frac{a^2(a^2-3b)^2}{2A}.
\]

The last summand belongs to $\pp_F^T$: by the bounds on $a,b,A$ above,
$v_F(a)\ge1$, $v_F(a^2-3b)\ge1$, and $v_F(A^{-1})=t$.

For the first pair apply \eqref{O:I:conversioneq} to
$a\in F\subset E$. Since
$v_E(a)\ge3\lceil t/3\rceil\ge t\ge \ceil{T/p}$, it gives
$\Psi_F(3a+a^3)=1$.

For the second pair apply \eqref{O:I:conversioneq} to $b/x$.
It also has $E$-valuation at least $t\ge \ceil{T/p}$, and
\[
 \Tr_{E/F}(b/x)=\frac{b^2}{A},\qquad
 \Nm_{E/F}(b/x)=\frac{b^3}{A}.
\]
Thus
$\Psi_F((b^2+b^3)/A)=1$.
The three character values prove $\Psi_F(\Lambda)=1$.
\end{proof}

\subsection{Stable twists and completion of the comparison}\label{sec:odd-ur-stable}
For $h\ge T$, the norm-approximation depth used in Lemma~\ref{O:I:chooseA} exceeds the break $t$, so we switch to stable twisting.
\begin{proposition}[Stable-twist comparison]\label{O:I:stable-comparison}
If $h\ge T$, then $\Delta_U(\theta_U)=\Delta_E(\theta_E)$.
\end{proposition}
\begin{proof}
Suppose $h\ge T$, so $m=m_F(\lambda)=T+h\ge2T$.
Choose an admissible $\Gamma$ and a stationary representative
$\beta$ for $\lambda$ as in \eqref{U:stationary-linearization}.
Lemma~\ref{U:tame-covector}, applied with $b_F=\beta/\Gamma$, gives
the same stationary coefficient after pullback to $U$. Since $U/F$ is
unramified, \eqref{U:trace-ideal} and
\eqref{U:unramified-pullback-conductor} show that $\Gamma$ remains
admissible, so $\beta$ is a stationary representative there.
For the pullback to $E$, Proposition~6.8(c) of \cite{Ueda} gives the
same $\Gamma,\beta$, since $m\ge2T$.
The characters $\chi_U^0,\chi_E$ have conductor $T$, at most half
the conductor of the corresponding pullbacks. Equation~\eqref{U:stable} gives
\[
 \Delta_U(\theta_U)=\chi_U^0(\Gamma/\beta)
       \Delta_U(\lambda\circ\Nm_{U/F}),
 \qquad
 \Delta_E(\theta_E)=\chi_E(\Gamma/\beta)
       \Delta_E(\lambda\circ\Nm_{E/F}).
\]
The multipliers agree by \eqref{U:det-character} for
$(\chi_U^0,\chi_E)$.
By \eqref{U:FML} for $L/F$ and \eqref{U:stable} over $F$,
\[
 \prod_{\mu\in S(L/F)}\Delta_F(\mu\lambda)
 =\Delta_F(\lambda)^p
   \prod_{\mu\in S(L/F)}\mu(\Gamma/\beta)
 =\Delta_F(\lambda)^p.
\]
The character product is $1$ by Proposition~\ref{U:conjugacy}, and
\eqref{U:odd-norm-factor} gives the product of its local constants as $1$.
Therefore
$\Delta_U(\lambda\circ\Nm_{U/F})=\Delta_E(\lambda\circ\Nm_{E/F})$ and the stable quotient
is exactly $1$.
\end{proof}

\begin{proof}[Proof of Theorem~\ref{O:I:main}]
For $h\ge T$, use Proposition~\ref{O:I:stable-comparison}.
For $0\le h\le t$, Proposition~\ref{O:I:cubicidentity} and
Corollary~\ref{U:odd-power} put the quotient of local constants
in $\mu_4\cap\mu_p=\{1\}$, since $p$ is odd.
The products over $S(U/F)$ and $S(E/F)$ are $1$ by
\eqref{U:odd-norm-factor}. The cases $0\le h\le t$ and $h\ge T$
exhaust all possible conductors.
\end{proof}

\section{Artin--Schreier coordinates and trace--norm estimates}\label{O:sec:as}\label{sec:odd-estimates}

Let $K/F$ be totally ramified and Galois with
$\operatorname{Gal}(K/F)=C_p^2$, where $p>2$ is the residue
characteristic. Choose a degree-$p$ intermediate field $L_1$ of
smallest lower ramification break, and let $L_2$ be any other degree-$p$
intermediate field. Write $t_1\le t_2$ for the breaks of $L_1/F,L_2/F$.
By Lemma~\ref{D:break-ledger}, the breaks of $K/L_2$ and $K/L_1$
are $t_1$ and $t'=t_1+p(t_2-t_1)$, respectively.
All valuations $v_K,v_{L_i},v_F$ are normalized on their own fields.  In
particular $v_K|_{L_i}=pv_{L_i}$, $v_{L_i}|_F=pv_F$, and
$v_F(\Nm_{L_i/F}(y))=v_{L_i}(y)$.


\begin{proposition}[An Artin--Schreier generator with prescribed valuation]\label{O:A:prop:AS}
One has $p\nmid t_1$. There exists $\Delta\in K$ such that
\[
 \begin{gathered}
 K=L_2(\Delta),\qquad \Delta^p-\Delta=a\in L_2,\\
 \Nm_{K/L_2}(\Delta)=a,\qquad v_K(\Delta)=v_{L_2}(a)=-t_1.
 \end{gathered}
\]
\end{proposition}
\begin{proof}
Let $\sigma$ generate $\operatorname{Gal}(K/L_2)$.  First suppose $\operatorname{char}F=0$,
and put $V=v_K(p)$.  Applying \eqref{U:trace-ideal} to $1$ on $L_2/F$
gives
\begin{equation}\label{O:A:eq:pbound}
 v_K(p)\ge p(p-1)t_2.
\end{equation}
Put $\kappa=V-(p-1)t_1>0$.
Equation~\eqref{U:trace-ideal} for $K/L_2$ gives
$\Tr_{K/L_2}(\pp_K^\kappa)=p\OO_{L_2}$.
Choose $\eta\in\pp_K^\kappa$ with $\Tr_{K/L_2}(\eta)=-p$. Since
$\Tr_{K/L_2}(1+\eta)=0$, Lemma~\ref{C:hilbert90} gives $d_0\in K$ with
$\sigma(d_0)-d_0=1+\eta$.
Apply Lemma~\ref{O:A:best-approximation} to $d_0$ over $L_2$. Choose
$b\in L_2$ and put $d=d_0-b$, so that
\[
 q:=v_K(d)=\max_{c\in L_2}v_K(d_0-c),\qquad p\nmid q.
\]
The displacement is unchanged and has valuation zero.
Lemma~\ref{C:leading-ramification} therefore gives
$0=v_K(\sigma(d)-d)=q+t_1$.
Thus $v_K(d)=-t_1$.

Put $a_0=d^p-d$. The equation $\sigma(d)-d=1+\eta$, the binomial theorem,
and $v_K(d)=-t_1$ give
$v_K(\sigma(a_0)-a_0)\ge\kappa$.
If $a_0\in L_2$, put $a=a_0$ and $\Delta=d$. Suppose instead that
$a_0\notin L_2$. Apply Lemma~\ref{O:A:best-approximation} to $a_0$ and
choose $a\in L_2$ with
$q_0=v_K(a_0-a)$ maximal. Since $p\nmid q_0$,
Lemma~\ref{C:leading-ramification} gives
$v_K(\sigma(a_0)-a_0)=q_0+t_1$.
Hence
$v_K(a_0-a)=q_0\ge\kappa-t_1=V-pt_1>0$.
For $f(X)=X^p-X-a$, the shifted polynomial $f(d+Z)$ has integral
coefficients, unit linear coefficient, and constant coefficient in
$\pp_K^{V-pt_1}$. Indeed $v_K(pd^{p-1})=\kappa>0$, and the other
intermediate coefficients are integral. Hensel's lemma gives a root
$\Delta\equiv d\pmod{\pp_K^{V-pt_1}}$.

In either case $v_K(\Delta)=-t_1$, and therefore
$v_K(\Delta^p-\Delta)=-pt_1$,
so $v_{L_2}(a)=-t_1$.

In characteristic $p$, Lemma~\ref{C:hilbert90} gives $d_0$ with
$\sigma(d_0)-d_0=1$. Apply Lemma~\ref{O:A:best-approximation} and
write $d=d_0-b$ with $b\in L_2$ and $v_K(d)$ maximal; then
$p\nmid v_K(d)$. Lemma~\ref{C:leading-ramification} gives
$v_K(d)=-t_1$. Now
$d^p-d$ is fixed by $\sigma$, so take
$\Delta=d, a=d^p-d\in L_2$.
Then $v_{L_2}(a)=-t_1$.

In both cases $p\nmid v_K(d)$ and $v_K(d)=-t_1$, hence $p\nmid t_1$. Since $v_K(L_2^\times)=p\mathbf Z$,
$\Delta\notin L_2$ and hence $K=L_2(\Delta)$. Its minimal polynomial
is $X^p-X-a$, whose constant term gives $\Nm_{K/L_2}(\Delta)=a$.
\end{proof}

For $1\le i\le p$, let $e_i$ be the $i$th elementary symmetric function
of the $K/L_1$-conjugates of $\Delta$; thus $e_p=\Nm_{K/L_1}(\Delta)$.

\begin{lemma}[Norms, power traces, and conjugates of the generator]\label{O:A:generator-data}
One has $v_{L_1}(e_p)=-t_1$ and
\begin{equation}\label{O:A:eq:moments}
 \Tr_{K/L_2}(\Delta^j)=0\ (1\le j<p-1),\quad
 \Tr_{K/L_2}(\Delta^{p-1})=p-1,\quad
 \Tr_{K/L_2}(\Delta^p)=pa.
\end{equation}
For every nonzero $\xi\in\mu_{p-1}\subset F$, there is a
$K/L_2$-conjugate $\Delta_\xi$ of $\Delta$ satisfying
\begin{equation}\label{O:A:conjugate-root}
 \Delta_\xi=
 \begin{cases}
  \Delta+\xi,&\operatorname{char}F=p,\\
  \Delta+\xi+\eta_\xi,\quad
  v_K(\eta_\xi)\ge v_K(p)-(p-1)t_1,&\operatorname{char}F=0.
 \end{cases}
\end{equation}
Together with $\Delta$, these are all the conjugates over $L_2$.
\end{lemma}
\begin{proof}
Total ramification gives the norm valuation. Newton's identities for
$X^p-X-a$ give \eqref{O:A:eq:moments}.
In characteristic $p$, the roots are exactly $\Delta+\xi$ for
$\xi\in\F_p$. In mixed characteristic, the residual of $X^p-X-a$
at $\Delta+\xi$ has valuation at least $v_K(p)-(p-1)t_1$,
since $\xi^p=\xi$. After translating by $\Delta+\xi$, the polynomial
has integral coefficients and unit linear coefficient, by
\eqref{O:A:eq:pbound}. Lemma~\ref{D:newton} gives the asserted root
and error bound. The corrections have positive valuation, so the
roots are distinct, and the degree is $p$.
\end{proof}

\subsection{The trace--norm congruence}
\begin{lemma}[Coefficient extraction]\label{O:A:lem:coeff}
If $\theta\in L_1$ and $\theta=\sum_{i=0}^{p-1}Y_i\Delta^i$ with
$Y_i\in L_2$, then
\[
 v_{L_2}(Y_i)\ge v_{L_1}(\theta)+it_1\quad(0\le i<p).
\]
\end{lemma}
\begin{proof}
Since $p\nmid t_1$, the valuations $pv_{L_2}(Y_i)-it_1$ are distinct modulo $p$.
Each is at least $pv_{L_1}(\theta)$, proving the assertion for $i=0$ and the
preliminary bound for every $i$.

For each nonzero $\xi\in\mu_{p-1}$, use the conjugate
$\Delta_\xi$ of Lemma~\ref{O:A:generator-data}.

Proceed by induction on $j$; the case $j=0$ was proved above.
Let $\sigma_\xi\in\Gal(K/L_2)$ satisfy $\sigma_\xi(\Delta)=\Delta_\xi$.
Its restriction to $L_1$ has break $t_1$, so
Lemma~\ref{C:leading-ramification} gives
$v_{L_1}(\sigma_\xi(\theta)-\theta)\ge v_{L_1}(\theta)+t_1$.
In equal characteristic replacing $\Delta_\xi$ by $\Delta+\xi$ introduces
no error.  In mixed characteristic the resulting error has valuation at least
$pv_{L_1}(\theta)+v_K(p)-(p-2)t_1$,
and its $\Delta^{j-1}$ coefficient has $v_{L_2}$ at least
$v_{L_1}(\theta)+jt_1$, because \eqref{O:A:eq:pbound} gives
$v_K(p)\ge p(p-1)t_1\ge(p-1)(j+1)t_1$.
The induction hypothesis applied to $\sigma_\xi(\theta)-\theta$ yields
\[
 \sum_{i=j}^{p-1}\binom{i}{j-1}\xi^{i-j+1}Y_i
 \in\pp_{L_2}^{v_{L_1}(\theta)+jt_1}.
\]
Multiply by $\xi^{-1}$ and sum over $\mu_{p-1}$.
The power sums cancel every term except $(p-1)jY_j$;
this multiplier is a unit.
\end{proof}

\begin{lemma}[Twisted trace and symmetric functions]\label{O:A:lem:twisted}
For $Y\in L_2$,
\begin{align*}
 v_{L_1}(\Tr_{K/L_1}(Y\Delta^i))&\ge(p-1)t_2-it_1+v_{L_2}(Y)
 &&(0\le i<p),\\
 v_{L_1}(e_i^{K/L_1}(Y\Delta))&\ge(p-1)t_2+i(v_{L_2}(Y)-t_1)
 &&(1\le i<p).
\end{align*}
In particular
\begin{equation}
 v_{L_1}(e_i)\ge(p-1)t_2-it_1\qquad(1\le i<p).\label{O:A:eq:ei}
\end{equation}
\end{lemma}
\begin{proof}
The case $i=0$ of the trace estimate, namely
$\Tr_{K/L_1}(Y)=\Tr_{L_2/F}(Y)\in F$, follows from \eqref{U:trace-ideal}.

Put $D_j=(p-1)(t_j+1)$ for the different exponents of $L_j/F$,
$j=1,2$. For $i\ge1$ put $c=(p-1)t_2-it_1+v_{L_2}(Y)$.
By \eqref{U:trace-dual}, it suffices to show $\Tr_{L_1/F}(\theta \Tr_{K/L_1}(Y\Delta^i))\in\OO_F$
for every $v_{L_1}(\theta)\ge-D_1-c$.
Write $\theta=\sum Y_j\Delta^j$ and use trace transitivity.
For $1\le d\le2p-2$, the only nonzero power traces
$\Tr_{K/L_2}(\Delta^d)$ are those of degrees $p-1,p,2p-2$, with
values $p-1,pa,p-1$.
Thus
\[
 \Tr_{K/L_2}(\theta\Delta^i)=(p-1)Y_{p-1-i}+paY_{p-i}
       +\ind_{i=p-1}(p-1)Y_{p-1}.
\]
Lemma~\ref{O:A:lem:coeff} shows that, after multiplication by $Y$, all three
terms lie in $\pp_{L_2}^{-D_2}$.  For the first this is exactly
$v_{L_2}(Y)+(p-1-i)t_1-D_1-c=-D_2$;
for the second use $v_{L_2}(pa)\ge-t_1$, and for the last there is an additional
nonnegative $(p-1)t_1$. This proves the trace estimate.

Newton's identity for $e_i(Y\Delta)$ involves the power traces
$\Tr_{K/L_1}(Y^j\Delta^j)$, $1\le j\le i<p$.  Each product with total weight $i$
has valuation at least $(p-1)t_2+i(v_{L_2}(Y)-t_1)$, and $i$ is a unit.
\end{proof}

\begin{lemma}[Norms of close elements]\label{O:A:lem:H17}
Let \(x,y\in L_2\) and \(w\in L_1\) satisfy 
\[
 xw-y\in\pp_K^{1+pt_2},\qquad
 v_K(xw)\ge1+t_2+(p-3)t_1.
\]
Then
\[
 \Nm_{K/L_1}(xw)-\Nm_{L_2/F}(y)\in\pp_{L_1}^{1+pt_2}.
\]
\end{lemma}
\begin{proof}
Write $K_0=1+pt_2$.
The zero case is immediate. Put $z=xw\ne0$ and $h_0=v_K(z)$, and expand
$z=\sum_{i=0}^{p-1}y_i\Delta^i$, with $y_i=xc_i\in L_2$.
Lemma~\ref{O:A:lem:coeff} gives
$v_K(y_i\Delta^i)\ge h_0+(p-1)it_1$.
Thus $v_K(y_0)=h_0$, while
$v_K(y_i\Delta^i)\ge h_0+(p-1)it_1>h_0$ for every $i\ge1$.
The valuations of the nonzero terms in
\[
 z-y=(y_0-y)+\sum_{i=1}^{p-1}y_i\Delta^i
\]
have distinct residues modulo $p$, because $p\nmid t_1$.
Since their sum has valuation at least $K_0$, each term does.
In particular $v_K(y_i\Delta^i)\ge K_0$ for $i\ge1$ and
$v_{L_2}(y_0-y)\ge t_2+1$.

First keep the linear term and put $z_\ell=y_0+y_1\Delta$.
Multiplicativity of the norm gives
\[
 \Nm_{K/L_1}(z_\ell)=\Nm_{L_2/F}(y_0)\Nm_{K/L_1}(1+(y_1/y_0)\Delta).
\]
We have $v_{L_1}(\Nm_{L_2/F}(y_0))=h_0$ and $v_{L_2}(y_1/y_0)\ge t_1$.
Lemma~\ref{O:A:lem:twisted} therefore bounds the intermediate
terms in $\Nm_{K/L_1}(z_\ell)-\Nm_{L_2/F}(y_0)$ by
\[
 h_0+(p-1)t_2+j\bigl(v_{L_2}(y_1/y_0)-t_1\bigr)
 \ge h_0+(p-1)t_2\ge K_0.
\]
The remaining term is $\Nm_{K/L_1}(y_1\Delta)$ and has $v_{L_1}$ equal to
$v_K(y_1\Delta)\ge K_0$. Hence
$\Nm_{K/L_1}(z_\ell)-\Nm_{L_2/F}(y_0)\in\pp_{L_1}^{K_0}$.

Now put $U=\sum_{i=2}^{p-1}y_i\Delta^i$. Then
$v_K(U/z)\ge2(p-1)t_1$ and $v_K(U)\ge K_0$. The identity
\[
 \Nm_{K/L_1}(z)-\Nm_{K/L_1}(z_\ell)=\Nm_{K/L_1}(z)\{1-\Nm_{K/L_1}(1-U/z)\}
\]
is expanded over $K/L_1$, whose different exponent is
$D'=(p-1)(t'+1)$.
For each intermediate coefficient, \eqref{U:symmetric} gives
\[
 v_{L_1}\bigl(\Nm_{K/L_1}(z)\,e_j^{K/L_1}(U/z)\bigr)
 \ge h_0+\floor{[2j(p-1)t_1+D']/p}
 \ge K_0+\floor{[(p-3)t_1+p-1]/p}\ge K_0.
\]
The minimum occurs at $j=1$, and the second inequality follows by inserting
$h_0\ge1+t_2+(p-3)t_1$ and $t'=t_1+p(t_2-t_1)$.
The remaining term is $\Nm_{K/L_1}(U)$, which has valuation $v_K(U)\ge K_0$.
Thus $\Nm_{K/L_1}(z)-\Nm_{K/L_1}(z_\ell)\in\pp_{L_1}^{K_0}$.

Finally $y_0,y$ are integral and $v_{L_2}(y_0-y)\ge t_2+1$.
For the last norm comparison, set $a=v_{L_2}(y)$.
If $a\ge t_2+1$ (or $y=0$), both norms have valuation at least $t_2+1$.
Otherwise factor out $y$, apply \eqref{U:symmetric} to
$(y_0-y)/y$, and note that its shallowest intermediate term has valuation
at least
\[
 a+\floor{[t_2+1-a+(p-1)(t_2+1)]/p}
 =t_2+1+\floor{(p-1)a/p}\ge t_2+1.
\]
The norm term has valuation $v_{L_2}(y_0-y)\ge t_2+1$.
Consequently $\Nm_{L_2/F}(y_0)-\Nm_{L_2/F}(y)\in\pp_F^{t_2+1}$, which lies in
$\pp_{L_1}^{p(t_2+1)}\subseteq\pp_{L_1}^{K_0}$.
Combining the three comparisons proves the assertion.
\end{proof}

\begin{proposition}[Comparison of traces of norms]\label{O:A:thm:H14}
One has
\[
 \Tr_{L_1/F}(e_p)-\Tr_{L_2/F}(a)\in\pp_F^{1+t_2}.
\]
\end{proposition}
\begin{proof}
Put $D_1=(p-1)(t_1+1)$, the different exponent of $L_1/F$.
Compare the coefficient of $Z^p$ in the identity
\[
 \Nm_{L_1/F}(\Nm_{K/L_1}(1-Z\Delta))=\Nm_{L_2/F}(\Nm_{K/L_2}(1-Z\Delta)).
\]
The two norm polynomials are
\[
 \Nm_{K/L_1}(1-Z\Delta)=1+\sum_{i=1}^{p-1}(-1)^ie_iZ^i-e_pZ^p,
 \qquad \Nm_{K/L_2}(1-Z\Delta)=1-Z^{p-1}-aZ^p.
\]
The right coefficient is $-\Tr_{L_2/F}(a)$.  Besides $-\Tr_{L_1/F}(e_p)$, the left coefficient
contains $-\Nm_{L_1/F}(e_1)$ and traces of monomials in conjugates of the $e_i$
with $i<p$, at least two factors, and total weight $p$.
Indeed the cyclic group permutes the selections from the $p$ factors;
every orbit has size $p$, except the constant selection $e_1$ from every
factor, which gives $\Nm_{L_1/F}(e_1)$.

By \eqref{O:A:eq:ei}, each trace of a nonconstant selection has valuation at least
$\floor{\frac{2(p-1)t_2-pt_1+D_1}{p}}$.
This is at least $1+t_2$ whenever
$(p-2)t_2-t_1-1\ge0$.  The same condition makes
$v_F(\Nm_{L_1/F}(e_1))=v_{L_1}(e_1)\ge(p-1)t_2-t_1\ge1+t_2$.
This covers $p\ge5$ and $p=3$, $t_2>t_1$.

For $p=3$, $t_2=t_1$, first note that
$\Tr_{L_1/F}(e_1)=\Tr_{L_1/F}(\Tr_{K/L_1}(\Delta))=0$ by
\eqref{O:A:eq:moments}. Equation~\eqref{O:A:eq:ei} gives
$e_1\in\pp_{L_1}^{t_1}$. Since $D_1=2(t_1+1)$,
\eqref{U:trace-ideal} gives
\[
 \Tr_{L_1/F}(\pp_{L_1}^{t_1})=\pp_F^{t_1},\qquad
 \Tr_{L_1/F}(\pp_{L_1}^{t_1+1})=\pp_F^{t_1+1}.
\]
Thus trace induces a surjective map
\[
 \pp_{L_1}^{t_1}/\pp_{L_1}^{t_1+1}
 \longrightarrow \pp_F^{t_1}/\pp_F^{t_1+1}.
\]
Both sides are one-dimensional $k$-vector spaces, so the map is an
isomorphism. The class of $e_1$ maps to zero, hence
$e_1\in\pp_{L_1}^{t_1+1}$ and
$v_F(\Nm_{L_1/F}(e_1))\ge t_1+1$. Every remaining weight-three selection has one
$e_1$ factor and one $e_2$ factor. Its trace has valuation at least
$\floor{(t_1+1+2(t_1+1))/3}=t_1+1$.
\end{proof}

\subsection{Weighted norm and trace congruences}
\begin{lemma}[Replacing conjugates by translates in norm sums]\label{O:A:lem:translate}
Let $x\in L_2$ and $1\le d\le p$, with $v_K(x\Delta^{d-1})\ge1$.
Then
\[
 \Nm_{L_2/F}(x)\Tr_{L_1/F}(e_p^d)
 \equiv \Nm_{L_2/F}(x)\sum_{\xi\in\{0\}\cup\mu_{p-1}}\Nm_{K/L_1}(\Delta+\xi)^d
 \pmod{\pp_{L_1}^{pt_2+v_K(x\Delta^{d-1})}}.
\]
\end{lemma}
\begin{proof}
Put $R=v_K(x\Delta^{d-1})=pv_{L_2}(x)-(d-1)t_1$.
If $\operatorname{char}F=p$, the conjugates are exactly $\Delta+\xi$, so
the two sums are equal.

Assume mixed characteristic, put $V=v_K(p)$ and
$K_{\mathrm{tr}}=pt_2+t_1$, and fix $\xi\ne0$. Write the conjugate near
$\Delta+\xi$ as $(\Delta+\xi)(1+z)$.
Polynomial division gives the polynomial
\[
 Q_\xi(\Delta)=\frac{(\Delta+\xi)^p-(\Delta+\xi)-(\Delta^p-\Delta)}
                         {\Delta+\xi}.
\]
It has degree $p-2$, leading coefficient $p\xi$, and every coefficient is
in $p\OO_F$.  The equation for $z$ is
\[
 z=Q_\xi+p(\Delta+\xi)^{p-1}z
   +(\Delta+\xi)^{p-1}\sum_{j=2}^p\binom pj z^j.
\]
Hence
\[
 v_K(z)\ge V-(p-2)t_1,\qquad
 v_K(z-Q_\xi)\ge2V-(2p-3)t_1.
\]
For $p\ge5$, one has
$v_{L_1}(\Tr_{K/L_1}(z-Q_\xi))\ge K_{\mathrm{tr}}$:
subtracting $K_{\mathrm{tr}}$ from the trace-ideal lower bound gives
\[
 (p-3)t_1-\floor{(p-2)t_1/p}+(2p-3)(t_2-t_1)\ge0.
\]
Likewise $v_{L_1}(e_i(z))\ge K_{\mathrm{tr}}$ for $2\le i<p$, and $v_{L_1}(\Nm_{K/L_1}(z))=v_K(z)\ge V-(p-2)t_1\ge K_{\mathrm{tr}}$.  Terms of degree at most $p-3$ in $Q_\xi$ have traces of
valuation at least
$2(p-1)t_2-(p-3)t_1\ge K_{\mathrm{tr}}$ by Lemma~\ref{O:A:lem:twisted}.

For $p=3$ retain one further iterate:
\[
 Q_\xi=3\xi\Delta,\qquad
 z=Q_\xi+3(\Delta+\xi)^2Q_\xi+E,\quad v_K(E)\ge3V-5t_1.
\]
The trace of $E$ has valuation at least $V-t_1+2(t_2-t_1)\ge K_{\mathrm{tr}}$.
Moreover
\[
 3(\Delta+\xi)^2Q_\xi
 =9\xi\Delta^3+18\xi^2\Delta^2+9\xi^3\Delta.
\]
Using $\Delta^3=\Delta+a$, each trace on the right has valuation at least
$6t_2-2t_1\ge K_{\mathrm{tr}}$.  The other norm coefficients satisfy
$v_{L_1}(e_2(z))\ge2V/3+2(t_2-t_1)\ge K_{\mathrm{tr}}$ and
$v_{L_1}(\Nm_{K/L_1}(z))\ge V-t_1\ge K_{\mathrm{tr}}$.
Thus in all cases
\begin{equation*}
 \Nm_{K/L_1}(1+z)\equiv1+\xi L\pmod{\pp_{L_1}^{K_{\mathrm{tr}}}},\qquad
 L=p\Tr_{K/L_1}(\Delta^{p-2}),\quad
 v_{L_1}(L)\ge pt_1+2(p-1)(t_2-t_1).
\end{equation*}
Since $2v_{L_1}(L)\ge K_{\mathrm{tr}}$, the same formula raised to the $d$th power has linear
term $d\xi L$.

Now $\Nm_{K/L_1}(\Delta+\xi)-e_p$ is integral by \eqref{O:A:eq:ei}; hence
$v_{L_1}(\Nm_{K/L_1}(\Delta+\xi)^d-e_p^d)\ge-dt_1+t_1$.
After multiplying by $\Nm_{L_2/F}(x)$, the error term in the equation above
is in $\pp_{L_1}^{pt_2+R}$ because $v_{L_1}(\Nm_{L_2/F}(x)e_p^d)=R-t_1$.
In the linear term, replacing $\Nm_{K/L_1}(\Delta+\xi)^d$ by $e_p^d$ gives an error
of valuation at least
$R+pt_1+2(p-1)(t_2-t_1)\ge R+pt_2$.
Consequently the weighted difference of the two sums is congruent modulo
$\pp_{L_1}^{pt_2+R}$ to
\[
 d\,\Nm_{L_2/F}(x)e_p^d L\sum_{\xi\in\mu_{p-1}}\xi=0.
\]
Since $\Gal(K/F)$ is abelian, the sum over conjugates is
$\Tr_{L_1/F}(e_p^d)$.
\end{proof}

\begin{lemma}[Comparison of weighted power sums]\label{O:A:lem:filtered}
Put
\[
 P_d=e_p^d+\sum_{\xi\in\mu_{p-1}}\Nm_{K/L_1}(\Delta+\xi)^d,
 \qquad Q_d=\Tr_{K/L_1}((\Delta^p-\Delta)^d)=\Tr_{L_2/F}(a^d).
\]
Let $x\in L_2$ and $1\le d\le p$, with $v_K(x\Delta^{d-1})\ge1$.
Then
\[
 \Nm_{L_2/F}(x)(P_d-Q_d)\equiv(p-1)\ind_{d=p-1}\Nm_{L_2/F}(x)(1+e_{p-1}^p)
       \pmod{\pp_{L_1}^{pt_2+v_K(x\Delta^{d-1})}}.
\]
\end{lemma}
\begin{proof}
Put $R=v_K(x\Delta^{d-1})$.
For a monomial $e_\lambda=\prod e_i^{\lambda_i}$ put
$r=\sum\lambda_i$, $n=\sum i\lambda_i$, and $\nu=\sum_{i<p}\lambda_i$.
The Teichm\"uller power sums show that $P_d$ contains it only if
$n\equiv d\pmod{p-1}$, with coefficient
$(p-1)\binom{d}{r}\frac{r!}{\prod\lambda_i!}$,
plus the separate term $e_p^d$. Write $n=(p-1)v+d$.

For indeterminates $Z_1,\ldots,Z_p$ and $r>0$, define the integer
\[
 D_\lambda=[Z^\lambda]
 \left(\sum_{i=1}^p iZ_i\right)
 \left(\sum_{i=1}^p Z_i\right)^{r-1}.
\]
Euler's identity for $\left(\sum_i Z_i\right)^r$ gives
\[
 rD_\lambda=n\binom r{\lambda_1,\ldots,\lambda_p}.
\]
Equivalently,
\[
 D_\lambda=\frac{n(r-1)!}{\prod\lambda_i!}.
\]
Newton's generating-function formula therefore shows that the coefficient
of $e_\lambda$ in $Q_d$, for $0\le v\le d$, is the integer
\[
 (-1)^{r+v}\binom{d}{v} D_\lambda
 =(-1)^{r+v}\binom{d}{v}\frac{n(r-1)!}{\prod\lambda_i!}.
\]

The valuation estimate \eqref{O:A:eq:ei} gives
\begin{equation*}
 v_{L_1}(\Nm_{L_2/F}(x)e_\lambda)\ge R-t_1+(p-1)\bigl((r-v)t_1+\nu(t_2-t_1)\bigr).
\end{equation*}
For $v\ge0$ and $r\ge v+2$, the inequality
$pr-n=v+p(r-v)-d\ge p$ implies $\nu\ge2$.
Then the displayed lower bound for $v_{L_1}(\Nm_{L_2/F}(x)e_\lambda)$ is at least
$R-t_1+2(p-1)t_2\ge R+pt_2$.
If $r=v$, only $v=r=d$, $e_\lambda=e_p^d$ is possible; its coefficients
are $1+(p-1)=p$ and $p$, so they cancel.

Suppose $r=v+1\le d$.  Their coefficient difference is exactly
\[
 \frac{(r-1)!}{\prod\lambda_i!}
 \left((p-1)r\binom{d}{r}+n\binom{d}{r-1}\right)
 =\frac{p d\binom{d}{r-1}(r-1)!}{\prod\lambda_i!}.
\]
Here $r\binom{d}{r}=(d-r+1)\binom{d}{r-1}$ and
$n=(p-1)(r-1)+d$.  The factorial denominators are units at $p$:
if $r=p$, then $d=p$ and $n\equiv1\pmod p$, so no $\lambda_i$ can equal $p$.
Except for the pure norm monomial, one has $\nu\ge1$.
The added valuation $v_{L_1}(p)\ge(p-1)t_2$ then raises the same lower bound to at least $R+pt_2$.
The pure norm possibility is $d=p,r=1,v=0$, $e_\lambda=e_p$;
its coefficient difference is $p^2$, which gives valuation at least
$R+pt_2$ after applying \eqref{O:A:eq:pbound}.

If $v=d$ and $r=d+1$, the first coefficient is absent and the difference is
$p d\,d!/\prod\lambda_i!$.
It is divisible by $p$ except when $d=p-1$ and
$e_\lambda=e_{p-1}^p$, in which case it is exactly $p-1$.
Indeed for $d\le p-2$ all factorial denominators are units;
for $d=p-1$ the only nonunit denominator comes from that constant selection;
for $d=p$ the numerator contains at least three factors of $p$ and the
factorial denominator at most one.
The remaining monomials with $r\ge v+2$ were already dealt with.

Finally $v=-1$ contributes only the constant monomial when $d=p-1$, giving
$p-1$, or $e_1$ when $d=p$, giving $p(p-1)e_1$.
The latter has valuation at least $R+pt_2$ by
\eqref{O:A:eq:ei} and \eqref{O:A:eq:pbound}.
These cases exhaust the possible monomials and leave exactly the two
terms displayed in the lemma.
\end{proof}

\begin{proposition}[Weighted norm--trace congruence]\label{O:A:thm:h}
Define
\[
 h(z)=\Tr_{L_1/F}(\Nm_{K/L_1}(\Delta z))-\Tr_{L_2/F}(\Nm_{K/L_2}(\Delta z)).
\]
For $x\in L_2$ and $0\le j\le p-1$ with $v_K(x\Delta^j)\ge1$,
\begin{equation*}
 h(x\Delta^j)\equiv(p-1)\ind_{j=p-2}\Nm_{L_2/F}(x)(1+e_{p-1}^p)
       \pmod{\pp_{L_1}^{pt_2+v_K(x\Delta^j)}}.
\end{equation*}
\end{proposition}
\begin{proof}
Put $R=v_K(x\Delta^j)$.
By multiplicativity of the norm, the difference between
$h(x\Delta^j)$ and
$\Nm_{L_2/F}(x)(\Tr_{L_1/F}(e_p^{j+1})-\Tr_{L_2/F}(a^{j+1}))$ is
$\Tr_{L_2/F}((\Nm_{L_2/F}(x)-x^p)a^{j+1})$.
By Lemma~\ref{O:A:lem:power}, the argument has $v_{L_2}$ at least
$R-t_1+(p-1)t_2$.  Its trace, viewed in $L_1$, has valuation at least
$R-t_1+2(p-1)t_2\ge R+pt_2$.

Apply Lemmas~\ref{O:A:lem:translate} and~\ref{O:A:lem:filtered}
with index $j+1$; their hypothesis is $R\ge1$.
\end{proof}

Put
\[
 g(z)=\Tr_{L_1/F}(e_p\Tr_{K/L_1}(z))-\Tr_{L_2/F}(a\Tr_{K/L_2}(z)).
\]
Expand $e_p=\sum_{l=0}^{p-1}Y_l\Delta^l$ over $L_2$ and
$-e_{p-1}=\sum s_l\Delta^l$ over $L_2$.
\begin{lemma}[Coefficients of the norm in Artin--Schreier coordinates]\label{O:A:lem:bcoeff}
The coefficients satisfy
\begin{align*}
 v_{L_2}(Y_l)&\ge(p-1)t_2-(p-l)t_1&& (2\le l<p),\\
 v_{L_2}(Y_0-a)&\ge(p-1)t_2-t_1,&
 v_{L_2}(Y_1-1+s_0)&\ge(p-1)t_2.
\end{align*}
\end{lemma}
\begin{proof}
Use $e_p=a+\Delta+\sum_{i=1}^{p-1}(-1)^ie_i\Delta^{p-i}$.
Expand each $e_i$ over $L_2$ by Lemma~\ref{O:A:lem:coeff}; its $\Delta^j$
coefficient has valuation at least $(p-1)t_2-it_1+jt_1$.
If $j+p-i\ge p$, reduce once using $\Delta^p=\Delta+a$.
The unreduced contribution to $Y_l$ has valuation at least
$(p-1)t_2-(p-l)t_1$; each reduced contribution has valuation at least
$(p-1)t_2+(l-1)t_1$.

For $l=0$ only the terms with $j=i$ contribute, with a factor $a$.
For $l=1$ the sole low term is the constant coefficient of $e_{p-1}$,
namely $-s_0$; all others have valuation at least $(p-1)t_2$.
\end{proof}

\begin{proposition}[Weighted trace estimates on monomials]\label{O:A:prop:g}
Let $x\in L_2$ and $0\le j<p$.
If $j\ne p-2$ and $x\in\pp_{L_2}$, then
\[
 g(x\Delta^j)\in\pp_F^{1+t_2}.
\]
There is $\omega\in L_2$, independent of $x$, such that for every
$x\in L_2$ with $v_K(x\Delta^{p-2})\ge1+(t_2-t_1)$,
\[
 g(x\Delta^{p-2})=\Tr_{L_2/F}(x(\omega-1)),\qquad
 x\omega+xe_{p-1}\in\pp_K^{1+pt_2}.
\]
\end{proposition}
\begin{proof}

Trace transitivity gives, for every $z\in K$,
\[
 g(z)=\Tr_{L_2/F}\bigl(\Tr_{K/L_2}((e_p-a)z)\bigr).
\]
Write $e_p-a=(Y_0-a)+\sum_{l=1}^{p-1}Y_l\Delta^l$.
Lemma~\ref{O:A:lem:bcoeff} gives $Y_1\in\OO_{L_2}$ and
$Y_0-a,Y_l\in\pp_{L_2}^{t_2}$ for $2\le l<p$.
The power traces $\Tr_{K/L_2}(\Delta^d)$ are integral for
$0\le d\le2p-2$. Moreover, by \eqref{O:A:eq:moments} and
$\Delta^p=\Delta+a$,
\[
 \Tr_{K/L_2}(\Delta^{j+1})\in\pp_{L_2}^{t_2}
 \qquad(0\le j<p,\ j\ne p-2);
\]
for $j=p-1$ this is $pa$, with $v_{L_2}(pa)\ge t_2$ by
\eqref{O:A:eq:pbound}, and in the other nonexceptional cases it is zero.
Hence, if $x\in\pp_{L_2}$ and $j\ne p-2$, every term in
\[
 \begin{aligned}
 \Tr_{K/L_2}((e_p-a)x\Delta^j)
   &=x(Y_0-a)\Tr_{K/L_2}(\Delta^j)\\
   &\quad+\sum_{l=1}^{p-1}xY_l\Tr_{K/L_2}(\Delta^{l+j})
 \end{aligned}
\]
belongs to $\pp_{L_2}^{1+t_2}$. Its trace therefore belongs to
$\pp_F^{1+t_2}$, proving the first assertion.

For $j=p-2$, put $\omega=1+(p-1)Y_1+paY_2$.
Equation~\eqref{O:A:eq:moments} gives
$g(x\Delta^{p-2})=\Tr_{L_2/F}(x(\omega-1))$.
Lemma~\ref{O:A:lem:bcoeff} gives
\[
 \omega+e_{p-1}=p(1-s_0)+(s_0+e_{p-1})+(p-1)(Y_1-1+s_0)+paY_2.
\]
After multiplication by $x$, every term except $s_0+e_{p-1}$ has valuation at
least $(1+pt_2)$, by \eqref{O:A:eq:pbound} and
Lemma~\ref{O:A:lem:bcoeff}. For the remaining term,
Lemma~\ref{O:A:lem:coeff} gives
$v_K(e_{p-1}+s_0)\ge p(p-1)(t_2-t_1)+(p-1)t_1$.
Since $v_K(x)=v_K(x\Delta^{p-2})+(p-2)t_1$, the required lower bound differs from this one by
$(p-3)t_1+(p-1)^2(t_2-t_1)\ge0$. Thus
$x\omega+xe_{p-1}\in\pp_K^{1+pt_2}$.
\end{proof}

\section{Compatible characters in totally ramified odd degree}\label{sec:odd-models}
\subsection{Construction of compatible characters of minimal conductor}\label{O:sec:models}
Retain the fields of Section~\ref{O:sec:as}, with $L_1/F$ of
smallest lower ramification break $t_1$ and $L_2/F$ of break $t_2$. Proposition~\ref{O:A:prop:AS} gives $p\nmid t_1$. 
Put $T'=t'+1$. Retain the generator $\Delta$ of $K/L_2$ from
Proposition~\ref{O:A:prop:AS},  and the coefficients $e_i$   for  $\Delta$.
Choose a nontrivial norm character $\tau_2\in S(L_2/F)$.

Fix a nontrivial additive character $\psi_F$ of $F$ and put
$\psi_L=\psi_F\circ\Tr_{L/F}$. Choose $\alpha_2\in F^\times$ so
\[
 \tau_2(1-z)=\psi_F(\alpha_2 P(z))
 \qquad(z\in\pp_F^{\ceil{(t_2+1)/p}}).
\]
Put $\Psi_L(x)=\psi_L(\alpha_2 x)$. Their largest trivial ideals are
 $\pp_F^{t_2+1}, \pp_{L_1}^{T'},\pp_{L_2}^{t_2+1},\pp_K^{T'}$.
Set
\[
 m_1=1+t_1+t',\quad m_2=1+t_1+t_2,\quad
 q_i=\ceil{m_i/p}.
\]
Define characters on $U_{L_i}^{q_i}$ by
\begin{equation}\label{O:M:models}
 R_1(1-z)=\Psi_{L_1}(e_pP(z)),\qquad R_2(1-z)=\Psi_{L_2}(aP(z)).
\end{equation}
Lemma~\ref{O:I:logchart}, with $r=q_i$ and $M=m_i$, makes
these formulas characters of $U_{L_i}^{q_i}$.
Their conductors are $m_i$, since $v_{L_1}(e_p)=v_{L_2}(a)=-t_1$.

\begin{lemma}[Norm filtration at the prescribed depths]\label{O:M:preimage-lemma}
One has
\begin{equation}\label{O:M:preimage}
 \Nm_{K/L_1}^{-1}(U_{L_1}^{q_1})=U_K^{q_1},\qquad
 \Nm_{K/L_2}(U_K^{q_1})\subset U_{L_2}^{q_2}.
\end{equation}
\end{lemma}
\begin{proof}
Writing
$q_1=(t_2-t_1)+\lceil(2t_1+1)/p\rceil$ and
$q_2=\lceil(2t_1+(t_2-t_1)+1)/p\rceil$ gives
\begin{equation}\label{O:M:q-depths}
 q_1\ge q_2,\qquad q_1\le t',\qquad q_2\le t_2,\qquad
 \floor{[q_1+(p-1)(t_1+1)]/p}\ge q_2.
\end{equation}
For the last inequality it is enough that
$q_1+(p-1)(t_1+1)\ge pq_2$; use
$pq_2\le2t_1+(t_2-t_1)+p$ and
$(p-3)t_1+\lceil(2t_1+1)/p\rceil-1\ge0$.
The middle two inequalities in \eqref{O:M:q-depths} follow from
$2t_1+1\le pt_1$ and $t_2-t_1\ge0$.

If $\Nm_{K/L_1}(x)\in U_{L_1}^{q_1}$, total ramification gives
$v_K(x)=0$ and $\bar x^p=1$, hence $x\in U_K^1$. The graded norm maps
at depths less than $t'$ are injective by \eqref{U:graded-norm}. Since \eqref{O:M:q-depths} gives $q_1\le t'$, one gets
$x\in U_K^{q_1}$, and the reverse inclusion follows from the same norm
filtration. The second inclusion in \eqref{O:M:preimage} follows from
\eqref{U:symmetric} and \eqref{O:M:q-depths}. The same argument
is unchanged after an unramified base extension.
\end{proof}

\begin{lemma}[Cancellation on monomials]\label{O:M:monophase}
With $g$ and $h$ as in Proposition~\ref{O:A:prop:g} and
Proposition~\ref{O:A:thm:h}, for $x\in L_2$, $0\le j<p$,
$v_K(x\Delta^j)\ge q_1$, one has
\[
 \Psi_F\bigl(g(x\Delta^j)+h(x\Delta^j)\bigr)=1.
\]
\end{lemma}
\begin{proof}
Put $K_0=1+pt_2$.
We have $v_K(x\Delta^j)\ge q_1\ge1+(t_2-t_1)$, hence $x\in\pp_{L_2}$.
For $j\ne p-2$, Proposition~\ref{O:A:prop:g} puts $g(x\Delta^j)$
in $\pp_F^{t_2+1}$. Proposition~\ref{O:A:thm:h} puts $h(x\Delta^j)$
in $F\cap\pp_{L_1}^{1+pt_2}=\pp_F^{t_2+1}$.

For $j=p-2$, $pv_{L_2}(x)\ge q_1+(p-2)t_1\ge t_2+1$, so $v_{L_2}(x)\ge\ceil{(t_2+1)/p}$.
The difference between $p-1$ and $-1$ in
Proposition~\ref{O:A:thm:h} is in $\pp_{L_1}^{K_0}$ because
\[
 v_{L_1}(p)+pv_{L_2}(x)\ge(p-1)t_2+q_1+(p-2)t_1\ge K_0,
\]
and $e_{p-1}$ is integral. Also, by \eqref{O:A:eq:ei},
\[
 v_K(xe_{p-1})\ge q_1+(p-2)t_1+p(p-1)(t_2-t_1)
       \ge1+t_2+(p-3)t_1.
\]
Proposition~\ref{O:A:prop:g} supplies $\omega$ with
$g(x\Delta^{p-2})=\Tr_{L_2/F}(x(\omega-1))$ and
$x\omega+xe_{p-1}\in\pp_K^{K_0}$. Apply Lemma~\ref{O:A:lem:H17}
with $w=-e_{p-1}$ and $y=x\omega$. Since $\Nm_{K/L_1}(-xe_{p-1})=\Nm_{L_2/F}(x)(-e_{p-1})^p$, it gives
$\Nm_{L_2/F}(x)(-e_{p-1})^p\equiv \Nm_{L_2/F}(x\omega)\pmod{\pp_{L_1}^{K_0}}$.
The approximation also implies
$v_{L_2}(x\omega)\ge\ceil{(t_2+1)/p}$. Thus $g(x\Delta^{p-2})+h(x\Delta^{p-2})$ is congruent in $F$ to
\[
 -\Tr_{L_2/F}(x)-\Nm_{L_2/F}(x)+\Tr_{L_2/F}(x\omega)+\Nm_{L_2/F}(x\omega)\pmod{\pp_F^{t_2+1}}.
\]
Corollary~\ref{O:I:conversion}, applied to $L_2/F$ and $\tau_2$, gives
character value one for each pair.
\end{proof}

For $z\in\pp_K^{q_1}$ put
\begin{equation}\label{O:M:poly}
 F_z(Z)=\Tr_{L_1/F}\bigl(e_pP(1-\Nm_{K/L_1}(1-Zz))\bigr)
       -\Tr_{L_2/F}\bigl(aP(1-\Nm_{K/L_2}(1-Zz))\bigr).
\end{equation}
Its degree is at most $p(p-1)<p^2$.

\begin{lemma}[Coefficients of the comparison polynomial]\label{O:M:polycoeff}
Every coefficient of $F_z$ defines an additive character of order dividing
$p$ on $\OO_F$. Its degree-one coefficient is $g(z)$; its degree-$p$
coefficient is $h(z)$ modulo $\pp_F^{t_2+1}$.
\end{lemma}
\begin{proof}
Equations~\eqref{U:symmetric} and~\eqref{O:M:q-depths} show that
the coefficients of $1-\Nm_{K/L_i}(1-Zz)$ belong to
$\pp_{L_i}^{q_i}$. Products of these
coefficients have at least that valuation. Multiplication by $e_p$ or $a$
loses $t_1$. In mixed characteristic $v_{L_i}(p)\ge(p-1)t_2$, and
\[
 (p-1)t_2-t_1+q_1\ge T',\qquad
 (p-1)t_2-t_1+q_2\ge t_2+1.
\]
For the first inequality the difference is
$(p-3)t_1+\lceil(2t_1+1)/p\rceil-1$; the second follows from
$(p-3)t_1+(p-2)(t_2-t_1)+q_2-1\ge0$.
Thus every coefficient defines a character of order dividing $p$
on $\OO_F$.  When $\operatorname{char}F=p$, this also follows directly
from the fact that the additive group of $F$ has exponent $p$.

The linear coefficient follows by differentiation at $Z=0$.
At degree $p$, the norm terms give $h(z)$. All other terms contain at
least two intermediate elementary coefficients of total weight $p$.
If $K/L_i$ has break $u$, their valuation after multiplication by $e_p$ or
$a$ is at least
$q_1-t_1+2(p-1)u/p$.
For $u=t_1$ this is strictly greater than $t_2$; for $u=t'$ it is
strictly greater than $T'-1$, using
$q_1=(t_2-t_1)+\lceil(2t_1+1)/p\rceil$. Integral valuations therefore put
each term in $\pp_{L_2}^{t_2+1}$ or $\pp_{L_1}^{T'}$,
respectively. Trace to $F$ puts it in
$\pp_F^{t_2+1}$. 
\end{proof}

\begin{lemma}[Compatibility of the prescribed characters]\label{O:M:compat}
The characters \eqref{O:M:models} satisfy $R_1\circ \Nm_{K/L_1}=R_2\circ \Nm_{K/L_2}$ on $U_K^{q_1}$.
\end{lemma}
\begin{proof}

Define on $U_K^{q_1}$
\[
 \Xi(u)=R_1(\Nm_{K/L_1}(u))R_2(\Nm_{K/L_2}(u))^{-1}.
\]
It is a multiplicative character by \eqref{O:M:preimage}, and
$\Xi(1-Zz)=\Psi_F(F_z(Z))$ for $z\in\pp_K^{q_1}$ and integral base scalars $Z$.
Put
$M_0=\max\{q_1,\,p\max(m_1,m_2)\}$.
The norm filtration in Section~\ref{U:norm-filtration} gives
\[
 \Nm_{K/L_i}(U_K^{M_0})\subset U_{L_i}^{m_i}\qquad(i=1,2),
\]
so $R_1\circ \Nm_{K/L_1}=R_2\circ \Nm_{K/L_2}=1$ on $U_K^{M_0}$.

Suppose $q_1\le s<M_0$ and that $\Xi$ is trivial on $U_K^{s+1}$.
Since $2s\ge s+1$, the map
\[
 z\longmapsto\Xi(1-z),\qquad z\in\pp_K^s,
\]
is additive: the error in replacing multiplication of $1-x,1-y$ by
addition of $x,y$ belongs to $U_K^{s+1}$.

The same statement holds after an unramified base extension, after
pulling $\Xi$ back by the norm; the induction hypothesis remains valid
because the norm preserves the unit depth. Choose such an extension of
degree prime to $p$ and absolute residue degree at least $3$. For
$z\in\pp_K^s$, Lemmas~\ref{O:M:rigidity} and \ref{O:M:polycoeff} applied
to \eqref{O:M:poly} give
\[
 \Xi(1-z)=\Psi_F(g(z)+h(z)).
\]

Write $z=\sum_{j=0}^{p-1}x_j\Delta^j$. Since $p\nmid t_1$, the valuations
of the terms are distinct modulo $p$, so every $x_j\Delta^j$ still lies
in $\pp_K^s$. Additivity and Lemma~\ref{O:M:monophase} give
\[
 \Xi(1-z)=\prod_{j=0}^{p-1}\Xi(1-x_j\Delta^j)=1.
\]
Thus $\Xi$ is trivial on $U_K^s$. Descending from $M_0$ to $q_1$ proves
$R_1\circ \Nm_{K/L_1}=R_2\circ \Nm_{K/L_2}$ on $U_K^{q_1}$.
\end{proof}

\begin{proposition}[Extension to compatible characters of minimal conductor]\label{O:M:global}
Every extension $\chi_2$ of $R_2$ to $L_2^\times$ has a compatible
extension $\chi_1$ of $R_1$ to $L_1^\times$. Both characters have
conductors $m_i$, and their common pullback is primitive.
\end{proposition}
\begin{proof}
Fix an extension $\chi_2$ of $R_2$; such extensions exist by
Lemma~\ref{C:character-extension}. Put $\Theta=\chi_2\circ \Nm_{K/L_2}$.
Let $\rho$ generate $\Gal(K/L_1)$, whose break is $t'$.
For $z\in K^\times$, Lemma~\ref{C:leading-ramification} gives
$\rho(z)/z\in U_K^{t'}\subset U_K^{q_1}$.
Its $\Nm_{K/L_1}$-norm is $1$. By Lemma~\ref{O:M:compat} and \eqref{O:M:preimage},
\[
 \Theta(\rho(z)/z)=R_2(\Nm_{K/L_2}(\rho(z)/z))
                 =R_1(\Nm_{K/L_1}(\rho(z)/z))=1.
\]
Hence $\Theta$ is invariant under $\Gal(K/L_1)$ on all of $K^\times$.

Lemmas~\ref{C:hilbert90} and~\ref{C:character-extension} give a
character $\widetilde\chi_1$ of $L_1^\times$ with
$\widetilde\chi_1\circ \Nm_{K/L_1}=\Theta$.
On $U_{L_1}^{q_1}\cap \Nm_{K/L_1}(K^\times)$ it agrees with $R_1$: any norm preimage belongs to
$U_K^{q_1}$ by \eqref{O:M:preimage}, where Lemma~\ref{O:M:compat} applies.
Thus $\widetilde\chi_1/R_1$ on $U_{L_1}^{q_1}$ is a character of the image of $U_{L_1}^{q_1}$
in $L_1^\times/\Nm_{K/L_1}(K^\times)$. Extend it across this finite cyclic quotient
and divide $\widetilde\chi_1$ by the resulting norm character. This gives
$\chi_1|_{U_{L_1}^{q_1}}=R_1$, without changing the common pullback.

The restrictions $R_i$ are nontrivial at depth $m_i-1$ and trivial
at depth $m_i$, so the conductors are $m_i$.

Let $\sigma$ generate $\Gal(L_2/F)$. Since $p\nmid t_1$,
Lemma~\ref{C:leading-ramification} shows that $v\mapsto\sigma^{-1}(v)/v$
induces an isomorphism
\[
 U_{L_2}^{t_1}/U_{L_2}^{t_1+1}
 \longrightarrow U_{L_2}^{t_1+t_2}/U_{L_2}^{t_1+t_2+1}.
\]
The character $\chi_2$ is nontrivial on the target quotient because
$m_2=t_1+t_2+1$. Thus $\chi_2^\sigma/\chi_2$ is nontrivial.
Since the common pullback $\Theta$ is $\Gal(K/F)$-invariant,
$\chi_2^\sigma/\chi_2$ is trivial on $\Nm_{K/L_2}(K^\times)$ and hence
belongs to $S(K/L_2)$. By \eqref{U:norm-character-conductor},
its conductor is $t_1+1$. If $\Theta$ descended from $F$,
then $\chi_2^\sigma/\chi_2$ would be trivial, a contradiction. Here the action of $\Gal(L_2/F)$ on $L_2^\times/\Nm_{K/L_2}(K^\times)$
is trivial, since
$|\operatorname{Aut}(C_p)|=p-1$.
\end{proof}

\begin{lemma}[A lower bound for the conductor]\label{O:G:lower}
Let $(\varphi_1,\varphi_2)$ be a primitive compatible pair for $K/F$.
If $L_i/F$ and $K/L_i$ have breaks $t_i$ and $u_i$, respectively, then
\[
 m_{L_i}(\varphi_i)\ge1+t_i+u_i\qquad(i=1,2).
\]
\end{lemma}
\begin{proof}
Let $\sigma_i$ generate $\Gal(L_i/F)$. By Proposition~\ref{U:conjugacy},
$\mu_i=\varphi_i^{\sigma_i}/\varphi_i$ is a nontrivial norm character for
$K/L_i$, of conductor $u_i+1$. Choose $y=1+z\in U_{L_i}^{u_i}$ with
$\mu_i(y)\ne1$. Lemma~\ref{C:leading-ramification} gives
\[
 v_{L_i}(\sigma_i^{-1}(z)-z)\ge t_i+v_{L_i}(z),
 \qquad \sigma_i^{-1}(y)/y\in U_{L_i}^{t_i+u_i}.
\]
The value of $\varphi_i$ on this quotient is $\mu_i(y)\ne1$.
Thus $m_{L_i}(\varphi_i)>t_i+u_i$.
\end{proof}

\begin{lemma}[The twisting character for a prescribed pair]\label{O:M:lambdaformula}
Let $(\varphi_1,\varphi_2)$ be a primitive compatible pair with
$\varphi_2|_{U_{L_2}^{m_2}}=1$, and let $(\chi_1,\chi_2)$ be a pair
supplied by Proposition~\ref{O:M:global}. There exist $j\in\{1,\ldots,p-1\}$,
a quasi-character $\lambda$ of $F^\times$, and $w\in L_2$ with
$v_{L_2}(w)\ge-\lfloor t_1/p\rfloor$ such that
\begin{equation}\label{O:M:actual2}
 \varphi_2=\chi_2^j(\lambda\circ \Nm_{L_2/F}).
\end{equation}
For every $z\in\pp_{L_2}^{q_2}$,
\[
 (\lambda\circ \Nm_{L_2/F})(1-z)=\psi_{L_2}\bigl(j\alpha_2(w^p-w)P(z)\bigr).
\]
\end{lemma}
\begin{proof}
Lemma~\ref{O:G:lower} gives $m_{L_i}(\varphi_i)\ge m_i$.
The upper bound on $\varphi_2$ therefore gives $m_{L_2}(\varphi_2)=m_2$.
Since $m_1>t'+1$ and $m_2>t_1+1$, the common norm pullback and
\eqref{U:norm-pullback-conductor} give
\[
 p\,m_{L_1}(\varphi_1)-(p-1)(t'+1)
   =p\,m_2-(p-1)(t_1+1).
\]
Using $t'=t_1+p(t_2-t_1)$, we obtain
\begin{equation}\label{O:M:forced-conductors}
 m_{L_1}(\varphi_1)=m_1,\qquad m_{L_2}(\varphi_2)=m_2.
\end{equation}

For a generator $\sigma$ of $\Gal(L_2/F)$, the characters
$\varphi_2^\sigma/\varphi_2$ and $\chi_2^\sigma/\chi_2$
generate $S(K/L_2)$ by Proposition~\ref{U:conjugacy}. Choose
$j\in\{1,\ldots,p-1\}$ so that these characters agree after
raising the second to the $j$th power. Then $\varphi_2\chi_2^{-j}$ is invariant under
$\Gal(L_2/F)$. Lemmas~\ref{C:hilbert90}
and~\ref{C:character-extension} give a quasi-character $\lambda$
of $F^\times$ satisfying \eqref{O:M:actual2}.
Put
\[
 h_0=\floor{t_1/p},\qquad
 q_0=h_0+\ceil{(t_2+1)/p}.
\]
Equation~\eqref{U:norm-pullback-conductor} and
$m_{L_2}(\lambda\circ \Nm_{L_2/F})\le m_2$ imply
$m_F(\lambda)\le t_2+1+h_0$.
The following inequalities include the case $p\mid t_2$:
\begin{equation*}
 q_0\le q_2,\quad pq_0\ge t_2+1+h_0,\quad
 \Nm_{L_2/F}(U_{L_2}^{q_2})\subset U_F^{q_0},\quad
 (p-2)t_2-ph_0+q_0-1\ge0.
\end{equation*}
To verify them, write $t_1=ph_0+b_0$, $1\le b_0<p$.
Then $q_2=h_0+\lceil(b_0+t_2+1)/p\rceil\ge q_0$.
For the norm inclusion use \eqref{U:trace-ideal} and
$q_2+(p-1)(t_2+1)\ge pq_0$; the latter follows from
$t_2\ge t_1=ph_0+b_0$ and the preceding ceiling formula.
The last inequality follows by replacing $t_2$ by $t_1$ and $q_0$ by
$h_0+\lceil(t_1+1)/p\rceil$; the resulting lower bound is nonnegative.

Put $\Psi'_L=\Psi_L(j\,\cdot)$.
Lemma~\ref{O:I:logchart} and \cite[Lemma~4.1]{Ueda}, applied to
$\pp_F^{q_0}/\pp_F^{t_2+1+h_0}$, give a coefficient
$\gamma\in F$, $v_F(\gamma/(j\alpha_2))\ge-h_0$, such that
\[
 \lambda(1-x)=\psi_F(\gamma P(x))\qquad(x\in\pp_F^{q_0}).
\]
If the restriction is trivial, take $\gamma=w=0$.
The inclusion $\Nm_{L_2/F}(U_{L_2}^{q_2})\subset U_F^{q_0}$ proved above
then proves the formula. Otherwise \eqref{U:norm-approximation},
applied to $L_2/F$, the
element $\gamma/(j\alpha_2)$, and $s=t_2$, gives
\begin{equation*}
 w\in L_2,\qquad v_{L_2}(w)=v_F(\gamma/(j\alpha_2))\ge-h_0,\qquad
 \gamma/(j\alpha_2)-\Nm_{L_2/F}(w)\in\pp_F^{v_{L_2}(w)+t_2}.
\end{equation*}

Set $x=1-\Nm_{L_2/F}(1-z)\in\pp_F^{q_0}$. Since
\[
 (\gamma/(j\alpha_2)-\Nm_{L_2/F}(w))P(x)\in\pp_F^{v_{L_2}(w)+t_2+q_0}
                         \subseteq\pp_F^{t_2+1},
\]
one has $\psi_F(\gamma P(x))=\Psi'_F(\Nm_{L_2/F}(w)P(x))$.
The norm-logarithm congruence from Lemma~\ref{O:I:normlog}, now with target $t_2+1+h_0$ and
input depth $q_0\le q_2$, yields
\[
 (\lambda\circ \Nm_{L_2/F})(1-z)
 =\Psi'_{L_2}(\Nm_{L_2/F}(w)P(z))\Psi'_F(\Nm_{L_2/F}(wP(z))).
\]
Since $v_{L_2}(w)+q_2\ge\ceil{(t_2+1)/p}$, Corollary~\ref{O:I:conversion},
applied to $L_2/F$ and $\tau_2^j$ with additive scale $j\alpha_2$, changes the
second factor to $\Psi'_{L_2}(-wP(z))$.
Lemma~\ref{O:A:lem:power}, applied to $w\in L_2$ over $F$, gives
$v_{L_2}(w^p-\Nm_{L_2/F}(w))\ge p v_{L_2}(w)+(p-1)t_2$. Multiplication by $P(z)$
puts this error in $\pp_{L_2}^{t_2+1}$ by $(p-2)t_2-ph_0+q_0-1\ge0$. This proves the formula.
\end{proof}

\begin{proposition}[Character formulas for a prescribed pair]
\label{O:M:realize}
Let $(\varphi_1,\varphi_2)$ be a primitive compatible pair and suppose
$\varphi_2|_{U_{L_2}^{m_2}}=1$. Then $m_{L_i}(\varphi_i)=m_i$ for $i=1,2$.
Moreover there exist $\alpha_2'\in\{\alpha_2,2\alpha_2,\ldots,(p-1)\alpha_2\}$ and
$\Delta\in K$ such that
\[
 \Delta^p-\Delta\in L_2,\qquad v_K(\Delta)=-t_1.
\]
For some $\rho\in\Gal(L_1/F)$,
\[
 \begin{aligned}
 \varphi_1^\rho(1-z)&=\psi_{L_1}(\alpha_2'\Nm_{K/L_1}(\Delta)P(z))
       &&(z\in\pp_{L_1}^{q_1}),\\
 \varphi_2(1-z)&=\psi_{L_2}(\alpha_2'\Nm_{K/L_2}(\Delta)P(z))
       &&(z\in\pp_{L_2}^{q_2}).
 \end{aligned}
\]
The characters $\varphi_1^\rho$ and $\varphi_1$ have the same local
constant.
\end{proposition}
\begin{proof}
Choose $(\chi_1,\chi_2)$ by Proposition~\ref{O:M:global}, and then
choose $j,\lambda,w$ by Lemma~\ref{O:M:lambdaformula}. Put
$\alpha_2'=j\alpha_2$. The conductor equalities are
\eqref{O:M:forced-conductors}. Write $t_1=ph_0+b_0$, where
$h_0=\lfloor t_1/p\rfloor$ and $1\le b_0<p$, and set
$a_\lambda=a+w^p-w$. Its valuation is $-t_1$ because
$pv_{L_2}(w)\ge-ph_0=-t_1+b_0>-t_1$. In mixed characteristic the error of
$\Delta+w$ in $X^p-X-a_\lambda$ has valuation at least
\[
 v_K(p)-(p-1)t_1+pv_{L_2}(w)\ge v_K(p)-pt_1+b_0>0.
\]
Its derivative is a unit. Lemma~\ref{D:newton}, applied after translating
by $\Delta+w$, gives a root $\Delta_\lambda$ of valuation $-t_1$. In equal characteristic
$\Delta+w$ is already that root. It generates $K/L_2$ because $p\nmid t_1$.
Thus $\Nm_{K/L_2}(\Delta_\lambda)=a_\lambda$.

Equation~\eqref{O:M:actual2} and Lemma~\ref{O:M:lambdaformula}
give the stated formula for $\varphi_2$ on $U_{L_2}^{q_2}$.

Lemmas~\ref{O:A:lem:coeff} and~\ref{O:A:lem:twisted} apply to
$\Delta_\lambda,a_\lambda$.
Apply Proposition~\ref{O:M:global} to the new generator and norm character, choosing
$\varphi_2$ itself as the second extension. It supplies $\widetilde\varphi_1$ with the stated formula on $U_{L_1}^{q_1}$
and the same norm pullback as $\varphi_2$. Thus $\widetilde\varphi_1/\varphi_1\in S(K/L_1)$.
By Proposition~\ref{U:conjugacy}, there is $\rho\in\Gal(L_1/F)$ with
$\widetilde\varphi_1=\varphi_1^\rho$.
Equation~\eqref{U:automorphism-invariance} gives equality of their local constants.
\end{proof}


\subsection{Norms of the symmetric coefficient}\label{O:sec:r2}
Retain the notation of Section~\ref{O:sec:models} and put
$q=(t_2-t_1)+\ceil{(t_1+1)/p}$.
Retain $\tau_2$ and its coefficient $\alpha_2$. Choose a nontrivial
norm character $\tau_1\in S(L_1/F)$ and a coefficient $\alpha_1$ by
\eqref{O:I:tauchart}, applied to $L_1/F$ with the same additive character $\psi_F$.
Put $\varrho=\alpha_1/\alpha_2$, so $v_F(\varrho)=t_2-t_1$.
Choose $\chi_1$ as the character on $L_1^\times$ supplied by
Proposition~\ref{O:M:global} for this $\tau_2$ and $\Delta$. Then
\begin{equation*}
 \chi_1(1-z)=\Psi_{L_1}(e_pP(z))\qquad(z\in\pp_{L_1}^{q_1}).
\end{equation*}
Its exact conductor is $m_1$. Let $\sigma$ be the restriction to $L_1$ of the
automorphism of $K/L_2$ carrying $\Delta$ to the root congruent to
$\Delta+1$ modulo $\pp_K$.

\begin{lemma}[Truncated-logarithm formula for the conjugate quotient]\label{O:R:comm-log}
For $z\in\pp_{L_1}^{q}$,
\[
 \frac{\chi_1^\sigma(1-z)}{\chi_1(1-z)}
       =\Psi_{L_1}((\sigma(e_p)-e_p)P(z)).
\]
\end{lemma}
\begin{proof}
Put $d=\sigma^{-1}(z)-z$, so $v_{L_1}(d)\ge q+t_1$.
The quotient $(1-\sigma^{-1}(z))/(1-z)$ is $1-d/(1-z)$.
We have $q+t_1\ge q_1$, hence the formula for $\chi_1(1-z)$ above applies to this quotient even
though it need not apply separately to $1-z$.

As an integral formal power series in $z,d$,
$P(d/(1-z))-P(z+d)+P(z)$
is divisible by $d$ and has total degree at least $p$. Its denominators
are products of integers $1,\ldots,p-1$ and powers of $1-z$.
Its valuation is therefore at least $(p-1)q+(q+t_1)=pq+t_1\ge T'+t_1=m_1$.
Multiplication by $e_p$ puts it in the kernel of $\Psi_{L_1}$.
Finally trace invariance gives
$\Psi_{L_1}(e_pP(\sigma^{-1}(z)))=\Psi_{L_1}(\sigma(e_p)P(z))$.
\end{proof}

\begin{lemma}[Conjugate increment of the norm coefficient]\label{O:R:shift}
One has $\sigma(e_p)-e_p\equiv1+e_{p-1}\pmod{\pp_{L_1}^{T'-q}}$.
\end{lemma}
\begin{proof}
Norm expansion gives
\[
 \Nm_{K/L_1}(\Delta+1)-e_p=1+e_{p-1}+\sum_{i=1}^{p-2}e_i.
\]
Each summand has valuation at least
$(p-1)t_2-(p-2)t_1=t_1+(p-1)(t_2-t_1)\ge T'-q$ by \eqref{O:A:eq:ei}.

If $\operatorname{char}F=p$, then $\sigma(\Delta)=\Delta+1$, so
$\sigma(e_p)=\Nm_{K/L_1}(\Delta+1)$ and the assertion follows from the preceding
norm expansion.  Assume mixed characteristic.  Write
$\sigma(\Delta)=(\Delta+1)(1+z)$ and put $V=v_K(p)$.
The elementary Hensel estimate gives
$v_K(z)\ge V-(p-2)t_1$.
For intermediate coefficients in $\Nm_{K/L_1}(1+z)-1$, the bound in
\eqref{U:symmetric} is smallest at the trace coefficient, and after multiplying by
$\Nm_{K/L_1}(\Delta+1)$ it gives at least
\begin{align*}
 -t_1+\floor{[V-(p-2)t_1+(p-1)(t'+1)]/p}
 &=V/p-t_1+(p-1)(t_2-t_1)+\ceil{(t_1+1)/p}\\
 &\ge(p-2)t_1+2(p-1)(t_2-t_1)+\ceil{(t_1+1)/p}\ge T'-q.
\end{align*}
The norm term has larger valuation.
\end{proof}

Lemmas~\ref{O:R:comm-log} and~\ref{O:R:shift} show that
$\mu=\chi_1^\sigma/\chi_1$ satisfies
\begin{equation}\label{O:R:S}
 \mu(1-z)=\Psi_{L_1}((1+e_{p-1})P(z))\qquad(z\in\pp_{L_1}^{q}).
\end{equation}
It is a nontrivial norm character for $K/L_1$, of conductor $T'$.

Apply \eqref{U:norm-approximation} to $L_1/F$, the element
$\varrho$, and $s=t_1$. Choose $c\in L_1^\times$ with
\begin{equation*}
 \Nm_{L_1/F}(c)=\varrho\varepsilon,\qquad\varepsilon\in U_F^{t_1},\qquad v_{L_1}(c)=(t_2-t_1).
\end{equation*}
\begin{lemma}[Truncated-logarithm formula for the norm pullback]\label{O:R:T}
On $U_{L_1}^{q}$ one has
\[
 (\tau_2\circ \Nm_{L_1/F})(1-z)=\Psi_{L_1}((1-\varrho/c)P(z)).
\]
\end{lemma}
\begin{proof}
For $x=1-\Nm_{L_1/F}(1-z)$, Section~\ref{U:norm-filtration} gives
$v_F(x)\ge\ceil{(t_2+1)/p}$, so \eqref{O:I:tauchart}, applied to $L_2/F$, gives the formula for $\tau_2$. The norm-logarithm congruence of Lemma~\ref{O:I:normlog} on $L_1/F$, with target
$t_2+1$ and input depth $(t_2-t_1)+\ceil{(t_1+1)/p}=q$, gives
\[
 (\tau_2\circ \Nm_{L_1/F})(1-z)=\Psi_{L_1}(P(z))\Psi_F(\Nm_{L_1/F}(P(z))).
\]
The error in replacing the second factor by
$\psi_F(\alpha_1 \Nm_{L_1/F}(P(z)/c))$ has valuation, before multiplication by
$\alpha_2$, at least $q+t_1\ge t_2+1$. The error is $\Nm_{L_1/F}(P(z))(\varepsilon^{-1}-1)$.
Since $v_{L_1}(P(z)/c)\ge q-(t_2-t_1)=\ceil{(t_1+1)/p}$, apply
Corollary~\ref{O:I:conversion} to $L_1/F$ and $\tau_1$, with additive
scale $\alpha_1$. It changes this factor to
$\psi_{L_1}(-\alpha_1 P(z)/c)$.
\end{proof}

\begin{proposition}[Identification of the conjugate quotient]\label{O:R:normalize}
One has $\mu=\tau_2\circ \Nm_{L_1/F}$.
\end{proposition}
\begin{proof}
By Proposition~\ref{D:norm-separation} and Lemma~\ref{D:crossed-norm},
$\tau_2\circ\Nm_{L_1/F}$ is a nontrivial character in $S(K/L_1)$; by
\eqref{U:norm-character-conductor}, its conductor is $T'$. Thus both are
nontrivial characters of the same cyclic norm quotient of order $p$. Their
restrictions to $U_{L_1}^{t'}/U_{L_1}^{t'+1}$ distinguish the $p-1$ possible
generators.

If $t_2>t_1$, \eqref{O:A:eq:ei} gives
$v_{L_1}(e_{p-1})\ge(p-1)(t_2-t_1)$, while
$v_{L_1}(\varrho/c)=(p-1)(t_2-t_1)$; both are positive. For
$z\in\pp_{L_1}^{t'}$, one has
$P(z)\equiv z\pmod{\pp_{L_1}^{t'+1}}$.
The formulas in \eqref{O:R:S} and Lemma~\ref{O:R:T} therefore agree on
$U_{L_1}^{t'}$. Hence the characters are equal.

If $t_2=t_1$, then $t'=t_1$. Use $z=Y/e_p$ for $Y\in\OO_F$,
whose classes span $\pp_{L_1}^{t_1}/\pp_{L_1}^{t_1+1}$.
For these $z$, $P(z)\equiv z\pmod{\pp_{L_1}^{t_1+1}}$.
The reciprocal-coefficient identities
\[
 -e_{p-1}/e_p=-\Tr_{K/L_1}(1/\Delta),\qquad
 \Tr_{K/L_2}(1/\Delta)=-1/a
\]
follow from the minimal polynomials; the second uses the coefficient
$-1$ of $X$ in $X^p-X-a$. Trace transitivity gives
$\Psi_{L_1}(-e_{p-1}z)=\Psi_{L_2}(Y/a)$. Since
$v_{L_2}(Y/a)\ge t_1\ge\ceil{(t_1+1)/p}$, Corollary~\ref{O:I:conversion},
applied to $L_2/F$ and $\tau_2$, together with norm transitivity gives
\[
 \Psi_{L_1}(-e_{p-1}z)=\Psi_{L_2}(Y/a)
 =\Psi_F(-\Nm_{L_2/F}(Y/a))=\Psi_F(-\Nm_{L_1/F}(z)).
\]
The intermediate coefficients of $\Nm_{L_1/F}(1-z)$
other than trace and norm are in $\pp_F^{t_1+1}$ by
\eqref{U:symmetric}. Thus
\[
 (\tau_2\circ \Nm_{L_1/F})(1-z)=\Psi_F(\Tr_{L_1/F}(z)+\Nm_{L_1/F}(z))
                =\Psi_{L_1}(z)\Psi_F(\Nm_{L_1/F}(z)).
\]
The identity just proved identifies this with $\Psi_{L_1}((1+e_{p-1})z)$, which is the
value of $\mu$ by \eqref{O:R:S}. Their restrictions to
$U_{L_1}^{t_1}$ agree, so the norm characters are equal.
\end{proof}

\begin{proposition}[Norm of the symmetric coefficient]\label{O:R:R2}
One has
\begin{equation*}
 \Nm_{L_1/F}(-e_{p-1})\equiv \varrho^{p-1}
       \pmod{\pp_F^{T'-q}}.
\end{equation*}
Consequently, for every integer $m$ with $pm>t_1$,
\[
 \Nm_{L_1/F}(-e_{p-1})\equiv\varrho^{p-1}
       \pmod{\pp_F^{1+t_2-(p-2)m}}.
\]
\end{proposition}
\begin{proof}

Proposition~\ref{O:R:normalize}, \eqref{O:R:S}, and Lemma~\ref{O:R:T} imply
\[
 \Psi_{L_1}((-e_{p-1}-\varrho/c)P(z))=1\qquad\text{for every }z\in\pp_{L_1}^q.
\]
By Lemma~\ref{O:I:logchart}, $P$ is a bijection of $\pp_{L_1}^q$.
By \cite[Lemma~4.1]{Ueda}, the additive pairing gives
\[
 -e_{p-1}-\varrho/c\in\pp_{L_1}^{T'-q}.
\]

Put $d_s=(p-1)(t_2-t_1)$. Both $\varrho/c$ and $e_{p-1}$ have valuation $d_s$, since
$T'-q-d_s=R_0:=t_1+1-\ceil{(t_1+1)/p}\ge1$. The quotient $(-e_{p-1})/(\varrho/c)$ lies in $U_{L_1}^{R_0}$.
Since $1\le R_0\le t_1$, the norm filtration in
Section~\ref{U:norm-filtration} gives
$\Nm_{L_1/F}(-e_{p-1})/\Nm_{L_1/F}(\varrho/c)\in U_F^{R_0}$. As
$v_F(\Nm_{L_1/F}(\varrho/c))=d_s$ and $d_s+R_0=T'-q$, it follows that
\[
 \Nm_{L_1/F}(-e_{p-1})\equiv \Nm_{L_1/F}(\varrho/c)=\varrho^{p-1}\varepsilon^{-1}
                  \pmod{\pp_F^{T'-q}}.
\]
The remaining error has valuation at least
$d_s+t_1\ge d_s+t_1+1-\ceil{(t_1+1)/p}=T'-q$. This proves the first asserted norm congruence, including $t_2=t_1$.

Finally $pm>t_1$ and $p\nmid t_1$ imply $m\ge \ceil{(t_1+1)/p}$, and
\[
 T'-q-[1+t_2-(p-2)m]=(p-2)((t_2-t_1)+m)-\ceil{(t_1+1)/p}\ge0.
\]
The final inequality proves the second assertion.
\end{proof}

\subsection{Norm representatives and stationary-factor decomposition}\label{O:sec:parameters}
\begin{lemma}[Choice of a norm representative and adjusted coefficient]
\label{O:P:choose}
Let $E/F$ be a totally ramified cyclic extension of degree $p$, with
break $t_1\ge1$, and let $\psi_F$ be a nontrivial additive character of $F$.
Let $\mu$ be a nontrivial norm character in $S(E/F)$. Suppose $\alpha_0\in F^\times$ satisfies
\[
 \mu(1-z)=\psi_F(\alpha_0P(z))
 \quad\bigl(z\in\pp_F^{\ceil{(t_1+1)/p}}\bigr),\qquad
 v_F(\alpha_0)=-n_F(\psi_F)-t_1-1.
\]
For every $\gamma\in F^\times$, there is $w\in E^\times$ such that,
with
\[
 \alpha=\frac{\gamma}{\Nm_{E/F}(w)},
\]
one has $\alpha/\alpha_0\in U_F^{t_1}$ and
\[
 \mu(1-z)=\psi_F(\alpha P(z))
 \qquad\bigl(z\in\pp_F^{\ceil{(t_1+1)/p}}\bigr).
\]
\end{lemma}
\begin{proof}
The first assertion is \eqref{U:norm-approximation} applied to
$c=\gamma/\alpha_0$ and $s=t_1$. For $z\in\pp_F^{\ceil{(t_1+1)/p}}$,
Lemma~\ref{O:I:logchart} gives $v_F(P(z))=v_F(z)$, so
\[
 v_F\bigl((\alpha-\alpha_0)P(z)\bigr)
 \ge -n_F(\psi_F)-1+\ceil{(t_1+1)/p}
 \ge -n_F(\psi_F).
\]
Hence $\psi_F(\alpha P(z))=\psi_F(\alpha_0P(z))=\mu(1-z)$.
\end{proof}

\begin{proposition}[Coefficient under norm pullback]\label{O:P:pullback}
Let $E/F$ be totally ramified cyclic of degree $p$ with break $t_1\ge1$.
Fix a nontrivial additive character $\psi_F$ of $F$ and put
$\psi_E=\psi_F\circ\Tr_{E/F}$. Let
$\lambda:F^\times\to\CC^\times$ be a quasi-character of conductor $1+r$, where
$r>t_1$, and let $\gamma$ be a coefficient for $\lambda$ in
\eqref{O:P:fullP}. Fix a nontrivial norm character $\mu\in S(E/F)$
and a coefficient $\alpha_0$ satisfying the hypotheses of
Lemma~\ref{O:P:choose}. Choose $\alpha,w$ by that lemma for this $\gamma$.
Then $\lambda\circ\Nm_{E/F}$ has conductor
$n=1+t_1+p(r-t_1)$ and
\begin{equation*}
 (\lambda\circ\Nm_{E/F})(1-z)
 =\psi_E((\gamma-\alpha w)P(z))
 \qquad\bigl(z\in\pp_E^{\ceil{n/p}}\bigr).
\end{equation*}
Moreover $v_E(\gamma-\alpha w)=p\,v_F(\gamma)$.
\end{proposition}
\begin{proof}
Put $h=r-t_1$ and $q=h+\ceil{(t_1+1)/p}=\ceil{n/p}$.
The conductor follows from \eqref{U:norm-pullback-conductor}.
Since $\alpha/\alpha_0\in U_F^{t_1}$, one has
$v_F(\alpha)=v_F(\alpha_0)$. Together with
$\Nm_{E/F}(w)=\gamma/\alpha$ and total ramification, this gives
$v_E(w)=v_F(\gamma/\alpha_0)=-h$. Hence
$v_E(\alpha w/\gamma)=(p-1)h>0$, so
$v_E(\gamma-\alpha w)=p\,v_F(\gamma)$.

For $z\in\pp_E^q$, the norm $\Nm_{E/F}(1-z)$ lies in
$U_F^{\ceil{(r+1)/p}}$, where \eqref{O:P:fullP} applies to $\lambda$.
The highest-degree term in the norm expansion has valuation at least $q$,
and the intermediate terms at least
$\floor{(q+(p-1)(t_1+1))/p}$. Both bounds are at least
$\ceil{(r+1)/p}$; for the second bound use
$\ceil{(t_1+1)/p}+(p-2)t_1-1\ge0$.
Lemma~\ref{O:I:normlog}, with $h=r-t_1$, gives
\[
 (\lambda\circ\Nm_{E/F})(1-z)
 =\psi_E(\gamma P(z))\psi_F(\gamma\Nm_{E/F}(P(z))).
\]
Since $wP(z)\in\pp_E^{\ceil{(t_1+1)/p}}$, Corollary~\ref{O:I:conversion},
applied to $\mu$ with coefficient $\alpha$, gives
\[
 \psi_E(\alpha wP(z))
 =\psi_F(-\gamma\Nm_{E/F}(P(z))).
\]
Multiplication proves the asserted pullback formula.
\end{proof}

\begin{proposition}[Stationary coefficients of the twisted pair]
\label{O:P:twisted-coefficients}
Let $F\subset L_1,L_2\subset K$ have the ramification data of
Section~\ref{O:sec:models}. Fix a nontrivial additive character $\psi_F$ of
$F$ and put $\psi_{L_i}=\psi_F\circ\Tr_{L_i/F}$. Let
$(\varphi_1,\varphi_2)$ be a primitive compatible pair with
\[
 m_{L_i}(\varphi_i)=m_i\qquad(i=1,2).
\]
Let $\lambda:F^\times\to\CC^\times$ be a quasi-character of conductor $1+r$,
put $m=r-t_2$ and
$\theta_i=\varphi_i(\lambda\circ\Nm_{L_i/F})$ for $i=1,2$, and suppose
$pm>t_1$. Choose $\gamma$ as in \eqref{O:P:fullP} for $\lambda$.

After replacing $\varphi_1$ and $\theta_1$ by the same
$\Gal(L_1/F)$-conjugates, choose $\Delta,\alpha_2'$ as in
Proposition~\ref{O:M:realize}. For this element put
$e_j=e_j^{K/L_1}(\Delta)$ for $1\le j\le p$ and
$a=\Nm_{K/L_2}(\Delta)$.
Choose $\alpha_i\in F^\times$ and $w_i\in L_i^\times$ by
Lemma~\ref{O:P:choose}: for $i=2$, start with the coefficient
$\alpha_2'$ and the norm character it represents in \eqref{O:I:tauchart};
for $i=1$, start with a coefficient in that formula for any nontrivial
norm character of $L_1/F$. These choices satisfy
\[
 \Nm_{L_i/F}(w_i)=\gamma/\alpha_i\quad(i=1,2),\qquad
 \alpha_2/\alpha_2'\in U_F^{t_2}.
\]
Then
\[
 m_{L_1}(\theta_1)=1+t'+pm,\qquad
 m_{L_2}(\theta_2)=1+t_2+pm.
\]
With
\begin{equation}\label{O:P:C}
 C_1=\gamma-\alpha_1w_1+\alpha_2 e_p,\qquad
 C_2=\gamma-\alpha_2 w_2+\alpha_2 a,
\end{equation}
one has $C_i\in L_i^\times$, $v_{L_i}(C_i)=p\,v_F(\gamma)$ for $i=1,2$, and
\begin{equation}\label{O:P:twist-charts}
\begin{aligned}
 \theta_1(1-z)&=\psi_{L_1}(C_1P(z))
 &&\bigl(z\in\pp_{L_1}^{\ceil{(1+t'+pm)/p}}\bigr),\\
 \theta_2(1-z)&=\psi_{L_2}(C_2P(z))
 &&\bigl(z\in\pp_{L_2}^{\ceil{(1+t_2+pm)/p}}\bigr).
\end{aligned}
\end{equation}
\end{proposition}
\begin{proof}
Choose $\Delta$ and $\alpha_2'$ as in Proposition~\ref{O:M:realize},
after conjugating $\varphi_1$ if necessary.
The same conjugation carries $\theta_1$ to the corresponding twist of
$\varphi_1$, since $\lambda\circ\Nm_{L_1/F}$ is $\Gal(L_1/F)$-invariant.
Equation~\eqref{U:automorphism-invariance} shows that neither local
constant changes.

Apply Lemma~\ref{O:P:choose} to the norm character of $L_2/F$ with
coefficient $\alpha_2'$ and to $\gamma$, obtaining $\alpha_2,w_2$.
Put $h=\alpha_2/\alpha_2'$. Since $h\in U_F^{t_2}$,
$v_{L_i}(h-1)\ge pt_2$. Also
$v_{L_1}(e_p)=v_{L_2}(a)=-t_1$. Hence, on the domains in
Proposition~\ref{O:M:realize},
\[
\begin{aligned}
 v_{L_1}((h-1)e_pP(z))
 &\ge pt_2-t_1+\ceil{(1+t_1+t')/p}\ge t'+1,\\
 v_{L_2}((h-1)aP(z))
 &\ge pt_2-t_1+\ceil{(1+t_1+t_2)/p}\ge t_2+1.
\end{aligned}
\]
By \eqref{U:trace-ideal}, the characters
$\psi_{L_1}(\alpha_2'\,\cdot)$ and
$\psi_{L_2}(\alpha_2'\,\cdot)$ are trivial on
$\pp_{L_1}^{t'+1}$ and $\pp_{L_2}^{t_2+1}$, respectively.
The two inequalities therefore show that the formulas for
$\varphi_1,\varphi_2$ also hold with $\alpha_2$.
Lemma~\ref{O:P:choose} already gives \eqref{O:I:tauchart} for the
same norm character with this coefficient. Hence
\begin{equation}\label{O:P:minimal-charts}
\begin{aligned}
 \varphi_1(1-z)&=\psi_{L_1}(\alpha_2 e_pP(z))
 &&\bigl(z\in\pp_{L_1}^{\ceil{(1+t_1+t')/p}}\bigr),\\
 \varphi_2(1-z)&=\psi_{L_2}(\alpha_2 aP(z))
 &&\bigl(z\in\pp_{L_2}^{\ceil{(1+t_1+t_2)/p}}\bigr).
\end{aligned}
\end{equation}

Apply Lemma~\ref{O:P:choose} to any nontrivial norm character of $L_1/F$,
with a coefficient provided by \eqref{O:I:tauchart}, to obtain
$\alpha_1,w_1$. These choices also satisfy the hypotheses of
Proposition~\ref{O:R:R2}.
The coefficient valuations in \eqref{O:I:tauchart} and
\eqref{O:P:fullP}, together with the norm identities, give
\[
 v_{L_2}(w_2)=-m,\qquad
 v_{L_1}(w_1)=-(m+t_2-t_1).
\]

By \eqref{U:norm-pullback-conductor}, the conductors of
$\lambda\circ\Nm_{L_1/F}$ and $\lambda\circ\Nm_{L_2/F}$ are
$1+t'+pm$ and $1+t_2+pm$. They exceed the conductors of $\varphi_1$ and
$\varphi_2$, respectively, by $pm-t_1>0$. Thus the twisted characters
have the asserted conductors. Proposition~\ref{O:P:pullback} gives
coefficients $\gamma-\alpha_1w_1$ and $\gamma-\alpha_2 w_2$ for these
two pullbacks. On their stated ideals the identities for $\varphi_1$ and
$\varphi_2$ also apply. Multiplying the character values proves
\eqref{O:P:twist-charts} with \eqref{O:P:C}.

Finally,
\[
\begin{aligned}
 v_{L_1}(\alpha_2 e_p/\gamma)
 &=v_{L_2}(\alpha_2 a/\gamma)=pm-t_1>0,\\
 v_{L_1}(\alpha_1w_1/\gamma)&=(p-1)(r-t_1)>0,\\
 v_{L_2}(\alpha_2 w_2/\gamma)&=(p-1)m>0.
\end{aligned}
\]
Hence $C_1,C_2$ are nonzero and
$v_{L_i}(C_i)=p\,v_F(\gamma)$.
\end{proof}

\begin{proposition}[Factorization of the local-constant quotient]
\label{O:P:four}
Fix the characters and elements of Proposition~\ref{O:P:twisted-coefficients}.
Define $A_1=\gamma+\alpha_2 e_p$ and put
\[
 Q_1=\frac{\varphi_2(C_2)}{\varphi_1(A_1)},\quad
 Q_2=\lambda\left(\frac{\Nm_{L_2/F}(C_2)}{\Nm_{L_1/F}(A_1)}\right),
\]
\[
 Q_3=\psi_{L_1}(\alpha_1 w_1)\theta_1(A_1/C_1),\qquad
 Q_4=\psi_{L_2}(-\alpha_2 w_2).
\]
Then
\begin{equation*}
 \frac{\Delta_{L_1}(\theta_1,\psi_{L_1})}
      {\Delta_{L_2}(\theta_2,\psi_{L_2})}
 =\frac{g_{L_1}(\theta_1,C_1)}{g_{L_2}(\theta_2,C_2)}
   \,Q_1Q_2Q_3Q_4.
\end{equation*}
If $Q_1Q_2Q_3Q_4=1$, this quotient belongs to $\mu_4$.
\end{proposition}
\begin{proof}
Compatibility gives $\theta_1(-1)=\theta_2(-1)$.
Proposition~\ref{O:I:enhanced}, applied to $\theta_i$ with
$\Psi_{L_i}=\psi_{L_i}$, $B=C=C_i$, and $\eta=0$, gives
\[
 \frac{\Delta_{L_1}(\theta_1,\psi_{L_1})}
      {\Delta_{L_2}(\theta_2,\psi_{L_2})}
 =\frac{g_{L_1}(\theta_1,C_1)}{g_{L_2}(\theta_2,C_2)}\frac{\theta_2(C_2)}{\theta_1(C_1)}
                      \psi_{L_1}(-C_1)\psi_{L_2}(C_2).
\]

Substitute \eqref{O:P:C}. The $\gamma$ terms cancel. The character factors satisfy
\[
 Q_1Q_2\theta_1(A_1/C_1)=\frac{\theta_2(C_2)}{\theta_1(C_1)}.
\]
The remaining additive factors are those of $Q_3Q_4$ and
\[
 \psi_F\bigl(\alpha_2[\Tr_{L_2/F}(a)-\Tr_{L_1/F}(e_p)]\bigr).
\]
By Proposition~\ref{O:A:thm:H14}, the difference of traces belongs to
$\pp_F^{1+t_2}$. Its product with $\alpha_2$ lies in
$\pp_F^{-n_F(\psi_F)}$, so the displayed factor is $1$.
This proves the asserted factorization.
\end{proof}

\section{Comparison of local constants in totally ramified odd degree}\label{sec:odd-calculation}

We retain the setup and notation of Sections~\ref{O:sec:as} and~\ref{sec:odd-models}.

\subsection{Minimal-conductor case}\label{O:sec:q1}

\begin{lemma}[Triviality on the deep norm-one subgroup]\label{O:G:deep}
Let $\theta_1$ be the character of $L_1^\times$ in a primitive
compatible pair. If
$ x\in U_{L_1}^{1+t_1+t'}$ and $\Nm_{L_1/F}(x)=1$,
then $\theta_1(x)=1$.
\end{lemma}
\begin{proof}
Choose a generator $\sigma$ of $\Gal(L_1/F)$.
Lemma~\ref{C:hilbert90} gives $y\in L_1^\times$ such that
$x=\frac{\sigma(y)}{y}$.
If $x=1$ there is nothing to prove. Otherwise $\sigma(y)\ne y$.
Apply Lemma~\ref{O:A:best-approximation} to $L_1/F$. Choose
$a\in F$ such that
\[
 q:=v_{L_1}(y-a)=\max_{b\in F}v_{L_1}(y-b),\qquad p\nmid q.
\]
Put $m=v_{L_1}(y)$. Then $q\ge m$.
Lemma~\ref{C:leading-ramification} gives
\[
 v_{L_1}(\sigma(y)-y)
 =v_{L_1}(\sigma(y-a)-(y-a))=q+t_1.
\]
Since
$x-1=\frac{\sigma(y)-y}{y}$
belongs to $\pp_{L_1}^{1+t_1+t'}$, one has
$q-m\ge t'+1$.
Thus $q>m$, so $a\ne0$ and $v_{L_1}(a)=m$. Replacing $y$ by
$y/a$ does not change $\sigma(y)/y$, and
\[
 y/a\in U_{L_1}^{q-m}\subset U_{L_1}^{t'+1}.
\]
By Proposition~\ref{U:conjugacy},
$\theta_1^\sigma/\theta_1$ is a nontrivial character in $S(K/L_1)$.
By \eqref{U:norm-character-conductor}, its conductor is $t'+1$.
It is trivial on $y/a$ and its conjugates, so
$\theta_1(\sigma(y)/y)=1$.
\end{proof}

\begin{proposition}[Reduction to minimal conductor by a common twist]\label{O:G:reduce}
For every primitive compatible pair $(\theta_1,\theta_2)$, there are
quasi-characters $\varphi_1$, $\varphi_2$, and $\lambda$ of $L_1^\times$,
$L_2^\times$, and $F^\times$, respectively, such that
\[
 \theta_1=\varphi_1(\lambda\circ \Nm_{L_1/F}),\qquad
 \theta_2=\varphi_2(\lambda\circ \Nm_{L_2/F}),
\]
where $\varphi_1,\varphi_2$ have primitive common norm pullback
$\varphi_1\circ \Nm_{K/L_1}=\varphi_2\circ \Nm_{K/L_2}$ and
\[
 m_{L_i}(\varphi_i)=m_i\qquad(i=1,2).
\]
If $m_{L_2}(\theta_2)>m_2$, then $m_F(\lambda)=1+r$ with $r>t_2$, and
$p(r-t_2)>t_1$.
\end{proposition}
\begin{proof}
Let $\Theta$ be the common norm pullback of $\theta_1,\theta_2$.
By Lemma~\ref{O:G:deep}, the restriction of $\theta_1$ to
$U_{L_1}^{m_1}$ is trivial on
$U_{L_1}^{m_1}\cap\ker\Nm_{L_1/F}$, hence induces a character on
\[
 \frac{U_{L_1}^{m_1}}
      {U_{L_1}^{m_1}\cap\ker\Nm_{L_1/F}}
 \xrightarrow[\sim]{\Nm_{L_1/F}}
 \Nm_{L_1/F}(U_{L_1}^{m_1}).
\]
The induced character is continuous. The norm image is open by
\eqref{U:unit-norm-surjective}. Lemma~\ref{C:character-extension}
extends it to a continuous quasi-character $\lambda$ of $F^\times$.
Define
\[
 \varphi_1=\theta_1/(\lambda\circ \Nm_{L_1/F}),\qquad
 \varphi_2=\theta_2/(\lambda\circ \Nm_{L_2/F}).
\]
Their common norm pullback is $\Theta/(\lambda\circ\Nm_{K/F})$,
so the pair is compatible and primitive. By construction
$\varphi_1|_{U_{L_1}^{m_1}}=1$; Lemma~\ref{O:G:lower} therefore gives
$m_{L_1}(\varphi_1)=m_1>t'+1$ and $m_{L_2}(\varphi_2)\ge m_2>t_1+1$.
Equation~\eqref{U:norm-pullback-conductor}, applied to their common
norm pullback, gives
\[
 p\,m_{L_1}(\varphi_1)-(p-1)(t'+1)
 =p\,m_{L_2}(\varphi_2)-(p-1)(t_1+1).
\]
It follows that $m_{L_2}(\varphi_2)=m_2$.

If $m_{L_2}(\theta_2)>m_2$, then $m_{L_2}(\lambda\circ \Nm_{L_2/F})=m_{L_2}(\theta_2)>m_2$.
By \eqref{U:norm-pullback-conductor} and \eqref{U:break-pullback-conductor},
a quasi-character of $F^\times$
of conductor at most $t_2+1$ pulls back with conductor at most $t_2+1$.
Thus $m_F(\lambda)=1+r>1+t_2$, and \eqref{U:norm-pullback-conductor} gives
$m_{L_2}(\theta_2)=1+t_2+p(r-t_2)>1+t_1+t_2$,
which is equivalent to the asserted range.
\end{proof}

\begin{proposition}[Minimal-conductor comparison of local constants]\label{O:G:minimal}
Let $(\varphi_1,\varphi_2)$ be a primitive compatible pair. If
$m_{L_1}(\varphi_1)=m_1$ and $m_{L_2}(\varphi_2)=m_2$, then
$\mathcal A_{L_1}(\varphi_1)=\mathcal A_{L_2}(\varphi_2)$.
\end{proposition}
\begin{proof}
By Proposition~\ref{O:M:realize}, after conjugating $\varphi_1$ without
changing its local constant, choose $\alpha_2'$ and $\Delta$ as in that
proposition. Recall $e_p=\Nm_{K/L_1}(\Delta)$ and
$a=\Nm_{K/L_2}(\Delta)$.

Applying  Proposition~\ref{O:I:enhanced}, with $B=C$, to the
two character formulas in Proposition~\ref{O:M:realize} gives
\[
\frac{\Delta_{L_2}(\varphi_2,\psi_{L_2})}
     {\Delta_{L_1}(\varphi_1,\psi_{L_1})}
=
\frac{g_2}{g_1}
\frac{\varphi_1(-\alpha_2'e_p)}
     {\varphi_2(-\alpha_2'a)}
\psi_F\!\left(
\alpha_2'\bigl[
\Tr_{L_1/F}(e_p)-\Tr_{L_2/F}(a)
\bigr]\right),
\]
where $g_i\in\mu_4$ are the  Gauss factors.

Proposition~\ref{U:determinant} gives
$\varphi_1|_{F^\times}=\varphi_2|_{F^\times}$, and compatibility evaluated
at $\Delta$ gives $\varphi_1(e_p)=\varphi_2(a)$. Thus the character quotient
is $1$. Proposition~\ref{O:A:thm:H14} gives
$\Tr_{L_1/F}(e_p)-\Tr_{L_2/F}(a)\in\pp_F^{1+t_2}$.
Since $\psi_F(\alpha_2'\,\cdot)$ is trivial on this ideal, the additive
factor is also $1$.

The quotient of local constants therefore belongs to $\mu_4$.
Its $p$th power is $1$ by Corollary~\ref{U:odd-power}, so it equals $1$,
since $p$ is odd. Equation~\eqref{U:odd-norm-factor} now gives
$\mathcal A_{L_1}(\varphi_1)=\mathcal A_{L_2}(\varphi_2)$.
\end{proof}

\subsection{Higher-conductor setup and the transfer formula}\label{O:sec:q1-high}

For the remaining calculations in Section~\ref{sec:odd-calculation}, let
$(\theta_1,\theta_2)$ be a primitive compatible pair with
$m_{L_2}(\theta_2)>m_2$.

Apply Proposition~\ref{O:G:reduce}, then choose
$\Delta,\gamma,\alpha_i,w_i$ as in
Proposition~\ref{O:P:twisted-coefficients}, retaining its symmetric
coefficients $e_j$ and norm $a$. Thus
\[
 \theta_i=\varphi_i(\lambda\circ\Nm_{L_i/F}),\qquad
 m_F(\lambda)=1+r,\qquad m=r-t_2,\qquad pm>t_1.
\]
Put $W=\gamma/\alpha_2=\Nm_{L_2/F}(w_2)$, $M=pm$,
$S=p(p-1)(t_2-t_1)$, and $\kappa_0=1+pt_2$.
Put $\Psi_E(x)=\psi_E(\alpha_2 x)$. Its largest trivial ideals on
$F,L_1,L_2$ have exponents $1+t_2,1+t',1+t_2$, respectively.
Use \eqref{O:P:minimal-charts} for the characters $\varphi_i$,
whose conductors are $m_1,m_2$.
Define the linear trace terms $\mathcal L_i$, the auxiliary trace $\mathcal T$,
and the quadratic trace term $\mathcal Q$ by
\begin{equation}\label{O:D:LQdef}
\begin{aligned}
 \mathcal L_i&=\Tr_{K/L_2}\!\left((\Delta/w_2)^i\frac{e_p}{1+e_p/W}\right)
       \quad(1\le i<p),\\
 \mathcal T&=\Tr_{K/L_1}((\Delta/w_2)^{p-1}),\\
 \mathcal Q&=\Tr_{K/L_2}\!\left(
        \frac{e_p\,(\Delta/w_2)^{p-1}\mathcal T}{(1+e_p/W)^2}\right).
\end{aligned}
\end{equation}
The minimal-character factor in Proposition~\ref{O:P:four} is
\[
 Q_1=\frac{\varphi_2(\gamma-\alpha_2 w_2+\alpha_2 a)}
           {\varphi_1(\gamma+\alpha_2 e_p)}.
\]

\begin{proposition}[Trace formula for the minimal-character transfer]\label{O:Q:logentry}
One has
\begin{equation}\label{O:Q:logentry-formula}
 Q_1=\Psi_{L_2}\!\left(
       \sum_{i=1}^{p-1}\frac{(-1)^i}{i}\mathcal L_i
       +\frac{\mathcal Q}{2(p-1)^2}\right).
\end{equation}

\end{proposition}
\begin{proof}
We first express $Q_1$ as a norm quotient. Put $z=\Delta/w_2$. The
Artin--Schreier polynomial gives
$\Nm_{K/L_2}(1+z)=1-w_2^{1-p}+a/w_2^p$.
Subtracting from $1-w_2/W+a/W$ gives
\[
 (1-w_2/W+a/W)-\Nm_{K/L_2}(1+z)
 =\left(\frac{w_2^p}{W}-1\right)\frac{a-w_2}{w_2^p}.
\]
Lemma~\ref{O:A:lem:power}, applied to $L_2/F$, gives
\[
 v_{L_2}\!\left(\frac{w_2^p}{W}-1\right)\ge(p-1)t_2.
\]
Hence the difference has valuation at least
$(p-1)t_2+pm-\max(t_1,m)\ge1+t_1+t_2=m_2$.
For $m\le t_1$, subtracting $m_2$ leaves
$(p-2)t_2+pm-2t_1-1\ge(p-3)t_1$.
For $m>t_1$ it leaves $(p-2)t_2+(p-1)m-t_1-1\ge0$.

All compared elements are units, so their $\varphi_2$-values agree.
Compatibility gives
\[
 (\varphi_2\circ\Nm_{K/L_2})(1+z)
 =(\varphi_1\circ\Nm_{K/L_1})(1+z).
\]
Finally $\varphi_1(\gamma)=\varphi_2(\gamma)$ by
Proposition~\ref{U:determinant}, and
$\Nm_{K/L_1}(z)=e_p/W$. It follows that
\[
 Q_1=\varphi_1\!\left(
 \frac{\Nm_{K/L_1}(1+z)}{1+e_p/W}\right).
\]

Put
\[
 E=\frac{\sum_{i=1}^{p-1}e_i(z)}{1+e_p/W}.
\]
The norm quotient just obtained gives $Q_1=\varphi_1(1+E)$.
Lemma~\ref{O:A:lem:twisted} gives
$v_{L_1}(e_i(z))\ge(p-1)t_2+i(m-t_1)$.
Write $m_0=\lceil t_1/p\rceil$; then $m\ge m_0$ because $p\nmid t_1$.
All the indicated bounds are increasing in $m$, and at $m=m_0\le t_1$
their minimum is
$\nu_0=(p-1)((t_2-t_1)+m_0)$.
Writing $t_1=p(m_0-1)+b_0$, $1\le b_0\le p-1$, gives
\[
 \nu_0\ge(t_2-t_1)+\ceil{(2t_1+1)/p}=\ceil{m_1/p},\qquad
 3\nu_0\ge m_1.
\]
Thus \eqref{O:P:minimal-charts} applies to $1+E$, and
\[
 \varphi_1(1+E)=\Psi_{L_1}\bigl(e_p\,(-E+E^2/2)\bigr).
\]

In a nonlinear Newton term, or in a product $e_i(z) e_j(z)$ other than
$e_{p-1}(z)^2$, the total symmetric index is at most $2p-3$ and there
are at least two factors. After multiplication by $e_p$, its
$K$-valuation is at least
$2p(p-1)t_2+(i+j)(pm-pt_1)-pt_1$.
Its worst bound occurs at $pm=t_1+1$ and $i+j=2p-3$; the excess over
$p(1+t')$ is
$(p-3)t_1+p(p-2)(t_2-t_1)+p-3\ge0$.
Terms with additional factors have no smaller bound.
Consequently only $e_{p-1}(z)^2$ can contribute among the quadratic terms.
Newton's identities, dividing by $1,\ldots,p-1$, give
\[
\begin{aligned}
 e_p\sum_{i=1}^{p-1}e_i(z)
 &\equiv e_p\sum_{i=1}^{p-1}\frac{(-1)^{i-1}}{i}
       \Tr_{K/L_1}(z^i)
       \pmod{\pp_{L_1}^{1+t'}},\\
 e_pe_{p-1}(z)^2&\equiv \frac{e_p\mathcal T^2}{(p-1)^2}
       \pmod{\pp_{L_1}^{1+t'}}.
\end{aligned}
\]
For the second congruence, every term of
$e_{p-1}(z)^2-\mathcal T^2/(p-1)^2$ contains at least three symmetric factors.
Consequently
\[
 v_{L_1}\!\left(
 e_p\left(e_{p-1}(z)^2-\frac{\mathcal T^2}{(p-1)^2}\right)\right)
 \ge-t_1+3\nu_0\ge1+t'.
\]
The remaining exponent is therefore
\[
 -\frac{e_p}{1+e_p/W}\sum_i\frac{(-1)^{i-1}}i
       \Tr_{K/L_1}(z^i)
 +\frac{e_p\mathcal T^2}{2(p-1)^2(1+e_p/W)^2}.
\]
Using $\mathcal T=\Tr_{K/L_1}(z^{p-1})$ and
\[
 \Psi_K=\Psi_{L_1}\circ\Tr_{K/L_1}
       =\Psi_{L_2}\circ\Tr_{K/L_2},
\]
this gives \eqref{O:Q:logentry-formula}.
\end{proof}

\subsection{Minimal-character transfer}\label{O:sec:q1-evaluation}

\subsubsection{The linear terms in the truncated logarithm}\label{O:sec:linear}

Let $\mathcal X=\{0\}\cup\mu_{p-1}\subset F$ and put
$\Delta_0=\Delta$. For $\xi\ne0$, let $\Delta_\xi$ be the
$K/L_2$-conjugate specified by \eqref{O:A:conjugate-root}.
These are all the $K/L_2$-conjugates of $\Delta$.

\begin{lemma}[Norm comparison for conjugates and translates]\label{O:W:actualnorm}
For each $\xi\in\mathcal X$,
\[
 v_{L_1}\bigl(\Nm_{K/L_1}(\Delta_\xi)-\Nm_{K/L_1}(\Delta+\xi)\bigr)\ge t_2+1.
\]
\end{lemma}
\begin{proof}
In mixed characteristic put $z=\eta_\xi/(\Delta+\xi)$, so
$v_K(z)\ge v_K(p)-(p-2)t_1$. By \eqref{O:A:eq:pbound},
$v_K(z)\ge t'+1$. The different exponent for $K/L_1$ is
$D'=(p-1)(t'+1)$. Equation~\eqref{U:symmetric} gives
\[
 v_{L_1}(\Nm_{K/L_1}(1+z)-1)\ge
 \min\left\{v_K(z),\left\lfloor\frac{v_K(z)+D'}p\right\rfloor\right\}.
\]
The second entry is no larger than the first. Since $v_{L_1}(\Nm_{K/L_1}(\Delta+\xi))=-t_1$, the required difference has valuation
at least
\[
 \frac{v_K(p)}{p}-t_1+(p-1)(t_2-t_1)+\left\lceil\frac{t_1}{p}\right\rceil
 \ge (p-2)t_1+2(p-1)(t_2-t_1)+\left\lceil\frac{t_1}{p}\right\rceil
 \ge t_2+1.
\]
For the last inequality, subtract $t_1+(t_2-t_1)+1$ and use $p\ge3$,
$t_1\ge1$, $\lceil t_1/p\rceil\ge1$. In characteristic $p$,
$\Delta_\xi=\Delta+\xi$, so the norm difference is zero.
\end{proof}

\begin{lemma}[Splitting the linear trace terms]\label{O:W:split}
For $1\le i\le p-1$, modulo $\pp_K^{\kappa_0}$,
\begin{equation*}
 \mathcal L_i\equiv{}w_2^{-i}\sum_{\xi\in\mathcal X}
 (\Delta+\xi)^i\frac{W(e_p+\xi)}{W+e_p+\xi}
 +\frac{e_{p-1}}{w_2^i(1+e_p/W)^2}
       \sum_{\xi\in\mu_{p-1}}\xi(\Delta+\xi)^i.
\end{equation*}
\end{lemma}
\begin{proof}
Define $f(X)=X/(1+X/W)$. For $v_K(X)=-pt_1$ its denominator is a unit.
The norm polynomial gives
\[
 \Nm_{K/L_1}(\Delta+\xi)=e_p+\xi+\xi e_{p-1}+R_\xi,
 \qquad v_K(R_\xi)\ge p((p-1)t_2-(p-2)t_1)\ge pt_2.
\]
The identity
\[
 f(X+z)-f(X)=\frac{z}{(1+X/W)(1+(X+z)/W)}
\]
and its first-order expansion show
\begin{equation*}
 f(\Nm_{K/L_1}(\Delta_\xi))\equiv
 \frac{W(e_p+\xi)}{W+e_p+\xi}+\frac{\xi e_{p-1}}{(1+e_p/W)^2}
 \pmod{\pp_K^{pt_2}}.
\end{equation*}
Indeed, the quadratic $e_{p-1}$-error has $K$-valuation at least
$pM+2S\ge pt_2$;
replacing $1+(e_p+\xi)/W$ by $1+e_p/W$ changes the linear term by an
 element of $\pp_K^{pM+S}\subseteq\pp_K^{pt_2}$.
Lemma~\ref{O:W:actualnorm} gives an error in $\pp_K^{p(t_2+1)}$.

The factor $(\Delta/w_2)^i$ in each summand has valuation
$i(M-t_1)\ge1$, so the equation above may be substituted at modulus $\kappa_0$.

If $\operatorname{char}F=p$, then $\Delta_\xi=\Delta+\xi$, so there
is no error in the power factor.  It remains to treat mixed
characteristic; put $V=v_K(p)$.

For $p\ge5$, replacing $\Delta_\xi$ by $\Delta+\xi$ in that power
factor gives an error of valuation at least
$V+iM-(2p+i-2)t_1\ge \kappa_0$.
At the weakest values $M=t_1+1$, $V=p(p-1)t_2$, the difference from $\kappa_0$
is $(p^2-4p+2)t_1+p(p-2)(t_2-t_1)+i-1\ge0$.

Suppose $p=3$. Hensel's equation at $\Delta+\xi$, $\xi=\pm1$,
gives
\begin{equation*}
 \Delta_\xi=\Delta+(1+3\Delta^2)\xi+\eta'_\xi,
 \qquad v_K(\eta'_\xi)\ge V-t_1.
\end{equation*}
To see this, the initial residual is
$3\xi\Delta^2+3\xi^2\Delta$; substituting the first Hensel correction
back into the equation adds terms of valuation at least $2V-4t_1$,
which is at least $V-t_1$. Thus the contribution of $\eta'_\xi$ to $\mathcal L_i$ has valuation at
least $V+iM-(i+3)t_1\ge \kappa_0$ for $i=1,2$.

Put $\rho_0=1+3\Delta^2$ and, for $i=1,2$, define
\[
 \mathcal S_i(Z)=w_2^{-i}\sum_{\xi\in\mathcal X}
 (\Delta+Z\xi)^i\frac{W(e_p+\xi)}{W+e_p+\xi}\qquad(Z\in K).
\]
Summing the three terms gives the identities
\begin{align*}
 \mathcal S_1(\rho_0)-\mathcal S_1(1)&=
 \frac{2(\rho_0-1)W^2}{w_2((W+e_p)^2-1)},\\
 \mathcal S_2(\rho_0)-\mathcal S_2(1)&=
 \frac{2W(\rho_0-1)}{w_2^2((W+e_p)^2-1)}
 \bigl((\rho_0+1)(e_p(W+e_p)-1)+2W\Delta\bigr).
\end{align*}
Their valuations are at least $V-2t_1+M$ and $V-5t_1+2M$, respectively,
hence at least $\kappa_0$. In the signed exceptional sum the formulas
\[
 \sum_{\xi=\pm1}\xi(\Delta+\rho_0\xi)=2\rho_0,
 \qquad \sum_{\xi=\pm1}\xi(\Delta+\rho_0\xi)^2=4\rho_0\Delta
\]
give errors with $K$-valuations at least $S+M+V-2t_1$ and
$S+2M+V-3t_1$, respectively; both are at least $\kappa_0$.
Consequently the stated congruence holds also when $p=3$.
\end{proof}

\begin{lemma}[Expansions of the norm coefficient]\label{O:W:norm-expansions}
There are $\omega,\rho\in K$ such that
\begin{align}\label{O:W:bforms}
 e_p&=\Delta+a+a\omega,& v_K(\omega)&\ge pt_2-t_1,\notag\\
 e_p&=\Delta+a+\Delta e_{p-1}+\rho,&
 v_K(\rho)&\ge p(p-1)t_2-(p^2-2p+2)t_1.
\end{align}
\end{lemma}
\begin{proof}
Each follows by solving the characteristic equation for its constant
term and applying Lemma~\ref{O:A:lem:twisted} to the remaining coefficients.
\end{proof}

\begin{lemma}[Expansion of the rational term]\label{O:W:rational}
For $1\le i\le p-2$,
\[
 \frac{W^2(\Delta-W-e_p)^i}
      {w_2^i((W+e_p)^p-W-e_p)}
 \equiv
 \frac{W^2(-a-W)^i}{w_2^i(W^p-W+a^p)}
 \pmod{\pp_K^{\kappa_0}}.
\]
For $i=p-1$,
\[
 \frac{W^2(\Delta-W-e_p)^{p-1}}
      {w_2^{p-1}((W+e_p)^p-W-e_p)}
 \equiv
 \frac{W^2\bigl((-a-W)^{p-1}
 -(p-1)(-a-W)^{p-2}\Delta e_{p-1}\bigr)}
      {w_2^{p-1}(W^p-W+a^p)}
 \pmod{\pp_K^{\kappa_0}}.
\]
\end{lemma}
\begin{proof}
The factors $(W+e_p)^p-W-e_p$ and $W^p-W+a^p$ in the denominators
both have $K$-valuation $-p^3m$. Put
\[
 E=(1+e_p/W)^p-1-(a/W)^p-W^{1-p}e_p/W.
\]
Then
\[
 (W+e_p)^p-W-e_p-(W^p-W+a^p)=W^pE.
\]
By \eqref{O:W:bforms},
\[
 \frac{e_p}{W}=\frac aW+\frac{\Delta}{W}+\frac{a\omega}{W}.
\]
Since $(\Delta/W)^p=(\Delta+a)/W^p$,
\[
 W^{1-p}\frac{e_p}{W}
 =\left(\frac{\Delta}{W}\right)^p
  +W^{1-p}\frac{a\omega}{W}.
\]
When $\operatorname{char}F=p$, Frobenius gives
\[
 E=\left(\frac{a\omega}{W}\right)^p-W^{1-p}\frac{a\omega}{W},
\]
and hence
\[
 v_K(E)\ge\min\bigl\{p(pM+pt_2-(p+1)t_1),\,
 p(p-1)M+pM+pt_2-(p+1)t_1\bigr\}\ge\kappa_0+M.
\]
In mixed characteristic the additional terms in the binomial expansion
of $(1+e_p/W)^p$ have coefficients in $p\mathbf Z$, so
\[
 \begin{aligned}
 v_K(E)&\ge\min\bigl\{v_K(p)+p(M-t_1),\ p(pM+pt_2-(p+1)t_1),\\
 &\hspace{28mm}p(p-1)M+pM+pt_2-(p+1)t_1\bigr\}\ge\kappa_0+M.
 \end{aligned}
\]
Replacing the denominator $(W+e_p)^p-W-e_p$ by $W^p-W+a^p$
changes the fraction by
\[
 E\,\frac{W^{p+2}(\Delta-W-e_p)^i}
 {w_2^i\bigl((W+e_p)^p-W-e_p\bigr)(W^p-W+a^p)}.
\]
The fraction multiplying $E$ has $K$-valuation
$[p(p-2)-(p-1)i]M\ge-M$. Since $v_K(E)\ge\kappa_0+M$,
this denominator replacement is valid modulo $\pp_K^{\kappa_0}$.

By \eqref{O:W:bforms},
\[
 \Delta-W-e_p=-a-W-\Delta e_{p-1}-\rho.
\]
For $1\le i\le p-2$, the term linear in $\Delta e_{p-1}$ satisfies
\[
 v_K\!\left(
 \frac{iW^2(-a-W)^{i-1}\Delta e_{p-1}}
      {w_2^i(W^p-W+a^p)}\right)
 \ge 2(p-1)M+S-t_1\ge\kappa_0.
\]
Terms containing $\rho$ or higher powers of $\Delta e_{p-1}$
have at least this valuation, by \eqref{O:W:bforms}.
For $i=p-1$, the terms containing $\rho$ have valuation at least
\[
 (p-1)M+p(p-1)t_2-(p^2-2p+2)t_1\ge\kappa_0,
\]
and the terms of degree at least two in $\Delta e_{p-1}$ have valuation
at least
\[
 (2p-1)M+2p(p-1)(t_2-t_1)-2t_1\ge\kappa_0.
\]
At $p=3$, $t_2=t_1$, $M=t_1+1$, the three excesses over $\kappa_0$
are $3$, $1$, and $4$, respectively.
\end{proof}

Put $C=1+a/W$, so $C\in U_{L_2}^1$.

\begin{proposition}[Evaluation of the linear trace terms]\label{O:W:final}
For   $1\le i<p$,
\[
 \begin{aligned}
 \mathcal L_i\equiv{}&
 (p-1)\frac{e_{p-1}}{w_2^{p-2}C^2}\ind_{i=p-2}
 +(p-1)\frac{W}{w_2^{p-1}}\ind_{i=p-1}\\
 &-(p-1)\frac{W^2(-a-W)^i}{w_2^i(W^p-W+a^p)}
 \pmod{\pp_K^{\kappa_0}}.
 \end{aligned}
\]
\end{proposition}
\begin{proof}
Combining Lemma~\ref{O:W:split} with
Lemma~\ref{O:D:interpolation}, evaluated at $Z=W+e_p$, gives
\begin{align}\label{O:W:Wpre}
 \mathcal L_i\equiv{}&\frac{pWe_p\Delta^i}{w_2^i(W+e_p)}
 +(p-1)\frac{W}{w_2^{p-1}}\ind_{i=p-1}
 -(p-1)\frac{W^2(\Delta-W-e_p)^i}{w_2^i((W+e_p)^p-W-e_p)}\notag\\
 &+(p-1)\frac{e_{p-1}}{(1+e_p/W)^2}
 \left(\frac{\ind_{i=p-2}}{w_2^{p-2}}
       +\frac{i\Delta\ind_{i=p-1}}{w_2^{p-1}}\right)
 \pmod{\pp_K^{\kappa_0}}.
\end{align}
When $\operatorname{char}F=p$, the first term is zero.  In mixed
characteristic its $K$-valuation is at least
$v_K(p)-pt_1+i(pm-t_1)\ge1+pt_2$ by \eqref{O:A:eq:pbound}.
Discard this first term and use Lemma~\ref{O:W:rational} for the third term.

For $i=p-2$, it remains to replace the denominator in the fourth term.
The difference of the inverse squares is
\[
 \frac{(a-e_p)/W\,[2+(a+e_p)/W]}
      {(1+e_p/W)^2(1+a/W)^2}.
\]
Since $v_K(a-e_p)\ge-t_1$,
\[
 v_K\!\left(\frac{e_{p-1}}{w_2^{p-2}}
   \left(\frac1{(1+e_p/W)^2}-\frac1{C^2}\right)\right)
 \ge S+2(p-1)M-t_1\ge\kappa_0.
\]
This gives the first term in the asserted formula.

For $i=p-1$, the fourth term in \eqref{O:W:Wpre} is
\[
 \frac{(p-1)^2\Delta e_{p-1}}
      {w_2^{p-1}(1+e_p/W)^2}.
\]
The part of the third term that is linear in $\Delta e_{p-1}$ is, by
Lemma~\ref{O:W:rational},
\[
 -(p-1)\,
 \frac{W^2\bigl(-(p-1)(-a-W)^{p-2}\Delta e_{p-1}\bigr)}
      {w_2^{p-1}(W^p-W+a^p)}.
\]
Using $-a-W=-WC$ and
\[
 W^p-W+a^p=W^p\bigl(1-W^{1-p}+(a/W)^p\bigr),
\]
and the preceding denominator replacement, their sum is congruent to
\[
 \frac{(p-1)^2\Delta e_{p-1}}{w_2^{p-1}}
 \left(
 \frac1{C^2}
 -\frac{C^{p-2}}{1-W^{1-p}+(a/W)^p}
 \right).
\]
The numerator of the parenthesis is
\[
 -W^{1-p}-\sum_{j=1}^{p-1}\binom pj(a/W)^j.
\]
The contribution of $W^{1-p}$ has $K$-valuation
$S-t_1+(p^2-1)M\ge\kappa_0$. In equal characteristic the remaining
terms vanish. In mixed characteristic they have $K$-valuation at least
$S-t_1+(p-1)M+v_K(p)+p(M-t_1)\ge\kappa_0$ by
\eqref{O:A:eq:pbound}. Hence the two linear contributions cancel modulo
$\pp_K^{\kappa_0}$.

The term with $i=p-2$ comes from the second sum in
Lemma~\ref{O:W:split}; all remaining terms are exactly those displayed.
\end{proof}

\subsubsection{The quadratic term in the truncated logarithm}\label{O:sec:quadratic}
Set $f(X)=X/(1+X/W)^2$.
\begin{lemma}[Replacement of conjugates by integral translates]
\label{O:D:conjugates}
In the trace defining $\mathcal Q$  in \eqref{O:D:LQdef} one may replace every
$K/L_2$ conjugate $\Delta_\xi$ of $\Delta$ by $\Delta+\xi$,
$\xi\in\mathcal X$, modulo $\pp_K^{\kappa_0}$.
\end{lemma}
\begin{proof}
If $\operatorname{char}F=p$, \eqref{O:A:conjugate-root} gives
$\Delta_\xi=\Delta+\xi$, so there is nothing to prove. Assume mixed
characteristic and write $e=v_K(p)$. By
\eqref{O:A:conjugate-root},
\[
 \Delta_\xi=\Delta+\xi+\eta_\xi,\qquad
 v_K(\eta_\xi)\ge e-(p-1)t_1.
\] The outer power difference, after
division by $w_2^{p-1}$, has valuation at least $e-(2p-3)t_1+(p-1)M$.
The inner trace difference has this bound plus $(p-1) t'$, by
\eqref{U:trace-ideal} for $K/L_1$. Both resulting product errors have
valuation at least
$e+S+2(p-1)M-3(p-1)t_1$.
Here $f(\Nm_{K/L_1}(\Delta_\xi))$ has valuation $-pt_1$, and the unchanged
outer power and inner trace have bounds $(p-1)(M-t_1)$ and $S+(p-1)M$.

For the norm argument itself put $a_0=e-(p-2)t_1$. The quotient
$\eta_\xi/(\Delta+\xi)$ has valuation at least $a_0$.
By \eqref{U:symmetric} for $K/L_1$,
\[
 v_K\{\Nm_{K/L_1}(\Delta_\xi)-\Nm_{K/L_1}(\Delta+\xi)\}
 \ge-pt_1+\min\{pa_0,a_0+(p-1)t'\}=e-(p-1)t_1+S.
\]
The equality for this lower bound uses $a_0\ge t'$, which follows
from $e\ge p(p-1)t_2$. The identity
\[
 f(X)-f(Y)=(X-Y)\frac{1-XY/W^2}
 {(1+X/W)^2(1+Y/W)^2}
\]
shows that $f$ does not lower this difference valuation. Multiplying
by the outer power and the inner trace gives the further bound
$e+2S+2(p-1)M-2(p-1)t_1$.

After $e\ge p(p-1)(t_1+(t_2-t_1))$ and $M\ge t_1+1$, the excesses of the two
bounds over $(1+pt_2)$ are at least
\[
 (p^2-3p+1)t_1+p(2p-3)(t_2-t_1)+2p-3,
\]
\[
 p(p-2)t_1+p(3p-4)(t_2-t_1)+2p-3,
\]
respectively. Both are positive for $p\ge3$, so the replacements
preserve $\mathcal Q$ modulo $\pp_K^{\kappa_0}$.
\end{proof}

\begin{lemma}[Reduction of the quadratic trace term to a rational sum]\label{O:D:onesum}
One has
\begin{equation}\label{O:D:onesumeq}
 \mathcal Q\equiv\frac{\mathcal T}{w_2^{p-1}}
     \sum_{\xi\in\mathcal X}(\Delta+\xi)^{p-1} f(e_p+\xi)
     \pmod{\pp_K^{\kappa_0}}.
\end{equation}
\end{lemma}
\begin{proof}
After Lemma~\ref{O:D:conjugates}, expand the translated inner trace as
\[
 \Tr_{K/L_1}((\Delta+\xi)^{p-1}/w_2^{p-1})
 =\sum_{j=0}^{p-1}\binom{p-1}{j}\xi^{p-1-j}\Tr_{K/L_1}(\Delta^j/w_2^{p-1}).
\]
For $j\le p-3$, the product valuation is at least
\[
 p(p-1)t_2+2(p-1)M-(2p-1+pj)t_1\ge \kappa_0.
\]
The weakest case $j=p-3$ has excess at least
$(p-1)t_1+p(p-2)(t_2-t_1)+2p-3$.

For $j=p-2$, its scalar is a multiple of $\xi$.
Subtract the constant product $f(e_p)\Delta^{p-1}/w_2^{p-1}$, whose weighted
sum is zero because $\sum_{\xi\ne0}\xi=0$.
One has
$v_K(((\Delta+\xi)^{p-1}-\Delta^{p-1})/w_2^{p-1})\ge (p-1)M-(p-2)t_1$
and $v_K(f(\Nm_{K/L_1}(\Delta+\xi))-f(e_p))\ge0$.
Lemma~\ref{O:A:lem:twisted}, applied with $Y=w_2^{-(p-1)}$ and
$i=p-2$, gives the trace-factor valuation at least
$p(p-1)t_2-p(p-2)t_1+p(p-1)m$. Together with this bound, the preceding
estimates give
\[
 \begin{aligned}
 &v_K\!\left(\Tr_{K/L_1}(\Delta^{p-2}/w_2^{p-1})f(e_p)
        \frac{(\Delta+\xi)^{p-1}-\Delta^{p-1}}{w_2^{p-1}}\right)
   \ge p(p-1)t_2+2(p-1)M-(p^2-2)t_1,\\
 &v_K\!\left(\Tr_{K/L_1}(\Delta^{p-2}/w_2^{p-1})\frac{(\Delta+\xi)^{p-1}}{w_2^{p-1}}
       \bigl(f(\Nm_{K/L_1}(\Delta+\xi))-f(e_p)\bigr)\right)\\
 &\hspace{3em}\ge p(p-1)t_2+2(p-1)M-(p^2-p-1)t_1.
 \end{aligned}
\]
Subtracting $\kappa_0$ from these lower bounds gives at least
$p(p-2)(t_2-t_1)+2p-3$ and
$(p-1)t_1+p(p-2)(t_2-t_1)+2p-3$, respectively.
Thus both summand differences belong to $\pp_K^{\kappa_0}$;
only the term with $j=p-1$, whose trace is $\mathcal T$, remains.

The norm polynomial now reads
\[
 \Nm_{K/L_1}(\Delta+\xi)=e_p+\xi+\xi e_{p-1}+r_\xi,
 \qquad v_K(r_\xi)\ge pt_2.
\]
First delete $r_\xi$ using the difference formula for $f$.
Next expand in $\xi e_{p-1}$. Expansion of the inverse square as a
geometric power series shows that its Taylor coefficients of degree
$j\ge2$ have valuation at least $(j-1)pM$; no division by a
factorial divisible by $p$ is used. The series converges because
$v_K(e_{p-1}/W)>0$. Finally replace
$f'(e_p+\xi)$ by $f'(e_p)$; their difference has valuation at least $pM$.
After multiplication by the power and trace factors, the respective
error bounds are
\[
 S+2(p-1)M-(p-1)t_1+pt_2,\quad
 3S+(3p-2)M-(p-1)t_1,\quad
 2S+(3p-2)M-(p-1)t_1.
\]
At $M=t_1+1$, their excesses over $\kappa_0$ are, respectively,
\[
 (p-1)t_1+S+2p-3,\quad (p-1)t_1+p(3p-4)(t_2-t_1)+3p-3,\quad
 (p-1)t_1+p(2p-3)(t_2-t_1)+3p-3,
\]
all positive. Increasing $M$ only improves them.

By Lemma~\ref{O:D:interpolation}, the remaining signed term contains
\[
 \sum_{\xi\ne0}\xi(\Delta+\xi)^{p-1}=(p-1)^2\Delta.
\]
The corresponding contribution
\[
 \frac{\mathcal T e_{p-1}f'(e_p)}{w_2^{p-1}}
 \sum_{\xi\ne0}\xi(\Delta+\xi)^{p-1}
\]
has valuation at least $2S+2(p-1)M-t_1$, whose excess is at least
$(p-3)t_1+p(2p-3)(t_2-t_1)+2p-3\ge0$.
This proves \eqref{O:D:onesumeq}.
\end{proof}

\begin{proposition}[Evaluation of the quadratic trace term]\label{O:D:quadratic}
We have
\begin{equation}\label{O:D:quadraticcong}
 \mathcal Q\equiv(p-1)\frac{a\mathcal T}{w_2^{p-1}C^2}
                      \pmod{\pp_K^{\kappa_0}}.
\end{equation}
\end{proposition}
\begin{proof}
Put $D_L=(W+e_p)^p-W-e_p$.
Thus $v_K(W+e_p)=v_K(\Delta-W-e_p)=-pM$ and $v_K(D_L)=-p^2M$.
Apply Lemma~\ref{O:D:interpolation} with $i=(p-1)$ and combine its
rational expression with its derivative in $Z$, holding $W,e_p,\Delta$
constant. Evaluating at $Z=W+e_p$ and collecting terms gives
\begin{equation*}
\begin{aligned}
 \sum_{\xi\in\mathcal X}(\Delta+\xi)^{p-1} f(e_p+\xi)
 ={}&\frac{(p-1)W^2(\Delta-W-e_p)^{p-2}(\Delta-e_p)}{D_L}\\
 &+\frac{pW^2\Delta^{p-1}e_p}{(W+e_p)^2}\\
 &-\frac{p(p-1)W^3(\Delta-W-e_p)^{p-2}}{D_L}\\
 &-\frac{(p-1)W^3(\Delta-W-e_p)^{p-1}
          (p(W+e_p)^{p-1}-1)}{D_L^2}.
\end{aligned}
\end{equation*}

The factor $\mathcal T/w_2^{p-1}$ has valuation at least $S+2(p-1)M$.
When $\operatorname{char}F=p$, the second and third terms of the rational-sum identity above are zero, and the factor $p(W+e_p)^{p-1}-1$ in the
fourth term is $-1$.  Thus only the $-1$ part of the fourth term remains
among these three contributions.  After multiplication by $\mathcal T/w_2^{p-1}$ its
valuation is at least $S+(p^2-2)M$.  At $M=t_1+1$ its excess over
$\kappa_0$ is
\[
 (p^2-p-2)t_1+p(p-2)(t_2-t_1)+p^2-3\ge0.
\]

Now assume mixed characteristic and put $e=v_K(p)$.  The valuations of
the last three terms, splitting the last one at $p(W+e_p)^{p-1}-1$, are bounded
below by
\[
 e-(2p-1)t_1,\qquad e-pM,\qquad e-pM,\qquad p(p-2)M.
\]
Consequently their products have the three lower bounds
\[
 e+S+2(p-1)M-(2p-1)t_1,\quad e+S+(p-2)M,
 \quad S+(p^2-2)M.
\]
At $e=p(p-1)(t_1+(t_2-t_1))$, $M=t_1+1$, their excesses over $\kappa_0$ are,
respectively,
\begin{align*}
 &(p^2-2p-1)t_1+p(2p-3)(t_2-t_1)+2p-3,\\
 &(p^2-p-2)t_1+p(2p-3)(t_2-t_1)+p-3,\\
 &(p^2-p-2)t_1+p(p-2)(t_2-t_1)+p^2-3.
\end{align*}
All are nonnegative for $p\ge3$; increasing $e$ or $M$ preserves this.

It remains to treat the first term of that rational-sum identity.
By \eqref{O:W:bforms}, $e_p-\Delta=a(1+\omega)$ and
$v_K(a\omega)\ge p(t_2-t_1)-t_1\ge-t_1$.
Both $(\Delta-W-e_p)/[-(W+a)]$ and $(W+e_p)/(W+a)$ belong to
$U_K^{pM-t_1}$, and
$D_L/(W+e_p)^p=1-(W+e_p)^{1-p}\in U_K^{p(p-1)M}$.
Since $p-2$ is odd,
\[
 \frac{W^2(\Delta-W-e_p)^{p-2}}{D_L}
 =-C^{-2}(1+\eta_0),\qquad v_K(\eta_0)\ge pM-t_1.
\]
The first term is therefore $(p-1) a C^{-2}(1+\omega)(1+\eta_0)$.
After multiplication by $\mathcal T/w_2^{p-1}$, one has
\[
 \begin{aligned}
 v_K\!\left(\frac{(p-1)a\mathcal T\omega}{w_2^{p-1}C^2}\right)
   &\ge S+2(p-1)M+pt_2-(p+1)t_1,\\
 v_K\!\left(\frac{(p-1)a\mathcal T\eta_0}{w_2^{p-1}C^2}\right)
   &\ge S+(3p-2)M-(p+1)t_1.
 \end{aligned}
\]
At $M=t_1+1$, subtracting $\kappa_0$ from these bounds gives
$(p-3)t_1+p(p-1)(t_2-t_1)+2p-3$ and
$(p-3)t_1+p(p-2)(t_2-t_1)+3p-3$, respectively, both nonnegative.
Moreover $v_K(\eta_0)\ge pM-t_1\ge1$, so
\[
 v_K\!\left(\frac{(p-1)a\mathcal T\omega\eta_0}{w_2^{p-1}C^2}\right)
 \ge \kappa_0+pM-t_1>\kappa_0.
\] Combining this with Lemma~\ref{O:D:onesum}
proves \eqref{O:D:quadraticcong}.
\end{proof}

\subsection{Norm-twist correction}\label{O:sec:q2}
With $C=1+a/W$, put
\[
  P_*=\Nm_{L_2/F}(C),\qquad
 Z=\frac{\Nm_{L_1/F}(1+e_p/W)}{P_*-W^{1-p}},\qquad
 x_2=\frac{w_2}{WC}.
\]
The elements $P_*,Z$ and $P_*-W^{1-p}$ are units, and
\[
 \Nm_{L_2/F}(x_2)=W^{1-p}/P_*,\qquad v_{L_2}(x_2)=(p-1)m.
\]
The definition of $Q_2$ in Proposition~\ref{O:P:four} gives
\begin{equation}\label{O:N:exactfactor}
 Q_2=\lambda(Z^{-1})\,
       \lambda\!\left(\frac{\Nm_{L_2/F}(1-x_2)}{1-\Nm_{L_2/F}(x_2)}\right).
\end{equation}

\begin{lemma}[Weighted symmetric coefficients]\label{O:N:moments}
For $1\le i<p$,
\[
 W^{-i}\bigl(e_i^{L_1/F}(e_p)-e_i^{L_2/F}(a)\bigr)
 \equiv-\ind_{i=p-1}W^{1-p}(1+e_{p-1}^p)
              \pmod{\pp_{L_1}^{1+pr}}.
\]
\end{lemma}
\begin{proof}
We first prove, for $1\le j<p$,
\[
 W^{-j}[\Tr_{L_1/F}(e_p^j)-\Tr_{L_2/F}(a^j)]
 \equiv(p-1)\ind_{j=p-1}W^{1-p}(1+e_{p-1}^p)
              \pmod{\pp_{L_1}^{1+pr}}.
\]
For $j=1$ Proposition~\ref{O:A:thm:H14}, multiplied by $W^{-1}$, gives the
required modulus.

For $j\ge2$, apply Proposition~\ref{O:A:thm:h} to
$w_2^{-j}\Delta^{j-1}$, whose $K$-valuation is $jpm-(j-1)t_1>0$.
Its error exponent $pt_2+jpm-(j-1)t_1$ is at least $1+pr$, since their
difference is $(j-1)(pm-t_1)-1\ge0$.
Replacing $w_2^{-pj}$ by $W^{-j}$ in the $L_2/F$ trace changes it by
\[
 \Tr_{L_2/F}\bigl((W^{-j}-w_2^{-pj})a^j\bigr).
\]
Lemma~\ref{O:A:lem:power}, applied to $L_2/F$, gives
$v_{L_2}(W^{-j}-w_2^{-pj})\ge jpm+(p-1)t_2$.
After adding the different, the trace is in $\pp_F^{r+1}$, because
\[
 (j-1)pm+(p-2)t_2-jt_1-1
 \ge (p-2)t_2-t_1+j-2\ge0.
\]
Apply Newton's identities to the $L_1/F$-conjugates of $e_p/W$ and
the $L_2/F$-conjugates of $a/W$. Induction over $i<p$ introduces only
integral factors and division by the unit $i$. For $i<p-1$ the
weighted difference is zero modulo $\pp_{L_1}^{1+pr}$. For $i=p-1$,
dividing $(p-1)W^{1-p}(1+e_{p-1}^p)$ by $-(p-1)$ gives the stated
congruence.
\end{proof}

Put $\mathcal N_a=\Nm_{L_2/F}(a)$ and $d_W=1+W^{-p}\mathcal N_a$.

\begin{proposition}[The weighted norm-quotient congruence]\label{O:N:Z}
One has
\[
 Z\equiv1+W^{1-p}\frac{\Nm_{L_1/F}(-e_{p-1})}{d_W}\pmod{\pp_F^{r+1}}.
\]
\end{proposition}
\begin{proof}

The highest-degree terms in the norm expansions of $e_p/W$ and $a/W$ agree by norm transitivity.
Summing Lemma~\ref{O:N:moments} over the remaining symmetric degrees gives
\[
 \Nm_{L_1/F}(1+e_p/W)-(P_*-W^{1-p})\equiv W^{1-p}(-e_{p-1})^p
             \pmod{\pp_{L_1}^{1+pr}}.
\]
Lemma~\ref{O:A:lem:power}, applied to $L_1/F$, gives
\[
 v_{L_1}\bigl(W^{1-p}((-e_{p-1})^p-\Nm_{L_1/F}(-e_{p-1}))\bigr)
 \ge p(p-1)m+p(p-1)(t_2-t_1)+(p-1)t_1\ge1+pr,
\]
because the excess is $p(p-2)(m+(t_2-t_1))-t_1-1\ge0$.
Both sides after this substitution belong to $F$, so intersection with
$F$ improves the modulus to $\pp_F^{r+1}$.

It remains to prove
\[
 W^{1-p}\Nm_{L_1/F}(-e_{p-1})
 \bigl((P_*-W^{1-p})^{-1}-d_W^{-1}\bigr)\in\pp_F^{r+1}.
\]
One has
\[
 (P_*-W^{1-p})-d_W=\sum_{i=1}^{p-1}e_i^{L_2/F}(a/W)-W^{1-p},
\]
and
\[
 v_F(e_i^{L_2/F}(a/W))\ge
 \floor{\frac{pm-t_1+(p-1)(t_2+1)}p}.
\]
The weighted intermediate terms have valuation at least
\[
 (p-1)m+(p-1)(t_2-t_1)
 +\floor{\frac{pm-t_1+(p-1)(t_2+1)}p}\ge r+1.
\]
Put $b_0=t_1-p\floor{t_1/p}$, so $1\le b_0<p$ and
$m\ge\floor{t_1/p}+1$. At $t_2=t_1$ the excess is at least
\[
 (p-3)\floor{t_1/p}+p-2-b_0
 +\floor{\frac{(p-2)b_0+p-1}{p}}\ge0.
\]
For $b_0\le p-2$ every indicated term is nonnegative; for $b_0=p-1$
the final floor is $p-2$, giving again a nonnegative result.
Increasing $(t_2-t_1)$ or $m$ cannot decrease the excess.

The weighted final $-W^{1-p}$ term has excess
\[
 (2p-3)m+(p-2)(t_2-t_1)-t_1-1\ge pm-t_1-1\ge0.
\]
All denominators are units.
\end{proof}

Put
 $F_*=\frac{P_*}{P_*-W^{1-p}}
     =\bigl(1-\Nm_{L_2/F}(x_2)\bigr)^{-1}$.

\begin{proposition}[Norm-twist correction]\label{O:N:Q2}
One has
\[
 Q_2=\Psi_F\!\left(W^{2-p}\frac{\Nm_{L_1/F}(-e_{p-1})}{d_W}\right)
          \Psi_{L_2}\!\left(WF_*P(x_2)\right).
\]
\end{proposition}
\begin{proof}
For the second factor of \eqref{O:N:exactfactor}, put
\[
 b=\floor{\frac{(p-1)m+(p-1)(t_2+1)}p}
  =\floor{\frac{(p-1)(r+1)}p}.
\]
Then $2b\ge r+1$.
The norm expansion is
\[
 \frac{\Nm_{L_2/F}(1-x_2)}{1-\Nm_{L_2/F}(x_2)}
 =1+Q,\qquad
 Q=\frac{\sum_{i=1}^{p-1}(-1)^i e_i^{L_2/F}(x_2)}
         {1-\Nm_{L_2/F}(x_2)}.
\]
The denominator is a unit. Equation~\eqref{U:symmetric} gives
$v_F(e_i^{L_2/F}(x_2))\ge b$ for $1\le i<p$, hence $v_F(Q)\ge b$.
Thus
\[
 P(-Q)+Q\in\pp_F^{2b}\subseteq\pp_F^{r+1}.
\]
The coefficient formula \eqref{O:P:fullP} for $\lambda$ gives
$\lambda(1+Q)=\psi_F(-\gamma Q)$.
Newton's identities give
\[
 e_j^{L_2/F}(x_2)\equiv(-1)^{j-1}\Tr_{L_2/F}(x_2^j)/j\pmod{\pp_F^{2b}},
 \qquad 1\le j<p.
\]
Here $v_F(\Tr_{L_2/F}(x_2^j))\ge b$ by \eqref{U:trace-ideal}, so every non-leading
term in the Newton recursion is in $\pp_F^{2b}$. Therefore
\[
 -Q\equiv
 \frac{\Tr_{L_2/F}(P(x_2))}{1-\Nm_{L_2/F}(x_2)}
 \pmod{\pp_F^{r+1}}.
\]
Since $\gamma=\alpha_2 W$, this evaluates the second factor as
\[
 \lambda\!\left(\frac{\Nm_{L_2/F}(1-x_2)}{1-\Nm_{L_2/F}(x_2)}\right)
 =\Psi_{L_2}\!\left(WF_*P(x_2)\right).
\]

For the first factor, \eqref{O:A:eq:ei} gives
$v_F(\Nm_{L_1/F}(-e_{p-1}))=v_{L_1}(e_{p-1})\ge(p-1)(t_2-t_1)$.
Since $d_W$ is a unit, this gives
\[
 v_F\!\left(\frac{W^{1-p}\Nm_{L_1/F}(-e_{p-1})}{d_W}\right)
 \ge(p-1)(m+t_2-t_1).
\]
Moreover,
\[
 2(p-1)(m+t_2-t_1)-(r+1)
 =(2p-3)m+(2p-3)(t_2-t_1)-t_1-1
 \ge pm-t_1-1\ge0.
\]
Proposition~\ref{O:N:Z} and \eqref{O:P:fullP} therefore give
\[
 \lambda(Z^{-1})
 =\Psi_F\!\left(W^{2-p}\frac{\Nm_{L_1/F}(-e_{p-1})}{d_W}\right).
\]
\end{proof}

\clearpage
\subsection{Cancellation of the transfer, twist, and linear terms}\label{O:sec:transition}

For $1\le i<p$, put
\[
 F_i=\frac{W}{W^p-W+a^p}\left(\frac{W+a}{w_2}\right)^i.
\]
The rational term of Proposition~\ref{O:W:final} contributes
$-(p-1)WF_i/i$ to the $\Psi_{L_2}$-argument in
\eqref{O:Q:logentry-formula}; Proposition~\ref{O:N:Q2} contributes
$WF_*P(x_2)$. Since $\Psi_{L_2}$ is trivial on
$\pp_{L_2}^{1+t_2}$ and $v_{L_2}(W)=-pm$, after factoring out $W$
it suffices to compare the remaining terms modulo
$\pp_{L_2}^{1+t_2+pm}$.
\begin{lemma}[A congruence for the rational terms]\label{O:T:rational}
One has
\begin{equation*}
 -F_{p-1}+\sum_{i=1}^{p-2}\frac{F_i}{i}+F_*P(x_2)
 \equiv x_2(1-W/w_2^p)\pmod{\pp_{L_2}^{1+t_2+pm}}.
\end{equation*}
\end{lemma}
\begin{proof}
Put $Q=1+W^{-p}a^p$ and $G=WC^p/(w_2^p(Q-W^{1-p}))$ in $L_2$.
Both $Q-W^{1-p}$ and $P_*-W^{1-p}$ are units. Multiplication gives
\begin{equation*}
 F_i=G x_2^{p-i},\qquad G x_2^p=\frac{W^{1-p}}{Q-W^{1-p}}.
\end{equation*}
For $2\le j\le p-1$,
$1/(p-j)+1/j=p/[j(p-j)]$.
Consequently
\[
 -F_{p-1}+\sum_{i=1}^{p-2}\frac{F_i}{i}
 \equiv-GP(x_2)\pmod{\pp_{L_2}^{1+t_2+pm}}.
\]
The difference is $pG\sum_{j=2}^{p-1}x_2^j/[j(p-j)]$.
In mixed characteristic its valuation is at least
$(p-1)t_2+2(p-1)m$, whose excess over $(1+t_2+pm)$ is
$(p-2)(t_2+m)-1\ge0$.  When $\operatorname{char}F=p$, these scalar
coefficients map to zero.
The coefficient of $F_{p-1}$ remains unchanged.

Lemma~\ref{O:A:lem:power} and the binomial theorem give
\begin{align*}
 v_{L_2}(P_*-C^p)&\ge(p-1)t_2,\\
 v_{L_2}(C^p-Q)&\ge(p-1)t_2+pm-t_1,\\
 v_{L_2}(1-W/w_2^p)&\ge(p-1)t_2.
\end{align*}
Hence $v_{L_2}(P_*-Q)\ge(p-1)t_2$.
One has
\[
 F_*-G-(1-W/w_2^p)
 =W^{1-p}\left(
 \frac1{P_*-W^{1-p}}-\frac{W}{w_2^p(Q-W^{1-p})}\right)
 -\frac{W(C^p-Q)}{w_2^p(Q-W^{1-p})}.
\]
The preceding bounds give
\begin{align*}
 v_{L_2}\!\left(
 \frac1{P_*-W^{1-p}}-\frac{W}{w_2^p(Q-W^{1-p})}
 \right)&\ge(p-1)t_2,\\
 v_{L_2}(F_*-G)&\ge(p-1)t_2.
\end{align*}
For $j\ge2$, the valuation of $(F_*-G)x_2^j$ is at least
$(p-1)t_2+2(p-1)m\ge (1+t_2+pm)$.
For $j=1$, multiply the identity for $F_*-G-(1-W/w_2^p)$ by $x_2$.
The first term on its right-hand side has $L_2$-valuation at least
$(p^2-1)m+(p-1)t_2\ge (1+t_2+pm)$.
The second term has $L_2$-valuation at least
$(p-1)t_2+pm-t_1+(p-1)m$; its excess is
\[
 (p-2)t_2+(p-1)m-t_1-1
 =(p-3)t_1+(p-2)(t_2-t_1)+(p-1)m-1\ge0.
\]
This proves the required rational-term congruence.
\end{proof}

For $z=\sum_{i=0}^{p-1}z_i\Delta^i$ with $z_i\in L_2$, let
$\pr(z)=z_0$. The exceptional $i=p-2$ term in Proposition~\ref{O:W:final}
and the main term in \eqref{O:D:quadraticcong} contribute to the
$\Psi_{L_2}$-argument in \eqref{O:Q:logentry-formula}; after substitution,
they give the two correction terms below. Put
\begin{equation}\label{O:T:XandY}
 X_e=\frac{-e_{p-1}}{w_2^{p-2}C^2},\qquad
 X_{\mathcal T}=\frac{a\mathcal T}{w_2^{p-1}C^2}.
\end{equation}
Their contributions are respectively
$\frac{p-1}{p-2}X_e$ and $\frac{1}{2(p-1)}X_{\mathcal T}$.

\begin{lemma}[Constant coefficient and its norm]\label{O:T:projection}
Let $z=xw_0$ with $x\in L_2$, $w_0\in L_1$. Suppose
\begin{equation*}
 v_K(z)\ge1+t',\qquad v_K(z)\ge1+t_2+(p-3)t_1.
\end{equation*}
Then $\pr(z)\in L_2$ satisfies
\begin{equation*}
 z\equiv \pr(z)\pmod{\pp_K^{\kappa_0}},\qquad
 \Nm_{K/L_1}(z)\equiv \Nm_{L_2/F}(\pr(z))\pmod{\pp_{L_1}^{\kappa_0}}.
\end{equation*}
If $z\ne0$, then $v_K(\pr(z))=v_K(z)$.
\end{lemma}
\begin{proof}
The zero case is immediate. Put $h=v_K(z)$ and write
$z=\sum z_i\Delta^i$ using
Lemma~\ref{O:A:lem:coeff} for $w_0$. It gives
\[
 v_{L_2}(z_i)\ge h/p+it_1,\qquad
 v_K(z_i\Delta^i)\ge h+i(p-1)t_1.
\]
Thus $v_K(z_i\Delta^i)\ge h+i(p-1)t_1>h$ for every $i\ge1$,
while the constant term $\pr(z)=z_0$ has valuation $h$. Since
$1+t'+(p-1)t_1=\kappa_0$, all nonconstant terms belong to $\pp_K^{\kappa_0}$.
This proves the first congruence.

Lemma~\ref{O:A:lem:H17}, applied to $z=xw_0$ and $\pr(z)=z_0$, gives
the norm congruence.
\end{proof}

\begin{lemma}[Replacement of the denominator]\label{O:T:denominator}
One has
\begin{align*}
 W^{2-p}\Nm_{L_1/F}(-e_{p-1})(P_*^{-2}-d_W^{-2})&\in\pp_F^{1+t_2},\\
 W^{1-p}\mathcal N_a\Nm_{L_1/F}(\mathcal T)(P_*^{-2}-d_W^{-2})&\in\pp_F^{1+t_2}.
\end{align*}
\end{lemma}
\begin{proof}
Since $P_*$ and $d_W$ are units,
$v_F(P_*^{-2}-d_W^{-2})\ge v_F(P_*-d_W)$.
The norm expansion and \eqref{U:symmetric} for $L_2/F$ give
\[
 v_F(P_*-d_W)\ge m+
 \floor{\frac{(p-2)t_1+(p-1)((t_2-t_1)+1)}p}.
\]
Indeed the $j$th intermediate coefficient has bound
$\floor{[j(pm-t_1)+(p-1)(t_2+1)]/p}$, nondecreasing in $j$.
Put $b_0=t_1-p\floor{t_1/p}$ and write
$m=\floor{t_1/p}+1+c$, where $1\le b_0<p$ and $c\ge0$.
In Lemma~\ref{O:A:lem:twisted}, the symmetric-function bound with
$Y=1$, $i=p-1$ bounds $e_{p-1}$, and the trace bound with
$Y=w_2^{1-p}$, $i=p-1$ bounds $\mathcal T$.
The two weighted valuation excesses over $1+t_2$ are bounded below by
\[
 (p-3)\floor{t_1/p}+(p-1)c+p-2+(p-2)(t_2-t_1)-b_0
 +\floor{\frac{(p-2)b_0+(p-1)(t_2-t_1+1)}p},
\]
\[
 (p-3)\floor{t_1/p}+(2p-1)c+2p-2+(p-2)(t_2-t_1)-2b_0
 +\floor{\frac{(p-2)b_0+(p-1)(t_2-t_1+1)}p}.
\]
The first is nonnegative for $b_0\le p-2$; if $b_0=p-1$, use
$\floor{((p-2)b_0+(p-1)(t_2-t_1+1))/p}\ge p-2$.
The second is nonnegative because $b_0\le p-1$.
\end{proof}

\begin{lemma}[Representatives of the two correction terms and their norms]\label{O:T:representatives}
The elements in \eqref{O:T:XandY} satisfy
\begin{equation*}
\begin{aligned}
 X_e&\equiv \pr(X_e)\pmod{\pp_K^{\kappa_0}},\\
 X_{\mathcal T}&\equiv \pr(X_{\mathcal T})\pmod{\pp_K^{\kappa_0}},\\
 \Nm_{L_2/F}(\pr(X_e))&\equiv W^{2-p}\Nm_{L_1/F}(-e_{p-1})/d_W^2\pmod{\pp_F^{1+t_2}},\\
 \Nm_{L_2/F}(\pr(X_{\mathcal T}))&\equiv W^{1-p}\mathcal N_a\Nm_{L_1/F}(\mathcal T)/d_W^2\pmod{\pp_F^{1+t_2}}.
\end{aligned}
\end{equation*}
Moreover $\pr(X_e),\pr(X_{\mathcal T})\in\pp_{L_2}^{\ceil{(1+t_2)/p}}$.
\end{lemma}
\begin{proof}
By Lemma~\ref{O:A:lem:twisted},
\[
 v_K(X_e)\ge(p-2)M+p(p-1)(t_2-t_1),\qquad
 v_K(X_{\mathcal T})\ge2(p-1)M-pt_1+p(p-1)(t_2-t_1).
\]
At $M\ge t_1+1$, their excesses over $1+t'$ are at least
\[
 (p-3)(t_1+1)+p(p-2)(t_2-t_1),\qquad
 (p-3)t_1+2p-3+p(p-2)(t_2-t_1),
\]
and their excesses over $1+t_2+(p-3)t_1$ are at least
\[
 p-3+[p(p-1)-1](t_2-t_1),\qquad
 2p-3+[p(p-1)-1](t_2-t_1).
\]
All are nonnegative, including $p=3$ and $t_2=t_1$.
Lemma~\ref{O:T:projection} therefore applies to both products and proves
the two asserted congruences with the constant coefficients. It also gives
$v_K(\pr(X_e)),v_K(\pr(X_{\mathcal T}))\ge1+t_2$, hence
$v_{L_2}(\pr(X_e)),v_{L_2}(\pr(X_{\mathcal T}))\ge\ceil{(1+t_2)/p}$.

The norm part of that lemma gives
\[
 \Nm_{L_2/F}(\pr(X_e))\equiv\frac{W^{2-p}(-e_{p-1})^p}{P_*^2},\qquad
 \Nm_{L_2/F}(\pr(X_{\mathcal T}))\equiv\frac{W^{1-p}\mathcal N_a\mathcal T^p}{P_*^2}
 \pmod{\pp_{L_1}^{\kappa_0}}.
\]

By Lemma~\ref{O:A:lem:power} on $L_1/F$,
\begin{align*}
 v_{L_1}\bigl(W^{2-p}((-e_{p-1})^p-\Nm_{L_1/F}(-e_{p-1}))\bigr)
 &\ge(p-2)M+p(p-1)(t_2-t_1)+(p-1)t_1,\\
 v_{L_1}\bigl(W^{1-p}\mathcal N_a(\mathcal T^p-\Nm_{L_1/F}(\mathcal T))\bigr)
 &\ge2(p-1)M+p(p-1)(t_2-t_1)-t_1.
\end{align*}
The first excess over $\kappa_0$ is at least
$(p-3)(t_1+1)+p(p-2)(t_2-t_1)$, and the second at least
$(p-3)t_1+2p-3+p(p-2)(t_2-t_1)$.
Thus the two powers may be replaced by $\Nm_{L_1/F}(-e_{p-1}),\Nm_{L_1/F}(\mathcal T)$. Now both sides lie
in $F$, so $F\cap\pp_{L_1}^{\kappa_0}=\pp_F^{1+t_2}$ applies.
Lemma~\ref{O:T:denominator} replaces $P_*^{-2}$ by $d_W^{-2}$ at that modulus, proving the remaining assertions.
\end{proof}

\begin{lemma}[Conversion of the correction-character values by norm]\label{O:T:actualnorm}
One has
\begin{align}\label{O:T:normhalves}
 \Psi_{L_2}(\pr(X_e)/2)&=\Psi_F(-W^{2-p}\Nm_{L_1/F}(-e_{p-1})/(2d_W^2)),\notag\\
 \Psi_{L_2}(-\pr(X_{\mathcal T})/2)&=\Psi_F(W^{1-p}\mathcal N_a\Nm_{L_1/F}(\mathcal T)/(2d_W^2)).
\end{align}
\end{lemma}
\begin{proof}
By Proposition~\ref{O:P:twisted-coefficients}, the norm character of
$L_2/F$ has coefficient $\alpha_2$ in \eqref{O:I:tauchart}.
Corollary~\ref{O:I:conversion}, applied to this
extension with break $t_2$, gives
\[
 \Psi_{L_2}(z)\Psi_F(\Nm_{L_2/F}(z))=1
       \qquad(z\in\pp_{L_2}^{\ceil{(1+t_2)/p}}).
\]

The product $\Psi_{L_2}(z/2)\Psi_F(\Nm_{L_2/F}(z)/2)$ has square $1$
and has $p$-power order. Indeed, for a continuous additive character $\psi$
of a local field of residue characteristic $p$ and each $x$,
$\psi(x)^{p^a}=\psi(p^a x)=1$ for sufficiently large $a$.
As $p$ is odd, this product is $1$.

Apply it to $\pr(X_e),\pr(X_{\mathcal T})$ and use Lemma~\ref{O:T:representatives}.
The errors after multiplication by $1/2$ still lie in the
trivial ideal $\pp_F^{1+t_2}$.
\end{proof}

\begin{proposition}[Cancellation of the transfer, twist, and linear terms]\label{O:T:Q124}
One has
\begin{equation}\label{O:T:Q124phase}
 Q_1Q_2Q_4
 =\Psi_F\left(
 W^{2-p}\frac{\Nm_{L_1/F}(-e_{p-1})}{d_W}-\frac{W^{2-p}\Nm_{L_1/F}(-e_{p-1})}{2d_W^2}
                   +\frac{W^{1-p}\mathcal N_a\Nm_{L_1/F}(\mathcal T)}{2d_W^2}\right).
\end{equation}
\end{proposition}
\begin{proof}
Propositions~\ref{O:Q:logentry}, \ref{O:W:final}, and~\ref{O:D:quadratic} express the argument of $\Psi_{L_2}$ in
$Q_1$ modulo $\pp_K^{\kappa_0}$ as
\[
 -WF_{p-1}+W\sum_{i=1}^{p-2}F_i/i+W/w_2^{p-1}
       +\frac{p-1}{p-2}X_e+\frac{1}{2(p-1)}X_{\mathcal T}.
\]
For $1\le i\le p-2$, replacing $-(p-1)/i$ by $1/i$ in this
expression changes the coefficient by $p/i$, hence introduces
$(p/i)WF_i$.  When $\operatorname{char}F=p$, the scalar $p/i$ is zero
in $F$.  In mixed characteristic, \eqref{O:A:eq:pbound} gives
\[
 v_{L_2}(F_i)=(p-1)(p-i)m,\qquad
 v_K(pWF_i)\ge v_K(p)+(p-2)pm\ge \kappa_0.
\]
The coefficient of $F_{p-1}$ is unchanged.

Replace $X_e,X_{\mathcal T}$ by $\pr(X_e),\pr(X_{\mathcal T})$ using Lemma~\ref{O:T:representatives}.
The resulting expression belongs to $L_2$, so its difference from
the argument in \eqref{O:Q:logentry-formula} lies in
$L_2\cap\pp_K^{\kappa_0}=\pp_{L_2}^{1+t_2}$ and has trivial $\Psi_{L_2}$ phase.

The terms discarded when replacing the two coefficients by $1/2$ and
$-1/2$, respectively, are $\frac{p}{2(p-2)}\pr(X_e)$ and
$\frac{p}{2(p-1)}\pr(X_{\mathcal T})$.
The projections are integral by Lemma~\ref{O:T:representatives}.  In mixed characteristic, \eqref{O:A:eq:pbound} gives
$v_K(p)\ge p(p-1)t_2\ge \kappa_0$, so both discarded terms lie in
$\pp_K^{\kappa_0}$.  When $\operatorname{char}F=p$, both scalar
coefficients are zero.  Consequently
\[
 Q_1=\Psi_{L_2}\!\left(
 -WF_{p-1}+W\sum_{i=1}^{p-2}F_i/i+W/w_2^{p-1}
       +(\pr(X_e)-\pr(X_{\mathcal T}))/2\right).
\]

Proposition~\ref{O:N:Q2} and Lemma~\ref{O:T:rational} combine the two rational
sums modulo $\pp_{L_2}^{1+t_2+pm}$. Since $v_{L_2}(W)=-pm$,
multiplication by $W$ puts this error in the trivial ideal
$\pp_{L_2}^{1+t_2}$ of $\Psi_{L_2}$.
Including $W/w_2^{p-1}$ and $Q_4=\Psi_{L_2}(-w_2)$, the remaining
argument of $\Psi_{L_2}$ is
\[
 W[x_2(1-W/w_2^p)+w_2^{-(p-1)}]-w_2
 =\frac{a}{w_2^{p-1}C}(1-w_2^p/W).
\]
By Lemma~\ref{O:A:lem:power},
$v_{L_2}(1-w_2^p/W)\ge(p-1)t_2$. Thus the displayed argument has
valuation at least $-t_1+(p-1)m+(p-1)t_2$, which is at least $1+t_2$,
since
\[
 (p-3)t_1+(p-2)(t_2-t_1)+(p-1)m-1\ge0.
\]
Its $\Psi_{L_2}$-value is therefore $1$.
The product is therefore
\[
 \Psi_F(W^{2-p}\Nm_{L_1/F}(-e_{p-1})/d_W)\,
 \Psi_{L_2}((\pr(X_e)-\pr(X_{\mathcal T}))/2).
\]
Finally \eqref{O:T:normhalves} gives \eqref{O:T:Q124phase}.
\end{proof}

\subsection{Quadratic stationary-phase correction}\label{O:sec:q3}

Recall $\varrho=\alpha_1/\alpha_2$ from Section~\ref{O:sec:r2}.
\begin{proposition}[Quadratic stationary-phase correction]\label{O:T:Q3}
One has
\begin{equation}\label{O:T:Q3phase}
 Q_3=\Psi_F\left(-\frac W2\frac{(\varrho/W)^{p-1}}{d_W}\right).
\end{equation}
\end{proposition}
\begin{proof}
Put $z=\alpha_1 w_1/A_1$. Then $v_{L_1}(z)=(p-1)(r-t_1)$,
and $\theta_1$ has conductor $1+t_1+p(r-t_1)$.
The inequalities
\[
 p(p-1)(r-t_1)\ge1+t_1+p(r-t_1),\qquad
 3(p-1)(r-t_1)\ge1+t_1+p(r-t_1)
\]
follow from $pm\ge t_1+1$: the second excess is
$(2p-3)(r-t_1)-t_1-1\ge0$, and the first is no smaller.
The formula \eqref{O:P:twist-charts} therefore applies at $z$, with
coefficient $C_1=A_1(1-z)$. Moreover
\[
 (1-z)P(z)-z+\frac{z^2}{2}
 =-\sum_{j=3}^{p-1}\frac{z^j}{j(j-1)}-\frac{z^p}{p-1}.
\]
The right-hand side has valuation at least $3v_{L_1}(z)$, so its
product with $A_1$ has $\psi_{L_1}$-value $1$. This includes $p=3$,
when the sum is empty and the remainder is $-z^3/2$. Thus
\[
 Q_3=\psi_{L_1}\bigl(-A_1[(1-z)P(z)-z]\bigr)
     =\psi_{L_1}(A_1z^2/2)
     =\psi_{L_1}\left(\frac{\alpha_1^2w_1^2}{2A_1}\right).
\]

Set
\[
 y=\frac{(\alpha_1/\gamma)w_1^2}{1+e_p/W}.
\]
Its valuation is $(p-2)(r-t_1)\ge\ceil{(t_1+1)/p}$.
Corollary~\ref{O:I:conversion} for $L_1/F$ with coefficient $\alpha_1$,
followed by taking the unique square roots in the group of
$p$-primary roots of unity, gives
\[
 \psi_{L_1}(\alpha_1 y/2)=\psi_F(-\alpha_1 \Nm_{L_1/F}(y)/2).
\]
The exact norm choice $\Nm_{L_1/F}(w_1)=\gamma/\alpha_1$ now implies
\begin{equation*}
 Q_3=\psi_F\left(-\frac\gamma2
       \frac{(\alpha_1/\gamma)^{p-1}}{\Nm_{L_1/F}(1+e_p/W)}\right).
\end{equation*}

The highest-degree term in the norm expansion in the denominator is
$\Nm_{L_1/F}(e_p/W)=W^{-p}\mathcal N_a$, so \eqref{U:symmetric} for $L_1/F$, whose break is $t_1$, gives
\[
 v_F(\Nm_{L_1/F}(1+e_p/W)-d_W)\ge
 m+\floor{\frac{(p-2)t_1+p-1}p}.
\]
Both denominators are units, and
$v_F((\alpha_1/\gamma)^{p-1})=(p-1)(m+(t_2-t_1))$. Hence
\[
 \begin{aligned}
 &v_F\!\left(\left(\frac{\alpha_1}{\gamma}\right)^{p-1}
       \left(\frac1{\Nm_{L_1/F}(1+e_p/W)}-\frac1{d_W}\right)\right)-(r+1)\\
 &\qquad\ge(p-1)m+(p-2)(t_2-t_1)
       +\floor{\frac{(p-2)t_1+p-1}p}-t_1-1\ge0.
 \end{aligned}
\]
Put $b_0=t_1-p\floor{t_1/p}$ and write
$m=\floor{t_1/p}+1+c$, where $1\le b_0<p$ and $c\ge0$.
The last expression becomes
\[
 (p-3)\floor{t_1/p}+(p-1)c+p-2-b_0+
 \floor{\frac{(p-2)b_0+p-1}p}+(p-2)(t_2-t_1).
\]
It is nonnegative for $b_0\le p-2$; for $b_0=p-1$ the floor is
$p-2$. This includes the equality boundary $p=3,b_0=2$.
The character $x\mapsto\psi_F(\gamma x)$ is trivial on
$\pp_F^{r+1}$, and $2$ is a unit. Thus replacing $\Nm_{L_1/F}(1+e_p/W)$ by
$d_W$ in the displayed expression for $Q_3$ leaves its value unchanged.
Using $\gamma=\alpha_2 W$ and $\alpha_1/\gamma=\varrho/W$ gives
\eqref{O:T:Q3phase}.
\end{proof}

\subsection{Higher-conductor comparison}\label{O:sec:weighted}
\begin{lemma}[Comparison of the power trace and symmetric coefficient]\label{O:E:trace}
One has
\begin{equation*}
 v_K\bigl(\mathcal T-w_2^{1-p}e_{p-1}\bigr)\ge (p-1)M+S+(p-1)t_1.
\end{equation*}
Moreover $v_K(\mathcal T),v_K(w_2^{1-p}e_{p-1})\ge (p-1)M+S$.
\end{lemma}
\begin{proof}
For $z=\Delta/w_2$, Lemma~\ref{O:A:lem:twisted} gives
\begin{equation*}
 v_K(e_i^{K/L_1}(z))\ge p(p-1)t_2+i(pm-pt_1),\qquad1\le i<p.
\end{equation*}

Put $B_0=(p-1)M+p(p-1)(t_2-t_1)+(p-1)t_1$.
The last assertion follows from the trace estimate in
Lemma~\ref{O:A:lem:twisted} and \eqref{O:A:eq:ei}.

Subtract $w_2^{-p}$ times the characteristic-polynomial identity for
$\Delta$ from the identity for $z$. The result is
\[
 \sum_{i=1}^{p-1}(-1)^i
  (e_i^{K/L_1}(z)-w_2^{-i}e_i)z^{p-i}
       -\Nm_{K/L_1}(\Delta)\,(W^{-1}-w_2^{-p})=0.
\]
For $i\le p-2$, \eqref{O:A:eq:ei} and the bound for $e_i^{K/L_1}(z)$ above bound the
$i$th summand below by
$pM+p(p-1)t_2-[p+(p-1)i]t_1$.
Lemma~\ref{O:A:lem:power} for $L_2/F$ bounds the last summand below
by $pM+p(p-1)t_2-pt_1$. Solving for the $i=p-1$ summand and dividing
by $z$, of valuation $M-t_1>0$, gives
\begin{equation*}
 v_K(e_{p-1}^{K/L_1}(z)-w_2^{1-p}e_{p-1})
 \ge(p-1)M+p(p-1)t_2-(p-1)^2t_1=B_0.
\end{equation*}

Newton's identity for degree $p-1$ says
\[
 \Tr_{K/L_1}(z^{p-1})=-(p-1)e_{p-1}^{K/L_1}(z)
 +\mathcal R(e_1^{K/L_1}(z),\ldots,e_{p-2}^{K/L_1}(z)),
\]
where $\mathcal R$ is an integral polynomial whose monomials have
weight $p-1$ and at least two elementary-symmetric factors.
By the same bound for $e_i^{K/L_1}(z)$, every such monomial has valuation at least
\[
 (p-1)M+p(p-1)(t_1+2(t_2-t_1))\ge B_0.
\]
The bound for $e_{p-1}^{K/L_1}(z)-w_2^{1-p}e_{p-1}$ above gives
$\mathcal T\equiv-(p-1)w_2^{1-p}e_{p-1}\pmod{\pp_K^{B_0}}$.
In mixed characteristic, \eqref{O:A:eq:pbound} gives
$v_K(p)\ge p(p-1)t_2$, so
\[
 v_K(pw_2^{1-p}e_{p-1})\ge(p-1)M+p(p-1)(t_1+2(t_2-t_1))\ge B_0.
\]
In equal characteristic this term is zero. This gives the asserted bound for $\mathcal T-w_2^{1-p}e_{p-1}$.
\end{proof}

\begin{lemma}[Weighted comparison of $p$th powers]\label{O:E:power}
One has
\begin{equation*}
 v_K\!\left(W^{1-p}\mathcal N_a(\mathcal T^p-W^{1-p}e_{p-1}^p)\right)
       \ge2p(p-1)M-pt_1+pS.
\end{equation*}
\end{lemma}
\begin{proof}

Put
$V=2p(p-1)M-pt_1+p^2(p-1)(t_2-t_1)$.
Put $R=(p-1)M+p(p-1)(t_2-t_1)$ and $B_0=R+(p-1)t_1$.
The binomial theorem gives
\[
 \mathcal T^p-w_2^{-p(p-1)}e_{p-1}^p
 =(\mathcal T-w_2^{1-p}e_{p-1})^p
 -\sum_{i=1}^{p-1}\binom pi \mathcal T^i(-w_2^{1-p}e_{p-1})^{p-i}.
\]
Lemma~\ref{O:E:trace} gives
\[
 v_K(\mathcal T-w_2^{1-p}e_{p-1})\ge B_0,\qquad
 v_K(\mathcal T),v_K(w_2^{1-p}e_{p-1})\ge R.
\]
The weight has valuation
$v_K(W^{1-p}\mathcal N_a)=p(p-1)M-p^2t_1$.
Its product with $(\mathcal T-w_2^{1-p}e_{p-1})^p$ has valuation at least $V$.
In mixed characteristic, \eqref{O:A:eq:pbound} and
$v_K\!\left(\binom pi\right)\ge v_K(p)$ for $1\le i<p$ give
$p(p-1)M-p^2t_1+v_K(p)+pR\ge V$
for every mixed binomial term.
When $\operatorname{char}F=p$, all coefficients $\binom pi$ with
$1\le i<p$ are zero, so the mixed binomial terms are absent.

It remains to replace $w_2^{-p(p-1)}e_{p-1}^p$ by
$W^{1-p}e_{p-1}^p$.
Lemma~\ref{O:A:lem:power}, applied to $w_2^{-1}\in L_2$, gives
$v_K(w_2^p/W-1)\ge p(p-1)t_2$,
so
\[
 v_K(W^{1-p}-w_2^{-p(p-1)})
   \ge p(p-1)M+p(p-1)t_2.
\]
After multiplying by $W^{1-p}\mathcal N_a(-e_{p-1})^p$, the bound is again at least
$V$ by \eqref{O:A:eq:ei}. This proves the lemma.
\end{proof}

\begin{proposition}[Valuation of the weighted norm--trace term]\label{O:E:main}
One has
\[
 L_F:=W^{1-p}\mathcal N_a\bigl(\Nm_{L_1/F}(\mathcal T)+W^{1-p}\Nm_{L_1/F}(-e_{p-1})\bigr)
       \in\pp_F^{1+t_2}.
\]
\end{proposition}
\begin{proof}
Lemma~\ref{O:A:lem:twisted} gives
$v_{L_1}(e_{p-1})\ge(p-1)(t_2-t_1)$ and $v_{L_1}(\mathcal T)\ge(p-1)(m+(t_2-t_1))$.
Using Lemma~\ref{O:A:lem:power} for $L_1/F$,
\begin{align*}
 v_{L_1}\bigl(W^{1-p}\mathcal N_a(\mathcal T^p-\Nm_{L_1/F}(\mathcal T))\bigr)
   &\ge2(p-1)M-t_1+p(p-1)(t_2-t_1),\\
 v_{L_1}\bigl(W^{2-2p}\mathcal N_a((-e_{p-1})^p-\Nm_{L_1/F}(-e_{p-1}))\bigr)
   &\ge2(p-1)M-t_1+p(p-1)(t_2-t_1).
\end{align*}
Together with Lemma~\ref{O:E:power}, these bounds give
\[
 v_{L_1}(L_F)\ge2(p-1)M-t_1+p(p-1)(t_2-t_1).
\]

Since $L_F\in F$ and $M=pm$, division by $p$ and integrality of
$v_F$ give
\[
 v_F(L_F)\ge2(p-1)m+(p-1)(t_2-t_1)-\floor{t_1/p}\ge1+t_2.
\]
Put $b_0=t_1-p\floor{t_1/p}$, so $1\le b_0<p$ and
$m\ge\floor{t_1/p}+1$. The right side minus $1+t_2$ is at least
\[
 (p-3)\floor{t_1/p}+2p-3-b_0+(p-2)(t_2-t_1)\ge p-2\ge1.
\]

\end{proof}

\begin{proposition}[Higher-conductor comparison of local constants]\label{O:T:actual-local-factors}
Let $(\theta_1,\theta_2)$ be a primitive compatible pair. If
\[
 m_{L_2}(\theta_2)>m_2,
\]
then
\[
 \Delta_{L_1}(\theta_1,\psi_{L_1})=
 \Delta_{L_2}(\theta_2,\psi_{L_2}).
\]
\end{proposition}
\begin{proof}
By Proposition~\ref{O:G:reduce}, write
\[
 \theta_i=\varphi_i(\lambda\circ\Nm_{L_i/F}),\qquad
 m_F(\lambda)=1+r,
\]
with $m=r-t_2$ and $pm>t_1$, and choose the data of
Proposition~\ref{O:P:twisted-coefficients}.
Propositions~\ref{O:T:Q124} and~\ref{O:T:Q3}, with
$d_W=1+W^{-p}\mathcal N_a$, give
\begin{equation}\label{O:T:factorphase}
 Q_1Q_2Q_3Q_4
 =\Psi_F\!\left(\frac{L_F}{2d_W^2}
 +\frac{W^{2-p}}{2d_W}
    (\Nm_{L_1/F}(-e_{p-1})-\varrho^{p-1})\right).
\end{equation}

Proposition~\ref{O:E:main} gives $L_F\in\pp_F^{1+t_2}$.
Since $pm>t_1$, the second assertion of Proposition~\ref{O:R:R2} gives
\[
 W^{2-p}(\Nm_{L_1/F}(-e_{p-1})-\varrho^{p-1})\in\pp_F^{1+t_2},
\]
using $v_F(W)=-m$. Since $d_W$ and $2$ are units, both summands in
the argument of $\Psi_F$ in \eqref{O:T:factorphase} lie in
$\pp_F^{1+t_2}$. Thus $Q_1Q_2Q_3Q_4=1$.
The final assertion of Proposition~\ref{O:P:four} puts the quotient of the
local constants in $\mu_4$. Corollary~\ref{U:odd-power} gives its
$p$th power equal to $1$.
Since $p$ is odd, $\mu_p\cap\mu_4=\{1\}$, proving equality.
\end{proof}

\subsection{Completion}\label{O:sec:total}
\begin{theorem}[The totally ramified case for odd $p$]\label{O:G:totalbranch}
Theorem~\ref{U:main} holds when $\ell=p>2$ and $K/F$ is totally
ramified, in mixed and equal characteristic.
\end{theorem}
\begin{proof}
First compare the smallest-break field $L_1$ chosen above with an arbitrary
other degree-$p$ intermediate field. If the second conductor is $m_2$,
Proposition~\ref{O:M:realize}, applied to this pair, gives first
conductor $m_1$. Proposition~\ref{O:G:minimal} therefore applies.

Otherwise Lemma~\ref{O:G:lower} gives $m_{L_2}(\theta_2)>m_2$,
so Proposition~\ref{O:T:actual-local-factors} gives equality of the local
constants. The factors indexed by $S(L_i/F)$ are $1$ by
\eqref{U:odd-norm-factor}.

For a pair not containing $L_1$, extend its invariant common pullback to $L_1$
by Lemmas~\ref{C:hilbert90} and~\ref{C:character-extension};
the two comparisons with $L_1$
then give the desired equality by transitivity.
\end{proof}

\section{Dyadic extensions with an unramified intermediate field}\label{sec:dyadic-ur}
Let $F$ be a nonarchimedean local field of residue characteristic $2$,
let $k$ be its residue field, and let $K/F$ be biquadratic. Fix a nontrivial additive character
$\psi_F$ of $F$ and put $\psi_L=\psi_F\circ\Tr_{L/F}$.
For a quadratic intermediate field $L$, let $\omega_L$ be the
nontrivial character of $F^\times/\Nm_{L/F}(L^\times)$.

\begin{theorem}[The case with an unramified quadratic intermediate field]\label{D:UR:URall}
Let $K/F$ be biquadratic over a local field of residue characteristic two,
and assume it has an unramified quadratic intermediate field $U$.
For every other quadratic intermediate field $E$ and every primitive
compatible pair $(\theta_U,\theta_E)$, one has
\[
 \Delta_U(\theta_U,\psi_U)\Delta_F(\omega_U,\psi_F)
 =\Delta_E(\theta_E,\psi_E)\Delta_F(\omega_E,\psi_F).
\]
Thus $\mathcal A_L(\theta_L)$ is independent of the quadratic intermediate field $L$.
\end{theorem}

Let $E/F$ be ramified quadratic of break $t\ge1$ and put $T=t+1$.
The stationary linearization \eqref{U:stationary-linearization},
applied to $\omega_E$ at depth $\ceil{T/2}$ and written with $1-z$, gives
$\alpha\in F^\times$ such that
\begin{equation}\label{D:UR:tauchart}
 \omega_E(1-z)=\Psi_F(z),\quad z\in\pp_F^{\ceil{T/2}},\qquad
 \Psi_L(w):=\psi_L(\alpha w),\qquad v_F(\alpha)=-n_F(\psi_F)-T.
\end{equation}
Equation~\eqref{U:trace-ideal} shows that both $\Psi_F$ and $\Psi_E$
have largest trivial ideal with exponent $T$. The norm-conversion
formula \eqref{D:UR:phase} therefore applies to these characters.
We use $v_F(2)=+\infty$ in characteristic two.

\begin{lemma}[Trace--norm normalization]\label{D:UR:oddT}
For $w\in E^\times$, put
\[
 c_w=\frac{\Tr_{E/F}(w)^2}{\Nm_{E/F}(w)}.
\]
If $v_F(c_w)=T-1$, then
\begin{equation}\label{D:UR:trace-norm-normalization}
 \Psi_F(c_wz)=(-1)^{\Tr_{k/\F_2}(\bar z)}
 \qquad(z\in\OO_F).
\end{equation}
In particular, if $T$ is odd, then $\operatorname{char}F=0$,
$T=2v_F(2)+1$, and
\[
 \Psi_F(4z)=(-1)^{\Tr_{k/\F_2}(\bar z)},\qquad z\in\OO_F.
\]
\end{lemma}
\begin{proof}
Assume $v_F(c_w)=T-1$ and put
$u=\Tr(w)w/\Nm(w)$. Then $v_E(u)=T-1>0$ and
\[
 \Nm_{E/F}(1+uz)=1+c_w(z+z^2),\qquad z\in\OO_F.
\]
Since $\omega_E$ is trivial on norms,
\[
 1=\omega_E\!\left(1+c_w(z+z^2)\right)
  =\Psi_F\!\left(-c_w(z+z^2)\right),
\]
where \eqref{D:UR:tauchart} applies because
$v_F(c_w(z+z^2))\ge T-1\ge\ceil{T/2}$.
Thus $\bar z\mapsto\Psi_F(c_wz)$ is trivial on the
Artin--Schreier image. It is well defined on $k$ because $\Psi_F$ is
trivial on $\pp_F^T$, and nontrivial because $v_F(c_w)=T-1$.
Hence it is the unique nontrivial additive character of $k$ trivial
on the Artin--Schreier image, proving
\eqref{D:UR:trace-norm-normalization}. This is the normalization used in
\cite[proof of Proposition~11.6]{Ueda}.

If $T$ is odd, \cite[Lemma~11.2(a)]{Ueda} gives
$\operatorname{char}F=0$ and $T=2v_F(2)+1$, so $v_F(4)=T-1$.
Apply \eqref{D:UR:trace-norm-normalization} with $w=1$.
\end{proof}

\subsection{Compatible characters of minimal conductor and their twists}
Assume from now through Section~\ref{D:UR:sec:critical} that $U/F$ is unramified
quadratic, $E/F$ ramified quadratic of break $t$, and $K=EU$.
Let $\omega_U$ and $\omega_E$ denote the nontrivial norm characters of $U/F$ and $E/F$, respectively.
By Lemmas~\ref{D:crossed-norm} and~\ref{D:quadratic-norms}, the nontrivial
norm character on $K/U$ is $\omega_E\circ\Nm_{U/F}$. All four
$\Psi_L=\psi_L(\alpha\,\cdot)$ have largest
trivial ideal $\pp_L^T$. Let $\kappa$ be the residue field of $U$, so
$k\subset\kappa$.
Choose $c\in\OO_F$ whose residue has absolute trace $1$, and choose
\[d\in\OO_U,\qquad d^2-d=c.\]
This polynomial gives the unramified quadratic extension. For its
nontrivial automorphism $\sigma$,
\[
 \sigma(d)=1-d,\quad \Tr_{U/F}(d)=1,\quad \Nm_{U/F}(d)=-c.
\]
In particular $c,d$ are units and $\bar d\notin k$.

\begin{proposition}[Construction of compatible characters of minimal conductor]\label{D:UR:models}
There exist characters $\chi_U^0,\chi_E^0$, both of conductor $T$, with
a primitive common norm pullback, such that
\begin{align}
 (\chi_U^0)^\sigma/\chi_U^0&=\omega_E\circ\Nm_{U/F},\notag\\
 \chi_U^0(1-z)&=\Psi_U(d^2z) &&(z\in\pp_U^{\ceil{T/2}}),\label{D:UR:modelU}\\
 \chi_E^0(1-z)&=\Psi_E(cz) &&(z\in\pp_E^{\ceil{T/2}}).\label{D:UR:modelE}
\end{align}

\end{proposition}
\begin{proof}
First define $\chi_U^0$ on $U_U^{\ceil{T/2}}$ by \eqref{D:UR:modelU}.
Since $2\ceil{T/2}\ge T$ and $d$ is a unit, this is a character of exact
conductor $T$. Unramified norm expansion shows
\[
 (\omega_E\circ\Nm_{U/F})(1-z)=\Psi_U(z)\quad(z\in\pp_U^{\ceil{T/2}}),
\]
because $v_F(\Nm_{U/F}(z))\ge2\ceil{T/2}\ge T$. Further,
$\sigma(d^2)-d^2=1-2d$. For $z\in\pp_U^{\ceil{T/2}}$, applying \eqref{U:trace-ideal} to $1$
gives $v_F(2)\ge\floor{T/2}$, hence
$v_U(2dz)\ge v_F(2)+\ceil{T/2}\ge T$, so
$\Psi_U(2dz)=1$. Thus the asserted conjugate-quotient identity holds on $U_U^{\ceil{T/2}}$.

Apply the character-extension argument in the proof of
Proposition~\ref{O:I:canonical} to these prescriptions on
$U_U^{\ceil{T/2}}$, with $\omega_E\circ\Nm_{U/F}$ in place of $\tau_U$.
It gives $\chi_U^0$ on $U^\times$ with the stated nontrivial conjugate
quotient. The descent and primitivity arguments in that proof then give
$\chi_E^0$ with the same primitive norm pullback.

For $z\in\pp_K^{\ceil{T/2}}$, the ramified quadratic norm gives
\[
 \chi_U^0(\Nm_{K/U}(1-z))
 =\Psi_K(d^2z)\Psi_U(-d^2\Nm_{K/U}(z))
 =\Psi_K((d^2-d)z)=\Psi_K(cz).
\]
The middle equality uses equation~\eqref{D:UR:phase} for $K/U$ on
$dz\in\pp_K^{\ceil{T/2}}$ and $\Nm_{K/U}(dz)=d^2\Nm_{K/U}(z)$. The surjectivity $\Nm_{K/E}(U_K^{\ceil{T/2}})=U_E^{\ceil{T/2}}$ from
\eqref{U:unit-norm-surjective}, together with the norm expansion modulo
$\pp_E^T$, gives \eqref{D:UR:modelE}. Since $c$ is
a unit, its conductor is exactly $T$.

\end{proof}

\begin{proposition}[Reduction to a common twist]\label{D:UR:actual}
For every primitive compatible pair $(\theta_U,\theta_E)$ there are
a quasi-character $\lambda$ of $F^\times$, a character $\chi_E$ of
$E^\times$, and an integer $h\ge0$ such that
\begin{equation}\label{D:UR:twistpair}
 \begin{aligned}
  \theta_U&=\chi_U^0(\lambda\circ\Nm_{U/F}),&
  m_U(\theta_U)&=T+h,\\
  \theta_E&=\chi_E(\lambda\circ\Nm_{E/F}),&
  m_E(\theta_E)&=T+2h.
 \end{aligned}
\end{equation}
For $h>0$, $m_F(\lambda)=T+h$; for $h=0$, $m_F(\lambda)\le T$.
The character $\chi_E$ has conductor $T$, satisfies
\eqref{D:UR:modelE}, and has the same norm pullback to $K$ as $\chi_U^0$.
\end{proposition}
\begin{proof}
By Proposition~\ref{U:conjugacy}, the conjugate quotients of
$\theta_U$ and $\chi_U^0$ both equal $\omega_E\circ\Nm_{U/F}$.
Apply the descent argument in the proof of
Proposition~\ref{O:I:actualtwist} with this character in place of
$\tau_U$. It gives $\lambda$ and $\mu_E\in S(K/E)$; set
$\chi_E=\chi_E^0\mu_E$. Since $K/E$ is unramified, $\mu_E$ is
unramified, so $\chi_E$ retains conductor $T$ and
\eqref{D:UR:modelE}. The conductor calculation in the same proof,
using \eqref{U:unramified-pullback-conductor}--\eqref{U:break-pullback-conductor}
with degree $2$, gives the asserted conductors and bounds on $m_F(\lambda)$.
\end{proof}

\subsection{Stationary coefficients and the minimal-conductor comparison}
For $0\le h\le t$, put
\[
 s_U=\ceil{(T+h)/2},\qquad s_E=\ceil{T/2}+h.
\]
Since $2s_U\ge T+h$, \cite[Lemma~4.1]{Ueda} gives $\gamma\in F$ such that
\begin{equation}\label{D:UR:lambdachart}
 \lambda(1-z)=\psi_F(\gamma z)\quad(z\in\pp_F^{s_U}).
\end{equation}
For $h>0$, $v_F(\gamma/\alpha)=-h$; at $h=0$, one may choose $\gamma/\alpha\in\OO_F$.
The coefficient $\gamma/\alpha$ is determined modulo $\pp_F^{T-s_U}$.

\begin{lemma}[Norm representatives for the twisting coefficient]\label{D:UR:choosex}
There is $x\in E$, with $A=\Nm_{E/F}(x)$, such that one may take $\gamma=\alpha A$ in
\eqref{D:UR:lambdachart}. If $h>0$, $v_E(x)=-h$; at $h=0$, $x$ is integral.
\end{lemma}
\begin{proof}
If $\gamma/\alpha\in\pp_F^{T-s_U}$, then $h=0$ and
$\lambda|_{U_F^{s_U}}=1$; take $x=A=0$.
Otherwise put $a=v_F(\gamma/\alpha)$ and $s=T-s_U-a$.
For $h>0$, one has $a=-h$ and
$1\le s=T-s_U+h=\floor{(T+h)/2}\le t$,
since $T+h\le2T-1$. For $h=0$, one has $0\le a<T-s_U$ and
$1\le s\le T-s_U=\floor{T/2}\le t$.
Equation~\eqref{U:norm-approximation}, applied to $c=\gamma/\alpha$, gives
$x\in E^\times$ with
\[
 v_E(x)=a,\qquad
 \gamma/\alpha-\Nm_{E/F}(x)\in\pp_F^{a+s}=\pp_F^{T-s_U}.
\]
Thus \eqref{D:UR:lambdachart} holds with $\gamma=\alpha A$,
where $A=\Nm_{E/F}(x)$.
\end{proof}
Choose $x,A$ as in Lemma~\ref{D:UR:choosex}. For this choice put
$B_U=A+d^2$ and $B_E=A-x+c$.
\begin{proposition}[Coefficients for $\theta_U$ and $\theta_E$]\label{D:UR:coeffs}
For the given characters in \eqref{D:UR:twistpair},
\begin{equation*}
 \theta_U(1-z)=\Psi_U(B_Uz)\ (z\in\pp_U^{s_U}),\qquad
 \theta_E(1-z)=\Psi_E(B_Ez)\ (z\in\pp_E^{s_E}).
\end{equation*}
Their coefficient valuations are $-h$ and $-2h$, respectively.
\end{proposition}
\begin{proof}
For the unramified pullback of $\lambda$, the norm term in the argument
of $\Psi_F$ has valuation at least $-h+2s_U\ge T$; use
\eqref{D:UR:modelU}.

For $z\in\pp_E^{s_E}$,
$\Tr_{E/F}(z)-\Nm_{E/F}(z)\in\pp_F^{s_U}$, since
\[
 s_E\ge s_U,\qquad \floor{(s_E+T)/2}\ge s_U.
\]
Then \eqref{D:UR:lambdachart} and equation~\eqref{D:UR:phase}, applied to $xz$ of
valuation at least $\ceil{T/2}$, give
\[
 \lambda(\Nm_{E/F}(1-z))=\Psi_F(A\Tr_{E/F}(z)-A \Nm_{E/F}(z))
                  =\Psi_E((A-x)z).
\]
Together with \eqref{D:UR:modelE} this proves the formulas.

For $h>0$, $v_U(B_U)=v_U(A)=-h$ and
$v_E(B_E)=v_E(A)=-2h$, since the other summands have larger valuation. At $h=0$, write
$\xi=\bar x\in k$. The residues are $\xi^2+\bar d^2$ and
$\xi^2-\xi+\bar c$. The first is nonzero since $\bar d\notin k$, and
the second has absolute trace $1$.
\end{proof}
For the same choice of $x$, put
\[
 C=x+d,\qquad Z_U=\Nm_{K/U}(C),\quad Z_E=\Nm_{K/E}(C),\quad b=\Tr_{E/F}(x).
\]
\begin{lemma}[Norm coefficients and coefficient differences]\label{D:UR:norm-differences}
One has
\[
 \begin{aligned}
 Z_U&=A+bd+d^2,& Z_E&=x^2+x-c,\\
 Z_U-B_U&=bd,& Z_E-B_E&=bx-2A+2x-2c.
 \end{aligned}
\]
If $x\ne0$ and $x'$ is its conjugate over $F$, then
$(x'/x)Z_E-B_E=b(1-c/x)$.
\end{lemma}
\begin{proof}
The two quadratic norm expansions give $Z_U,Z_E$.
Use $x^2=bx-A$ to subtract $B_E$. For the last identity,
$(x'/x)Z_E=A+x'-c(x'/x)$, while $x+x'=b$.
\end{proof}

\begin{proposition}[Minimal-conductor comparison]\label{D:UR:minimal}
Every primitive compatible pair with
$m_U(\theta_U)=m_E(\theta_E)=T$ satisfies
\begin{equation*}
 \Delta_U(\theta_U)\Delta_F(\omega_U)
 =\Delta_E(\theta_E)\Delta_F(\omega_E).
\end{equation*}
\end{proposition}
\begin{proof}
Write $e=v_F(2)$.
We treat even and odd conductor separately.

First suppose $T=2r$ is even. The element $x$ in Lemma~\ref{D:UR:choosex} is integral. The differences
$Z_U-B_U=bd$ and
$Z_E-B_E=bx-2A+2x-2c$
have respective valuations at least $r$ in $U$ and $2r$ in $E$, by
\eqref{U:trace-ideal} and $v_F(2)\ge r$. Thus $Z_U,Z_E$ are stationary coefficients for $\theta_U,\theta_E$
in Lemma~\ref{U:stationary-factor}. All three conductors $T$ are even, so there are
no residual sums.

Taking traces gives
\begin{equation*}
 S:=\Tr_{U/F}(Z_U)-\Tr_{E/F}(Z_E)=1+4c-\mathcal D_x,
 \qquad \mathcal D_x=b^2-4A=(2x-b)^2.
\end{equation*}
Here $\Tr_{U/F}(d)=1$, $\Tr_{U/F}(d^2)=1+2c$, and
$\Tr_{E/F}(x^2)=b^2-2A$.
Put $c_0=-\alpha^{-1}$. Compatibility gives
$\theta_U(Z_U)=\theta_E(Z_E)$; Proposition~\ref{U:determinant} gives
$\theta_U(c_0)/\theta_E(c_0)=\omega_U(c_0)\omega_E(c_0)$.
Using \eqref{U:normalized-factor} and
$\Delta_F(\omega_E)=\omega_E(c_0)\Psi_F(-1)$ therefore gives
\begin{equation*}
 \frac{\Delta_U(\theta_U)\Delta_F(\omega_U)}
      {\Delta_E(\theta_E)\Delta_F(\omega_E)}
 =(-1)^T\Psi_F(1-S).
\end{equation*}
Here $\Delta_F(\omega_U)=(-1)^{n_F(\psi_F)}$ and
$\omega_U(c_0)=(-1)^{n_F(\psi_F)+T}$.

Now $2x-b$ is trace zero and has $E$-valuation at least $2r$.
Equation~\eqref{D:UR:phase} gives $\Psi_F(\mathcal D_x)=1$ since
$\Nm_{E/F}(2x-b)=-\mathcal D_x$. Also $v_F(4c)\ge2r=T$.
Substituting $S=1+4c-\mathcal D_x$ proves that the displayed quotient is $1$.

Suppose now that $T$ is odd. By Lemma~\ref{D:UR:oddT}, $T=2e+1$, $\operatorname{char}F=0$ and $e\ge1$.
We have $x\in\OO_E$ and $b\in\pp_F^e$.
The differences $Z_U-B_U$ and $Z_E-B_E$ have valuations at least
$e$ and $2e$, respectively. Both therefore satisfy the stationary
coefficient congruences modulo $\pp_U^e$ and $\pp_E^e$. In addition put
\begin{equation*}
 W_E=\Nm_{E/F}(2x+1),\qquad W_U=\Nm_{U/F}(b+2d).
\end{equation*}
The element $2x+1$ is a unit and
$W_E=1+2b+4A\equiv1\pmod{\pp_F^{2e}}$. Thus $W_E$ is a stationary coefficient for $\omega_E$ relative to $\Psi_F$.
Moreover $\omega_E(W_E)=1$.

Use the stationary coefficients $Z_U,Z_E,W_E$ for
$\theta_U,\theta_E,\omega_E$ in \eqref{U:normalized-factor}, and denote their
normalized sums by $g_U,g_E,g_{\omega_E}$. Compatibility,
\eqref{U:det-character}, and Lemma~\ref{U:e2} applied to $C$ give
\begin{equation}\label{D:UR:oddratio}
 \frac{\Delta_U(\theta_U)\Delta_F(\omega_U)}
      {\Delta_E(\theta_E)\Delta_F(\omega_E)}
 =(-1)^T\Psi_F(W_U)\frac{g_U}{g_Eg_{\omega_E}}.
\end{equation}
The prefactor on the right is $1$. Indeed $a=b/2\in\OO_F$ and
\[
 W_U=4\Nm_{U/F}(d+a),\qquad
 \overline{\Nm_{U/F}(d+a)}=\bar a^2+\bar a-\bar c.
\]
Its absolute trace is $1$. Lemma~\ref{D:UR:oddT} gives $\Psi_F(W_U)=-1$,
which cancels $(-1)^T=-1$.

Choose a uniformizer $\Pi$ of $E$ and put $\pi=\Nm_{E/F}(\Pi)$; this is a
uniformizer of $F$ and $U$. In \eqref{U:normalized-residual} use critical elements
$\pi^e$ for $U,F$ and $\Pi^e$ for $E$, and write $H_U,H_E,H_{\omega_E}$
for the critical functions of $\theta_U,\theta_E,\omega_E$ with these choices.
Keep $C$ fixed and, for $z\in\OO_K$, introduce
$Y_z=C(1+\Pi^e z)$.
The three maps on the residue fields are
\begin{equation*}
\begin{aligned}
 \overline{\frac{\Nm_{K/U}(Y_z)/Z_U-1}{\pi^e}}
   &=\bar z^2 &&\text{in }\kappa,\\
 \overline{\frac{\Nm_{K/E}(Y_z)/Z_E-1}{\Pi^e}}
   &=\Tr_{\kappa/k}(\bar z) &&\text{in }k,\\
 \overline{\frac{\Nm_{E/F}(\Tr_{K/E}(Y_z))/W_E-1}{\pi^e}}
   &=\bigl(\Tr_{\kappa/k}(\bar C\bar z)\bigr)^2 &&\text{in }k.
\end{aligned}
\end{equation*}
To prove the first congruence, note that $\Tr_{K/U}(\Pi^ez)$ has
$U$-valuation at least $\floor{(e+T)/2}\ge e+1$, while the norm term
has $U$-valuation at least $e$ and residue $\bar z^2$ after division
by $\pi^e$. For the second congruence, the trace term of
$\Nm_{K/E}(1+\Pi^ez)$ has $E$-valuation at least $e$ and residue
$\Tr_{\kappa/k}(\bar z)$ after division by $\Pi^e$, whereas its
quadratic term has $E$-valuation at least $2e\ge e+1$.
For the third, the relative trace change has $E$-valuation at least
$e$ and residue $\Tr_{\kappa/k}(\bar C\bar z)$ after division by
$\Pi^e$, since
$\overline{2x+1}=1$; apply the first ramified calculation to its norm.

Moreover
$\Nm_{U/F}(\Tr_{K/U}(Y_z))-W_U\in\pp_F^T$.
Indeed $v_U(b+2d)=e$, since $b/2\in\OO_F$ and $\bar d\notin k$.
The trace change $\Tr_{K/U}(C\Pi^ez)$ has valuation at least
$\floor{(e+T)/2}\ge e+1$. Expanding its unramified quadratic norm proves
the equation above.

Apply Lemma~\ref{U:common-function} to $C,Y_z$
and the fields $U,E$; the required traces are nonzero by the estimates above.
The norm-character values of both relative trace norms are $1$.
The preceding equation removes the $U$-trace-norm additive factor.
Equation~\eqref{U:normalized-residual}, with the three critical coordinates computed above, therefore gives
\begin{equation*}
 H_U(z^2)=H_E(\Tr_{\kappa/k}(z))
           H_{\omega_E}\bigl((\Tr_{\kappa/k}(\bar C z))^2\bigr),
 \qquad z\in\kappa.
\end{equation*}

The map
\[
 z\longmapsto\bigl(\Tr_{\kappa/k}(z),(\Tr_{\kappa/k}(\bar C z))^2\bigr)
        :\kappa\longrightarrow k\times k
\]
is bijective. Write $z=u+v\bar d$ and $\bar x=\xi\in k$; its coordinates
are $v$ and $(u+(\xi+1)v)^2$. Frobenius on $k$ is bijective.
Since Frobenius on $\kappa$ is bijective too, summing the identity for $H_U(z^2)$ above gives
\[
 \sum_{\kappa}H_U
     =\left(\sum_kH_E\right)\left(\sum_kH_{\omega_E}\right),
 \qquad g_U=g_Eg_{\omega_E}.
\]
Together with \eqref{D:UR:oddratio}, this proves
Proposition~\ref{D:UR:minimal} when $T$ is odd.
\end{proof}

\subsection{Twists below the stable range}
\label{D:UR:sec:critical}
Throughout this subsection let $1\le h<T$. Write the pair as in Proposition~\ref{D:UR:actual}, and choose
$x,A$ as in Lemma~\ref{D:UR:choosex}.
For this choice use the associated $B_U,B_E,C,Z_U,Z_E,b$ above.
In particular $v_E(x)=-h$. Let $x'$ be its nontrivial
$E/F$ conjugate, and put
$\widehat Z_E=(x'/x)Z_E$.
Unlike in the proof of Proposition~\ref{D:UR:minimal}, for $1\le h<T$ the difference $Z_E-B_E$ need not have the
valuation required for a stationary coefficient. We replace $Z_E$
by $\widehat Z_E$.
Proposition~\ref{D:UR:correctedreps} gives both its coefficient congruence
and its character value.

\begin{lemma}[The sign determined by the trace]\label{D:UR:traceparity}
Put $\mathfrak u=b^2/A$. If $T-h$ is even, then $v_F(\mathfrak u)\ge T$.
If $T-h$ is odd, then
\[
 v_F(\mathfrak u)=T-1,\qquad
 \Psi_F(\mathfrak u z)=(-1)^{\Tr_{k/\F_2}(\bar z)}
 \quad(z\in\OO_F).
\]
Consequently, 
$(-1)^{T+h}\Psi_F(-c\mathfrak u)=1$.
\end{lemma}
\begin{proof}
The trace-ideal formula \eqref{U:trace-ideal} gives
$v_F(b)\ge B:=\floor{(T-h)/2}$.
If $T-h$ is even, $2B+h=T$, proving the first assertion.

If $T-h$ is odd, the trace images of the two consecutive source ideals
are exactly $\pp_F^B$ and $\pp_F^{B+1}$. Hence trace induces a
nonzero map between one-dimensional $k$-spaces, so an element of exact
valuation $-h$ has trace of exact valuation $B$. Thus $b\ne0$ and
$v_F(\mathfrak u)=2B+h=T-1$.

Apply \eqref{D:UR:trace-norm-normalization} with $w=x$. Since
$\Tr_{E/F}(x)=b$ and $\Nm_{E/F}(x)=A$, one has $c_w=\mathfrak u$,
and therefore
\[
 \Psi_F(\mathfrak u z)=(-1)^{\Tr_{k/\F_2}(\bar z)}
 \qquad(z\in\OO_F).
\]
Since $\bar c$ has absolute trace one,
$\Psi_F(-c\mathfrak u)=-1$ in this parity and is $1$ in the even
parity, giving the asserted sign identity.
\end{proof}

\begin{proposition}[Corrected stationary coefficients and their character values]\label{D:UR:correctedreps}
For $1\le h<T$, the norm $Z_U=\Nm_{K/U}(C)$ and the coefficient
$\widehat Z_E$ defined above satisfy
$\theta_U(1-z)=\Psi_U(Z_Uz)$ for $z\in\pp_U^{s_U}$ and
$\theta_E(1-z)=\Psi_E(\widehat Z_Ez)$ for $z\in\pp_E^{s_E}$. Moreover
\begin{align}
 \widehat S:=\Tr_{U/F}(Z_U)-
                   \Tr_{E/F}(\widehat Z_E)&=1+c\mathfrak u,\notag\\
 \frac{\theta_E(\widehat Z_E)}{\theta_U(Z_U)}
                 &=(-1)^h.\label{D:UR:newmult}
\end{align}
\end{proposition}
\begin{proof}
Put $\rho=x'/x$. Lemma~\ref{D:UR:norm-differences} gives
$Z_U-B_U=bd$ and $\widehat Z_E-B_E=b(1-c/x)$.
Since $v_E(x)=-h<0$, the latter parenthesis is a unit.
Put $B=\floor{(T-h)/2}$. The trace bound for $b$ gives
\[
 v_U(Z_U-B_U)\ge B,\qquad
 v_E(\widehat Z_E-B_E)\ge2B.
\]
Here $T-s_U=B$ and $T-s_E=\floor{T/2}-h$.
One has
$2B-(T-s_E)=\ceil{T/2}-((T-h)\bmod2)\ge0$.
Consequently
\[
 Z_U-B_U\in\pp_U^{T-s_U},\qquad
 \widehat Z_E-B_E\in\pp_E^{T-s_E}.
\]
The character formulas therefore follow from Proposition~\ref{D:UR:coeffs}.
That proposition also gives $v_U(B_U)=-h$ and $v_E(B_E)=-2h$.
Since $B\ge0$ and $h>0$, the two displayed valuation bounds give
$v_U(Z_U)=-h$ and $v_E(\widehat Z_E)=-2h$.

One has $\Tr_{E/F}(\rho)=(b^2-2A)/A=\mathfrak u-2$, and
\[
 \Tr_{E/F}(\widehat Z_E)=2A+b+2c-c\mathfrak u,
 \qquad \Tr_{U/F}(Z_U)=2A+b+1+2c.
\]
This proves the asserted formula for $\widehat S$.

The uncorrected $Z_U,Z_E$ are the two norms of $C=x+d$, so compatibility gives $\theta_U(Z_U)=\theta_E(Z_E)$.
Proposition~\ref{U:conjugacy} on $E$ gives
\[
 \theta_E(\rho)=\frac{\theta_E(x')}{\theta_E(x)}
      =\omega_U(\Nm_{E/F}(x))=(-1)^{v_F(A)}=(-1)^h,
\]
which proves \eqref{D:UR:newmult}.
\end{proof}

\begin{proposition}[The quotient of local constants]\label{D:UR:criticalratio}
Let $g_U,g_E,g_{\omega_E}$ be the normalized sums in
Lemma~\ref{U:stationary-factor} for $\theta_U,\theta_E,\omega_E$ with coefficients
$Z_U,\widehat Z_E,1$, respectively, taking each factor to be $1$
when its character has even conductor. Then
\begin{equation*}
 \frac{\Delta_U(\theta_U)\Delta_F(\omega_U)}
      {\Delta_E(\theta_E)\Delta_F(\omega_E)}
  =\frac{g_U}{g_Eg_{\omega_E}}.
\end{equation*}
\end{proposition}
\begin{proof}
Let $c_0=-\alpha^{-1}$. Proposition~\ref{U:determinant} gives
$\theta_U(c_0)/\theta_E(c_0)=\omega_U(c_0)\omega_E(c_0)$.
The product $\omega_U(c_0)\Delta_F(\omega_U)$ is $(-1)^T$, since
$v_F(c_0)=n_F(\psi_F)+T$ and
\eqref{U:unramified-local-constant} gives
$\Delta_F(\omega_U)=(-1)^{n_F(\psi_F)}$.
The quotient of the multiplicative stationary values is \eqref{D:UR:newmult}.
The additive phase from \eqref{U:normalized-factor} and the norm-character factor is
$\Psi_F(1-\widehat S)=\Psi_F(-c\mathfrak u)$.
Thus the quotient of local constants is
\[
 (-1)^{T+h}\Psi_F(-c\mathfrak u)\frac{g_U}{g_Eg_{\omega_E}}.
\]
Lemma~\ref{D:UR:traceparity} makes the prefactor $1$, proving the stated quotient formula.
\end{proof}

\begin{proposition}[Evaluation of the critical sum over $U$]\label{D:UR:resU}
Suppose $T+h$ is odd and $h>0$. Choose a critical element
$\delta_U\in F$ with $v_F(\delta_U)=(T+h-1)/2$.
The critical function $H_U$ for $\theta_U$ with coefficient $Z_U$ and
critical element $\delta_U$ is $1$ on $k\subset\kappa$.
Its normalized sum is $g_U=1$.
\end{proposition}
\begin{proof}
Put $a=(T+h-1)/2$. For $y\in\OO_F$, put $z=\delta_Uy$. One has
$2a-s_E=\floor{T/2}-1\ge0$.
Also $a\ge \ceil{T/2}$: for even $T$ this follows from odd $h\ge1$; for odd $T$
the present parity requires even $h\ge2$. Thus $\theta_E(1+z)$ can be evaluated by
Proposition~\ref{D:UR:correctedreps}, and $\omega_E(1+z)$ by
\eqref{D:UR:tauchart}.

Equation~\eqref{U:det-character} and triviality of $\omega_U$ on units give
$\theta_U(1+z)=\theta_E(1+z)\omega_E(1+z)$.
Substitution into \eqref{U:normalized-residual} gives
\[
 H_U(\bar y)=\Psi_F((1-\widehat S)z)
            =\Psi_F(-c\mathfrak u\,\delta_Uy)=1.
\]
Indeed $v_F(\mathfrak u)\ge T-1$ and $a\ge1$.

By Lemma~\ref{U:stationary-factor}, the polar pairing of $H_U$
on the additive group of $\kappa$ is nondegenerate. Its restriction to $k$
is $1$, so $k$ is isotropic for its polar pairing. Since $|k|^2=|\kappa|$,
nondegeneracy gives $k^\perp=k$. Summation on a coset of $k$ is zero unless
the coset is $k$ itself, and the sum on $k$ is $|k|$. Thus
$|\kappa|^{-1/2}\sum_\kappa H_U=1$. 
\end{proof}

\begin{proposition}[Reciprocity of the two ramified critical sums]\label{D:UR:resE}
For every $1\le h<T$ when $T$ is odd,
\[
 g_E=g_{\omega_E}^{-1}.
\]
\end{proposition}
\begin{proof}
Write $e=v_F(2)$.
For $z\in\pp_E^e$ one has
$\omega_E(1+\Nm_{E/F}(z))=\Psi_E(z)$.
Indeed, the trace of $z$ has valuation at least
$\floor{(e+T)/2}\ge e+1$, and $v_F(\Nm_{E/F}(z))\ge e$.
The identity
\[
 \Nm_{E/F}(1+z)=(1+\Nm_{E/F}(z))
             \left(1+\frac{\Tr_{E/F}(z)}{1+\Nm_{E/F}(z)}\right)
\]
can therefore be evaluated by $\omega_E$, using \eqref{D:UR:tauchart} on the second factor only. Since the left side is a norm,
\[
 \omega_E(1+\Nm_{E/F}(z))
   =\Psi_F\!\left(\frac{\Tr_{E/F}(z)}{1+\Nm_{E/F}(z)}\right)
   =\Psi_E\!\left(\frac{z}{1+\Nm_{E/F}(z)}\right).
\]
The difference between $z/(1+\Nm_{E/F}(z))$ and $z$ has $E$-valuation at least
$3e\ge2e+1=T$. This proves the identity on $\pp_E^e$.

Choose a uniformizer $\Pi$ of $E$ and put $\pi=\Nm_{E/F}(\Pi)$.
In \eqref{U:normalized-residual}, use coefficients $\widehat Z_E,1$
and critical elements $\Pi^{e+h},\pi^e$ for $\theta_E,\omega_E$, respectively.
For $z\in\pp_E^{e+h}$, both the
formulas \eqref{D:UR:modelE} and \eqref{D:UR:lambdachart} apply. The first is
valid since $e+h\ge e+1=\ceil{T/2}$. For the second, the two terms in
$\Nm_{E/F}(1+z)-1=\Tr_{E/F}(z)+\Nm_{E/F}(z)$ have $F$-valuations at least
\[
 \floor{(3e+h+1)/2}\quad\hbox{and}\quad e+h,
\]
respectively, both at least
$s_U=e+\ceil{(h+1)/2}$. It follows that
\begin{equation*}
 \theta_E(1+z)=\Psi_E(-(A+c)z)\Psi_F(-A \Nm_{E/F}(z)).
\end{equation*}

The difference between $\widehat Z_E$ and $A-x+c$ is $b(1-c/x)$.
Its product with $z$ has valuation at least
$2\floor{(T-h)/2}+e+h =3e+(h\bmod2)\ge T$.
Equation~\eqref{U:normalized-residual} therefore gives
\begin{align*}
 H_E(\overline{z/\Pi^{e+h}})
   &=\Psi_E(-\widehat Z_Ez)\theta_E(1+z)^{-1}\\
   &=\Psi_E(xz)\Psi_F(A \Nm_{E/F}(z))\\
   &=H_{\omega_E}\!\left(\overline{\Nm_{E/F}(xz)/\pi^e}\right)^{-1}.
\end{align*}
For the last equality use $\omega_E(1+\Nm_{E/F}(xz))=\Psi_E(xz)$, proved above, and \eqref{U:normalized-residual}, since $v_E(xz)\ge e$ and
$\Nm_{E/F}(xz)=A\Nm_{E/F}(z)$.

Write $x=\Pi^{-h}u$ with $u\in\OO_E^\times$. On the two residue fields $k$ the
map $z\mapsto \Nm_{E/F}(xz)$ induces
\[
 y\longmapsto\bar u^2 y^2.
\]
This is a bijection. Taking normalized sums, as in
\eqref{U:normalized-residual}, gives
\[
 g_E=|k|^{-1/2}\sum_{y\in k}H_{\omega_E}(y)^{-1}
      =\overline{g_{\omega_E}}=g_{\omega_E}^{-1}.
\]
The last equality follows from $|g_{\omega_E}|=1$ in
\eqref{U:critical-sum-size}.
\end{proof}


\subsection{Stable twists and completion}\label{sec:dyadic-ur-stable}
For $h\ge T$, the norm-approximation depth used in Lemma~\ref{D:UR:choosex} exceeds the break $t$, so we switch to stable twisting.
\begin{lemma}[Stationary coefficients under quadratic norm pullback]\label{D:UR:highcov}
Let $L/F$ be quadratic with different exponent $D$, and let
$\lambda$ be a quasi-character of $F^\times$ of conductor $n\ge2$ with $n>D$.
Let $b_F$ be a stationary coefficient for $\lambda$, so
$\lambda(1+z)=\psi_F(b_Fz)$ for $z\in\pp_F^{\ceil{n/2}}$.

If $L/F$ is ramified, every stationary coefficient $b_L$ for
$\lambda\circ\Nm_{L/F}$ satisfies
$b_L/b_F\in U_L^{n-D}$.
If $L/F$ is unramified, $b_F$ is a stationary coefficient
for the pullback.
\end{lemma}
\begin{proof}
If $L/F$ is ramified, \eqref{U:norm-pullback-conductor} gives
pullback conductor $2n-D$, and \eqref{U:trace-ideal} gives
$n_L(\psi_L)=2n_F(\psi_F)+D$. Thus $v_L(b_L)=v_L(b_F)=-2(n+n_F(\psi_F))$.
For $z\in\pp_L^n$, its trace belongs to $\pp_F^{\ceil{n/2}}$ because
$\floor{(n+D)/2}\ge\ceil{n/2}$, and its norm belongs to $\pp_F^n$.
The norm expansion and the formula for $\lambda$ give
$(\lambda\circ\Nm_{L/F})(1+z)=\psi_L(b_Fz)$ for $z\in\pp_L^n$.
The linearization \eqref{U:stationary-linearization} for $b_L$ also applies there.
By \cite[Lemma~4.1]{Ueda}, their difference belongs to the annihilator
of $\pp_L^n$, so
$b_L-b_F\in\pp_L^{-2n_F(\psi_F)-D-n}$.
Divide by $b_F$ to obtain $b_L/b_F\in U_L^{n-D}$.

If $L/F$ is unramified, apply Lemma~\ref{U:tame-covector}.
\end{proof}

\begin{proposition}[A stable-twist formula]\label{D:UR:stable}
Let $(\chi_1,\chi_2)$ be a primitive compatible pair in a biquadratic
extension $K/F$.
Let $\lambda$ be a quasi-character of $F^\times$ of conductor $n\ge2$.
Assume for every ramified $L_i/F$ that
$n\ge m_{L_i}(\chi_i)+m_F(\omega_{L_i})$, $n\ge2m_F(\omega_{L_i})$,
and for every unramified $L_i/F$ that $n\ge2m_{L_i}(\chi_i)$.
Let $b_F$ be a stationary coefficient for $\lambda$, so that
$\lambda(1+z)=\psi_F(b_Fz)$ for $z\in\pp_F^{\ceil{n/2}}$.

For $\theta_i=\chi_i(\lambda\circ\Nm_{L_i/F})$,
\begin{equation*}
 \mathcal A_{L_i}(\theta_i)
 =\chi_i(b_F^{-1})\omega_{L_i}(b_F^{-1})\Delta_F(\lambda)^2.
\end{equation*}
In particular $\mathcal A_{L_i}(\theta_i)$ is independent of $i$.
\end{proposition}
\begin{proof}
Put $a_i=m_{L_i}(\chi_i)$ and $\delta_i=m_F(\omega_{L_i})$.
For each $i$, choose a stationary coefficient
$b_{L_i}$ for $\lambda\circ \Nm_{L_i/F}$ using
\eqref{U:stationary-linearization}.
For ramified $L_i/F$, one has
$\lfloor m_{L_i}(\lambda\circ \Nm_{L_i/F})/2\rfloor=n-\ceil{\delta_i/2}\ge a_i$;
for unramified $L_i/F$, the bound follows from $n\ge2a_i$.
Thus \eqref{U:stable} gives
\[
 \Delta_{L_i}(\chi_i(\lambda\circ \Nm_{L_i/F}))
 =\chi_i(b_{L_i}^{-1})\Delta_{L_i}(\lambda\circ \Nm_{L_i/F}).
\]
Lemma~\ref{D:UR:highcov} and $n-\delta_i\ge a_i$ show
$\chi_i(b_{L_i}^{-1})=\chi_i(b_F^{-1})$.
For unramified $L_i/F$, take $b_{L_i}=b_F$.

Equation~\eqref{U:FML} for $L_i/F$ gives
\[
 \Delta_{L_i}(\lambda\circ \Nm_{L_i/F})\Delta_F(\omega_{L_i})
 =\Delta_F(\lambda)\Delta_F(\lambda\omega_{L_i}).
\]
Equation~\eqref{U:stable} over $F$ applies since $n\ge2\delta_i$ (and
$\delta_i=0$ when $L_i/F$ is unramified). Its right side is
$\omega_{L_i}(b_F^{-1})\Delta_F(\lambda)^2$.
Multiply the two formulas and use Proposition~\ref{U:determinant}.
\end{proof}

\begin{proof}[Proof of Theorem~\ref{D:UR:URall}]
The other quadratic fields are ramified, since an unramified local Galois
group is cyclic. For a comparison with either such $E$, let
$T=m_F(\omega_E)$. Proposition~\ref{D:UR:actual} gives conductors
$T+h,T+2h$ with $h\ge0$.
For $h=0$, use Proposition~\ref{D:UR:minimal}.
For $1\le h<T$, choose the coefficients of
Proposition~\ref{D:UR:correctedreps}. Proposition~\ref{D:UR:criticalratio}
expresses the quotient as $g_U/(g_Eg_{\omega_E})$.
If $T+h$ is even, Lemma~\ref{U:stationary-factor} gives $g_U=1$;
otherwise Proposition~\ref{D:UR:resU} does. If $T$ is even, both
$g_E$ and $g_{\omega_E}$ are $1$; otherwise Proposition~\ref{D:UR:resE}
gives $g_Eg_{\omega_E}=1$. Thus the quotient is $1$ throughout this range.
For $h\ge T$, the characters
supplied by Proposition~\ref{D:UR:models} and
Proposition~\ref{D:UR:actual} have conductors $T$, the different
exponents of $U/F$ and $E/F$ are $0,T$, and
$m_F(\lambda)=T+h\ge2T$. Thus Proposition~\ref{D:UR:stable} applies.
This exhausts the conductors. Comparing each ramified field with $U$
gives independence of all three by transitivity.
\end{proof}

\begin{remark}[An alternative without the First Main Lemma]\label{D:UR:FML-free-alternative}
The use of the First Main Lemma in the stable range can be avoided.
For any $h>0$, choose $\gamma$ with $\lambda(1-z)=\psi_F(\gamma z)$ on
$\pp_F^{\ceil{(T+h)/2}}$ and $u\in U_F^t$ with
$\omega_E(u)=\omega_E(\gamma/\alpha)$. Replacing the auxiliary scale $\alpha$ by
$\alpha u$ makes $\gamma/(\alpha u)$ an exact norm from $E$, while preserving
\eqref{D:UR:tauchart}, \eqref{D:UR:modelU}, and \eqref{D:UR:modelE}
for the fixed characters. With this normalization, the coefficient and
critical-sum arguments from Lemma~\ref{D:UR:traceparity} through
Proposition~\ref{D:UR:resE} extend to every $h>0$; Proposition~\ref{D:UR:minimal}
treats $h=0$.
\end{remark}

\section{Totally ramified biquadratic extensions in characteristic two}\label{sec:dyadic-equal}
Let $F$ be a nonarchimedean local field of characteristic $2$ with residue field $k$.
Let $K/F$ be totally ramified and Galois with group $C_2^2$.
Write $L_1,L_2,L_3$ for its quadratic intermediate fields and
$\omega_i$ for the nontrivial character of
$F^\times/\Nm_{L_i/F}(L_i^\times)$.
\begin{theorem}[The totally ramified case in characteristic two]\label{D:EQ:main}
Let $(\theta_1,\theta_2,\theta_3)$ be a primitive compatible family of
quasi-characters of $L_i^\times$, with common norm pullback $\Theta$.
For any nontrivial additive character $\psi_F$, put
$\psi_{L_i}=\psi_F\circ\Tr_{L_i/F}$. Then
\begin{equation*}
 \mathcal A_{L_i}(\theta_i)=
 \Delta_{L_i}(\theta_i,\psi_{L_i})\Delta_F(\omega_i,\psi_F)
\end{equation*}
is independent of $i$.
\end{theorem}

Fix a uniformizer $\pi$ of $K$ and take
\begin{equation}\label{D:EQ:compatiblepi}
 \Pi_i=\Nm_{K/L_i}(\pi),\qquad \varpi=\Nm_{K/F}(\pi)=\Nm_{L_i/F}(\Pi_i).
\end{equation}
With the coefficient field $k$, one has $F=k((\varpi))$.
We first prove the theorem for
\begin{equation}\label{D:EQ:canonicalpsi}
 \psi_F(x)=(-1)^{\Tr_{k/\F_2}(\Res_F(x\,d\varpi))}.
\end{equation}
Its largest trivial ideal is $\OO_F$. The passage to an arbitrary
additive character is given at the end.

\subsection{Artin--Schreier norm characters}
\begin{lemma}[The Artin--Schreier norm character]\label{D:EQ:ASsymbol}
Let $z^2+z=f\in F$ define a nontrivial quadratic extension $E/F$.
Then its nontrivial norm character is
\begin{equation}\label{D:EQ:symbol}
 \omega_f(u)=(-1)^{\Tr_{k/\F_2}(\Res_F(f\,du/u))},\qquad u\in F^\times.
\end{equation}
If $f=f_0+\sum_{j>0,\ j\text{ odd}}f_j\varpi^{-j}$ is reduced and
its largest pole is $t>0$, this character has conductor $t+1$.
For every $w\in F$ and $u\in F^\times$, one also has
\begin{equation}\label{D:EQ:AS-invariance}
 \Tr_{k/\F_2}(\Res_F\bigl((w^2+w)\,du/u\bigr))=0.
\end{equation}
\end{lemma}
\begin{proof}
For a Laurent differential in a uniformizer $v$, define
\[
 \Cart\left(\sum_n a_nv^n\,dv\right)
       =\sum_m a_{2m+1}^{1/2}v^m\,dv.
\]
Then
\[
 \Cart(f^2\eta)=f\Cart(\eta),\qquad
 \Cart(dv/v)=dv/v,\qquad
 \Tr_{k/\F_2}(\Res(\Cart(\eta)))=\Tr_{k/\F_2}(\Res(\eta)).
\]
We also need $\Cart(du/u)=du/u$ for every nonzero Laurent series $u$. Write $u=v^nP$ with
$P\in k[[v]]^\times$ and put $q=P'/P$. The term $n\,dv/v$ is fixed by
$\Cart$. In characteristic two, $P''=0$, so $q'=q^2$. If
$q=\sum_m a_mv^m$, comparison of the coefficients of $v^{2m}$ gives
$a_{2m+1}=a_m^2$, and hence
\[
 \Cart(q\,dv)=\sum_m a_{2m+1}^{1/2}v^m\,dv
             =\sum_m a_mv^m\,dv=q\,dv.
\]
Therefore
\[
 \Tr_{k/\F_2}(\Res_F(w^2du/u))
   =\Tr_{k/\F_2}(\Res_F(\Cart(w^2du/u)))
   =\Tr_{k/\F_2}(\Res_F(wdu/u)),
\]
which proves \eqref{D:EQ:AS-invariance}. Thus \eqref{D:EQ:symbol} is
unchanged when $w^2+w$ is added to $f$. It is multiplicative in $u$.

Every Artin--Schreier class has a reduced representative of the stated
form. Remove a negative even power by adding $w^2+w$ and repeat; remove
the positive-depth part using the bijectivity of $w\mapsto w^2+w$ on
the maximal ideal. Constants are taken modulo the same map on $k$.

For $u=\Nm_{E/F}(v)$, Theorem~\ref{B:trace-residue} and
Lemma~\ref{B:dlog-norm} make its exponent equal to the absolute trace of
$\Res_E((z^2+z)\,dv/v)$, which is zero by \eqref{D:EQ:AS-invariance}, applied over $E$.
Thus the formula is trivial on norms; see also~\cite{TateResidue}
for the trace--residue identity.

For a reduced representative with pole $t$, it is trivial on $U_F^{t+1}$,
where $f\,du/u$ has nonnegative valuation. At
$u=1+c\varpi^t$, the exponent of the induced character on
$U_F^t/U_F^{t+1}$ is $\Tr_{k/\F_2}(f_tc)$, which is
nontrivial as $c$ varies. If instead the class is an unramified nonzero
constant, evaluation at a uniformizer is nontrivial.
Consequently the formula is nontrivial for every nonzero
Artin--Schreier class. The quadratic norm quotient has order two
by \eqref{U:norm-character-group}. Thus this is its nontrivial norm character. The conductor assertion follows as well.
\end{proof}

\subsection{Generators and stationary coefficients}
Choose the labeling so that $L_1/F$ has the smallest break, and take
generators $z_i\in L_i$ with reduced Artin--Schreier equations
\[
 z_i^2+z_i=f_i\quad(i=1,2),\qquad z_3=z_1+z_2,\quad f_3=f_1+f_2.
\]
Because $K/F$ is totally ramified, all three nonzero classes have a pole.
Write their breaks as $t_1\le t_2=t_3$, all odd, and put
\begin{equation*}
 T_1=t_1+1=2r_1,\qquad T_2=T_3=t_2+1=2r_2.
\end{equation*}
In the one-break case the two leading pole coefficients are distinct;
otherwise their sum would have smaller conductor and would be the
first field. In the two-break case $t_2-t_1=2(r_2-r_1)$.

Define $\zeta_i$ to be $\varpi^{t_2}$ times the negative-power part of $f_i$; then
\begin{equation*}
 \zeta_i=h_i^2\in\OO_F,\qquad h_3=h_1+h_2.
\end{equation*}
The square roots exist in $F$: every exponent $t_2-j$ is even, and
$k$ is perfect. Thus $v_F(h_1)=r_2-r_1$ and $h_2,h_3$ are units.
Set
\begin{equation*}
\begin{gathered}
 Y=h_1z_2+h_2z_1,\qquad
 B=\zeta_1f_2+\zeta_2f_1\in\OO_F,\\
 \kappa_1=h_2h_3,\qquad \kappa_2=h_1h_3,\qquad \kappa_3=h_1h_2.
\end{gathered}
\end{equation*}
Let $g_i$ be the nontrivial element of $\Gal(K/L_i)$, so
$g_1g_2=g_3$. Thus $g_1$ changes $z_2$ by $1$ and fixes $z_1$,
while $g_2$ changes $z_1$ by $1$ and fixes $z_2$.
Put $a_i=\Nm_{K/L_i}(Y)$.
\begin{lemma}[Norms and traces of the origin]\label{D:EQ:origin-identities}
For $i=1,2,3$, one has
\begin{align}
 g_i(Y)-Y&=h_i,&\Tr_{K/L_i}(Y)&=h_i,&
 a_i&=\kappa_i z_i+B.                                      \label{D:EQ:exact-origin}
\end{align}
Moreover,
\begin{equation}\label{D:EQ:originval}
 \begin{aligned}
 v_K(Y)&=v_{L_i}(a_i)=-t_1,\qquad v_F(\zeta_i)=T_2-T_i,\\
 \Nm_{K/F}(Y)&=\Nm_{L_i/F}(a_i),\qquad v_F(\Nm_{K/F}(Y))=-t_1.
 \end{aligned}
\end{equation}
The breaks of $K/L_i$ are $u_1=2t_2-t_1$ and $u_2=u_3=t_1$,
and their different exponents are
\begin{equation}\label{D:EQ:upperdifferent}
 D'_1=2T_2-T_1,\qquad D'_2=D'_3=T_1.
\end{equation}
\end{lemma}
\begin{proof}
The quadratic equations give $Y^2=\zeta_1z_2+\zeta_2z_1+B$,
which proves \eqref{D:EQ:exact-origin}.
For $i=2,3$, $v_{L_i}(\kappa_i)=2(r_2-r_1)=t_2-t_1$;
for $i=1$ its valuation is zero. Since $B\in\OO_F$, these observations
give $v_{L_i}(a_i)=-t_1$. Total ramification and transitivity of norms
then give \eqref{D:EQ:originval}; the asserted valuations of $\zeta_i$
follow from $\zeta_i=h_i^2$.

Apply Lemma~\ref{C:leading-ramification} to $Y$ in
\eqref{D:EQ:exact-origin}. The valuation of $h_1$ in $K$ is
$4(r_2-r_1)$, and those of $h_2,h_3$ are zero. This gives the stated
breaks and hence \eqref{D:EQ:upperdifferent}.
\end{proof}

Put
\begin{equation*}
 \alpha=\varpi^{-T_2},\qquad
 \Psi_F(x)=\psi_F(\alpha x),\qquad
 \Psi_{L_i}(x)=\Psi_F(\Tr_{L_i/F}(x)).
\end{equation*}
The largest trivial ideal for $\Psi_F$ is $\pp_F^{T_2}$.
For $\Psi_{L_i}$ the exponents are
$J_1=2T_2-T_1, J_2=J_3=T_2$.
Put $r_3=r_2$.

\begin{lemma}[Simultaneous coefficients for the norm characters]\label{D:EQ:lower}
For $i=1,2,3$, one has
\begin{equation}\label{D:EQ:lowerchart}
 \omega_i(1-z)=\Psi_F(\zeta_i z)\qquad(z\in\pp_F^{\ceil{T_i/2}}).
\end{equation}
\end{lemma}
\begin{proof}
Since $T_2$ is even, $\Psi_F(x^2)=1$ for every $x\in F$.
For $z\in\pp_F^{r_i}$, one has
\[
 \frac{f_i\,dz}{1-z}-f_i\,dz=\frac{f_i z\,dz}{1-z}.
\]
The coefficient of $d\varpi$ on the right has valuation at least
$-t_i+r_i+(r_i-1)=0$, so its residue vanishes.
Integration by parts gives $\Res_F(f_i\,dz)=\Res_F(z\,df_i)$
in characteristic two. Since $df_i=\alpha\zeta_i\,d\varpi$,
Lemma~\ref{D:EQ:ASsymbol} proves \eqref{D:EQ:lowerchart}.
The valuations $v_F(\zeta_i)=T_2-T_i$ from
Lemma~\ref{D:EQ:origin-identities} identify $\zeta_i$
as stationary coefficients for $\omega_i$ relative to $\Psi_F$
in Lemma~\ref{U:stationary-factor}.
\end{proof}

\subsection{Compatible characters and minimal-conductor comparison}
Define the minimal conductors and stationary depths
\begin{equation*}
 m_1=2T_2-1,\qquad m_2=m_3=T_1+T_2-1,\qquad
 s_i=\ceil{m_i/2}.
\end{equation*}
Thus $s_1=T_2$, $s_2=s_3=r_1+r_2$, and $m_i=J_i+t_1=2s_i-1$.
The formulas
\begin{equation*}
 R_i(1-z)=\Psi_{L_i}(a_i z),\qquad z\in\pp_{L_i}^{s_i}
\end{equation*}
define characters of $U_{L_i}^{s_i}$ with exact conductor $m_i$.
Multiplicativity follows since $2s_i-t_1\ge J_i$.

\begin{lemma}[The conjugate quotient of the prescribed character]\label{D:EQ:commmodel}
Let $\sigma$ be the nontrivial automorphism of $L_1/F$ and
$\nu_1=\omega_2\circ \Nm_{L_1/F}$, the nontrivial character in $S(K/L_1)$.
For every $v\in L_1^\times$ with $\sigma(v)/v\in U_{L_1}^{s_1}$,
\begin{equation*}
 R_1(\sigma(v)/v)=\nu_1(v).
\end{equation*}
\end{lemma}
\begin{proof}
The description of $\nu_1$ follows from Lemmas~\ref{D:crossed-norm}
and \ref{D:quadratic-norms}, whose characteristic-two proof uses only
\eqref{D:EQ:symbol} and Theorem~\ref{B:trace-residue}.
Multiplication of $v$ by an element of $F^\times$
does not affect either side. Reduce its valuation to $0$ or $1$.
If $v_{L_1}(v)=1$, Lemma~\ref{C:leading-ramification} gives
$v_{L_1}(\sigma(v)/v-1)=t_1<s_1$. Thus only units occur.

Write $v=a+bz_1$, with $a,b\in F$. The two summands have valuations of
opposite parity, since $t_1$ is odd. Consequently $a$ is a unit and
$v_F(b)\ge r_1$. The conjugate quotient is $1+b/v$; the assumption
$v_{L_1}(\sigma(v)/v-1)\ge s_1=2r_2$ forces $v_F(b)\ge r_2$. Put $z=b/a$, so $v_F(z)\ge r_2$.
The case $b=0$ is immediate. Quadratic trace calculation gives
\begin{equation*}
 \Tr_{L_1/F}\left(a_1\frac b v\right)
   =\frac{\kappa_1ab+B b^2}{\Nm_{L_1/F}(v)}
   =\frac{\kappa_1z+B z^2}{1+z+z^2f_1}.
\end{equation*}

To evaluate $\nu_1(v)$, put
$R=z^2f_1$, $U=1+R$ and $V=U+z$. The elements $U$ and $V$ are units, since
$v_F(R)\ge2r_2-t_1\ge1$. Factor $V=U(1+z/U)$.
Equation~\eqref{D:EQ:symbol} evaluates $\omega_2(U)$, and
\eqref{D:EQ:lowerchart} applies to $1+z/U$. They give
\[
 \omega_2(V)=
 \Psi_F\left(\frac{\zeta_1f_2z^2+\zeta_2z}{U}\right)
 =\Psi_F\left(\frac{\zeta_2R+\zeta_2z}{U}\right).
\]
The last equality follows from $Bz^2/U\in\pp_F^{T_2}$,
since $\zeta_1f_2+\zeta_2f_1=B\in\OO_F$ and $v_F(z)\ge r_2$.

Write $R_0=\varpi W^2$, where $W=\varpi^{-r_2}h_1z\in\OO_F$, and let $f_{1,0}$ be the constant term of $f_1$.
Then $R-R_0=f_{1,0}z^2\in\pp_F^{T_2}$, and
$U^{-1}-(1+R_0)^{-1}\in\pp_F^{T_2}$. Put $f=\varpi^{-r_2}h_2W$. Since
$d(\varpi(1+R_0))=d\varpi$, equation~\eqref{D:EQ:AS-invariance},
with $w=\varpi f$ and $u=\varpi(1+R_0)$, gives
\[
 \Tr_{k/\F_2}(\Res_F\left(
 \left(\frac{\varpi f^2}{1+R_0}+\frac f{1+R_0}\right)d\varpi
 \right))=0.
\]
Thus the two residue traces are equal. In the present variables this is
\[
 \Psi_F\left(\frac{\zeta_2R_0}{1+R_0}\right)
   =\Psi_F\left(\frac{h_1h_2z}{1+R_0}\right).
\]

It follows that $\omega_2(V)=\Psi_F(\kappa_1z/U)$ because
$\zeta_2+h_1h_2=h_2h_3=\kappa_1$.
Then
\[
 \frac{\kappa_1z}{U}-\frac{\kappa_1z}{V}
   =\frac{\kappa_1z^2}{UV}\in\pp_F^{T_2}.
\]
Also $Bz^2/V\in\pp_F^{T_2}$ in the trace formula for $a_1b/v$ above.
Since $\Nm_{L_1/F}(v)=a^2V$ and $\omega_2(a^2)=1$,
that trace formula proves $R_1(\sigma(v)/v)=\nu_1(v)$.
\end{proof}

\begin{lemma}[Compatibility of the prescribed characters on principal units]\label{D:EQ:Rcompat}
The norm pullbacks of the characters $R_i$ are defined and agree on
$U_K^{T_2}$.
\end{lemma}
\begin{proof}
Let $z\in\pp_K^{T_2}$. The norm term has $L_i$-valuation at least $T_2$.
By \eqref{U:trace-ideal} and \eqref{D:EQ:upperdifferent},
$\Tr_{K/L_1}(z)$ and $\Tr_{K/L_2}(z),\Tr_{K/L_3}(z)$ have respective depths at least
$3r_2-r_1$ and $r_1+r_2$; these are at least $s_1$ and $s_2$.

Apply Lemma~\ref{U:e2} to $Y(1-z)$ and $Y$.
Put $w_i=\Tr_{K/L_i}(Yz)/h_i$.
Equation~\eqref{U:trace-ideal}, applied to $K/L_i$, and
\eqref{D:EQ:originval} give
\[
 v_{L_1}(\Tr_{K/L_1}(Yz))\ge r_1+3(r_2-r_1),\qquad
 v_{L_2}(\Tr_{K/L_2}(Yz)),v_{L_3}(\Tr_{K/L_3}(Yz))\ge r_2.
\]
After division by $h_1$, of $L_1$-valuation $2(r_2-r_1)$, the first
bound is also $r_2$. Thus $v_{L_i}(w_i)\ge r_i$ for every $i$.
Since $\Tr_{K/L_i}(Y(1-z))=h_i(1-w_i)$, its $L_i/F$-norm divided by
$\zeta_i$ is $\Nm_{L_i/F}(1-w_i)\in U_F^{r_i}$.
The character $\omega_i$ is trivial on norms, so \eqref{D:EQ:lowerchart} gives
\[
 1=\omega_i(\Nm_{L_i/F}(1-w_i))
  =\Psi_F\!\left(\zeta_i-\Nm_{L_i/F}(\Tr_{K/L_i}(Y(1-z)))\right).
\]
The remaining trace term is
$\Tr_{L_i/F}(a_i(\Tr_{K/L_i}(z)+\Nm_{K/L_i}(z)))$, the argument of $\Psi_F$ in
$R_i(\Nm_{K/L_i}(1-z))$. Therefore this value is
$\Psi_F(e_2(Y(1-z))-e_2(Y))$, independent of $i$.
\end{proof}

\begin{proposition}[Extension to compatible characters of minimal conductor]\label{D:EQ:core-exists}
There are characters $\chi_i$ on $L_i^\times$, with exact conductors
$m_i$, primitive common norm pullback, and
\begin{equation}\label{D:EQ:corechart}
 \chi_i(1-z)=\Psi_{L_i}(a_i z)
                    \qquad(z\in\pp_{L_i}^{s_i}).
\end{equation}
\end{proposition}
\begin{proof}

Let $\sigma,\nu_1$ be as in Lemma~\ref{D:EQ:commmodel}, and put
$I=\{\sigma(v)/v:v\in L_1^\times\}$. The rule
$\sigma(v)/v\mapsto\nu_1(v)$ is well defined because
$\nu_1|_{F^\times}=1$, and agrees with $R_1$ on
$I\cap U_{L_1}^{s_1}$ by Lemma~\ref{D:EQ:commmodel}. Thus
\[
 \chi_1\bigl((\sigma(v)/v)u\bigr)=\nu_1(v)R_1(u),
 \qquad v\in L_1^\times,\quad u\in U_{L_1}^{s_1},
\]
defines a character on $I U_{L_1}^{s_1}$, trivial on $U_{L_1}^{m_1}$.
Extend it to $L_1^\times$ by Lemma~\ref{C:character-extension}.
Its conductor is $m_1$, and $\chi_1^\sigma/\chi_1=\nu_1$.

The character $\chi_1\circ \Nm_{K/L_1}$ is invariant under $\Gal(K/F)$: its other conjugate
quotient is $\nu_1\circ \Nm_{K/L_1}=\omega_2\circ \Nm_{K/F}=1$.
By Lemma~\ref{C:hilbert90}, it is trivial on $\ker \Nm_{K/L_i}$ for
$i=2,3$. Prescribe
\[
 \chi_i(\Nm_{K/L_i}(x))=\chi_1(\Nm_{K/L_1}(x)),\qquad x\in K^\times,
\]
and extend these characters from $\Nm_{K/L_i}(K^\times)$ to
$L_i^\times$ by Lemma~\ref{C:character-extension}.

For $i=2,3$, the break of $K/L_i$ is $t_1$;
\eqref{U:unit-norm-surjective} gives
\[
 \Nm_{K/L_i}(U_K^{T_2+1})=U_{L_i}^{s_i}\qquad(i=2,3),
\]
since $2s_i-t_1=T_2+1$. Lemma~\ref{D:EQ:Rcompat} then proves
\eqref{D:EQ:corechart} on $U_{L_i}^{s_i}$.
The restrictions $R_i$ are trivial on $U_{L_i}^{m_i}$ and nontrivial
on $U_{L_i}^{m_i-1}$, so $m_{L_i}(\chi_i)=m_i$.

Finally, descent of the common pullback from $F$ would make $\chi_1$
a product of a norm pullback from $F$ and a character in $S(K/L_1)$.
Both are $\sigma$-invariant, contradicting
$\chi_1^\sigma/\chi_1=\nu_1\ne1$.
\end{proof}

\begin{proposition}[Common twists and their conductors]\label{D:EQ:realize}
For every primitive compatible family $(\theta_i)$ there are characters $(\chi_i)$ satisfying
\eqref{D:EQ:corechart}, and a quasi-character $\lambda$ of $F^\times$, such that
\begin{equation}\label{D:EQ:actual-family}
 \theta_i=\chi_i(\lambda\circ \Nm_{L_i/F}).
\end{equation}
Writing $n=m_F(\lambda)$, one has
\[
 m_{L_i}(\theta_i)=
 \begin{cases}
 m_i,&n\le T_2+r_1-1,\\
 2n-T_i,&n\ge T_2+r_1.
 \end{cases}
\]
The conductors in the second range are even.
\end{proposition}
\begin{proof}
Choose characters $\chi_i^0$ from
Proposition~\ref{D:EQ:core-exists}. By Proposition~\ref{U:conjugacy},
$\theta_1/\chi_1^0$ is invariant under $\Gal(L_1/F)$.
Lemmas~\ref{C:hilbert90} and~\ref{C:character-extension}
give $\theta_1/\chi_1^0=\lambda\circ \Nm_{L_1/F}$. Set $\chi_1=\chi_1^0$.
For $i=2,3$, put
\[
 \mu_i=\frac{\theta_i}{\chi_i^0(\lambda\circ \Nm_{L_i/F})}.
\]
Compatibility gives $\mu_i\in S(K/L_i)$. Define
$\chi_i=\chi_i^0\mu_i$. Then
$\theta_i=\chi_i(\lambda\circ \Nm_{L_i/F})$, the common norm pullback is
unchanged, and \eqref{D:EQ:corechart} is unchanged because each
$\mu_i$ is either trivial or has conductor $T_1\le s_i$.

Let $\sigma_i$ be the nontrivial automorphism of $L_i/F$.
Primitivity forces $m_{L_i}(\theta_i)\ge t_i+u_i+1=m_i$:
by Lemma~\ref{C:leading-ramification}, the first ramification term in
$\sigma_i(v)/v$ on $U_{L_i}^{u_i}/U_{L_i}^{u_i+1}$ has exact depth
$t_i+u_i$. Its coefficient is nonzero since $u_i$ is odd; a character of
smaller conductor would make $\theta_i^{\sigma_i}/\theta_i$ trivial on
$U_{L_i}^{u_i}$, contrary to Proposition~\ref{U:conjugacy}.

By \eqref{U:norm-pullback-conductor}, the pullback $\lambda\circ \Nm_{L_i/F}$ has
conductor $2n-T_i$ for $n>T_i$, and at most $T_i$ for $n\le T_i$.
The number $2n-T_i$ is even, whereas $m_i$ is odd, so cancellation
at equal conductors is impossible. The same threshold
$n=T_2+r_1$ is obtained for all three fields. This proves the assertions.
\end{proof}

We compare $\mathcal A_{L_1}(\theta_1)$ and $\mathcal A_{L_2}(\theta_2)$.
Interchanging $L_2,L_3$ gives the comparison with
$\mathcal A_{L_3}(\theta_3)$.
For $X\in L_3$, put
\begin{equation}\label{D:EQ:adjust}
 A=\Nm_{L_3/F}(X),\quad p_X=\Tr_{L_3/F}(X),\quad
 Q_i=p_X Y+h_iX\in L_i\quad(i=1,2).
\end{equation}
\begin{lemma}[Correction identities and estimates]\label{D:EQ:bounds}
For $i=1,2$, one has
\begin{align}
 \Nm_{K/L_i}(Y+X)&=a_i+A+Q_i,&\Tr_{K/L_i}(Y+X)&=h_i+p_X,\label{D:EQ:commonC}\\
 \Nm_{L_i/F}(Q_i)&=\zeta_iA+\mathcal E,&
 \mathcal E&=p_X(p_Xa_3+\kappa_3X)\in F.\notag
\end{align}
If $X\ne0$ and $v_{L_3}(X)=-h$, then
\begin{align}
 v_F(p_X)&\ge\floor{(T_2-h)/2},&
 v_F(\mathcal E)&\ge T_2-r_1-h,\label{D:EQ:E-bound}\\
 v_{L_1}(Q_1)&\ge T_2-T_1-h,&
 v_{L_2}(Q_2)&\ge-h.\label{D:EQ:Q-bound}
\end{align}
\end{lemma}
\begin{proof}
The trace and norm formulas for $Y+X$ follow by expansion.
Write $X=p+qz_3$. Then $p_X=q$,
$Q_1=h_1p+qh_3z_1$ and $Q_2=h_2p+qh_3z_2$.
Taking the two quadratic norms gives the common error
$\mathcal E=q^2B+q\kappa_3p=p_X(p_Xa_3+\kappa_3X)$.

The trace-ideal formula gives $v_F(p_X)\ge r_2-\ceil{h/2}
=\floor{(T_2-h)/2}$.
The invariant element $p_Xa_3+\kappa_3X$ has $L_3$-valuation at least
$2(r_2-r_1)-h$, hence $F$-valuation at least
$(r_2-r_1)-\floor{h/2}$. Combining this with the trace bound proves
the bound for $\mathcal E$.
Finally $Q_i=\Tr_{K/L_i}(Yg_i(X))$. Apply \eqref{U:trace-ideal}
to $Yg_i(X)$, of $K$-valuation $-t_1-2h$, using the different
exponents in \eqref{D:EQ:upperdifferent}.
\end{proof}

\begin{proposition}[Stationary coefficients at minimal conductor]\label{D:EQ:min-origin}
Let the family be realized as in \eqref{D:EQ:actual-family}, with
$n\le T_2+r_1-1$. There exists $C\in Y+\pp_{L_3}^{1-r_1}$ with $v_K(C)=-t_1$ such that,
for $i=1,2$,
\begin{align}
 \theta_i(1-z)&=\Psi_{L_i}(\Nm_{K/L_i}(C)\,z)
                            &&(z\in\pp_{L_i}^{s_i}),       \notag\\
 \omega_i(1-z)&=\Psi_F(\Nm_{L_i/F}(\Tr_{K/L_i}(C))\,z)
                            &&(z\in\pp_F^{\ceil{T_i/2}}).\label{D:EQ:min-lower}
\end{align}
For $i=1,2$, the traces $\Tr_{K/L_i}(C)$ are nonzero and satisfy
$v_{L_i}(\Tr_{K/L_i}(C))=v_{L_i}(h_i)$.
If $t_1=t_2$, the same statements hold for all three indices.
\end{proposition}
\begin{proof}
For $z\in\pp_{L_i}^{s_i}$, both $\Tr_{L_i/F}(z)$ and $\Nm_{L_i/F}(z)$ lie in
$\pp_F^{r_1+r_2}$. For $i=1$ the trace bound is exactly $r_1+r_2$;
for $i=2,3$ its bound is at least $r_1+r_2$.
If $n\le r_1+r_2$, take $X=0$ and $C=Y$ below.

Otherwise put $h=n-T_2$, so
$1-(r_2-r_1)\le h\le r_1-1$.
Since $\lceil n/2\rceil\le r_1+r_2<n$,
\eqref{U:stationary-linearization} supplies $\gamma\in F^\times$ with
\[
 \lambda(1-z)=\psi_F(\gamma z)\quad(z\in\pp_F^{r_1+r_2}),\qquad
 v_F(\gamma/\alpha)=-h,
\]
where $\gamma/\alpha$ is determined modulo
$\pp_F^{r_2-r_1}$.
Put $s=r_2-r_1+h$. By the bounds on $h$ above,
$1\le s\le r_2-1\le t_2$.
Applying \eqref{U:norm-approximation} to $L_3/F$ and $c=\gamma/\alpha$ gives
$X\in L_3^\times$ such that
\[
 v_{L_3}(X)=-h,\qquad
 \gamma/\alpha-\Nm_{L_3/F}(X)\in\pp_F^{-h+s}=\pp_F^{r_2-r_1}.
\]
The character formula for $\lambda$ on $\pp_F^{r_1+r_2}$ is unchanged
when $\gamma$ is replaced by $\alpha A$.

Put $C=Y+X$. In either case $C-Y=X\in\pp_{L_3}^{1-r_1}$.
The case $X=0$ follows from \eqref{D:EQ:corechart} and \eqref{D:EQ:lowerchart}. Otherwise
$-2h>-t_1$ since $h\le r_1-1$, so $C$ has the valuation of $Y$.
For $z\in\pp_{L_i}^{s_i}$ and $i=1,2$, \eqref{D:EQ:Q-bound} gives
\[
 v_{L_1}((Q_1/\zeta_1)z)\ge T_2-2(r_2-r_1)-h=2r_1-h\ge r_1,
 \quad
 v_{L_2}((Q_2/\zeta_2)z)\ge r_1+r_2-h\ge r_2.
\]
Under the preceding bounds, equation~\eqref{D:UR:phase}, applied to
$L_i/F$ and $\omega_i$ with additive scale $\alpha\zeta_i$ using
\eqref{D:EQ:lowerchart}, and the formula for $\Nm_{L_i/F}(Q_i)$ above give
\begin{equation}\label{D:EQ:pullback-adjust}
 \Psi_{L_i}(Q_iz)
 =\Psi_F\bigl((A+\mathcal E/\zeta_i)\Nm_{L_i/F}(z)\bigr).
\end{equation}
For these $z$, \eqref{D:EQ:E-bound} gives
\[
\begin{aligned}
 v_F\!\left((\mathcal E/\zeta_1)\Nm_{L_1/F}(z)\right)&\ge T_2+r_1-h\ge T_2,\\
 v_F\!\left((\mathcal E/\zeta_2)\Nm_{L_2/F}(z)\right)&\ge3r_2-h\ge T_2.
\end{aligned}
\]
Since $\lambda(1-x)=\Psi_F(Ax)$ on $\pp_F^{r_1+r_2}$ and both
$\Tr_{L_i/F}(z)$ and $\Nm_{L_i/F}(z)$ lie in this ideal, the norm expansion
and \eqref{D:EQ:pullback-adjust} give
$(\lambda\circ \Nm_{L_i/F})(1-z)=\Psi_{L_i}((A+Q_i)z)$.
Multiplication by \eqref{D:EQ:corechart} gives the asserted stationary formula for $\theta_i$,
since $a_i+A+Q_i=\Nm_{K/L_i}(C)$ by \eqref{D:EQ:commonC}.

Moreover $\Nm_{L_i/F}(\Tr_{K/L_i}(C))=(h_i+p_X)^2=\zeta_i+p_X^2$.
The bound $2v_F(p_X)\ge2r_2-2\ceil{h/2}\ge2r_2-r_1=T_2-r_1$
proves \eqref{D:EQ:min-lower} for $z\in\pp_F^{r_i}$.
It also shows $v_F(p_X)>(r_2-r_1)$, so the traces $h_i+p_X$ are nonzero with
the original valuations.

When $t_1=t_2$, one has $h\le r_1-1$.
For $i=3$, equation~\eqref{D:UR:phase}, applied to $L_3/F$ and $\omega_3$
with additive scale $\alpha\zeta_3$ using \eqref{D:EQ:lowerchart}, gives the
coefficient $A+h_3X$, because
$\Nm_{L_3/F}(X/h_3)=A/\zeta_3$. The corresponding argument of
$\Psi_F$ has valuation at least $2r_1-h\ge r_1$.
But
\[
 \Nm_{K/L_3}(C)=a_3+X^2+h_3X
       =a_3+A+h_3X+p_X X,
\]
and $v_{L_3}(p_X X)\ge2r_1-2h-(h\bmod2)\ge0$.
Hence $\Psi_{L_3}(p_X Xz)=1$ for $z\in\pp_{L_3}^{2r_1}$.
Finally $\Tr_{K/L_3}(C)=h_3$, so $\Nm_{L_3/F}(\Tr_{K/L_3}(C))=\zeta_3$.
This proves the assertion for all three fields in the one-break case.
\end{proof}

\begin{proposition}[Minimal-conductor comparison]\label{D:EQ:min-two}\label{D:EQ:min-one}
For every primitive compatible family with conductors $m_i$, one has
\[
 \mathcal A_{L_1}(\theta_1)=\mathcal A_{L_2}(\theta_2)
                         =\mathcal A_{L_3}(\theta_3).
\]
\end{proposition}
\begin{proof}
By Proposition~\ref{D:EQ:realize}, realize the given family as in
\eqref{D:EQ:actual-family}. Since its conductors are $m_i$, that
proposition gives $n=m_F(\lambda)\le T_2+r_1-1$.
Choose $C$ as in Proposition~\ref{D:EQ:min-origin}. Use the following
construction for $i=1,2$, and also for $i=3$ when $t_1=t_2$.
The common restriction $D_\theta=\theta_i|_{F^\times}\omega_i$
is the one in Proposition~\ref{U:conjugacy}.
Use the stationary coefficients $\Nm_{K/L_i}(C)$ for $\theta_i$ and $\Nm_{L_i/F}(\Tr_{K/L_i}(C))$
for $\omega_i$ from Proposition~\ref{D:EQ:min-origin}. For $\theta_i$, take
critical element $\Pi_i^{\lfloor m_i/2\rfloor}$ in
Lemma~\ref{U:stationary-factor}, and denote its critical function and
normalized sum by $H_{\theta_i}$ and $g_{\theta_i}$.
Since $T_i$ is even, the normalized sum for $\omega_i$ is $1$.
Lemma~\ref{U:common-origin}, with the coefficients just specified, gives
\begin{equation}\label{D:EQ:minfull}
 \mathcal A_{L_i}(\theta_i)
  =D_\theta(\alpha^{-1})\Theta(C)^{-1}
                       \Psi_F(e_2(C))g_{\theta_i}.
\end{equation}

Keep $C$ fixed and introduce $C_z=C(1+\pi^{t_2}z)$
for $z\in\OO_K$.
The relative norms satisfy
\[
 v_{L_i}\!\left(\frac{\Nm_{K/L_i}(C_z)}{\Nm_{K/L_i}(C)}-1\right)
 \ge\floor{m_i/2}.
\]
The changes $\Tr_{K/L_i}(C\pi^{t_2}z)$ have valuations at least
$r_1+3(r_2-r_1)$ in $L_1$ and $r_2$ in $L_2$, respectively.
By Proposition~\ref{D:EQ:min-origin}, $\Tr_{K/L_i}(C)$ has the valuation
of $h_i$ for $i=1,2$. After division by $\Tr_{K/L_i}(C)$, both relative changes
have valuation at least $r_2\ge r_i$. Evaluating the norm quotient by \eqref{D:EQ:min-lower} gives
\[
 \Psi_F\!\left(\Nm_{L_i/F}(\Tr_{K/L_i}(C_z))
             -\Nm_{L_i/F}(\Tr_{K/L_i}(C))\right)=1.
\]
These relative trace changes also show that $\Tr_{K/L_i}(C_z)\ne0$.
When $t_1=t_2$, the same trace-ideal calculation applies to $i=3$,
since all $r_i$ and all relative different exponents agree.
Apply Lemma~\ref{U:common-function} to $C,C_z$. The relative norms
lie in the critical domains of $\theta_i$, while stationarity makes
the norm-character factors $1$ on the depth-$r_i$ domains just checked.
Equation~\eqref{U:normalized-residual}, in the chosen critical coordinates, gives
\begin{equation*}
 H_{\theta_i}(\rho_i(\bar z))=Q(\bar z),
\end{equation*}
Here $\rho_i:k\to k$ is the map induced by
$\Nm_{K/L_i}:U_K^{t_2}\to U_{L_i}^{\floor{m_i/2}}$ on the corresponding successive
unit quotients, using the uniformizers in \eqref{D:EQ:compatiblepi}.
The function $Q:k\to\C^\times$ is independent of $i$.

Suppose first that $t_1<t_2$.
For $y\in\pp_K^{t_2}$, equation~\eqref{U:trace-ideal} for $K/L_1$,
whose break is $2t_2-t_1$, gives
\[
 v_{L_1}(\Tr_{K/L_1}(y))\ge t_2+\frac{t_2-t_1}{2}\ge t_2+1.
\]
Thus the trace term is zero modulo $\pp_{L_1}^{t_2+1}$, whereas
$\Nm_{K/L_1}(y)\in\pp_{L_1}^{t_2}$. The compatible uniformizers in
\eqref{D:EQ:compatiblepi} therefore give $\rho_1(z)=z^2$.

For $K/L_2$, the trace induces an isomorphism of the one-dimensional
residue lines at source depth $t_2$ and target depth $(t_1+t_2)/2$:
\[
 \floor{(t_2+T_1)/2}=(t_1+t_2)/2,\qquad
 \floor{(t_2+1+T_1)/2}=(t_1+t_2)/2+1.
\]
For $y\in\pp_K^{t_2}$,
$v_{L_2}(\Nm_{K/L_2}(y))=v_K(y)\ge t_2\ge(t_1+t_2)/2+1$.
Thus the norm term is zero in the target quotient, and $\rho_2$ is the
nonzero $k$-linear map induced by the trace. Both $\rho_i$ are bijections of $k$.

Summing the common-pullback identity above gives $g_{\theta_1}=g_{\theta_2}$, and
\eqref{D:EQ:minfull} proves equality. The same comparison with $L_3$
proves equality for all three intermediate fields when $t_1<t_2$.

Suppose now that $t_1=t_2=t$. For the nonidentity
$g_i\in\Gal(K/L_i)$, put
\[
 \lambda_i=
 \overline{\frac{g_i(\pi)-\pi}{\pi^{t+1}}}\in k^\times.
\]
Lemma~\ref{D:break-ledger} gives $G_t=G$ and $G_{t+1}=1$.
The injective ramification map of Lemma~\ref{D:inertia} therefore
makes $\lambda_1,\lambda_2,\lambda_3$ pairwise distinct.
The critical norm formula \cite[Proposition~3.4]{Ueda}, with the
compatible uniformizers in \eqref{D:EQ:compatiblepi}, gives
\[
 \rho_i(z)=z^2+\lambda_i z,\qquad
 \ker\rho_i=\{0,\lambda_i\}.
\]
The kernels are pairwise distinct, each has order two, and each image
has index two in $(k,+)$.

Equation~\eqref{U:critical-polar} gives the nondegenerate polar
pairing of each critical function $H_{\theta_i}$. The identity $H_{\theta_i}\circ\rho_i=Q$ established above gives their common pullback.
Proposition~\ref{U:missing-coset}, applied with source and targets $(k,+)$,
gives equality of the normalized sums $g_{\theta_i}$.
Equation~\eqref{D:EQ:minfull} proves
$\mathcal A_{L_1}(\theta_1)=\mathcal A_{L_2}(\theta_2)=\mathcal A_{L_3}(\theta_3)$.
\end{proof}

\subsection{Even-conductor comparison before stable twisting}
\begin{proposition}[Comparison before stable twisting]\label{D:EQ:intermediate}
Let $(\theta_i)$ be a primitive compatible family written as in
\eqref{D:EQ:actual-family}. If
$T_2+r_1\le m_F(\lambda)\le2T_2+r_1-2$, then
$\mathcal A_{L_1}(\theta_1)=\mathcal A_{L_2}(\theta_2)$.
\end{proposition}
\begin{proof}
Put $n=m_F(\lambda)$ and $h=n-T_2$, so
$r_1\le h\le r_1+2r_2-2$.
By Proposition~\ref{D:EQ:realize}, the conductors and stationary depths are
\begin{equation*}
 m_{L_i}(\theta_i)=2n-T_i,\qquad b_i=n-r_i;
 \quad b_1=T_2+h-r_1,\quad b_2=r_2+h.
\end{equation*}
Put
$q_F=r_2+\floor{(h+r_1)/2}$.
The trace and norm of each variable $z\in\pp_{L_i}^{b_i}$, $i=1,2$,
belong to $\pp_F^{q_F}$. Moreover
$q_F\ge\ceil{n/2}$ and $b_i\ge s_i$.
By \eqref{U:stationary-linearization}, choose $\gamma\in F^\times$ such that
$\lambda(1-z)=\psi_F(\gamma z)$ for $z\in\pp_F^{q_F}$.
Then $v_F(\gamma/\alpha)=-h$, and $\gamma/\alpha$ is determined modulo $\pp_F^{T_2-q_F}$.
Put $s=h+T_2-q_F$. The stated conductor range gives
\[
 1\le r_2\le s=r_2+\ceil{(h-r_1)/2}\le2r_2-1=t_2.
\]
Equation~\eqref{U:norm-approximation}, applied to $L_3/F$ and $c=\gamma/\alpha$,
gives $X\in L_3^\times$ such that
\[
 v_{L_3}(X)=-h,\qquad
 \gamma/\alpha-\Nm_{L_3/F}(X)\in\pp_F^{-h+s}=\pp_F^{T_2-q_F}.
\]
Thus the character formula holds with $\gamma=\alpha A$;
use \eqref{D:EQ:adjust} with this $X$, and put $C=Y+X$.

The pullback calculation \eqref{D:EQ:pullback-adjust} in the proof of
Proposition~\ref{D:EQ:min-origin} applies because $(Q_i/\zeta_i)z$ has depth
at least $r_1$ and $r_2$, respectively:
$-h-2(r_2-r_1)+b_1=r_1, -h+b_2=r_2$.
The terms discarded from the argument of $\Psi_F$ satisfy
\[
\begin{aligned}
 v_F\!\left((\mathcal E/\zeta_1)\Nm_{L_1/F}(z)\right)
 &\ge(r_2+(r_2-r_1)-h)-2(r_2-r_1)+b_1=T_2,\\
 v_F\!\left((\mathcal E/\zeta_2)\Nm_{L_2/F}(z)\right)
 &\ge(r_2+(r_2-r_1)-h)+b_2=T_2+(r_2-r_1).
\end{aligned}
\]
Since $\lambda(1-x)=\Psi_F(Ax)$ on $\pp_F^{q_F}$ and both
$\Tr_{L_i/F}(z)$ and $\Nm_{L_i/F}(z)$ lie in this ideal, the norm-expansion
step following \eqref{D:EQ:pullback-adjust} in the proof of
Proposition~\ref{D:EQ:min-origin} gives
$(\lambda\circ \Nm_{L_i/F})(1-z)=\Psi_{L_i}((A+Q_i)z)$ for
$z\in\pp_{L_i}^{b_i}$. Combining it with \eqref{D:EQ:corechart} proves that
$\Nm_{K/L_i}(C)=a_i+A+Q_i$ is the stationary coefficient for $\theta_i$.
Here $v_K(X)=-2h<-t_1=v_K(Y)$, so
$v_{L_i}(\Nm_{K/L_i}(C))=-2h$, as required by $J_i-m_{L_i}(\theta_i)=-2h$.

The conductors of $\theta_i$ and $\omega_i$ are even. Apply
Lemma~\ref{U:stationary-factor} with coefficients $\Nm_{K/L_i}(C)$ for
$\theta_i$ and $\zeta_i$ for $\omega_i$. This gives
\[
 \mathcal A_{L_i}(\theta_i)
   =D_\theta(\alpha^{-1})\Theta(C)^{-1}
                    \Psi_F(\Tr_{L_i/F}(\Nm_{K/L_i}(C))+\zeta_i).
\]
By Lemma~\ref{U:e2} and \eqref{D:EQ:commonC}, its last exponent is
$e_2(C)+(h_i+p_X)^2+h_i^2=e_2(C)+p_X^2$,
the same for $i=1,2$. Hence
$\mathcal A_{L_1}(\theta_1)=\mathcal A_{L_2}(\theta_2)$.
Relabeling the third field proves the other comparison.
\end{proof}

\subsection{Stable twists and completion}
\begin{proposition}[Stable-twist comparison]\label{D:EQ:highcomparison}
Let $(\theta_i)$ be a primitive compatible family written as in
\eqref{D:EQ:actual-family}. If $m_F(\lambda)\ge2T_2+r_1-1$, then
$\mathcal A_{L_1}(\theta_1)=\mathcal A_{L_2}(\theta_2)$.
\end{proposition}
\begin{proof}
Put $n=m_F(\lambda)$ and $h=n-T_2$. The hypothesis gives
$h\ge r_1+2r_2-1$.
By \eqref{U:stationary-linearization}, choose $\gamma\in F^\times$ such that
\[
 \lambda(1-z)=\psi_F(\gamma z)\qquad(z\in\pp_F^{\ceil{n/2}}).
\]
Put $A=\gamma/\alpha$ and $c_F=\gamma^{-1}$. Then $v_F(A)=-h$.
For each $i=1,2$, choose a norm approximation
\begin{equation*}
 \frac{\Nm_{L_i/F}(x_i)}{A/\zeta_i}\in U_F^{t_i},\qquad
 v_{L_i}(x_i)=-h-(T_2-T_i).
\end{equation*}
Apply \eqref{U:norm-approximation} to $L_i/F$, the element
$A/\zeta_i$, and $s=t_i$ to obtain these $x_i$.

We first show that $A+\zeta_ix_i$ is a stationary coefficient of
$\lambda\circ\Nm_{L_i/F}$.

Put $\varepsilon_i=\zeta_i \Nm_{L_i/F}(x_i)/A\in U_F^{t_i}$.
For $z\in\pp_{L_i}^{b_i}$, where $b_i=n-r_i$, one has
\[
 \Tr_{L_i/F}(z),\ \Nm_{L_i/F}(z)\in\pp_F^{\ceil{n/2}}.
\]
Also $x_iz\in\pp_{L_i}^{r_i}$, so equation~\eqref{D:UR:phase} applies to
$L_i/F$ and $\omega_i$ with additive scale $\alpha\zeta_i$, using
\eqref{D:EQ:lowerchart}. The error in the norm approximation for $x_i$ above, after multiplication by $\Nm_{L_i/F}(z)$, has
$F$-valuation at least
$-h+t_i+b_i=T_2+r_i-1\ge T_2$.
Quadratic norm expansion gives the coefficient $A+\zeta_i x_i$
for $\lambda\circ \Nm_{L_i/F}$ on $\pp_{L_i}^{b_i}$.
Since
$v_{L_i}(\zeta_ix_i/A)=h+T_2-T_i>0$,
one has $v_{L_i}(A+\zeta_i x_i)=v_{L_i}(A)$.

Put $c_i^{\mathrm{tw}}=(\alpha(A+\zeta_ix_i))^{-1}$ and
$\xi_i=\zeta_ix_i/A$. The lower bound $h\ge r_1+2r_2-1$ places
$\xi_i$ in $\pp_{L_i}^{s_i}$, where \eqref{D:EQ:corechart} gives
\[
 \chi_i(c_i^{\mathrm{tw}}/c_F)=\chi_i((1+\xi_i)^{-1})
                        =\Psi_{L_i}(a_i\xi_i).
\]

For the norm conversion, the variable
$a_ix_i/A$ has valuation
$h-t_1-(T_2-T_i)\ge r_i$.
For $i=1$ this inequality is $h\ge r_1+2r_2-1$; for $i=2$ it is
weaker. Apply equation~\eqref{D:UR:phase} to $L_i/F$ and $\omega_i$ with
additive scale $\alpha\zeta_i$, using \eqref{D:EQ:lowerchart} and
$\Nm_{L_i/F}(a_i)=\Nm_{K/F}(Y)$ from \eqref{D:EQ:originval}, to get
\[
 \Psi_{L_i}(\zeta_i a_ix_i/A)
                      =\Psi_F(\Nm_{K/F}(Y)\varepsilon_i/A).
\]
By the same lower bound on $h$,
\[
 v_F\!\left(\frac{\Nm_{K/F}(Y)}A(\varepsilon_i-1)\right)
 \ge h-t_1+t_i\ge T_2\qquad(i=1,2).
\]
Thus
\begin{equation*}
 \chi_i(c_i^{\mathrm{tw}})=\chi_i(c_F)\Psi_F(\Nm_{K/F}(Y)/A).
\end{equation*}

The conductors $m_i$ satisfy the stable-twist bounds over $L_i$:
\[
 b_1=T_2+h-r_1\ge 2T_2-1=m_1,\qquad
 b_2=r_2+h\ge T_1+T_2-1=m_2.
\]
The stable-twist formula \eqref{U:stable} gives
$\Delta_{L_i}(\theta_i)=\chi_i(c_i^{\mathrm{tw}})\Delta_{L_i}(\lambda\circ \Nm_{L_i/F})$.

Equation~\eqref{U:FML} for $L_i/F$ gives
\[
 \Delta_{L_i}(\lambda\circ \Nm_{L_i/F})\Delta_F(\omega_i)
                =\Delta_F(\lambda)\Delta_F(\lambda\omega_i).
\]
Since $n\ge2T_2\ge2T_i$, the stable-twist formula over $F$ applies, so
the right side is $\omega_i(c_F)\Delta_F(\lambda)^2$.
Multiplication and the character-value formula for $\chi_i(c_i^{\mathrm{tw}})$ above give
\[
 \mathcal A_{L_i}(\theta_i)
 =\chi_i(c_F)\omega_i(c_F)\Psi_F(\Nm_{K/F}(Y)/A)\Delta_F(\lambda,\psi_F)^2
 \qquad(i=1,2).
\]
The right-hand side is independent of $i$ by Proposition~\ref{U:determinant}.
\end{proof}

\begin{proof}[Proof of Theorem~\ref{D:EQ:main}]
Apply Proposition~\ref{D:EQ:realize} to write $\theta_i=\chi_i(\lambda\circ \Nm_{L_i/F})$. If $n\le T_2+r_1-1$, use
Proposition~\ref{D:EQ:min-two}.
Otherwise $h=n-T_2\ge r_1$. The interval
$r_1\le h\le r_1+2r_2-2$ is Proposition~\ref{D:EQ:intermediate}, and the
remaining integers are Proposition~\ref{D:EQ:highcomparison}.
Their endpoints are consecutive, including $r_1=r_2=1$.
In the two-break case repeat the comparison of $L_1$ with $L_3$.
This proves independence of all three.

Every nontrivial additive character of $F=k((\varpi))$ is a scaling
$x\mapsto\psi_F(cx)$ of \eqref{D:EQ:canonicalpsi}, with $c\in F^\times$.
By \eqref{U:additive-scaling}, this change multiplies
$\mathcal A_{L_i}(\theta_i)$ by
$\theta_i(c)\omega_i(c)=D_\theta(c)$, independent of $i$ by
Proposition~\ref{U:conjugacy}. This proves the stated arbitrary-character result.
\end{proof}

\begin{remark}[An alternative without the First Main Lemma]\label{D:EQ:FML-free-alternative}
The higher-conductor comparison also admits a proof without the First
Main Lemma. Choose a stationary coefficient $\gamma$ for $\lambda$
relative to $\psi_F$ in the $1-z$ convention, and $u\in U_F^{t_2}$ with
$\omega_3(u)=\omega_3(\gamma/\alpha)$.
Replacing the auxiliary scale $\alpha$ by $\alpha u$ makes $\gamma/(\alpha u)$ an
exact norm from $L_3$ and preserves \eqref{D:EQ:lowerchart} and
\eqref{D:EQ:corechart}. The argument of Proposition~\ref{D:EQ:intermediate}
then applies to all $m_F(\lambda)\ge T_2+r_1$, with no upper bound.
At the new scale, the square term $(\Tr_{L_3/F}(X))^2$ in the final
additive argument is retained: it is common to both sides, whether or
not its character value is $1$.
\end{remark}

\section{Totally ramified biquadratic extensions in mixed characteristic: nonmaximal~breaks}\label{sec:dyadic-nonmaximal}
Let $F/\mathbf Q_2$ be finite, with residue field $k$, uniformizer
$\varpi$, and normalized valuation $v_F$; put $e=v_F(2)$.
Let $K/F$ be totally ramified with Galois group $C_2^2$.
For its quadratic intermediate fields $L_1,L_2,L_3$, let $t_i$
denote the lower ramification break of $L_i/F$.
Assume that the fields are labeled so that
\begin{equation}\label{D:NM:breaks}
 t_1=2r_1-1,\qquad t_2=t_3=2r_2-1,\qquad 1\le r_1\le r_2\le e.
\end{equation}
Let $\omega_i$ be the nontrivial character in $S(L_i/F)$.
For an arbitrary nontrivial additive character $\psi_F$, use
$\psi_{L_i}=\psi_F\circ\Tr_{L_i/F}$.

\begin{theorem}[The mixed-characteristic case with nonmaximal breaks]\label{D:NM:main}
For every primitive compatible family $(\theta_1,\theta_2,\theta_3)$,
where $\theta_i$ is a quasi-character of $L_i^\times$, the three expressions
\begin{equation*}
 \mathcal A_{L_i}(\theta_i)=
 \Delta_{L_i}(\theta_i,\psi_{L_i})\Delta_F(\omega_i,\psi_F)
\end{equation*}
are equal. 
\end{theorem}
Fix such a family and write $\Theta$ for its common norm pullback to $K$.
Put $T_1=2r_1$ and $T_2=T_3=2r_2$.
These $T_i$ are the different exponents of $L_i/F$ and the conductors
of $\omega_i$.  Lemma~\ref{D:break-ledger} gives the different exponents of $K/L_i$
\begin{equation}\label{D:NM:upper}
 D'_1=2T_2-T_1,\qquad D'_2=D'_3=T_1.
\end{equation}

\subsection{Quadratic generators and norm-character coefficients}
\begin{lemma}[Alignment of the quadratic generators]\label{D:NM:alignment}
There exist generators $z_i\in L_i$, elements $f_i\in F$ $(i=1,2)$, and $d\in F$ such that
\begin{equation*}
 z_1^2+z_1=f_1,\qquad z_2^2+z_2=f_2,\qquad
 v_F(f_1)=1-2r_1,\quad v_F(f_2)=1-2r_2.
\end{equation*}
\begin{equation}\label{D:NM:align}
 v_F(d)=r_2-r_1,\qquad E:=f_1+d^2f_2\in\pp_F^{1-r_1}.
\end{equation}
Moreover $1+d$ is a unit, including when $t_1=t_2$.
The nontrivial automorphism of $L_i/F$ sends $z_i$ to $-1-z_i$.
\end{lemma}
\begin{proof}
Choose a uniformizer $\Pi$ of $L_1$, with minimal polynomial
$X^2-cX+b$ over $F$. This polynomial is Eisenstein, so $v_F(b)=1$.
Since $\OO_{L_1}=\OO_F[\Pi]$, the monogenic different formula in
Appendix~\ref{app:diamond-foundations} gives
\[
 v_{L_1}(2\Pi-c)=2r_1<2e+1=v_{L_1}(2\Pi).
\]
Hence $v_F(c)=r_1$. The initial choices
$z_1^{(0)}=-\Pi/c$ and $f_1^{(0)}=-b/c^2$ satisfy
$(z_1^{(0)})^2+z_1^{(0)}=f_1^{(0)}$ and
$v_F(f_1^{(0)})=1-2r_1$, and $z_1^{(0)}$ generates $L_1$.
The same construction in $L_2$ gives $z_2,f_2$.

Put $S_0=1+2z_1^{(0)}$, $U_0=S_0^2=1+4f_1^{(0)}$, and
$q=\varpi^{2e-2r_1+1}$. Set
\[
\begin{gathered}
 S=S_0(1+\varpi^{e-r_1+1}w),\qquad d=\varpi^{r_2-r_1}v,\\
 z_1=(S-1)/2,\qquad f_1=(S^2-1)/4.
\end{gathered}
\]
The two base variables $v,w$ will be integral, with $v$ a unit.
Modulo $\pp_F^{r_1}$,
\[
 \frac{S^2-1+4d^2f_2}{q}
 \equiv u_0+\varpi U_0w^2+v_0v^2,
 \qquad u_0=4f_1^{(0)}/q,\quad v_0=4\varpi^{2r_2-2r_1}f_2/q.
\]
The omitted linear term has valuation at least $r_1$; $u_0,v_0,U_0$ are
units.
The quotient $\OO_F/\pp_F^{r_1}$ has characteristic two and a coefficient
field $k$, hence is $k[\varpi]/(\varpi^{r_1})$. For any
$A,B,C\in\OO_F^\times$, solving successively in the degree of
$\varpi$ gives $v\in\OO_F^\times$, $w\in\OO_F$ such that
\begin{equation}\label{D:MX:coefficient-solving}
 Bv^2+\varpi Cw^2\equiv A\pmod{\pp_F^{r_1}}.
\end{equation}
At even degrees one solves for the next coefficient of $v^2$, and at
odd degrees for that of $w^2$; Frobenius on $k$ is bijective.
Equation~\eqref{D:MX:coefficient-solving}, with
$A=u_0$, $B=v_0$, and $C=U_0$, gives $v,w$ satisfying
$v_0v^2+\varpi U_0w^2\equiv u_0\pmod{\pp_F^{r_1}}$, with $v$ a unit.
Thus $v_F(E)\ge1-r_1$. Since
$v_F(4f_1-4f_1^{(0)})>2e-2r_1+1$, one still has $v_F(f_1)=1-2r_1$.
The factor $1+\varpi^{e-r_1+1}w$ is a unit of $F$, so
$F(z_1)=F(S)=F(S_0)=L_1$. The nontrivial automorphism of $L_i/F$
sends $z_i$ to the other root $-1-z_i$ of $X^2+X-f_i$.

If $t_1<t_2$, then $d\in\pp_F$. Suppose $t_1=t_2$ and $\bar d=1$.
Then $f_1+f_2$ has valuation at least $2-2r_1$, and so does
$f_1+f_2+4f_1f_2$, since $e\ge r_1$. The generator
$z_3=z_1+z_2+2z_1z_2\in L_3$ satisfies $z_3^2+z_3=f_1+f_2+4f_1f_2$.
The integral element $\varpi^{r_1-1}z_3$ generates an order whose
polynomial discriminant has valuation $2r_1-2$. By
Lemma~\ref{C:order-discriminant}, the different exponent is
at most $2r_1-2$, contradicting $T_3=2r_1$. Thus $1+d$ is a unit.
\end{proof}

Put
\begin{equation*}
 \begin{gathered}
 z_3=z_1+z_2+2z_1z_2\in L_3,\qquad Y=z_1+dz_2,\\
 a_i=\Nm_{K/L_i}(Y),\qquad h_i=\Tr_{K/L_i}(Y),\qquad
 \zeta_i=\Nm_{L_i/F}(h_i).
 \end{gathered}
\end{equation*}
\begin{lemma}[Norms and traces of the origin]\label{D:NM:origin-identities}
One has
\begin{equation*}
 z_3^2+z_3=f_1+f_2+4f_1f_2,
\end{equation*}
and
\begin{align}\label{D:NM:htable}
 a_1&=2f_1-E-(1+d)z_1,&h_1&=2z_1-d,&\zeta_1&=d^2+2d-4f_1,\notag\\
 a_2&=E-2f_1-d(1+d)z_2,&h_2&=2dz_2-1,&\zeta_2&=1+2d-4d^2f_2,\notag\\
 a_3&=-E-dz_3,&h_3&=-(1+d),&\zeta_3&=(1+d)^2.
\end{align}
For $i=1,2,3$, these elements satisfy
\[
\begin{gathered}
 v_K(Y)=v_{L_i}(a_i)=-t_1,\qquad v_F(\zeta_i)=T_2-T_i,\\
 v_{L_1}(h_1)=T_2-T_1,\qquad v_{L_2}(h_2)=v_{L_3}(h_3)=0,\\
 v_{L_3}(z_3)=-t_2,\qquad \Tr_{K/F}(Y)=-2(1+d).
\end{gathered}
\]
\end{lemma}
\begin{proof}
Expanding the quadratic equations for $z_1,z_2$ gives
the asserted equation for $z_3$ and \eqref{D:NM:htable}.
If $t_1<t_2$, then $v_F(f_2)=1-2r_2<v_F(f_1)$.
If $t_1=t_2$, Lemma~\ref{D:NM:alignment} gives
$v_F(f_1+f_2)=1-2r_2$. In both cases $v_F(4f_1f_2)>1-2r_2$.
Thus $v_F(f_1+f_2+4f_1f_2)=1-2r_2$,
and $v_{L_3}(z_3)=-t_2$.

In \eqref{D:NM:htable}, the constant terms of $a_i$ have base
valuation at least $1-r_1$, whereas the nonconstant terms have
$L_i$-valuation $1-2r_1$. Hence $v_{L_i}(a_i)=-t_1$,
and total ramification gives $v_K(Y)=-t_1$.
For $h_1$, $v_{L_1}(2z_1)=2e-2r_1+1>2(r_2-r_1)$;
the other traces are units. Taking their norms proves
$v_F(\zeta_i)=T_2-T_i$. The total trace follows from
\eqref{D:NM:htable}.
\end{proof}

The stationary linearization \eqref{U:stationary-linearization},
applied to $\omega_2$ at depth $r_2$, and the fact that $\zeta_2$ is a
unit give $\alpha\in F^\times$ such that
\begin{equation}\label{D:NM:normalization}
 \begin{gathered}
 \omega_2(1-z)=\Psi_F(\zeta_2z)\quad(z\in\pp_F^{r_2}),\\
 \Psi_L(z)=\psi_L(\alpha z),\qquad v_F(\alpha)=-n_F(\psi_F)-T_2.
 \end{gathered}
\end{equation}
The largest trivial ideals for $\Psi_F,\Psi_{L_1},\Psi_{L_2},\Psi_{L_3}$
have exponents
$T_2$, $J_1=2T_2-T_1$, $J_2=J_3=T_2$.

\begin{lemma}[Simultaneous coefficients for the norm characters]\label{D:NM:lower}
For $i=1,2,3$, one has
\begin{equation}\label{D:NM:lowerformula}
 \omega_i(1-z)=\Psi_F(\zeta_i z)\qquad(z\in\pp_F^{\ceil{T_i/2}}).
\end{equation}
\end{lemma}
\begin{proof}
Using \eqref{U:stationary-linearization}
for $\omega_1$ at depth $r_1$, choose $b_1$ such that
$\omega_1(1-z)=\psi_F(b_1z)$ on $\pp_F^{r_1}$.
The coefficient valuations are
$v_F(b_1)=-n_F(\psi_F)-2r_1$ and
$v_F(\alpha\zeta_2)=-n_F(\psi_F)-2r_2$.
By Lemma~\ref{D:quadratic-norms}, $\omega_3=\omega_1\omega_2$ has
coefficient $b_1+\alpha\zeta_2$ on $\pp_F^{r_2}$.

Write $(A,B)$ for the quadratic norm symbol of Section~\ref{D:quadratic-symbols}.
For $z\in\pp_F^{r_1}$, put
$C=1+4f_1z$ and $w=(f_1/f_2)z$. Then
\[
 v_F(C-1)\ge2e-r_1+1,\qquad v_F(w)\ge2r_2-r_1\ge r_2.
\]
Multiplying the second symbol argument by the norm $-4f_1$ and using
\eqref{D:MX:symbol-transform} gives
\[
 \omega_1(1-z)=(C,f_1(1-z)(1+4f_1)).
\]
The ratio $f_1/(-d^2f_2)$ lies in $U_F^{r_1}$, by \eqref{D:NM:align}.
Lemma~\ref{D:MX:unit-symbol} removes this ratio and $1-z$ from the second
argument. The result is $(C,-f_2(1+4f_1))$. Applying the same transformation
with $f_2,w$ gives
$\omega_2(1-w)=(C,f_2(1-w)(1+4f_2))=(C,f_2(1+4f_2))$.
Also $(C,-1)=1$ by Lemma~\ref{D:MX:unit-symbol}, since
$v_F(C-1)+v_F(-2)\ge3e-r_1+1>2e$.
Consequently, for $z\in\pp_F^{r_1}$,
\begin{equation*}
 \omega_1(1-z)=\omega_2(1-(f_1/f_2)z)\,\omega_3(1+4f_1z).
\end{equation*}
In every use of the lemma, the sum of the two relevant valuations is
greater than $2e$; in particular $(2e-r_1+1)+r_1=2e+1$.
Since $v_F(4f_1z)\ge2e-r_1+1\ge r_2$, the formulas for
$b_1,\alpha\zeta_2,b_1+\alpha\zeta_2$ apply.
By \cite[Lemma~4.1]{Ueda}, the annihilator of $\pp_F^{r_1}$ under
$(b,z)\mapsto\psi_F(bz)$ is $\pp_F^{-n_F(\psi_F)-r_1}$.
The equation above therefore gives
\begin{equation*}
 b_1-\alpha\zeta_2 R\in\pp_F^{-n_F(\psi_F)-r_1},\qquad
 R=\frac{f_1(1-4f_2)}{f_2(1+4f_1)}.
\end{equation*}
Here multiplication of the variable $z$ by the unit $1+4f_1$ preserves
the ideal $\pp_F^{r_1}$.

We compare $R$ with $\zeta_1/\zeta_2$.
Expansion gives
\[
 f_1(1-4f_2)\zeta_2-f_2(1+4f_1)\zeta_1
 =E(1+2d+16f_1f_2)-8f_1f_2d(d+2)-2f_2d(d^2+d+1).
\]
Every term has valuation at least $1-r_1$. For the middle term,
$v_F(d+2)=(r_2-r_1)<e$, so its valuation is at least
$3e+2-4r_1\ge1-r_1$; for the last it is at least
$e-r_2-r_1+1\ge1-r_1$. The first follows from \eqref{D:NM:align} and
$v_F(16f_1f_2)\ge2$. Division by $f_2(1+4f_1)\zeta_2$, of valuation $1-2r_2$,
gives $R-\zeta_1/\zeta_2\in\pp_F^{2r_2-r_1}$.
Together with the bound for $b_1-\alpha\zeta_2R$ above, this proves
\eqref{D:NM:lowerformula} for $i=1$.

The second is \eqref{D:NM:normalization}. Finally,
$\zeta_3-\zeta_1-\zeta_2=-2d+4E\in\pp_F^{r_2}$,
which proves the third formula on $\pp_F^{r_2}$.
\end{proof}

\subsection{Compatible characters, twists, and norm representatives}
Put
\begin{equation*}
 m_1=2T_2-1,\qquad m_2=m_3=T_1+T_2-1,\qquad
 s_i=\ceil{m_i/2}.
\end{equation*}
Thus $s_1=T_2$, $s_2=s_3=r_1+r_2$, and $m_i=J_i+t_1=2s_i-1$.
The formulas
\begin{equation}\label{D:NM:Rcharts}
 R_i(1-z)=\Psi_{L_i}(a_i z)\qquad(z\in\pp_{L_i}^{s_i})
\end{equation}
define characters with exact conductor $m_i$ on the indicated groups:
$2s_i-t_1\ge J_i$, and the induced character on
$U_{L_i}^{m_i-1}/U_{L_i}^{m_i}$ is nontrivial.

\begin{lemma}[The conjugate quotient of the prescribed character]\label{D:NM:commmodel}
Let $\sigma$ be the nontrivial automorphism of $L_1/F$ and put
$\nu_1=\omega_2\circ \Nm_{L_1/F}$, the nontrivial character in $S(K/L_1)$.
For every $v\in L_1^\times$ with $\sigma(v)/v\in U_{L_1}^{s_1}$,
\begin{equation*}
 R_1(\sigma(v)/v)=\nu_1(v).
\end{equation*}
\end{lemma}
\begin{proof}
Multiplication of $v$ by an element of $F^\times$ changes neither side. If $v_{L_1}(v)$ were odd, Lemma~\ref{C:leading-ramification} would give
$v_{L_1}(\sigma(v)/v-1)=t_1<2r_2$, so this case is excluded. Otherwise scale to a unit and
write $v=P+Qz_1$. Parity of the two summand valuations makes $P$ a unit
and $v_F(Q)\ge r_1$. The element $1+2z_1$ is a unit of $L_1$. Thus the assumption
$v_{L_1}(\sigma(v)/v-1)\ge2r_2$ forces $z=Q/P\in\pp_F^{r_2}$.

Put
\[
 V_0=1-z-f_1z^2=\Nm_{L_1/F}(1+z z_1),\quad
 Q_0=1-dz-d^2f_2z^2=\Nm_{L_2/F}(1+dz z_2),\quad H_0=1+d^2f_2z^2.
\]
All three are units. Write $c=f_1-d^2f_2=E-2d^2f_2$, so $v_F(c)\ge1-r_1$.
The difference $Q_0-V_0=(1-d)z+cz^2$ lies in $\pp_F^{r_2}$.
Since $Q_0$ is a norm from $L_2$, \eqref{D:NM:lowerformula} gives
\begin{equation*}
 \nu_1(v)=\omega_2(V_0)=
 \Psi_F\bigl(\zeta_2(Q_0-V_0)/Q_0\bigr).
\end{equation*}
On the other hand $\sigma(v)/v=1-z(1+2z_1)/(1+z z_1)$.
Using $a_1=f_1-d^2f_2-(1+d)z_1$, a quadratic trace calculation yields
\begin{equation*}
 R_1(\sigma(v)/v)=\Psi_F(B_0/V_0),\qquad
 B_0=- (1+4f_1)(cz^2+(1+d)z).
\end{equation*}

We compare the two arguments of $\Psi_F$ modulo $\pp_F^{2r_2}$.
Both numerators in the preceding two character formulas lie in
$\pp_F^{r_2}$. Also
\[
 V_0-H_0=-z-Ez^2\in\pp_F^{r_2},\quad
 Q_0-H_0=-dz-2d^2f_2z^2\in\pp_F^{r_2}.
\]
Since the denominators are units, these bounds show that replacing
$V_0,Q_0$ by $H_0$ changes each argument by an element of $\pp_F^{2r_2}$
and leaves its $\Psi_F$-value unchanged.
With $U=1+4f_1$, the remaining numerator difference is
\[
 B_0-\zeta_2(Q_0-V_0)
 =-(U+\zeta_2)cz^2-(U(1+d)+\zeta_2(1-d))z.
\]
The identities
\[
 U+\zeta_2=2(1+d)+4c,\qquad
 U(1+d)+\zeta_2(1-d)=2(1+d-d^2)+4E(1+d)-8d^2f_2
\]
show that the first coefficient has valuation at least $e$, and the
second at least $e\ge r_2$. The first summand therefore has valuation
at least $2r_2+e+1-r_1\ge2r_2$ and the second at least $2r_2$.
The two character formulas now prove the assertion.
\end{proof}

\begin{lemma}[Compatibility of the prescribed characters on principal units]\label{D:NM:compat}
The norm pullbacks of the characters $R_i$ are defined and agree on
$U_K^{T_2}$.
\end{lemma}
\begin{proof}
For $z\in\pp_K^{2r_2}$, equation~\eqref{U:trace-ideal} with the
different exponents in \eqref{D:NM:upper}, together with total ramification, gives
\begin{align*}
 v_{L_1}(\Tr_{K/L_1}(z))&\ge3r_2-r_1,\\
 v_{L_i}(\Tr_{K/L_i}(z))&\ge r_2+r_1 &&(i=2,3),\\
 v_{L_i}(\Nm_{K/L_i}(z))&\ge2r_2 &&(i=1,2,3).
\end{align*}
Thus both trace and norm belong to $\pp_{L_i}^{s_i}$, where the formula for $R_i$ applies.

By \eqref{U:trace-ideal},
$v_{L_1}(\Tr_{K/L_1}(Yz))\ge3r_2-2r_1$ and
$v_{L_i}(\Tr_{K/L_i}(Yz))\ge r_2$ for $i=2,3$.
Thus $v_{L_i}(w_i)\ge r_2\ge T_i/2$ for
$w_i=\Tr_{K/L_i}(Yz)/h_i$,
using \eqref{D:NM:htable}. Thus $\Nm_{L_i/F}(1-w_i)\in U_F^{T_i/2}$.
Since $\Tr_{K/L_i}(Y(1-z))=h_i(1-w_i)$ and $\omega_i$ is trivial on norms,
\eqref{D:NM:lowerformula} gives
\[
 1=\omega_i(\Nm_{L_i/F}(1-w_i))
  =\Psi_F\bigl(\zeta_i-\Nm_{L_i/F}(\Tr_{K/L_i}(Y(1-z)))\bigr).
\]
With the mixed-characteristic signs, Lemma~\ref{U:e2} therefore gives
\[
 R_i(\Nm_{K/L_i}(1-z))
 =\Psi_F(-e_2(Y(1-z))+e_2(Y)),
\]
independent of $i$.
\end{proof}

\begin{proposition}[Extension to compatible characters of minimal conductor]\label{D:NM:core-exists}
The characters $R_i$ extend to a primitive compatible family of
characters of $L_i^\times$ with conductors $m_i$.
\end{proposition}
\begin{proof}
Apply the character-extension construction in the proof of
Proposition~\ref{D:EQ:core-exists}, using
Lemmas~\ref{D:NM:commmodel} and~\ref{D:NM:compat} in place of
Lemmas~\ref{D:EQ:commmodel} and~\ref{D:EQ:Rcompat}.
For $i=2,3$, the break of $K/L_i$ is $t_1$ and
$2s_i-t_1=T_2+1$, so \eqref{U:unit-norm-surjective} supplies the
norm surjectivity used in that construction. It therefore gives
primitive compatible characters $\chi_i$ with
$\chi_i|_{U_{L_i}^{s_i}}=R_i$.
Each $R_i$ is trivial on $U_{L_i}^{m_i}$ and nontrivial on
$U_{L_i}^{m_i-1}$, so $m_{L_i}(\chi_i)=m_i$.
\end{proof}

\begin{proposition}[Common twists and their conductors]\label{D:NM:realize}
For every primitive compatible family $(\theta_i)$, there are
characters $\chi_i$ as in
Proposition~\ref{D:NM:core-exists} and a quasi-character $\lambda$ of $F^\times$
such that
\begin{equation}\label{D:NM:family}
 \theta_i=\chi_i(\lambda\circ \Nm_{L_i/F}).
\end{equation}
Writing $n=m_F(\lambda)$, one has
\[
 m_{L_i}(\theta_i)=
 \begin{cases}
 m_i,&n\le T_2+r_1-1,\\
 2n-T_i,&n\ge T_2+r_1.
 \end{cases}
\]
The conductors in the second range are even.
\end{proposition}
\begin{proof}
Apply the common-twist construction in the proof of
Proposition~\ref{D:EQ:realize} to characters supplied by
Proposition~\ref{D:NM:core-exists}. Its changes to $\chi_2,\chi_3$
are twists by $S(K/L_i)$, whose characters have conductor at most
$T_1\le s_i$. Hence the evaluations \eqref{D:NM:Rcharts} remain valid,
and the construction gives \eqref{D:NM:family}.
The conductor calculation in that proof, using
\eqref{U:norm-pullback-conductor}--\eqref{U:break-pullback-conductor}
with $T_i$ even and $m_i$ odd, gives the two stated ranges.
\end{proof}

Fix $\chi_i$ and $\lambda$ as in Proposition~\ref{D:NM:realize}, and put
$n=m_F(\lambda)$.

\begin{lemma}[Normalization for an exact norm representative]\label{D:NM:freedom}
For $u\in U_F^{T_2-1}$, replacing $\alpha$ by $\alpha u$
preserves \eqref{D:NM:lowerformula} and \eqref{D:NM:Rcharts}.
For every $\gamma\in F^\times$, one may choose such a $u$ with
$\gamma/(\alpha u)\in\Nm_{L_3/F}(L_3^\times)$.
\end{lemma}
\begin{proof}
For $z\in\pp_F^{T_i/2}$ the change in the argument of $\Psi_F$ is
$(u-1)\zeta_i z$, with valuation at least $T_2-1+v_F(\zeta_i)+T_i/2\ge T_2$.
For $z\in\pp_{L_i}^{s_i}$ the change in the argument of $\Psi_{L_i}$
is $(u-1)a_i z$, with valuation at least $2(T_2-1)-t_1+s_i\ge J_i$. These inequalities include $r_1=r_2=1$.
Thus both formulas are preserved.

Since $\omega_3$ is nontrivial on $U_F^{T_2-1}$, choose
$u$ in that group with $\omega_3(u)=\omega_3(\gamma/\alpha)$. Then
$\gamma/(\alpha u)\in\Nm_{L_3/F}(L_3^\times)$.
\end{proof}

Write $g_i$ for the nontrivial automorphism of $K/L_i$.
For the comparison of $L_1$ and $L_2$, take $X\in L_3$ and put
\begin{equation}\label{D:NM:adjust}
 A=\Nm_{L_3/F}(X),\quad p_X=\Tr_{L_3/F}(X),\quad
 Q_i=\Tr_{K/L_i}(Yg_i(X))\quad(i=1,2).
\end{equation}
\begin{lemma}[Correction identities and estimates]\label{D:NM:bounds}
With primes denoting conjugation in $L_3/F$, one has, for $i=1,2$,
\begin{align}
 \Nm_{K/L_i}(Y-X)&=a_i+A-Q_i,&\Tr_{K/L_i}(Y-X)&=h_i-p_X,\label{D:NM:commonC}\\
 \Nm_{L_i/F}(Q_i)&=\zeta_iA+\mathcal E,&
 \mathcal E&=(X'-X)(a_3X'-a_3'X)\in F,\notag\\
 \Nm_{L_i/F}(\Tr_{K/L_i}(Y-X))&=\zeta_i+\Delta,&
 \Delta&=p_X^2+2(1+d)p_X.\label{D:NM:common-Delta}
\end{align}
If $X\ne0$ and $v_{L_3}(X)=-h$, then
\begin{align*}
 v_F(p_X)&\ge\floor{(T_2-h)/2},&
 v_F(\mathcal E)&\ge T_2-r_1-h,\\
 v_{L_1}(Q_1)&\ge T_2-T_1-h,&
 v_{L_2}(Q_2)&\ge-h.
\end{align*}
\end{lemma}
\begin{proof}
Expansion of the quadratic norms gives the identities; the last uses
$\Tr_{L_i/F}(h_i)=-2(1+d)$ for both fields.
Write $X=P+Qz_3$. Parity of valuations gives
$v_F(P)\ge-\floor{h/2}$ and $v_F(Q)\ge r_2-\ceil{h/2}$.
Since $e\ge r_2$, the trace $2P-Q$ has valuation at least
$r_2-\ceil{h/2}=\floor{(T_2-h)/2}$.
Using $a_3=-E-dz_3$, the common error is
\[
 \mathcal E=Q(dP-EQ)(1+4(f_1+f_2+4f_1f_2)).
\]
Its last factor is a unit. The first product term has $F$-valuation
at least $r_2+(r_2-r_1)-h=T_2-r_1-h$; the second has valuation
at least $1-r_1+2r_2-2\ceil{h/2}\ge T_2-r_1-h$.

Finally $Yg_i(X)$ has $K$-valuation $-t_1-2h$.
Apply \eqref{U:trace-ideal} with the different exponents in
\eqref{D:NM:upper} to obtain the bounds for $Q_i$.
\end{proof}

\subsection{Stationary coefficients and comparison at minimal conductor}
\begin{proposition}[Stationary coefficients at minimal conductor]\label{D:NM:minimal-origin}
Assume $n\le T_2+r_1-1$ in \eqref{D:NM:family}.
After a normalization allowed by Lemma~\ref{D:NM:freedom}, there is
$C\in Y+\pp_{L_3}^{1-r_1}$ with $v_K(C)=-t_1$ such that, for $i=1,2$,
\begin{align}
 \theta_i(1-z)&=\Psi_{L_i}(\Nm_{K/L_i}(C)\,z)
                              &&(z\in\pp_{L_i}^{s_i}),  \label{D:NM:upper-stationary}\\
 \omega_i(1-z)&=\Psi_F(\Nm_{L_i/F}(\Tr_{K/L_i}(C))\,z)
                              &&(z\in\pp_F^{\ceil{T_i/2}}).\label{D:NM:lower-stationary}
\end{align}
For $i=1,2$, the traces $\Tr_{K/L_i}(C)$ are nonzero and satisfy
$v_{L_i}(\Tr_{K/L_i}(C))=v_{L_i}(h_i)$.
If $t_1=t_2$, the same statements hold for all three indices.
\end{proposition}
\begin{proof}
Both trace and norm of a variable
in $\pp_{L_i}^{s_i}$ lie in $\pp_F^{r_1+r_2}$ by \eqref{U:trace-ideal}
and total ramification.
If $n\le r_1+r_2$, take $X=0$.
Otherwise $h=n-T_2$ satisfies
$1-(r_2-r_1)\le h\le r_1-1$.
Since $\lceil n/2\rceil\le r_1+r_2<n$,
\eqref{U:stationary-linearization} gives $\gamma\in F^\times$ with
$\lambda(1-z)=\psi_F(\gamma z)$ on
$\pp_F^{r_1+r_2}$ and $v_F(\gamma/\alpha)=-h$.
By Lemma~\ref{D:NM:freedom}, choose $u\in U_F^{T_2-1}$ and
$X\in L_3^\times$ with $A=\gamma/(\alpha u)=\Nm_{L_3/F}(X)$, so
$v_{L_3}(X)=-h$. Replace $\alpha$ by $\alpha u$, so that
$A=\gamma/\alpha$, and use $\Psi_L(x)=\psi_L(\alpha x)$
with the new scale henceforth.

Put $C=Y-X$. When $X=0$, the assertions follow from
\eqref{D:NM:Rcharts}, \eqref{D:NM:lowerformula}, and
Lemma~\ref{D:NM:origin-identities}. Otherwise $h\le r_1-1$ gives $-2h>-t_1$,
so $v_K(C)=v_K(Y)$. For $z\in\pp_{L_i}^{s_i}$,
\[
 v_{L_1}((Q_1/\zeta_1)z)\ge2r_1-h\ge r_1,\qquad
 v_{L_2}((Q_2/\zeta_2)z)\ge r_1+r_2-h\ge r_2.
\]
Under these bounds, equation~\eqref{D:UR:phase}, applied to $L_i/F$ and
$\omega_i$ with additive scale $\alpha\zeta_i$ using
\eqref{D:NM:lowerformula}, gives
\begin{equation}\label{D:NM:Qphase}
 \Psi_{L_i}(Q_iz)
   =\Psi_F\bigl((A+\mathcal E/\zeta_i)\Nm_{L_i/F}(z)\bigr).
\end{equation}
The discarded terms satisfy
\[
\begin{aligned}
 v_F\!\left((\mathcal E/\zeta_1)\Nm_{L_1/F}(z)\right)&\ge T_2+r_1-h\ge T_2,\\
 v_F\!\left((\mathcal E/\zeta_2)\Nm_{L_2/F}(z)\right)&\ge3r_2-h\ge T_2.
\end{aligned}
\]
Since $\lambda(1-x)=\Psi_F(Ax)$ on $\pp_F^{r_1+r_2}$ and both
$\Tr_{L_i/F}(z)$ and $\Nm_{L_i/F}(z)$ lie in this ideal, the norm expansion
and \eqref{D:NM:Qphase} give
$(\lambda\circ \Nm_{L_i/F})(1-z)=\Psi_{L_i}((A-Q_i)z)$.
Multiplication by \eqref{D:NM:Rcharts} and \eqref{D:NM:commonC} proves
\eqref{D:NM:upper-stationary}.

By Lemma~\ref{D:NM:bounds},
$v_F(\Delta)\ge\min\{2r_2-2\ceil{h/2},\ e+r_2-\ceil{h/2}\}$.
Adding $r_1$ to both bounds gives at least $T_2$, since $h\le r_1-1$;
adding $r_2$ does also. Equation \eqref{D:NM:common-Delta} therefore proves
\eqref{D:NM:lower-stationary}. Moreover
$2r_2-2\ceil{h/2}>2(r_2-r_1)$ and $r_2-\ceil{h/2}>0$. Since
$\Tr_{K/L_i}(C)=h_i-p_X$, the bound on $p_X$ gives
$v_{L_1}(p_X)>v_{L_1}(h_1)$ and $v_{L_2}(p_X)>v_{L_2}(h_2)$.
Hence $v_{L_i}(\Tr_{K/L_i}(C))=v_{L_i}(h_i)$ for $i=1,2$.

When $t_1=t_2$, one has $s_3=T_2$ and $h\le r_2-1$.
For $L_3$, equation~\eqref{D:UR:phase}, applied to $L_3/F$ and $\omega_3$
with additive scale $\alpha\zeta_3$ using \eqref{D:NM:lowerformula}, at
$(X/(1+d))z$ gives the
pullback coefficient $A-(1+d)X$, since $\zeta_3=(1+d)^2$.
But
\[
 \Nm_{K/L_3}(C)=a_3+X^2+(1+d)X,
 \qquad (X^2+(1+d)X)-(A-(1+d)X)=p_X X-2A+2(1+d)X.
\]
The three terms on the right have $L_3$-valuations at least
$2r_2-2h-(h\bmod2)$, $2e-2h$, and $2e-h$, all nonnegative.
Their products with $z\in\pp_{L_3}^{T_2}$ lie in the trivial ideal
of $\Psi_{L_3}$, proving \eqref{D:NM:upper-stationary} for $i=3$.
Also $\Tr_{K/L_3}(C)=-(1+d)-2X$ is a unit and
\[
 \Nm_{L_3/F}(\Tr_{K/L_3}(C))-\zeta_3=2(1+d)p_X+4A.
\]
Both terms have $F$-valuation at least $r_2$, proving
\eqref{D:NM:lower-stationary} for $i=3$.
\end{proof}

\begin{proposition}[Minimal-conductor comparison]\label{D:NM:minimal-two}\label{D:NM:minimal-one}
For every primitive compatible family with conductors $m_i$, one has
\[
 \mathcal A_{L_1}(\theta_1)=\mathcal A_{L_2}(\theta_2)
                         =\mathcal A_{L_3}(\theta_3).
\]
\end{proposition}
\begin{proof}
The assumed conductors $m_i$ give $n\le T_2+r_1-1$ by
Proposition~\ref{D:NM:realize}.
Choose $C$ as in Proposition~\ref{D:NM:minimal-origin}, using the
normalization allowed there. Use the following construction for
$i=1,2$, and also for $i=3$ when $t_1=t_2$.
Take the coefficients $\Nm_{K/L_i}(C)$ for $\theta_i$ and $\Nm_{L_i/F}(\Tr_{K/L_i}(C))$ for
$\omega_i$ from Proposition~\ref{D:NM:minimal-origin}.
Choose a uniformizer $\pi$ of $K$ and compatible uniformizers
$\Pi_i=\Nm_{K/L_i}(\pi)$, $\varpi=\Nm_{L_i/F}(\Pi_i)$. These choices need not be the ones
used to construct the generators earlier. In \eqref{U:normalized-residual},
use $\delta_{\theta_i}=\Pi_i^{\lfloor m_i/2\rfloor}$ for $\theta_i$, and write
$H_{\theta_i}$ for the resulting critical function and $g_{\theta_i}$
for its normalized sum.
Put $D_\theta=\theta_i|_{F^\times}\omega_i$, independent of $i$ by
Proposition~\ref{U:conjugacy}. Since $T_i$ is even, Lemma~\ref{U:common-origin} gives
\begin{equation}\label{D:NM:minimal-factor}
 \mathcal A_{L_i}(\theta_i)=D_\theta(-\alpha^{-1})\Theta(C)^{-1}
                         \Psi_F(-e_2(C))g_{\theta_i}.
\end{equation}

Keep $C$ fixed and introduce $C_z=C(1+\pi^{t_2}z)$
for $z\in\OO_K$.
The critical depths for $\theta_1$ and $\theta_2,\theta_3$ are $t_2$ and $r_1+r_2-1$, respectively.
The changes satisfy
$v_{L_1}(\Tr_{K/L_1}(C\pi^{t_2}z))\ge3r_2-2r_1$ and
$v_{L_i}(\Tr_{K/L_i}(C\pi^{t_2}z))\ge r_2$ for $i=2,3$.
After division by $\Tr_{K/L_i}(C)$, the valuation is at least
$r_2\ge T_i/2$.
Their norm quotients lie in $U_F^{T_i/2}$ and have $\omega_i$-value $1$.
The traces $\Tr_{K/L_i}(C_z)$ remain nonzero. When $t_1=t_2$, these estimates also apply to $i=3$, since
all stationary depths and relative different exponents agree.

Apply Lemma~\ref{U:common-function} to $C,C_z$.
The domains just checked and \eqref{U:normalized-residual} identify
the $\theta_i$ factors with their critical-function values, while
stationarity in Proposition~\ref{D:NM:minimal-origin} makes the
$\omega_i$ factors $1$. Writing $Q(\bar z)$ for the common value, we obtain
\begin{equation*}
 H_{\theta_i}(\rho_i(\bar z))=Q(\bar z),
\end{equation*}
where $\rho_i:k\to k$ is induced by $\Nm_{K/L_i}$ on the successive unit
quotients of depths $t_2$ in $K$ and $\lfloor m_i/2\rfloor$ in
$L_i$, using $\pi,\Pi_i$.

The residue-map calculation in the proof of
Proposition~\ref{D:EQ:min-two} applies to these maps: it uses only
\eqref{U:trace-ideal}, the breaks of $K/L_i$, and the critical norm
formula of \cite[Proposition~3.4]{Ueda} over the residue field $k$.
For $t_1<t_2$, the two maps are bijective, so
$g_{\theta_1}=g_{\theta_2}$. For $t_1=t_2$, their kernels have order
two and are pairwise distinct; \eqref{U:critical-polar} and
Proposition~\ref{U:missing-coset} give equality of all three normalized
sums. Equation~\eqref{D:NM:minimal-factor} proves the corresponding
equalities of $\mathcal A_{L_i}(\theta_i)$.
When $t_1<t_2$, repeat the comparison with $L_3$ in place of $L_2$.
\end{proof}

\subsection{Higher-conductor comparison and completion}
\begin{proposition}[Higher-conductor comparison]\label{D:NM:higher}
Let $(\theta_i)$ be a primitive compatible family written as in
\eqref{D:NM:family}. If $m_F(\lambda)\ge T_2+r_1$, then
$\mathcal A_{L_1}(\theta_1)=\mathcal A_{L_2}(\theta_2)$.
\end{proposition}
\begin{proof}
Put $n=m_F(\lambda)$ and $h=n-T_2\ge r_1$. The conductors of $\theta_i$
and their stationary depths are
\begin{equation*}
 m_{L_i}(\theta_i)=2n-T_i=2b_i,\qquad
 b_1=2r_2+h-r_1,\quad b_2=r_2+h.
\end{equation*}
For a variable in either $\pp_{L_i}^{b_i}$, $i=1,2$, its trace and norm
belong to the base stationary ideal
$\pp_F^{\ceil{n/2}}$. For the first trace this follows from
$r_2+\floor{(h+r_1)/2}\ge r_2+\ceil{h/2}$; for the second use $r_2\ge1$.
Using \eqref{U:stationary-linearization}, choose $\gamma\in F^\times$ with
$v_F(\gamma/\alpha)=-h$ and
$\lambda(1-z)=\psi_F(\gamma z)$ for $z\in\pp_F^{\ceil{n/2}}$.
By Lemma~\ref{D:NM:freedom}, choose $u\in U_F^{T_2-1}$ and
$X\in L_3^\times$ with $A=\gamma/(\alpha u)=\Nm_{L_3/F}(X)$, so
$v_{L_3}(X)=-h$. Replace $\alpha$ by $\alpha u$, so that
$A=\gamma/\alpha$, and use $\Psi_L(x)=\psi_L(\alpha x)$
with the new scale.
Use \eqref{D:NM:adjust} with this $X$, and put $C=Y-X$.

For $z\in\pp_{L_i}^{b_i}$,
\begin{align*}
 v_{L_1}((Q_1/\zeta_1)z)&\ge(2(r_2-r_1)-h)-4(r_2-r_1)+b_1=r_1,\\
 v_{L_2}((Q_2/\zeta_2)z)&\ge-h+b_2=r_2,
\end{align*}
Thus the pullback calculation \eqref{D:NM:Qphase} in the proof of Proposition~\ref{D:NM:minimal-origin} applies. The discarded terms satisfy
\[
\begin{aligned}
 v_F\!\left((\mathcal E/\zeta_1)\Nm_{L_1/F}(z)\right)
 &\ge(T_2-r_1-h)-2(r_2-r_1)+b_1=T_2,\\
 v_F\!\left((\mathcal E/\zeta_2)\Nm_{L_2/F}(z)\right)
 &\ge(T_2-r_1-h)+b_2=3r_2-r_1\ge T_2.
\end{aligned}
\]
The norm-expansion step following \eqref{D:NM:Qphase} in the proof of
Proposition~\ref{D:NM:minimal-origin} gives
$(\lambda\circ \Nm_{L_i/F})(1-z)=\Psi_{L_i}((A-Q_i)z)$ for
$z\in\pp_{L_i}^{b_i}$. Since $b_i\ge s_i$, formula
\eqref{D:NM:Rcharts} applies and gives
$\theta_i(1-z)=\Psi_{L_i}(\Nm_{K/L_i}(C)\,z)$.
Since $v_K(X)=-2h<-t_1=v_K(Y)$, one has
$v_{L_i}(\Nm_{K/L_i}(C))=-2h=J_i-m_{L_i}(\theta_i)$.

Use $\zeta_i=\Nm_{L_i/F}(h_i)$ as the coefficient for $\omega_i$.
The conductors of $\theta_i,\omega_i$ are even, so all critical-sum factors
are $1$. Lemma~\ref{U:stationary-factor} gives
\begin{align*}
 \mathcal A_{L_i}(\theta_i)
 &=D_\theta(-\alpha^{-1})\Theta(C)^{-1}
   \Psi_F(-\Tr_{L_i/F}(\Nm_{K/L_i}(C))-\zeta_i)\\
 &=D_\theta(-\alpha^{-1})\Theta(C)^{-1}\Psi_F(-e_2(C)+\Delta).
\end{align*}
The last identity follows from Lemma~\ref{U:e2} and
\eqref{D:NM:common-Delta}.
The resulting expression is independent of $i$.
Relabel the third field for the other comparison.
\end{proof}

\begin{proof}[Proof of Theorem~\ref{D:NM:main}]
Lemmas~\ref{D:NM:alignment} and~\ref{D:NM:lower} give $Y$ and
coefficients $\zeta_i$ for $\omega_i$.
Proposition~\ref{D:NM:core-exists} extends the characters $R_i$ to
compatible characters $\chi_i$. By Proposition~\ref{D:NM:realize}, the
given family can be written as $\theta_i=\chi_i(\lambda\circ \Nm_{L_i/F})$.

For $n\le T_2+r_1-1$ use Propositions~\ref{D:NM:minimal-origin} and~\ref{D:NM:minimal-two}.
For $n\ge T_2+r_1$ use Proposition~\ref{D:NM:higher}.
These two conductor intervals are consecutive. The one-break and two-break
ramification possibilities in \eqref{D:NM:breaks} are both included.
\end{proof}

\section{Totally ramified biquadratic extensions in mixed characteristic: maximal~break}\label{sec:dyadic-maximal}
Let $F/\mathbf Q_2$ be finite, with residue field $k$, normalized
valuation $v_F$, and put $e=v_F(2)$. Fix a uniformizer $\varpi$ of $F$.
Suppose $K/F$ is totally ramified with
$\Gal(K/F)=C_2^2$. For its quadratic intermediate fields $L_1,L_2,L_3$,
let $t_i$ denote the lower ramification break of $L_i/F$.
Assume that the fields are labeled so that
\begin{equation*}
 t_1=2r_1-1,\qquad t_2=t_3=2e,\qquad 1\le r_1\le e.
\end{equation*}
Let $\omega_i$ be the nontrivial character in $S(L_i/F)$.
Fix a nontrivial additive character $\psi_F$ of $F$ and put
$\psi_L=\psi_F\circ\Tr_{L/F}$.
\begin{theorem}[The mixed-characteristic case with maximal break]\label{D:MX:main}
For every primitive compatible family $(\theta_1,\theta_2,\theta_3)$,
where $\theta_i$ is a quasi-character of $L_i^\times$, the quantities
\begin{equation*}
 \mathcal A_{L_i}(\theta_i)=
 \Delta_{L_i}(\theta_i,\psi_{L_i})\Delta_F(\omega_i,\psi_F)
\end{equation*}
are equal. 
\end{theorem}
Fix such a family and write $\Theta$ for its common norm pullback to $K$.
We first prove $\mathcal A_{L_1}(\theta_1)=\mathcal A_{L_2}(\theta_2)$.
Interchanging $L_2,L_3$ gives $\mathcal A_{L_1}(\theta_1)=\mathcal A_{L_3}(\theta_3)$.

Put $T_1=2r_1$, $T_2=T_3=2e+1$. The different exponents of $K/L_i$ are
\begin{equation*}
 D'_1=2T_2-T_1,\qquad D'_2=D'_3=T_1.
\end{equation*}
The breaks of $K/L_1$ and $K/L_2,K/L_3$ are $2t_2-t_1$ and
$t_1$, respectively, by Lemma~\ref{D:break-ledger}.
The restriction of $v_K$ to $F$ is $4v_F$.

\subsection{Kummer generators and norm-character coefficients}

\begin{lemma}[Alignment of the Kummer generators]\label{D:MX:align}
There are generators $S\in L_1$ and $R\in L_2$ with
$S^2=1+\epsilon\in F$ and $R^2\in F$ such that
\begin{equation*}
 v_F(\epsilon)=v_F(R^2)=2e-2r_1+1,\qquad
 \epsilon/R^2\in U_F^{r_1}.
\end{equation*}
\end{lemma}
\begin{proof}
Put $q=\varpi^{2e-2r_1+1}$.
There are generators $s$ of $L_1/F$ and $R_0$ of $L_2/F$ with
\[
 s^2=1+q u_0,\qquad R_0^2=q w_0,\qquad u_0,w_0\in\OO_F^\times.
\]
Choose uniformizers $\Pi_i$ of $L_i$ for $i=1,2$.
The monogenic different formula recalled in
Appendix~\ref{app:diamond-foundations} gives
$v_{L_i}(2\Pi_i-\Tr_{L_i/F}(\Pi_i))=T_i$.
Since $v_{L_i}(2\Pi_i)=2e+1$ and nonzero elements of $F$ have even
$L_i$-valuations, it follows that
$v_F(\Tr_{L_1/F}(\Pi_1))=r_1$ and
$v_F(\Tr_{L_2/F}(\Pi_2))\ge e+1$.
For $s=1-2\Pi_1/\Tr_{L_1/F}(\Pi_1)$, the identity
$s^2-1=-4\Nm_{L_1/F}(\Pi_1)/\Tr_{L_1/F}(\Pi_1)^2$ and
$v_F(\Nm_{L_1/F}(\Pi_1))=1$ give the required valuation $2e-2r_1+1$.
Also $v_{L_2}(\Tr_{L_2/F}(\Pi_2)/2)\ge2$, so
$\Pi_2-\Tr_{L_2/F}(\Pi_2)/2$ is a trace-zero uniformizer.
Multiplying it by $\varpi^{e-r_1}$ gives $R_0$.

Take $S=s(1+\eta)$ and $R=R_0v$, where
$\eta=\varpi^{e-r_1+1}w$, $w\in\OO_F$, and $v\in\OO_F^\times$.
Modulo $\pp_F^{r_1}$,
\[
 \frac{S^2-1-R^2}{q}
 \equiv u_0-w_0v^2+\varpi(1+q u_0)w^2.
\]
The omitted term $2\eta/q$ has valuation at least $r_1$.
Apply the coefficient-solving calculation \eqref{D:MX:coefficient-solving}
in the proof of Lemma~\ref{D:NM:alignment}, with
$A=u_0$, $B=w_0$, and $C=1+qu_0$ in the formula for
$(S^2-1-R^2)/q$. This proves the asserted alignment; the correction to
$u_0$ lies in $\pp_F$, so the corrected element is a unit with the same residue as $u_0$.
\end{proof}

Choose $S,R$ by Lemma~\ref{D:MX:align}. Then $L_3=F(SR)$.
Write $g_i$ for the nontrivial automorphism
of $K/L_i$, and define
\begin{equation}\label{D:MX:origin}
 \begin{aligned}
 Y&=(S+R-1)/2,& a_i&=\Nm_{K/L_i}(Y),\\
 h_i&=\Tr_{K/L_i}(Y),& \zeta_i&=\Nm_{L_i/F}(h_i).
 \end{aligned}
\end{equation}
Put $B_0=(\epsilon-R^2)/4$.
\begin{lemma}[Norms and traces of the origin]\label{D:MX:origin-identities}
One has
\begin{align}\label{D:MX:origin-table}
 a_1&=(1-S)/2+B_0,&h_1&=S-1,&\zeta_1&=-\epsilon,\notag\\
 a_2&=-R/2-B_0,&h_2&=R-1,&\zeta_2&=1-R^2,\notag\\
 a_3&=-SR/2-(\epsilon+R^2)/4,&h_3&=-1,&\zeta_3&=1.
\end{align}
For $i=1,2,3$, these elements satisfy
\[
\begin{gathered}
 v_K(Y)=v_{L_i}(a_i)=-t_1,\qquad v_F(\zeta_i)=T_2-T_i,\\
 v_{L_1}(h_1)=T_2-T_1,\qquad v_{L_2}(h_2)=v_{L_3}(h_3)=0,\\
 \Tr_{K/F}(Y)=-2.
\end{gathered}
\]
\end{lemma}
\begin{proof}
The quadratic conjugations give \eqref{D:MX:origin-table}.
The alignment in Lemma~\ref{D:MX:align} gives
$v_F(B_0),v_F((\epsilon+R^2)/4)\ge1-r_1$; for the plus sign use also
$v_F(2R^2)\ge2e-r_1+1$.
For instance, $v_{L_2}(-R/2)=1-2r_1<2-2r_1\le v_{L_2}(B_0)$,
so $v_{L_2}(a_2)=-t_1$.
Since $v_F(\epsilon)<2e$, the identity $(S-1)(S+1)=\epsilon$ gives
$v_{L_1}(S-1)=v_{L_1}(S+1)=2e-2r_1+1$.
The other norm valuations follow from \eqref{D:MX:origin-table}.
Total ramification gives the valuation of $Y$; the trace valuations
and the total trace follow from the same table. Taking the norms of
$h_i$ gives $v_F(\zeta_i)=T_2-T_i$.
\end{proof}

The stationary linearization \eqref{U:stationary-linearization},
applied to $\omega_3$ at depth $e+1$, gives $\alpha\in F^\times$
such that
\begin{equation*}
 \omega_3(1-z)=\Psi_F(z)\quad(z\in\pp_F^{e+1}),\qquad
 \Psi_L(x)=\psi_L(\alpha x),\quad v_F(\alpha)=-n_F(\psi_F)-T_2.
\end{equation*}
The largest trivial ideals of $\Psi_F,\Psi_{L_1},\Psi_{L_2},\Psi_{L_3}$
have respective exponents
$T_2$, $J_1=2T_2-T_1$, $J_2=J_3=T_2$.

\begin{lemma}[Simultaneous coefficients for the norm characters]\label{D:MX:lower}
For $i=1,2,3$, one has
\begin{equation}\label{D:MX:lowerformula}
 \omega_i(1-z)=\Psi_F(\zeta_i z)\qquad(z\in\pp_F^{\ceil{T_i/2}}).
\end{equation}
\end{lemma}
\begin{proof}
Write $(A,B)$ for the quadratic norm symbol of Section~\ref{D:quadratic-symbols}.
For $z\in\pp_F^{r_1}$ put $C=1+\epsilon z$. The Steinberg identity and
\eqref{D:MX:symbol-transform} give
\[
 (1+\epsilon,1-z)
 =(1+\epsilon,-\epsilon(1-z))
 =(C,\epsilon(1-z)(1+\epsilon)).
\]
Here $v_F(C-1)\ge 2e-r_1+1$, and $(2e-r_1+1)+r_1=2e+1$.
Lemma~\ref{D:MX:unit-symbol} removes $1-z$ from the second argument and replaces
$\epsilon$ by $R^2$, since $\epsilon/R^2\in U_F^{r_1}$. By symmetry the result
is $(R^2(1+\epsilon),C)=\omega_3(C)$.
Also $2e-r_1+1\ge e+1$, so the defining formula for $\omega_3(1-z)$ gives
$\omega_1(1-z)=\Psi_F(-\epsilon z)$ on the entire ideal.

On $U_F^{e+1}$ the product $\omega_2=\omega_1\omega_3$ therefore has
coefficient $1-\epsilon$. For $z\in\pp_F^{e+1}$,
\[
 v_F((R^2-\epsilon)z)\ge 2e-r_1+1+e+1=3e-r_1+2\ge T_2,
\]
so $\Psi_F((1-\epsilon)z)=\Psi_F((1-R^2)z)$.
This gives \eqref{D:MX:lowerformula}; the third is the defining formula for $\omega_3(1-z)$.
\end{proof}

\subsection{Compatible characters and conductor reduction}
Put
\begin{equation*}
 m_1=2T_2-1,\qquad m_2=m_3=T_1+T_2-1,\qquad
 s_i=\ceil{m_i/2}.
\end{equation*}
Thus $s_1=T_2$, $s_2=s_3=e+r_1$, and $m_i=J_i+t_1$.
The formulas
\begin{equation}\label{D:MX:R}
 R_i(1-z)=\Psi_{L_i}(a_i z),\qquad z\in\pp_{L_i}^{s_i},
\end{equation}
define characters of their indicated unit groups with exact conductor
$m_i$: $2s_i-t_1\ge J_i$, and the induced character on
$U_{L_i}^{m_i-1}/U_{L_i}^{m_i}$ is nontrivial.

\begin{lemma}[Compatibility of the prescribed characters on principal units]\label{D:MX:Rcompat}
The norm pullbacks of the characters $R_i$ are defined and agree on
$U_K^{T_2}$.
\end{lemma}
\begin{proof}
Let $z\in\pp_K^{2e+1}$.
The highest-degree terms in the norm expansions have valuation at least $2e+1$ on the target field.
Equation~\eqref{U:trace-ideal} gives
\[
 v_{L_1}(\Tr_{K/L_1}(z))\ge3e-r_1+1,\qquad
 v_{L_i}(\Tr_{K/L_i}(z))\ge e+r_1\quad(i=2,3),
\]
so these traces lie in $\pp_{L_i}^{s_i}$.

The quantities $w_i=\Tr_{K/L_i}(Yz)/h_i$ have valuation at least $e+1$ on all
three fields: $v_{L_1}(\Tr_{K/L_1}(Yz))\ge3e-2r_1+2$, from which one
subtracts $2e-2r_1+1$, and
$v_{L_i}(\Tr_{K/L_i}(Yz))\ge e+1$ for $i=2,3$.
Thus $\Nm_{L_i/F}(\Tr_{K/L_i}(Y(1-z)))/\zeta_i=\Nm_{L_i/F}(1-w_i)$
belongs to $U_F^{\ceil{T_i/2}}$.
With these domain bounds, the calculation in the proof of
Lemma~\ref{D:NM:compat}, using \eqref{D:MX:lowerformula} in place of
\eqref{D:NM:lowerformula}, gives
$R_i(\Nm_{K/L_i}(1-z))=\Psi_F(-e_2(Y(1-z))+e_2(Y))$,
independent of $i$.
\end{proof}

\begin{proposition}[Extension to compatible characters of minimal conductor]\label{D:MX:cores}
The characters $R_i$ extend to a primitive compatible family of
characters of $L_i^\times$ with conductors $m_i$.
\end{proposition}
\begin{proof}
Let $\sigma$ be the nontrivial automorphism of $L_2/F$ and let
$\nu_2=\omega_1\circ \Nm_{L_2/F}$, the nontrivial character in $S(K/L_2)$ by
Lemmas~\ref{D:quadratic-norms} and~\ref{D:crossed-norm}.
Extend $R_2$ to a character $\chi_2$ of $L_2^\times$ by
Lemma~\ref{C:character-extension}. Its conjugate quotient for $L_2/F$ is determined
by $R_2$ on every element of $L_2^\times$, because
$\sigma(v)/v\in U_{L_2}^{2e}\subseteq U_{L_2}^{s_2}$ by
Lemma~\ref{C:leading-ramification}.
For $v=\Nm_{K/L_2}(Z)$, this ratio is $\Nm_{K/L_2}(g_1(Z)/Z)$.
One has $g_1(Z)/Z\in U_K^{4e-2r_1+1}\subseteq U_K^{2e+1}$, and its
$\Nm_{K/L_1}$-norm is $1$. Lemma~\ref{D:MX:Rcompat} therefore proves that the conjugate
quotient of $\chi_2$ is trivial on $\Nm_{K/L_2}(K^\times)$.

It is nontrivial: on $U_{L_2}^{t_1}$ the first ramification
term in Lemma~\ref{C:leading-ramification} maps onto
$U_{L_2}^{t_1+t_2}/U_{L_2}^{t_1+t_2+1}$,
since $t_1$ is odd. The restriction of $R_2$ to this quotient is
nontrivial. Hence $\chi_2^\sigma/\chi_2=\nu_2$.

Consequently $\chi_2\circ \Nm_{K/L_2}$ is invariant under $\Gal(K/F)$;
Lemma~\ref{C:hilbert90} makes it trivial on $\ker \Nm_{K/L_1}$ and $\ker \Nm_{K/L_3}$. The characters defined by
$\Nm_{K/L_i}(x)\mapsto\chi_2(\Nm_{K/L_2}(x))$ on $\Nm_{K/L_i}(K^\times)$ extend to
$\widetilde\chi_1$ on $L_1^\times$ and $\chi_3$ on $L_3^\times$ by
Lemma~\ref{C:character-extension}. Equation~\eqref{U:unit-norm-surjective} gives
$\Nm_{K/L_3}(U_K^{2e+1})=U_{L_3}^{e+r_1}$, since
$2(e+r_1)-t_1=2e+1$. Lemma~\ref{D:MX:Rcompat} gives $\chi_3|_{U_{L_3}^{s_3}}=R_3$.

For $\Nm_{K/L_1}$, every preimage of $U_{L_1}^{s_1}$ under $\Nm_{K/L_1}$ lies in
$U_K^{s_1}$: the norm has nonzero graded map at every smaller positive
depth, all less than the break $4e-2r_1+1$ of $K/L_1$, and its residue map is
Frobenius. Norm valuation excludes nonunits. Hence
$\widetilde\chi_1$ agrees with $R_1$ on
$U_{L_1}^{s_1}\cap \Nm_{K/L_1}(K^\times)$.
This agreement and Lemma~\ref{C:character-extension} give
$\mu\in S(K/L_1)$ with
\[
 \mu|_{U_{L_1}^{s_1}}
 =\widetilde\chi_1|_{U_{L_1}^{s_1}}/R_1.
\]
Put $\chi_1=\widetilde\chi_1\mu^{-1}$. Then
\[
 \chi_1|_{U_{L_1}^{s_1}}=R_1,\qquad
 \chi_1\circ\Nm_{K/L_1}=\widetilde\chi_1\circ\Nm_{K/L_1}.
\]
By \eqref{D:MX:R}, each extension is trivial on $U_{L_i}^{m_i}$ and
nontrivial on $U_{L_i}^{m_i-1}$, so its conductor is $m_i$.

If the common pullback were $\lambda\circ \Nm_{K/F}$, then
$\chi_2/(\lambda\circ \Nm_{L_2/F})$ would belong to $S(K/L_2)$.
Both factors would be $\Gal(L_2/F)$-invariant, contrary to
$\chi_2^\sigma/\chi_2=\nu_2\ne1$.
\end{proof}

\begin{proposition}[Common twists and their conductors]\label{D:MX:realize}
For every primitive compatible family $(\theta_i)$, after conjugating
$\theta_1$ if necessary, there is a quasi-character $\lambda$ of $F^\times$ such that
\begin{equation}\label{D:MX:twisted}
 \theta_i=\chi_i(\lambda\circ \Nm_{L_i/F})\qquad(i=1,2)
\end{equation}
with $\chi_i$ as in Proposition~\ref{D:MX:cores}. This conjugation changes
neither local constant nor the common norm pullback.
Writing $n=m_F(\lambda)$, one has, for $i=1,2$,
\begin{equation}\label{D:MX:higher-conductors}
 m_{L_i}(\theta_i)=
 \begin{cases}
 m_i,&n\le T_2+r_1-1,\\
 2n-T_i,&n\ge T_2+r_1.
 \end{cases}
\end{equation}
In the second range the conductors of $\theta_1,\theta_2$ are even
and odd, respectively.
\end{proposition}
\begin{proof}
Let $\sigma_i$ be the nontrivial automorphism of $L_i/F$, and put
$\nu_i=\omega_j\circ \Nm_{L_i/F}$ for $j\ne i$. Proposition~\ref{U:conjugacy} gives
$\theta_i^{\sigma_i}/\theta_i=\chi_i^{\sigma_i}/\chi_i=\nu_i$.
The quotient $\theta_2/\chi_2$ is invariant, so
Lemmas~\ref{C:hilbert90} and~\ref{C:character-extension} give
$\theta_2/\chi_2=\lambda\circ \Nm_{L_2/F}$.
The quotient $\theta_1/(\chi_1(\lambda\circ \Nm_{L_1/F}))$ is either $1$ or $\nu_1$.
Conjugate $\theta_1$ in the latter case. Its local constant is unchanged
by \eqref{U:automorphism-invariance};
its norm pullback is unchanged because $\nu_1\circ \Nm_{K/L_1}=1$.
This proves \eqref{D:MX:twisted}.

Finally \eqref{U:norm-pullback-conductor} gives the alternatives for the
norm-pullback conductors:
when $n\le T_2+r_1-1$, $m_{L_i}(\lambda\circ \Nm_{L_i/F})<m_i$; when
$n\ge T_2+r_1$, $m_{L_i}(\lambda\circ \Nm_{L_i/F})>m_i$.
The two conductors cannot be equal: $m_1$ is odd and
$2n-2r_1$ is even, whereas $m_2$ is even and $2n-T_2$ is odd.
\end{proof}

Choose $\lambda$ as in Proposition~\ref{D:MX:realize}, conjugating $\theta_1$
if necessary, and put $n=m_F(\lambda)$.

\subsection{Norm representatives and comparison at minimal conductor}
\begin{lemma}[Normalization for an exact norm representative]\label{D:MX:freedom}
For $u\in U_F^{T_2-1}$, replacing $\alpha$ by $\alpha u$
preserves \eqref{D:MX:lowerformula} and \eqref{D:MX:R}.
For every $\gamma\in F^\times$, one may choose such a $u$ with
$\gamma/(\alpha u)\in\Nm_{L_3/F}(L_3^\times)$.
\end{lemma}
\begin{proof}
For \eqref{D:MX:lowerformula}, the change has valuation at least
$2e+v_F(\zeta_i)+\ceil{T_i/2}\ge T_2$. For \eqref{D:MX:R} the change has
$L_i$-valuation at least $4e-t_1+s_i\ge J_i$.
Thus both formulas are unchanged on their stated ideals.

Since $\omega_3$ is nontrivial on $U_F^{T_2-1}$, the exact-norm
argument in the proof of Lemma~\ref{D:NM:freedom} applies and gives
$\gamma/(\alpha u)\in\Nm_{L_3/F}(L_3^\times)$ for a suitable $u$.
\end{proof}

For the comparison of $L_1$ and $L_2$, take $X\in L_3$ and put
\[
 A=\Nm_{L_3/F}(X),\qquad p_X=\Tr_{L_3/F}(X),\quad
 Q_i=\Tr_{K/L_i}(Yg_i(X))\ (i=1,2).
\]
\begin{lemma}[Correction identities and estimates]\label{D:MX:error-estimates}
With primes denoting conjugation in $L_3/F$, one has, for $i=1,2$,
\begin{align}\label{D:MX:adjust-identities}
 \Nm_{K/L_i}(Y-X)&=a_i+A-Q_i,&\Tr_{K/L_i}(Y-X)&=h_i-p_X,\notag\\
 \Nm_{L_i/F}(Q_i)&=\zeta_iA+\mathcal E,&
 \mathcal E&=(X'-X)(a_3X'-a_3'X)\in F,\\
 \Nm_{L_i/F}(\Tr_{K/L_i}(Y-X))&=\zeta_i+\Delta,&
 \Delta&=p_X^2+2p_X.\label{D:MX:Delta}
\end{align}
If $X\ne0$ and $v_{L_3}(X)=-h$, then
\begin{align}
 v_F(p_X)&\ge\floor{(T_2-h)/2},&
 v_F(\mathcal E)&\ge T_2-r_1-h,\notag\\
 v_{L_1}(Q_1)&\ge T_2-T_1-h,&
 v_{L_2}(Q_2)&\ge-h.\label{D:MX:error-bounds}
\end{align}
\end{lemma}
\begin{proof}
The identities follow by multiplying the four conjugates, with $X$
fixed by $g_3$; the last uses $\Tr_{L_i/F}(h_i)=\Tr_{K/F}(Y)=-2$.
Let $\Pi_3=SR/\varpi^{e-r_1}$, a trace-zero uniformizer of $L_3$.
Write $X=P+Q\Pi_3$; parity of valuations gives
$v_F(P)\ge-\floor{h/2}$ and $v_F(Q)\ge-\ceil{h/2}$.
The trace $2P$ has valuation at least
$e-\floor{h/2}=\floor{(T_2-h)/2}$.
Write $a_3=A_3+B_3\Pi_3$, where $v_F(A_3)\ge1-r_1$ and
$v_F(B_3)=-r_1$ by \eqref{D:MX:origin-table}.
Now $X'-X=-2Q\Pi_3$ and
$a_3X'-a_3'X=2(B_3P-A_3Q)\Pi_3$.
Their product has $F$-valuation at least $2e+1-r_1-h=T_2-r_1-h$.

Finally $Yg_i(X)$ has $K$-valuation $-t_1-2h$.
The trace-ideal formula \eqref{U:trace-ideal} gives the bounds
$2e-2r_1+1-h=T_2-T_1-h$ and $-h$ for $Q_1,Q_2$, respectively.
\end{proof}

\begin{proposition}[Stationary coefficients at minimal conductor]\label{D:MX:minimal-origin}
Assume $n\le T_2+r_1-1$. After a normalization allowed by Lemma~\ref{D:MX:freedom}, there is
$C\in Y+\pp_{L_3}^{1-r_1}$ with $v_K(C)=-t_1$ such that, for $i=1,2$,
\begin{align}\label{D:MX:minimal-stationary}
 \theta_i(1-z)&=\Psi_{L_i}(\Nm_{K/L_i}(C)\,z)
                                &&(z\in\pp_{L_i}^{s_i}),\notag\\
 \omega_i(1-z)&=\Psi_F(\Nm_{L_i/F}(\Tr_{K/L_i}(C))\,z)
                                &&(z\in\pp_F^{\ceil{T_i/2}}).
\end{align}
For $i=1,2$, the traces $\Tr_{K/L_i}(C)$ are nonzero and satisfy
$v_{L_i}(\Tr_{K/L_i}(C))=v_{L_i}(h_i)$.
\end{proposition}
\begin{proof}
Both trace and norm of a variable in
$\pp_{L_i}^{s_i}$ lie in $\pp_F^{e+r_1}$, for $i=1,2$.
If $n\le e+r_1$, the twist is trivial there, and use $X=0$.
Otherwise put $h=n-T_2$, so
$r_1-e\le h\le r_1-1$.
Since $\lceil n/2\rceil\le e+r_1<n$,
\eqref{U:stationary-linearization} gives a coefficient $\gamma\in F^\times$ with
$v_F(\gamma/\alpha)=-h$ and
$\lambda(1-z)=\psi_F(\gamma z)$ for $z\in\pp_F^{e+r_1}$.
By Lemma~\ref{D:MX:freedom}, choose $u\in U_F^{T_2-1}$ and
$X\in L_3^\times$ with $A=\gamma/(\alpha u)=\Nm_{L_3/F}(X)$, so
$v_{L_3}(X)=-h$. Replace $\alpha$ by $\alpha u$, so that
$A=\gamma/\alpha$, and use $\Psi_L(x)=\psi_L(\alpha x)$
with the new scale henceforth.

Put $C=Y-X$. When $X=0$, the assertions follow from \eqref{D:MX:R},
\eqref{D:MX:lowerformula}, and Lemma~\ref{D:MX:origin-identities}.
Otherwise $v_K(X)=-2h>-t_1=v_K(Y)$, so $v_K(C)=v_K(Y)=-t_1$.
For $z\in\pp_{L_i}^{s_i}$, Lemma~\ref{D:MX:error-estimates} gives
\[
 v_{L_1}((Q_1/\zeta_1)z)\ge2r_1-h\ge r_1+1,\qquad
 v_{L_2}((Q_2/\zeta_2)z)\ge e+r_1-h\ge e+1.
\]
Under these bounds, equation~\eqref{D:UR:phase}, applied to $L_i/F$ and
$\omega_i$ with additive scale $\alpha\zeta_i$ using
\eqref{D:MX:lowerformula}, together with \eqref{D:MX:adjust-identities}, gives
\begin{equation}\label{D:MX:Qphase}
 \Psi_{L_i}(Q_i z)
   =\Psi_F\bigl((A+\mathcal E/\zeta_i)\Nm_{L_i/F}(z)\bigr).
\end{equation}
The discarded terms satisfy
\[
\begin{aligned}
 v_F\!\left((\mathcal E/\zeta_1)\Nm_{L_1/F}(z)\right)&\ge T_2+r_1-h\ge T_2,\\
 v_F\!\left((\mathcal E/\zeta_2)\Nm_{L_2/F}(z)\right)&\ge T_2+e-h\ge T_2.
\end{aligned}
\]
Since $\lambda(1-x)=\Psi_F(Ax)$ on $\pp_F^{e+r_1}$ and both
$\Tr_{L_i/F}(z)$ and $\Nm_{L_i/F}(z)$ lie in this ideal, the norm expansion
and \eqref{D:MX:Qphase} give
$(\lambda\circ \Nm_{L_i/F})(1-z)=\Psi_{L_i}((A-Q_i)z)$.
Multiplication by \eqref{D:MX:R} and \eqref{D:MX:adjust-identities}
proves the first stationary formula.

For the second, \eqref{D:MX:error-bounds} and $h\le r_1-1$ give
$v_F(\Delta)\ge\min\{2e-2\floor{h/2},\ 2e-\floor{h/2}\}$.
Adding $r_1$ to either bound gives at least $T_2$; adding $e+1$ does also.
Thus replacing $\zeta_i$ by $\Nm_{L_i/F}(\Tr_{K/L_i}(C))$ preserves \eqref{D:MX:lowerformula}.
Moreover $v_{L_1}(p_X)\ge2e-2\floor{h/2}>2e-2r_1+1$, while $p_X$ has positive
valuation on $L_2$. Since $\Tr_{K/L_i}(C)=h_i-p_X$, it follows that
$v_{L_i}(\Tr_{K/L_i}(C))=v_{L_i}(h_i)$ for $i=1,2$.
\end{proof}

\begin{proposition}[Minimal-conductor comparison]\label{D:MX:minimal}
Let $(\theta_i)$ be a primitive compatible family whose first two characters
are written as in \eqref{D:MX:twisted}. If $m_F(\lambda)\le T_2+r_1-1$, then
$\mathcal A_{L_1}(\theta_1)=\mathcal A_{L_2}(\theta_2)$.
\end{proposition}
\begin{proof}
Choose $C$ as in Proposition~\ref{D:MX:minimal-origin}, using the
normalization allowed there. Denote the normalized sums in Lemma~\ref{U:stationary-factor} by
$g_{\theta_i},g_{\omega_i}$.
The character $D_\theta=\theta_i|_{F^\times}\omega_i$ is independent
of $i$ by Proposition~\ref{U:conjugacy}.
Apply Lemma~\ref{U:common-origin} with the coefficients and nonzero
traces supplied by Proposition~\ref{D:MX:minimal-origin}. It gives
\begin{equation*}
 \mathcal A_{L_i}(\theta_i)=D_\theta(-\alpha^{-1})\Theta(C)^{-1}
                    \Psi_F(-e_2(C))g_{\theta_i}g_{\omega_i}.
\end{equation*}
Only $\theta_1$ and $\omega_2$ have odd conductors, with critical depths
$2e$ and $e$ respectively.
Choose a uniformizer $\pi$ of $K$ and put $\Pi_1=\Nm_{K/L_1}(\pi)$.
Use the critical elements $\delta_{\theta_1}=\Pi_1^{2e}$ and
$\delta_{\omega_2}=\varpi^e$. Keep $C$ fixed and introduce
$C_z=C(1+\pi^{2e}z)$ for $z\in\OO_K$. The norm expansion over $L_1$ has
$v_{L_1}(\Nm_{K/L_1}(\pi^{2e}z))\ge2e$ and
$v_{L_1}(\Tr_{K/L_1}(\pi^{2e}z))\ge3e-r_1+1\ge2e+1$.
The induced residue map is a nonzero scalar times Frobenius, hence bijective.
Over $L_2$, one has
\[
 \frac{\Nm_{K/L_2}(C_z)}{\Nm_{K/L_2}(C)}-1
 =\Tr_{K/L_2}(\pi^{2e}z)+\Nm_{K/L_2}(\pi^{2e}z)
 \in\pp_{L_2}^{e+r_1},
\]
since the trace and norm terms have $L_2$-valuations at least $e+r_1$
and $2e\ge e+r_1$, respectively.

The trace change $\Tr_{K/L_1}(C\pi^{2e}z)$ has $L_1$-valuation at least
$3e-2r_1+1$; after division by $\Tr_{K/L_1}(C)$ its $L_1$-valuation is
at least $e\ge r_1$.
The quotient of the $L_1/F$-norms therefore lies in $U_F^{r_1}$,
where \eqref{D:MX:lowerformula} applies to $\omega_1$.
The consecutive trace ideals for $K/L_2$ give an isomorphism
\[
 \Tr_{K/L_2}:\pp_K^{2e-2r_1+1}/\pp_K^{2e-2r_1+2}
 \xrightarrow{\sim}\pp_{L_2}^{e}/\pp_{L_2}^{e+1}.
\]
Division by the unit $\Tr_{K/L_2}(C)$ preserves this isomorphism.
For $y\in\pp_{L_2}^e$, one has $v_F(\Nm_{L_2/F}(y))\ge e$ and
$v_F(\Tr_{L_2/F}(y))\ge\floor{(3e+1)/2}\ge e+1$.
Thus the induced residue norm map is again bijective.

All character arguments are nonzero.
At these increments, the first term for $i=2$ (involving $\theta_2$)
and the second term for $i=1$ (involving $\omega_1$) in
Lemma~\ref{U:common-function} are each $1$ by
\eqref{D:MX:minimal-stationary}.
Apply Lemma~\ref{U:common-function} to $C,C_z$.
Equation~\eqref{U:normalized-residual}, in the chosen critical coordinates, gives
\[
 H_{\theta_1}(\rho_{\theta_1}(z))=H_{\omega_2}(\rho_{\omega_2}(z))\qquad(z\in k),
\]
Here $\rho_{\theta_1},\rho_{\omega_2}:k\to k$ are induced,
in the chosen critical coordinates, by
\[
 \frac{\Nm_{K/L_1}(C_z)}{\Nm_{K/L_1}(C)}-1,
 \qquad
 \frac{\Nm_{L_2/F}(\Tr_{K/L_2}(C_z))}
      {\Nm_{L_2/F}(\Tr_{K/L_2}(C))}-1,
\]
respectively. Both maps are bijective by the preceding calculations.
Summing gives $g_{\theta_1}=g_{\omega_2}$. The factorization of $\mathcal A_{L_i}(\theta_i)$ above
proves $\mathcal A_{L_1}(\theta_1)=\mathcal A_{L_2}(\theta_2)$.
\end{proof}

\subsection{Higher-conductor comparison and completion}
Now $n\ge T_2+r_1$, and put $h=n-T_2\ge r_1$.
\begin{proposition}[Stationary coefficients at higher conductor]\label{D:MX:higher-coeff}
After a normalization allowed by Lemma~\ref{D:MX:freedom}, there is
$X\in L_3$ with $v_{L_3}(X)=-h$ such that $C=Y-X$ satisfies
\[
\begin{aligned}
 \theta_1(1-z)&=\Psi_{L_1}(\Nm_{K/L_1}(C)z)
     &&(z\in\pp_{L_1}^{n-r_1}),\\
 \theta_2(1-z)&=\Psi_{L_2}(\Nm_{K/L_2}(C)z)
     &&(z\in\pp_{L_2}^{n-e}).
\end{aligned}
\]
Moreover $v_{L_i}(\Nm_{K/L_i}(C))=-2h$.
\end{proposition}
\begin{proof}
Put $b_1=n-r_1$ and $b_2=n-e$.
Using \eqref{U:stationary-linearization}, choose $\gamma\in F^\times$ such that
$\lambda(1-z)=\psi_F(\gamma z)$ for $z\in\pp_F^{\ceil{n/2}}$.
The coefficient satisfies $v_F(\gamma/\alpha)=-h$. Both trace and norm of
$z\in\pp_{L_i}^{b_i}$ lie in that base ideal by \eqref{U:trace-ideal}
and total ramification.
By Lemma~\ref{D:MX:freedom}, choose $u\in U_F^{T_2-1}$ and
$X\in L_3^\times$ with $A=\gamma/(\alpha u)=\Nm_{L_3/F}(X)$, so
$v_{L_3}(X)=-h$. Replace $\alpha$ by $\alpha u$, so that
$A=\gamma/\alpha$, and use $\Psi_L(x)=\psi_L(\alpha x)$ with the new scale.

For $z\in\pp_{L_i}^{b_i}$,
\[
 v_{L_1}((Q_1/\zeta_1)z)\ge r_1,\qquad
 v_{L_2}((Q_2/\zeta_2)z)\ge e+1,
\]
so the pullback calculation \eqref{D:MX:Qphase} in the proof of
Proposition~\ref{D:MX:minimal-origin} applies. The discarded terms satisfy
\[
\begin{aligned}
 v_F\!\left((\mathcal E/\zeta_1)\Nm_{L_1/F}(z)\right)
 &\ge(T_2-r_1-h)-(2e-2r_1+1)+b_1=T_2,\\
 v_F\!\left((\mathcal E/\zeta_2)\Nm_{L_2/F}(z)\right)
 &\ge(T_2-r_1-h)+b_2=3e+2-r_1\ge T_2.
\end{aligned}
\]
The norm-expansion step following \eqref{D:MX:Qphase} in the proof of
Proposition~\ref{D:MX:minimal-origin} therefore gives the pulled-back
coefficient $A-Q_i$ on these ideals. They are
contained in $\pp_{L_i}^{s_i}$, so \eqref{D:MX:R} adds the coefficient $a_i$, giving $\Nm_{K/L_i}(C)$.
Since $v_K(X)=-2h<-t_1=v_K(Y)$, one has
$v_K(C)=-2h$ and $v_{L_i}(\Nm_{K/L_i}(C))=-2h$, as required by \eqref{D:MX:higher-conductors}.
\end{proof}

\begin{proposition}[Higher-conductor comparison]\label{D:MX:higher}
Let $(\theta_i)$ be a primitive compatible family whose first two characters
are written as in \eqref{D:MX:twisted}. If $m_F(\lambda)\ge T_2+r_1$, then
$\mathcal A_{L_1}(\theta_1)=\mathcal A_{L_2}(\theta_2)$.
\end{proposition}
\begin{proof}
Put $b_1=n-r_1$ and $b_2=n-e$.
Choose $X,C$ as in Proposition~\ref{D:MX:higher-coeff}, using the
normalization allowed there and the associated $p_X,\Delta$ in
\eqref{D:MX:adjust-identities}--\eqref{D:MX:Delta}.
Write $D_\theta=\theta_i|_{F^\times}\omega_i$ as in Proposition~\ref{U:conjugacy}.
Denote the normalized sums by
$g_{\theta_i},g_{\omega_i}$. Use the coefficients $\Nm_{K/L_i}(C)$ from Proposition~\ref{D:MX:higher-coeff} for
$\theta_i$, and $\zeta_i=\Nm_{L_i/F}(h_i)$ from Lemma~\ref{D:MX:lower} for
$\omega_i$. Lemma~\ref{U:e2}, together with
\eqref{D:MX:Delta} and \eqref{U:normalized-factor}, gives
\begin{equation}\label{D:MX:higher-full}
 \mathcal A_{L_i}(\theta_i)=D_\theta(-\alpha^{-1})\Theta(C)^{-1}
              \Psi_F(-e_2(C)+\Delta)g_{\theta_i}g_{\omega_i}.
\end{equation}
The factor preceding $g_{\theta_i}g_{\omega_i}$ is independent of $i$.
The conductors of $\theta_1,\omega_1$ are even, so
$g_{\theta_1}=g_{\omega_1}=1$. The conductors of $\theta_2,\omega_2$ are odd.

Choose a uniformizer $\pi$ of $K$, and put $\Pi_2=\Nm_{K/L_2}(\pi)$.
Use the critical elements $\delta_{\theta_2}=\Pi_2^{e+h}$ and
$\delta_{\omega_2}=\varpi^e$. Keep $C,Y$ fixed and introduce
\[
 V=Y\pi^{2e}z,\qquad C_z=C+V,\qquad Y_z=Y+V\quad(z\in\OO_K).
\]
We evaluate $\theta_i$ at $\Nm_{K/L_i}(C_z)/\Nm_{K/L_i}(C)$ and $\omega_i$ at
$\Nm_{L_i/F}(\Tr_{K/L_i}(Y_z))/\zeta_i$. The coefficients for these two characters
are $\Nm_{K/L_i}(C)$ and $\zeta_i=\Nm_{L_i/F}(h_i)$, respectively.
For $i=1$, denote the two arguments by $1+u$ and $1+v$,
respectively. Since $v_K(V)\ge2e-2r_1+1$ and $v_K(C)=-2h$, one has
\begin{align*}
 v_K(V/C)&\ge2e-2r_1+1+2h,\\
 v_{L_1}(\Tr_{K/L_1}(V/C))&\ge3e-2r_1+1+h\ge b_1,\\
 v_{L_1}(\Nm_{K/L_1}(V/C))&\ge2e-2r_1+1+2h\ge b_1.
\end{align*}
Thus $u\in\pp_{L_1}^{b_1}$, where the stationary formula in
Proposition~\ref{D:MX:higher-coeff} applies.

For $\theta_2$, identify $z$ with the class of $V$ in
$\pp_K^{2e-2r_1+1}/\pp_K^{2e-2r_1+2}$. Write $\mathrm{gr}\,\Nm_{K/L_2}$ for the map induced by
$x\mapsto \Nm_{K/L_2}(1+x)-1$, and $\mathrm{gr}\,\Nm_{L_2/F}$ for the map induced by
$x\mapsto \Nm_{L_2/F}(1+x)-1$.
The residue map $\rho_{\theta_2}$ is the composition
\begin{equation*}
 \frac{\pp_K^{2e-2r_1+1}}{\pp_K^{2e-2r_1+2}}
 \xrightarrow{\,/C\,}
 \frac{\pp_K^{2e-2r_1+1+2h}}{\pp_K^{2e-2r_1+2h+2}}
 \xrightarrow{\,\mathrm{gr}\,\Nm_{K/L_2}\,}
 \frac{\pp_{L_2}^{e+h}}{\pp_{L_2}^{e+h+1}}.
\end{equation*}
The first map is an isomorphism because $v_K(C)=-2h$.
For the second, the break of $K/L_2$ is $2r_1-1$ and
\[
 e+h>2r_1-1,\qquad 2e-2r_1+1+2h=2(e+h)-(2r_1-1).
\]
Thus \eqref{U:graded-norm} proves that this map, and hence $\rho_{\theta_2}$,
is bijective.

For $\omega_1$, $v_{L_1}(\Tr_{K/L_1}(V)/h_1)\ge
e\ge r_1$, so $v\in\pp_F^{r_1}$. The stationary formulas in
Proposition~\ref{D:MX:higher-coeff} and \eqref{D:MX:lowerformula} give
\begin{equation*}
 \Psi_{L_1}(-\Nm_{K/L_1}(C)u)\theta_1(1+u)^{-1}
 =\Psi_F(-\zeta_1v)\omega_1(1+v)^{-1}=1.
\end{equation*}
For $i=2$, the residue map $\rho_{\omega_2}$ is the composition
\begin{equation*}
 \frac{\pp_K^{2e-2r_1+1}}{\pp_K^{2e-2r_1+2}}
 \xrightarrow{\,\Tr_{K/L_2}\,}
 \frac{\pp_{L_2}^{e}}{\pp_{L_2}^{e+1}}
 \xrightarrow{\,/h_2\,}
 \frac{\pp_{L_2}^{e}}{\pp_{L_2}^{e+1}}
 \xrightarrow{\,\mathrm{gr}\,\Nm_{L_2/F}\,}
 \frac{\pp_F^e}{\pp_F^{e+1}}.
\end{equation*}
By \eqref{U:trace-ideal}, the trace ideals at source depths $2e-2r_1+1,2e-2r_1+2$ have
depths $e,e+1$. The induced trace map is therefore a surjection between
one-dimensional $k$-spaces, hence an isomorphism. Division by $h_2$ is an
isomorphism because $h_2$ is a unit. The last map is bijective by the
graded norm map in \eqref{U:graded-norm}, since
$L_2/F$ has break $2e$ and $e<2e$.
Thus $\rho_{\omega_2}$ is bijective.
The traces $\Tr_{K/L_i}(Y_z)$ remain nonzero.

Since $\Tr_{K/L_i}(C)=h_i-p_X$, quadratic norm expansion gives
\[
 [\Nm_{L_i/F}(\Tr_{K/L_i}(C_z))-\Nm_{L_i/F}(\Tr_{K/L_i}(C))]-[\Nm_{L_i/F}(\Tr_{K/L_i}(Y_z))-\zeta_i]=-p_X\Tr_{K/F}(V).
\]
Using Lemma~\ref{U:e2}, we obtain
\begin{equation*}
 \Tr_{L_i/F}(\Nm_{K/L_i}(C_z)-\Nm_{K/L_i}(C))+(\Nm_{L_i/F}(\Tr_{K/L_i}(Y_z))-\zeta_i)
  =e_2(C_z)-e_2(C)+p_X\Tr_{K/F}(V),
\end{equation*}
The right side is independent of $i$. Compatibility gives
\[
 \theta_i\!\left(\frac{\Nm_{K/L_i}(C_z)}{\Nm_{K/L_i}(C)}\right)
 =\Theta(C_z/C),
\]
and $\omega_i(\Nm_{L_i/F}(\Tr_{K/L_i}(Y_z))/\zeta_i)=1$.
Together with the stationary products of value $1$ above and \eqref{U:normalized-residual}, these identities give
\[
 H_{\theta_2}(\rho_{\theta_2}(z))H_{\omega_2}(\rho_{\omega_2}(z))=1\qquad(z\in k).
\]
Both residue maps are additive bijections. Therefore
$\sum H_{\theta_2}=\overline{\sum H_{\omega_2}}$, after reindexing.
Each normalized sum has absolute value $1$ by
\eqref{U:critical-sum-size}. Consequently
$g_{\theta_2}g_{\omega_2}=1$.
Equation~\eqref{D:MX:higher-full} proves the assertion for every $h\ge r_1$.
\end{proof}

\begin{proof}[Proof of Theorem~\ref{D:MX:main}]
Construct $Y$ by \eqref{D:MX:origin} and apply
Lemma~\ref{D:MX:lower} and Proposition~\ref{D:MX:cores} to obtain
$\zeta_i$ and $\chi_i$. Proposition~\ref{D:MX:realize} writes the given
characters as in \eqref{D:MX:twisted}, after conjugating $\theta_1$
if necessary. This does not change its norm pullback or local constant.
Every base conductor is either $n\le T_2+r_1-1$, covered by
Proposition~\ref{D:MX:minimal}, or $n\ge T_2+r_1$, covered by
Proposition~\ref{D:MX:higher}. Their endpoints are consecutive.
Repeat the construction with the two maximal fields interchanged to
compare $L_1$ with $L_3$. This proves independence
of all three expressions for the original given family.
\end{proof}

\section{Completion of the Second Main Lemma}\label{sec:completion}
\label{U:exhaustion}

\begin{lemma}[Ramification cases]\label{U:dispatch}
Every Galois extension $K/F$ with group $C_\ell^2$ satisfies the
hypotheses of one of Theorems~\ref{O:F:tame}, \ref{O:I:main},
\ref{O:G:totalbranch}, \ref{D:UR:URall}, \ref{D:EQ:main},
\ref{D:NM:main}, or~\ref{D:MX:main}.
\end{lemma}
\begin{proof}
By Lemma~\ref{D:inertia} the finite residue Galois group is cyclic;
a $C_\ell^2$ extension cannot be unramified. Its inertia has order $\ell$ or $\ell^2$.
If $\ell\ne p$, wild inertia is trivial and tame inertia is cyclic,
so inertia has order $\ell$. Abelianness and the residue-Frobenius
conjugation action imply $\ell\mid |k_F|-1$. There is a unique unramified
degree-$\ell$ intermediate field, and every other degree-$\ell$
intermediate field is tamely ramified. Theorem~\ref{O:F:tame} applies.

Suppose $\ell=p$. If inertia has order $p$, there is one unramified degree-$p$
intermediate field and every other degree-$p$ intermediate field is ramified.
Theorem~\ref{O:I:main} treats odd $p$, and Theorem~\ref{D:UR:URall}
treats $p=2$.

If inertia has order $p^2$, the extension is totally ramified.
For odd $p$, Lemma~\ref{D:break-ledger} selects a degree-$p$ field
$L_1$ such that every other degree-$p$ field $L_2$ has breaks
$t_1\le t_2$ over $F$, while $K/L_1$ and $K/L_2$ have breaks
$t_1+p(t_2-t_1)$ and $t_1$, respectively. Proposition~\ref{O:A:prop:AS} gives $p\nmid t_1$.
These are the hypotheses of
the totally ramified proof in Sections~\ref{sec:odd-estimates}--\ref{sec:odd-calculation}.

It remains to classify totally ramified biquadratic extensions.
In characteristic two, the reduced Artin--Schreier representatives
of Section~\ref{sec:dyadic-equal} have positive odd breaks and the two largest equal, including
the one-break case. Theorem~\ref{D:EQ:main} treats all of them.

In mixed characteristic, put $e=v_F(2)$ and write $D$ for the
different exponent of a ramified quadratic extension. By \eqref{U:trace-ideal}, $2\le D\le2e+1$.
Lemma~\ref{D:UR:oddT} shows that an odd $D$ equals $2e+1$ and that
the corresponding norm character on $U_F^{2e}/U_F^{2e+1}$, in
coordinates $1+4z$, is
$(-1)^{\Tr_{k_F/\F_2}(\bar z)}$. Thus a ramified quadratic norm
character has conductor $2a$, with $1\le a\le e$, or $2e+1$;
the product of two characters of conductor $2e+1$ is trivial on
$U_F^{2e}$.

In mixed characteristic the three nontrivial norm characters of $L_i/F$ obey
$\omega_3=\omega_1\omega_2$ by Lemma~\ref{D:quadratic-norms},
and all are ramified. The product of two
characters of unequal conductors has the larger conductor, so the maximum
of their three conductors occurs at least twice. If this maximum is at most $2e$,
their conductors are $(2a,2r,2r)$ after relabeling, with
$1\le a\le r\le e$, giving the breaks $(2a-1,2r-1,2r-1)$ of Section~\ref{sec:dyadic-nonmaximal}.
If the maximum is $2e+1$, the preceding calculation makes the third conductor smaller,
and the lower breaks are $(2a-1,2e,2e)$ with $1\le a\le e$, precisely Section~\ref{sec:dyadic-maximal}.
The formulas for the different exponents of $K/L_i$ follow from
Lemma~\ref{D:break-ledger}.
\end{proof}

\begin{proof}[Proof of Theorem~\ref{U:main}]
Apply the theorem specified by Lemma~\ref{U:dispatch}. It compares
$\mathcal A_L(\theta_L)$ for a chosen degree-$\ell$ intermediate field
with $\mathcal A_{L^{\prime}}(\theta_{L^{\prime}})$ for every
other degree-$\ell$ intermediate field $L^{\prime}$.
Equality for any pair follows by transitivity.

Theorem~\ref{D:EQ:main} already treats an arbitrary nontrivial
additive character; all other cases use $\psi_F$ as given.
This proves the assertion for the stated family.
\end{proof}
\appendix
\addtocontents{toc}{\protect\setcounter{tocdepth}{1}}
\section{Gauss sums and polynomial identities}\label{app:leading-gauss}
We prove the leading Gauss congruence and its base-$q$-digit form
used in Section~\ref{sec:tame}; see also~\cite{Stickelberger}.
The proof uses character orthogonality, Jacobi sums, and induction
on the sum of the base-$p$ digits.

\subsection{The coefficient field and notation}\label{A:coefficient-field}
Let $k=\F_q$, $q=p^f$, and choose an unramified extension $K_0/\Q_p$
with residue field $k$. To construct it, lift a monic irreducible polynomial of degree $f$
over $\F_p$ to a monic polynomial over $\mathbf Z_p$, and take its
quotient ring $R$. It is finite free and $p$-adically complete over
$\mathbf Z_p$, with $R/pR=k$. Every element with nonzero residue is a
unit: lift an inverse modulo $p$ and invert the remaining element of
$1+pR$ by its convergent geometric series. Every nonzero element is
$p^n$ times such a unit, since $R$ is finite free and separated over
$\mathbf Z_p$. Thus $R$ is a discrete valuation ring with uniformizer
$p$, and its fraction field is the required $K_0$.

Hensel lifting of $X^{q-1}-1$ gives a unique multiplicative lift
$x\mapsto[x]$ from $k^\times$ into $R^\times$: its derivative is a
unit at every nonzero residue root. The lifts multiply because their
products are again roots with the prescribed residue. Choose a primitive
$p$th root $\zeta$ and put
\[
 \varpi=\zeta-1,\qquad K_1=K_0(\zeta).
\]
For $p>2$, the polynomial
\[
 \frac{(1+X)^p-1}{X}=p+\binom p2X+\cdots+X^{p-1}
\]
is Eisenstein; hence $\varpi$ is a uniformizer of $K_1$ and its residue
field is $k$. For $p=2$, $K_1=K_0$ and $\varpi=-2$ is already a
uniformizer. All root-of-unity values below are viewed in $K_1$; the
identities are algebraic identities among these values.

Let $\psi(x)=\zeta^{\Tr_{k/\F_p}(x)}$. For $0<a<q-1$, write
\[
 a=\sum_{i=0}^{f-1}a_ip^i,\qquad
 s_p(a)=\sum_i a_i.
\]
With $\omega(x)=[x]$, use the convention of Section~\ref{sec:tame}:
\[
 \tau_k(\omega^a)=-\sum_{x\in k^\times}[x]^{-a}\psi(x).
\]

\subsection{A first-order calculation and a Jacobi sum}
\begin{lemma}[First-order Gauss congruence]\label{A:gauss-first}
If $q>2$, then
$\tau_k(\omega)\equiv\varpi\pmod{\varpi^2\OO_{K_1}}$.
Moreover $\tau_k(\omega^{p^j})=\tau_k(\omega)$ whenever
$0\le j<f$ and $p^j<q-1$.
\end{lemma}
\begin{proof}
For each $x$, let $t_x\in\{0,\ldots,p-1\}$ represent
$\Tr_{k/\F_p}(x)$. Then
$\psi(x)=(1+\varpi)^{t_x}\equiv1+t_x\varpi\pmod{\varpi^2}$.
The identity $\sum_{x\ne0}[x]^{-1}=0$ follows by summing a
nontrivial character of the cyclic group $k^\times$. Therefore
\[
 \tau_k(\omega)\equiv-\varpi\sum_{x\ne0}[x]^{-1}t_x\pmod{\varpi^2}.
\]
The residue of the last sum is
\[
 \sum_{x\ne0}x^{-1}\Tr_{k/\F_p}(x)
   =\sum_{i=0}^{f-1}\sum_{x\ne0}x^{p^i-1}=-1\quad\text{in }k.
\]
The $i=0$ inner sum is $-1$; all the others are zero.

Finally, Frobenius permutes
$k^\times$ and preserves its absolute trace. Reindexing by the inverse
Frobenius shows that $\tau_k(\omega^{pa})=\tau_k(\omega^a)$, with indices reduced modulo $q-1$.
Iteration gives the last assertion.
\end{proof}

\begin{lemma}[Jacobi sums when the digits add without carries]
\label{A:jacobi-unit}
Suppose $a,b>0$ and $a+b<q-1$. Let $a_i,b_i$ be their base-$p$
digits, and assume $a_i+b_i<p$ for every $i$. Put
\[
 J(a,b)=\sum_{x\in k\setminus\{0,1\}}[x]^{-a}[1-x]^{-b}.
\]
Then
\begin{equation*}
 \tau_k(\omega^a)\tau_k(\omega^b)=-J(a,b)\tau_k(\omega^{a+b}),
\end{equation*}
and $J(a,b)$ is a unit with
\begin{equation*}
 J(a,b)\equiv-\prod_{i=0}^{f-1}\binom{a_i+b_i}{a_i}
                  \pmod{\varpi\OO_{K_1}}.
\end{equation*}
\end{lemma}
\begin{proof}
Expand the product of the Gauss sums and separate $t=x+y$. When $t=0$
the contribution is zero because the character $[x]^{-(a+b)}$ is
nontrivial. When $t\ne0$, write $x=tu$, $y=t(1-u)$. The two sums
separate; the remaining unnegated Gauss sum is
$-\tau_k(\omega^{a+b})$. This gives the asserted product identity.

For the reduction of the Jacobi sum, both exponents below are positive,
so the terms at $0$ and $1$ may be included. In $k$ the reduction is
\[
 \sum_{x\in k}x^{q-1-a}(1-x)^{q-1-b}
 =\sum_{j=0}^{q-1-b}(-1)^j\binom{q-1-b}{j}
                          \sum_{x\in k}x^{q-1-a+j}.
\]
Only $j=a$ contributes, so the reduction is
$-(-1)^a\binom{q-1-b}{a}$.

To evaluate this binomial coefficient, use in $\F_p[X]$ the identity
\[
 (1+X)^{q-1-b}
    =\prod_{i=0}^{f-1}(1+X^{p^i})^{p-1-b_i}.
\]
Comparison of coefficients gives
$\binom{q-1-b}{a}=\prod_i\binom{p-1-b_i}{a_i}$ modulo $p$;
each exponent on the right has a unique digit expansion. Since
$a_i+b_i<p$, ordinary factorials give
\[
 \binom{p-1-b_i}{a_i}
       \equiv(-1)^{a_i}\binom{a_i+b_i}{a_i}\pmod p.
\]
For odd $p$, $a\equiv\sum_i a_i\pmod2$, so the signs cancel.
For $p=2$ the same sign equality holds in the residue field.
This proves the asserted Jacobi-sum congruence. None of its factorials is
zero modulo $p$, proving that $J(a,b)$ is a unit.
\end{proof}

\subsection{The leading coefficient}
\begin{proposition}[Leading Gauss congruence]\label{A:leading-gauss}
For every prime $p$, every $q=p^f$, and $0<a<q-1$,
\begin{equation*}
 \tau_k(\omega^a)\equiv\frac{\varpi^{s_p(a)}}{\prod_i a_i!}
                   \pmod{\varpi^{s_p(a)+1}\OO_{K_1}}.
\end{equation*}
\end{proposition}
\begin{proof}
If $q=2$ there is no index in the specified interval. Otherwise use
induction on $s_p(a)$. For digit sum one, $a=p^j$, and
Lemma~\ref{A:gauss-first} proves the assertion.

For larger digit sum choose $j$ with $a_j>0$ and put $c=a-p^j$.
Then $c>0$, and $c+p^j=a<q-1$ is an addition without carries.
Lemma~\ref{A:jacobi-unit} gives $J(c,p^j)\equiv-a_j\pmod\varpi$.
It is a unit, so the product formula in Lemma~\ref{A:jacobi-unit}
and the induction hypothesis give
\[
 \tau_k(\omega^a)
 =-\frac{\tau_k(\omega^c)\tau_k(\omega^{p^j})}{J(c,p^j)}
 \equiv\frac{\varpi^{s_p(a)}}{
                a_j(a_j-1)!\prod_{i\ne j}a_i!}
                \pmod{\varpi^{s_p(a)+1}}.
\]
This proves the stated Gauss-sum congruence; the $\varpi$-valuation of
$\tau_k(\omega^a)$ is $s_p(a)$.
The argument includes $p=2$, with $\varpi=-2$.
\end{proof}

For nonzero $X,Y$ in the coefficient field, write $X\equiv_\times Y$
if $X/Y\in1+\varpi\OO_{K_1}$.

\begin{corollary}[Leading Gauss congruence in base-$q$ digits]\label{A:grouped-gauss}
Let $q=p^f$, with $f\ge1$, and let $r\ge1$. Use the coefficient-field construction of
Section~\ref{A:coefficient-field} with residue field $\F_{q^r}$,
and let $\omega$ be its Teichm\"uller character. Take the additive character
$x\mapsto\zeta^{\Tr_{\F_{q^r}/\F_p}(x)}$. For $0<a<q^r-1$, write
\[
 a=\sum_{i=0}^{r-1}\gamma_iq^i,\qquad 0\le\gamma_i<q.
\]
Then
\begin{equation*}
 \tau_{\F_{q^r}}(\omega^a)
 \equiv_\times\frac{\varpi^{\sum_i\gamma_i}}{\prod_i\gamma_i!}.
\end{equation*}
\end{corollary}
\begin{proof}
For $n\ge0$ with base-$p$ digits $n_j$, separating the multiples of $p$
and applying Wilson's theorem successively gives
\[
 n!\equiv_\times(-p)^{(n-\sum_j n_j)/(p-1)}\prod_j n_j!.
\]
Write $\gamma_i=\sum_{j=0}^{f-1}\gamma_{ij}p^j$.
Since $\varpi^{p-1}\equiv_\times-p$, also for $p=2$ with $\varpi=-2$,
the factorial congruence gives
\[
 \frac{\varpi^{\sum_i\gamma_i}}{\prod_i\gamma_i!}
 \equiv_\times
 \frac{\varpi^{\sum_{i,j}\gamma_{ij}}}{\prod_{i,j}\gamma_{ij}!}.
\]
The $\gamma_{ij}$ are the base-$p$ digits of $a$, so
Proposition~\ref{A:leading-gauss} gives the assertion.
\end{proof}

\subsection{Polynomial character values and power sums}
\begin{lemma}[Additivity of polynomial character values]\label{O:M:rigidity}
Let $M$ be a nonarchimedean local field of residue characteristic $p$,
and let $\Psi_M$ be an additive character. Write
$f(Z)=\sum_{n=1}^{D}a_nZ^n\in M[Z]$, where $D<p^2$, and suppose
$x\mapsto\Psi_M(a_nx)$ has order dividing $p$ on $\OO_M$ for each $n$.
Let $M'/M$ be unramified of degree prime to $p$, with absolute residue
degree at least $3$, and set $\Psi_{M'}=\Psi_M\circ\Tr_{M'/M}$.
For $\lambda\in k_{M'}$, let $[\lambda]$ be its Teichm\"uller lift
(the coefficient-field lift in equal characteristic). Suppose
\[
 \lambda\longmapsto\Psi_{M'}(f([\lambda]))
\]
is additive on $k_{M'}$. Then
\[
 \Psi_M(f(1))=\Psi_M(a_1+a_p),
\]
where missing coefficients are zero.
\end{lemma}
\begin{proof}
Put $\psi_{\mathbf F_p}(x)=\exp(2\pi i x/p)$ for $x\in\mathbf F_p$.
For $x,y\in k_{M'}$,
$[x+y]-[x]-[y]\in p\OO_{M'}$.
In equal characteristic this difference is zero. Each coefficient
character is trivial on $p\OO_{M'}$, so its values on the lifts form
an additive character of $k_{M'}$, of the form
$(\psi_{\mathbf F_p}\circ\Tr_{k_{M'}/\mathbf F_p})(\beta_n\lambda)$;
unramified trace shows that each $\beta_n$ is the scalar extension of its
base coefficient. The displayed function is consequently
\[
 (\psi_{\mathbf F_p}\circ\Tr_{k_{M'}/\mathbf F_p})
       \Bigl(\sum_n\beta_n\lambda^n\Bigr).
\]

By the assumed additivity, this function is itself an additive character
of $k_{M'}$. The nondegenerate finite-field trace pairing therefore gives
$\beta\in k_{M'}$ such that
\[
 \Tr_{k_{M'}/\mathbf F_p}\!\left(\sum_n\beta_n\lambda^n\right)
   =\Tr_{k_{M'}/\mathbf F_p}(\beta\lambda)
       \qquad(\lambda\in k_{M'}).
\]
Expand both traces as sums of Frobenius powers and reduce positive
exponents modulo $|k_{M'}|-1$. The monomials of degrees
$1,\ldots,|k_{M'}|-1$ are linearly independent as functions on $k_{M'}$,
so the contributions from each Frobenius orbit agree separately.
The right side has only the orbit of $1$. Since the absolute residue
degree is at least $3$ and $D<p^2$, the only exponents of $f$ in that
orbit are $1,p$. Every other orbit has zero trace contribution.
Evaluating at $\lambda=1$ gives
$\Psi_{M'}(f(1))=\Psi_{M'}(a_1+a_p)$.

Both sides are $p$th roots of unity, and the unramified degree is prime
to $p$, so the equality descends to $M$.
\end{proof}

\begin{lemma}[Finite interpolation and power sums]\label{O:D:interpolation}
Let $p$ be an odd prime and let $L$ be a field of characteristic $0$ or $p$
containing all $(p-1)$st roots of unity. Put
$\mathcal X=\{0\}\cup\mu_{p-1}$, and let $\Delta\in L$.
For $0\le i<p$ and $Z\in L\setminus(-\mathcal X)$,
\[
 \sum_{\xi\in\mathcal X}\frac{(\Delta+\xi)^i}{Z+\xi}
 =\frac{p\Delta^i}{Z}+(p-1)\frac{(\Delta-Z)^i}{Z^p-Z}.
\]
For $1\le i<p$, the accompanying power sums are
\begin{align*}
 \sum_{\xi\in\mathcal X}(\Delta+\xi)^i
 &=p\Delta^i+(p-1)\ind_{i=p-1},\\
 \sum_{\xi\in\mu_{p-1}}\xi(\Delta+\xi)^i
 &=(p-1)(\ind_{i=p-2}+i\Delta\ind_{i=p-1}).
\end{align*}
\end{lemma}
\begin{proof}
Regard the first identity as an equality of rational functions in $Z$.
Both sides have only simple poles at $-\mathcal X$ and vanish at infinity.
At zero the residue on each side is $\Delta^i$; at $-\xi$ with
$\xi^{p-1}=1$, it is $(\Delta+\xi)^i$, since the derivative of
$Z^p-Z$ there is $p-1$. Subtraction gives a rational function with no
poles and zero value at infinity, hence zero. This works in both
characteristics because $p-1$ is nonzero.

The last two identities follow by the binomial theorem and
$\sum_{\xi\in\mu_{p-1}}\xi^j=0$ unless $p-1$ divides $j$, when the
sum is $p-1$.
\end{proof}

\section{Formal local residues and their trace compatibility}\label{app:local-residues}
This appendix proves the local identity used in the characteristic-two
argument. A differential on $k((t))$ means a \emph{formal differential}
$h(t)\,dt$, with termwise differentiation of Laurent series, and
\[
 \Res_t\left(\sum_n h_nt^n\,dt\right)=h_{-1}.
\]
We prove substitution and trace compatibility, including for wildly ramified extensions.

\subsection{Substitution and a finite free expansion}
\begin{lemma}[Formal substitution]\label{B:substitution}
Let $k$ be a field and let $u(t)\in t^e k[[t]]$ have exact order
$e\ge1$. For every $H(s)\in k((s))$,
\begin{equation*}
 \Res_t\bigl(H(u(t))u'(t)\,dt\bigr)
                         =e\Res_s(H(s)\,ds).
\end{equation*}
In particular the residue is unchanged by a change of uniformizer.
\end{lemma}
\begin{proof}
It suffices to treat $H(s)=s^n$: only finitely many of its Laurent
monomials can contribute to the residue after substitution. Write
$u=t^e v$ with $v\in k[[t]]^\times$. If $n=-1$, then
$u'/u=e/t+v'/v$ and $v'/v\in k[[t]]$, proving the assertion.

For $n\ne-1$ one must not divide by $n+1$ in characteristic dividing
$n+1$. Instead, perform the calculation first over the torsion-free
ring
\[
 R=\mathbf Z[c_0,c_0^{-1},c_1,c_2,\ldots],\qquad
 v=c_0+c_1t+c_2t^2+\cdots.
\]
The coefficient giving the residue uses only finitely many $c_i$ and
lies in this ring, also for negative $n$. The residue of a derivative
is zero, so differentiating $u^{n+1}$ gives
$(n+1)\Res_t(u^n u'\,dt)=0$. Torsion-freeness gives residue zero in
$R$. Specializing the coefficients to any field proves the desired
identity without division in that field. Linearity now proves the substitution identity; for a change of uniformizer, $e=1$.
\end{proof}

\begin{lemma}[A basis compatible with coefficients]\label{B:free-basis}
Let $R$ be a commutative ring, let $e\ge1$ be an integer, and let
$u(t)=t^e(c_0+c_1t+\cdots)$, where $c_0\in R^\times$.
Under $s\mapsto u(t)$, the module $R[[t]]$ is free over $R[[s]]$
with basis $1,t,\ldots,t^{e-1}$. After inverting $s$ it is
$R((t))$, free over $R((s))$ with the same basis. These expansions
and the resulting multiplication matrices commute with specialization
of the coefficients.
\end{lemma}
\begin{proof}
For $n=ej+i$, $0\le i<e$, the series $s^jt^i$ has leading term
$c_0^j t^n$. Given a power series, cancel its lowest coefficient
using a multiple of $s^jt^i$ and continue successively.
This constructs an expansion $\sum_{i=0}^{e-1}H_i(s)t^i$, with
$H_i\in R[[s]]$.

It is unique: the first nonzero terms of the
nonzero $H_i(s)t^i$ have different degrees modulo $e$, and hence
cannot cancel. The algorithm uses only ring operations and $c_0^{-1}$.
Every output coefficient is determined after finitely many operations,
which proves compatibility with specialization.

Since $s=t^e v(t)$ with $v(t)$ a unit, inverting $s$ is equivalent
to inverting $t$. This proves the assertions for Laurent series as well.
In particular, the algebraic trace of multiplication by an element of
$R((t))$ is an element of $R((s))$, obtained by the trace of its finite
matrix, and it commutes with specialization.
\end{proof}

\subsection{Trace compatibility}
\begin{lemma}[Totally ramified substitution]\label{B:trace-substitution}
For a field $k$ and $s=u(t)$ of positive order $e$,
\begin{equation*}
 \Res_s\left(\Tr_{k((t))/k((s))}(h(t))\,ds\right)
       =\Res_t\left(h(t)u'(t)\,dt\right)
                 \qquad(h\in k((t))).
\end{equation*}
The trace here is the algebraic trace on the finite free extension.
The identity remains valid even when $u'=0$.
\end{lemma}
\begin{proof}
First suppose that $k$ has characteristic zero. After a finite extension
of the coefficient field, we can take an $e$th root of the leading
coefficient of $u$ and include all $e$th roots of unity. This scalar
extension preserves the trace matrix in Lemma~\ref{B:free-basis}, so an
identity proved there descends. The binomial series then gives a formal
parameter $z=t(c_0+c_1t+\cdots)^{1/e}$ with $s=z^e$.
Write $h=\sum_n b_nz^n$. The extension $k((z))/k((s))$ is now
Galois with automorphisms $z\mapsto\xi z$, $\xi^e=1$. Thus
\[
 \Tr_{k((z))/k((s))}(h)=e\sum_{j\in\mathbf Z}b_{ej}s^j,
 \qquad h\,ds=e\sum_n b_nz^{n+e-1}\,dz.
\]
Both residues equal $e b_{-e}$. The change from $t$ to $z$ preserves
residue by Lemma~\ref{B:substitution}. The trace is continuous for the
Laurent-series topology, either from the displayed automorphisms or
from its finite multiplication matrix, so these calculations apply to
infinite Laurent series, not just Laurent polynomials.

Now take a field of arbitrary characteristic. Lift every coefficient
of $u(t)$ and $h(t)$ to an independent indeterminate in a polynomial
ring over $\mathbf Z$, inverting the nonzero leading coefficient of
$u$. This ring is a characteristic-zero integral domain. The two residues in the asserted identity belong to this ring:
Lemma~\ref{B:free-basis} gives this for the trace, and the right side
is an ordinary coefficient of a product and derivative. Each residue
uses only finitely many of the indeterminates. Over the fraction field
of this ring the characteristic-zero proof gives equality. It is
therefore an equality in the ring itself. Specialize the coefficients
to their values in $k$, using Lemma~\ref{B:free-basis} to specialize
the trace. This proves the identity in every characteristic.
This coefficient argument does not invert the possibly zero integer
$e$ in $k$ and does not assume a tame extension.
\end{proof}

\begin{theorem}[Local trace--residue identity]\label{B:trace-residue}
Let $E/F$ be a finite separable extension of equal-characteristic local
fields with finite residue fields. For every formal differential
$\eta$ on $E$,
\begin{equation*}
 \Tr_{k_E/k_F}(\Res_E(\eta))=\Res_F(\Tr_{E/F}(\eta)).
\end{equation*}
Here if $s$ is a uniformizer of $F$ and $\eta=x\,ds$, the trace of
the differential is $\Tr_{E/F}(x)\,ds$.
\end{theorem}
\begin{proof}
In an equal-characteristic complete discrete valuation field with finite
residue field $k$ of size $q$, each residue has a unique lift solving
$X^q-X=0$, by Hensel's lemma with derivative $-1$. In characteristic
$p$ these lifts form a field: sums and products are again roots with
the expected residues. Choosing a uniformizer $t$ and successively
subtracting residue coefficients gives a unique convergent Laurent
expansion. Thus the field is $k((t))$. In a finite extension, these
coefficient fields are compatible, because the lifts of elements of
the smaller residue field are the same unique roots.

Write $F=k_F((s))$ and $E=k_E((t))$. The intermediate field
$F'=k_E((s))$ is unramified over $F$; its trace is the coefficientwise
finite-field trace. In $E$, $s=u(t)$ has positive order $e$, and
Lemma~\ref{B:free-basis} gives $[E:F']=e$. Separability implies
$u'(t)\ne0$: otherwise $u\in k_E((t^p))$, and the intermediate
extension $k_E((t))/k_E((t^p))$ would be purely inseparable of degree
$p$, contradicting separability of $E/F'$. Therefore $ds\ne0$ and
every formal differential on $E$ can be written as $x\,ds$.

Apply Lemma~\ref{B:trace-substitution} to $E/F'$. It gives
$\Res_{F'}(\Tr_{E/F'}(x)\,ds)=\Res_E(x\,ds)$ in $k_E$.
Apply $\Tr_{k_E/k_F}$ and commute it with extraction of the
$s^{-1}$ coefficient. Transitivity of matrix traces proves the trace--residue identity. Residues do not depend on the chosen
uniformizers by Lemma~\ref{B:substitution}.
\end{proof}

\begin{lemma}[Logarithmic derivative of a norm]\label{B:dlog-norm}
For $v\in E^\times$ in the same separable extension,
\begin{equation*}
 \frac{d\Nm_{E/F}(v)}{\Nm_{E/F}(v)}
                      =\Tr_{E/F}\left(\frac{dv}{v}\right).
\end{equation*}
\end{lemma}
\begin{proof}
Let $D_F$ be differentiation with respect to a Laurent parameter of $F$,
and let $D_E$ be its extension to $E$; separability gives this extension.
Choose an $F$-basis $e_1,\ldots,e_n$ of $E$. Let $M=M(v)$ be the matrix of
multiplication by $v$, and let $B$ be the matrix whose $j$th column is the
coordinate vector of $D_E(e_j)$. Differentiating the coordinate identity
for $ve_j$ gives
\[
 D_F(M)+BM=M(D_E(v))+MB,
\]
where $D_F(M)$ is obtained by differentiating each entry and $M(D_E(v))$ is
the matrix of multiplication by $D_E(v)$.

Multiply the displayed identity
on the left by $M^{-1}$ and take matrix traces. Cyclic invariance of the
matrix trace gives
\[
 \operatorname{tr}(M^{-1}D_F(M))
 =\operatorname{tr}(M^{-1}M(D_E(v)))
 =\Tr_{E/F}(D_E(v)/v).
\]
Jacobi's determinant identity gives
\[
 \frac{D_F(\det M)}{\det M}
   =\operatorname{tr}(M^{-1}D_F(M)).
\]
Multiplication by the base differential gives the asserted logarithmic-derivative identity.
\end{proof}

\section{Elementary descent and ramification arguments}\label{app:descent}
The following lemmas are used to extend characters and to compute
ramification. The results for extensions with Galois group $C_\ell^2$
are proved in Appendix~\ref{app:diamond-foundations}.

\begin{lemma}[Cyclic Hilbert 90, in both forms]\label{C:hilbert90}
Let $E/F$ be a finite cyclic extension, and let $\sigma$ be a
generator of $\Gal(E/F)$. Then
\[
 \ker\Nm_{E/F}=\{y/\sigma(y):y\in E^\times\},\qquad
 \ker\Tr_{E/F}=\{\sigma(y)-y:y\in E\}.
\]
\end{lemma}
\begin{proof}
Put $n=[E:F]$.
Distinct field embeddings are linearly independent as maps. Indeed, take a nonzero relation of shortest length. Evaluating
at $tx$ and subtracting one embedding's value at $t$ times the relation
at $x$ removes that embedding. Choose $t$ on which another embedding
differs from it. The resulting shorter nonzero relation is a
contradiction.

If $a\in E^\times$ has norm one, put $b_0=1$ and
$b_i=\prod_{j=0}^{i-1}\sigma^j(a)$ for $1\le i\le n$.
By the independence just proved the map
$x\mapsto\sum_{i=0}^{n-1}b_i\sigma^i(x)$ is not identically zero.
Choose $x$ for which its value $y$ is nonzero. Since $b_n=1$ and
$a\sigma(b_i)=b_{i+1}$, one has $a\sigma(y)=y$.
This gives $a=y/\sigma(y)$. The reverse inclusion telescopes under
the norm.

For the additive statement, the trace map is nonzero by the same
independence, and hence is a surjective $F$-linear map to $F$.
Its kernel has dimension $n-1$. The map $\sigma-1$ has kernel $F$,
so its image also has dimension $n-1$. Its image lies in the trace
kernel by telescoping, proving equality. This proof applies when
the characteristic divides $n$.
\end{proof}

\begin{lemma}[Extension of characters]
\label{C:character-extension}
Let $E$ be a nonarchimedean local field and let $H$ be a subgroup of
$E^\times$ containing $U_E^m$ for some
$m\ge1$. Let $\chi:H\to\CC^\times$ be a character trivial on $U_E^m$.
Then $\chi$ extends to a continuous character of $E^\times$.
The same assertion holds if $\chi$ is first prescribed on several
subgroups, agrees on their intersections, and defines a character on
the subgroup they generate containing such a $U_E^m$.
\end{lemma}
\begin{proof}
The group $E^\times/U_E^m$ is a product of an infinite cyclic group
and the finite group $\OO_E^\times/U_E^m$. Extend a character of a
subgroup one generator at a time. If the first relation of a new
generator $g$ with the existing subgroup is $g^n=h$, choose a nonzero
complex $n$th root of the prescribed value of $h$. If there is no
such relation, choose any nonzero complex value. This defines an
extension at each step, and finitely many steps suffice. Pullback to
$E^\times$ is continuous because $U_E^m$ is open and remains in the
kernel. Apply this argument to the generated subgroup for the last assertion.
\end{proof}

\begin{lemma}[Truncated logarithms on unit groups]\label{O:I:logchart}
Let $L$ be a nonarchimedean local field of odd residue characteristic $p$,
and put $P(Z)=\sum_{j=1}^{p-1}Z^j/j$. The polynomial
$P(z+w-zw)-P(z)-P(w)$ has no monomial of total degree less than $p$.
Consequently, for integers $1\le r\le M$ with $pr\ge M$, the map
\[
 1-z\longmapsto P(z)
\]
is a group isomorphism $U_L^r/U_L^M\longrightarrow
\pp_L^r/\pp_L^M$. The map $P:\pp_L^r\longrightarrow\pp_L^r$ is bijective
and preserves valuations.
\end{lemma}
\begin{proof}
Compare degrees less than $p$ in the formal logarithm identity over
$\mathbf Q$. The truncated expression has coefficients in
$\mathbf Z_{(p)}$, so the assertion is valid in both characteristics.
Under the hypothesis $pr\ge M$, this identity makes the induced map
additive modulo $\pp_L^M$.

Since $P(z)=z+z^2R(z)$ with $R$ integral and
$P'(z)\equiv1\pmod{\pp_L}$, Lemma~\ref{D:newton} solves $P(z)=w$
uniquely in $\pp_L^r$. The same expression for $P$ shows that it
preserves valuations, so the induced map is an isomorphism.
\end{proof}

\begin{lemma}[Best approximation]\label{O:A:best-approximation}
Let $L/M$ be a finite totally ramified extension of local fields with
ramification index $e$, and let $\sigma$ be an $M$-automorphism
of $L$. If $y\in L$ and $\sigma(y)\ne y$, there is $a\in M$ such that
\[
 q:=v_L(y-a)=\max_{b\in M}v_L(y-b),
 \qquad e\nmid q.
\]
Moreover
\begin{equation*}
 q\le v_L(\sigma(y)-y).
\end{equation*}
\end{lemma}
\begin{proof}
For every $b\in M$,
\[
 v_L(y-b)\le v_L(\sigma(y-b)-(y-b))=v_L(\sigma(y)-y).
\]
Choose $a\in M$ attaining the maximum $q$ of these bounded integral valuations.
If $q=en$, choose $u\in\OO_M^\times$ with the same residue as
$(y-a)\pi_M^{-n}$, where $\pi_M$ is a uniformizer of $M$.
Then $v_L(y-a-u\pi_M^n)>q$, a contradiction.
\end{proof}

\begin{lemma}[The first wild ramification term]\label{C:leading-ramification}
Let $K/F$ be a finite Galois extension of local fields, let $\sigma$
be an inertia automorphism, and let $\pi$ be a uniformizer of $K$.
Suppose
\[
 v_K(\sigma(\pi)-\pi)=u+1,\qquad u\ge1.
\]
Put $c=(\sigma(\pi)-\pi)/\pi^{u+1}\in\OO_K^\times$.
For $x\in K^\times$ with $v_K(x)=m$, one has
$v_K(\sigma(x)-x)\ge m+u$ and
\[
 \overline{\frac{\sigma(x)-x}{\pi^{m+u}}}
 =m\bar c\,\overline{\frac{x}{\pi^m}}.
\]
In particular, $v_K(\sigma(x)-x)=m+u$ when $m$ is prime to the residue characteristic.
\end{lemma}
\begin{proof}
By definition, $\sigma(\pi)=\pi(1+c\pi^u)$.
Expand $x=\sum_{j\ge m}a_j\pi^j$ using the multiplicative lifts of
residue representatives. Inertia fixes those lifts: they are the
unique lifts of their residues among roots of unity of order prime
to the residue characteristic (and zero). The series converges in
the local field, with only finitely many negative powers. For any
integer $j$, the binomial expansion at the principal unit gives
\[
 (1+c\pi^u)^j-1\equiv jc\pi^u\pmod{\pp_K^{u+1}}.
\]
This is also valid for negative $j$, by the convergent inverse series.
Consequently the $j$th summand changes by an element of valuation at
least $j+u$. After division by $\pi^{m+u}$, only the $j=m$ summand can
have nonzero residue, and its residue is $m\bar c\,\bar a_m$.
Since $\bar a_m=\overline{x/\pi^m}$, this proves the displayed formula.
Its right side is nonzero when $m$ is prime to the residue characteristic,
which proves the equality of valuations.

Without the condition $u\ge1$, the first coefficient is instead
$\overline{(\sigma(\pi)/\pi)^m-1}$, which can vanish even when
$m$ is prime to the residue characteristic.
\end{proof}

\begin{lemma}[Transitivity of the different]\label{C:different-tower}
For a tower of finite separable local extensions $L/E/F$,
\[
 \mathfrak D_{L/F}
      =\mathfrak D_{L/E}\mathfrak D_{E/F}\OO_L.
\]
\end{lemma}
\begin{proof}
Write $A=\OO_F$, $B=\OO_E$ and $C=\OO_L$. These are finite free
modules along the tower. Via the nondegenerate trace pairings their
dual modules identify with the inverse differents. The elementary
adjunction
\[
 \operatorname{Hom}_A(C,A)
   \simeq\operatorname{Hom}_B(C,\operatorname{Hom}_A(B,A))
\]
sends a functional $\ell$ to $c\mapsto(b\mapsto\ell(bc))$; its
inverse evaluates at $b=1$. Write the principal fractional $B$-ideal
$\mathfrak D_{E/F}^{-1}=dB$. Under the trace identifications, the
right side is $d\operatorname{Hom}_B(C,B)
=d\mathfrak D_{L/E}^{-1}$. Transitivity of field trace shows that
these identifications send $x\in L$ to the same functional
$c\mapsto\Tr_{L/F}(xc)$. Thus
$\mathfrak D_{L/F}^{-1}=d\mathfrak D_{L/E}^{-1}$ as fractional
$C$-ideals. Inversion gives the assertion.
\end{proof}

\begin{lemma}[The discriminant of an integral order]\label{C:order-discriminant}
Let $E/F$ be finite separable and let $B\subset\OO_E$ be an
$\OO_F$-order of full rank. Then its discriminant ideal is the field
discriminant ideal multiplied by the square of its index ideal.
In particular the field discriminant divides the order discriminant.
For a totally ramified extension, the different exponent
is at most the $F$-valuation of the order discriminant.
\end{lemma}
\begin{proof}
Choose $\OO_F$-bases of $B$ and $\OO_E$, and let $M$ be the
inclusion matrix. The two trace Gram matrices satisfy
$G_B=M^{\mathsf T}G_{\OO_E}M$, so
$\det G_B=(\det M)^2\det G_{\OO_E}$. This proves the first
assertion and divisibility.

The valuation of the field discriminant
is the length over $\OO_F$ of the trace-dual lattice modulo
$\OO_E$: this follows by diagonalizing the trace matrix over the
discrete valuation ring. The trace dual is
$\mathfrak D_{E/F}^{-1}$ by the definition of the inverse different.
If the different exponent is $d$, this quotient has $\OO_F$-length $[k_E:k_F]d$, since each successive
$\OO_E$-lattice quotient is $k_E$. In the totally ramified case
$[k_E:k_F]=1$, proving the stated bound.
\end{proof}

\addtocontents{toc}{\protect\enlargethispage{12pt}}
\section{Norm subgroups and ramification in \texorpdfstring{$C_\ell^2$}{C ell squared}-extensions}
\label{app:diamond-foundations}
We prove the facts about $C_\ell^2$ extensions used in the proof of
Theorem~\ref{U:main}. We use the cyclic-prime results recalled in
Section~\ref{U:norm-characters},
together with \cite[Proposition~3.4]{Ueda}, and the standard local-field
results recalled below. For their general forms,
see~\cite[Chapters~III--V, XIII--XIV]{Serre}.

\subsection{Ramification and intermediate fields}

For a finite separable extension $E/F$, write $e(E/F)$ and $f(E/F)$
for its ramification index and residue degree. We use
\[
 [E:F]=e(E/F)f(E/F)
\]
and the monogenic different formula
\[
 \mathfrak D_{E/F}=(f'(\pi))
 \quad\text{if }\OO_E=\OO_F[\pi],
\]
where $f$ is the minimal polynomial of $\pi$ over $F$.
For a totally ramified extension of degree $n$, any uniformizer $\pi$
of $E$ satisfies this hypothesis. Indeed, the degree formula gives
$F(\pi)=E$. In the basis $1,\pi,\ldots,\pi^{n-1}$ the nonzero terms
$a_i\pi^i$ have valuations $nv_F(a_i)+i$, which are distinct modulo
$n$. If their sum is integral, each $a_i$ is integral.

We also use the following choices of generators. In characteristic
$p$, let $E/F$ be a cyclic extension of degree $p$, and let $\sigma$
be a generator of $\Gal(E/F)$. Lemma~\ref{C:hilbert90}, applied to
$1$, gives $\sigma(z)-z=1$.
Then $z^p-z\in F$ and $E=F(z)$.
For a quadratic extension $E/F$ in characteristic different from two,
let $\sigma$ be its nontrivial automorphism, choose $u\in E$ with
$u\ne\sigma(u)$, and put $w=u-\sigma(u)$.
Then $\sigma(w)=-w$, so $w^2\in F$ and $E=F(w)$.

\begin{lemma}[Newton approximation]
\label{D:newton}
Let $R$ be a complete discrete valuation ring with valuation $v$, let
$f\in R[X]$, and let $a\in R$. Suppose $s=v(f'(a))<\infty$ and
$v(f(a))>2s$. Then there is a unique root $b$ in the ball
$v(b-a)>s$. Moreover $v(b-a)\ge v(f(a))-s$.
\end{lemma}
\begin{proof}
The case $f(a)=0$ is immediate, including uniqueness by the argument
below. Otherwise iterate $a_{n+1}=a_n-f(a_n)/f'(a_n)$.
Taylor's polynomial identity with integral coefficients shows
inductively that $v(f'(a_n))=s$, and, writing
$r_n=v(f(a_n))-2s>0$, that
\[
 v(a_{n+1}-a_n)=s+r_n,\qquad r_{n+1}\ge2r_n.
\]
If an iterate is a root, stop. The derivative differences have valuation
at least $s+r_n>s$, so they cannot change $s$; the linear Taylor term
cancels and the remaining terms have valuation at least $2(s+r_n)$.
Thus the iterates are integral and converge to a root with the asserted
error bound.

For $b,c$ in the stated ball, divided differences give
$f(b)-f(c)=(b-c)(f'(a)+\varepsilon)$ with $v(\varepsilon)>s$.
The second factor is nonzero; hence two roots there are equal.
Rescaling the variable and the polynomial gives the same conclusion
for the fractional ideals used in the preceding sections.
\end{proof}

\begin{lemma}[Inertia, tame action, and subgroup numbering]\label{D:inertia}
For a finite Galois extension $K/F$ of local fields, let $G=\Gal(K/F)$,
write $G_j$ for its lower ramification groups, let $I$ be the kernel
of its residue action, and put $q=|k_F|$.
The residue action identifies $G/I$ with the cyclic group
$\Gal(k_K/k_F)$. Its fixed field $K^I/F$ is unramified.
The quotient $I/G_1$ is cyclic of order prime to the residue
characteristic $p$, and $G_1$ is a $p$-group.
A lift of residue Frobenius conjugates $I/G_1$ by the $q$th power map.
For $H\subset G$, the lower ramification groups on $K/K^H$ are
$H_j=H\cap G_j$ for every integer $j\ge0$.
For any uniformizer $\pi$ of $K$ and $j\ge1$, the map
\[
 G_j/G_{j+1}\longrightarrow(k_K,+),\qquad
 gG_{j+1}\longmapsto
 \overline{\frac{g(\pi)-\pi}{\pi^{j+1}}}
\]
is an injective homomorphism.
\end{lemma}
\begin{proof}
Choose a generator $\bar a$ of $k_K/k_F$, lift its monic irreducible
polynomial to $\OO_F[X]$, and lift $\bar a$ to a root $a\in\OO_K$
by Lemma~\ref{D:newton}. The field $F(a)$ has degree at most
$f=[k_K:k_F]$ and residue degree at least $f$, hence degree exactly
$f$ and ramification index one. All residue conjugates of $\bar a$
lift uniquely inside $F(a)$, so $F(a)/F$ is Galois and reduction
identifies its group with $\Gal(k_K/k_F)$. Normality of $K/F$
extends its automorphisms to $K$. An element of $G$ acts trivially on
the residue field precisely when it fixes $a$, by uniqueness of the
simple-root lifts. This proves the assertions about $I$ and $G/I$.

Choose a uniformizer $\pi$ of $K$. Inertia fixes the multiplicative
lifts of every residue element. Expanding elements of $\OO_K$ in powers of $\pi$
therefore tests lower ramification on $\pi$ alone. In particular,
\[
 I\longrightarrow k_K^\times,\qquad
 g\longmapsto\overline{g\pi/\pi}
\]
is a homomorphism with kernel $G_1$. For $j\ge1$ the map
\[
 G_j\longrightarrow k_K,\qquad
 g\longmapsto\overline{(g\pi/\pi-1)/\pi^j}
\]
is additive with kernel $G_{j+1}$. These statements follow by
multiplying the two uniformizer ratios and reducing the first
nonzero coefficient. The finite subgroup of $k_K^\times$ is cyclic
and has order prime to $p$; each additive image is a $p$-group.
The filtration eventually equals $1$, since every nonidentity inertia
automorphism moves $\pi$. This proves the assertion about $G_1$.

The tame residue is unchanged on replacing $\pi$ by $u\pi$ for a
unit $u$, since inertia fixes its residue. For a Frobenius lift $\phi$,
compute using $\pi'=\phi^{-1}(\pi)$:
\[
 \overline{(\phi g\phi^{-1})(\pi)/\pi}
   =\overline{\phi}\bigl(\overline{g(\pi')/\pi'}\bigr)
   =\overline{g\pi/\pi}^{\,q}.
\]

Finally, membership in a lower ramification group is the condition
$v_K(gx-x)\ge j+1$ for all $x\in\OO_K$. For an element of $H$
this condition is identical over $F$ and over $K^H$: both use the same
normalized valuation on $K$. This proves the subgroup assertion.
\end{proof}

\begin{lemma}[Ramification breaks in degree $p^2$]
\label{D:break-ledger}
Let $K/F$ be Galois with group $G=C_p^2$, where $p$ is the residue
characteristic. If $K/F$ is totally ramified, there are integers
$1\le t\le b$ such that its lower filtration is $G$ through $t$,
possibly one subgroup $H_0$ of order $p$ through $b$, and then $1$.
Choose $L_0=K^{H_0}$ if $b>t$, and any degree-$p$ intermediate field
if $b=t$. The breaks of $L_0/F$ and $K/L_0$ are $t$ and $b$,
respectively. Every other degree-$p$ intermediate field $L$ satisfies
\[
 L/F\text{ has break }s=t+(b-t)/p,\qquad
 K/L\text{ has break }t.
\]
In particular $s$ is an integer and $b=t+p(s-t)$.

If the inertia group has order $p$, there is a unique unramified
degree-$p$ intermediate field $U$. Every other degree-$p$
intermediate field $E$ is ramified over $F$, and the breaks of
$E/F$ and $K/U$ are equal.
\end{lemma}
\begin{proof}
Suppose first that $K/F$ is totally ramified.
Lemma~\ref{D:inertia} gives $G_1=G$. A decreasing filtration of
$C_p^2$ can contain at most one proper nontrivial subgroup.
The subgroup assertion gives the breaks of $K/L_0$ and $K/L$.

For a uniformizer $\pi$ of $K$, the monogenic different formula gives
\[
 D_{K/F}=v_K\left(\prod_{g\ne1}(\pi-g\pi)\right)
       =\sum_{j\ge0}(|G_j|-1)
       =(p^2-1)(t+1)+(p-1)(b-t).
\]
Here each nonidentity $g$ is counted exactly
$v_K(\pi-g\pi)$ times. Lemma~\ref{C:different-tower} and the cyclic-prime different
formula in Section~\ref{sec:ramification} give
\[
 D_{K/F}=(p-1)(b+1)+p\,D_{L_0/F}
         =(p-1)(t+1)+p\,D_{L/F}.
\]
Solving these equalities gives the breaks of $L_0/F$ and $L/F$, including the integrality of $s$.

If inertia has order $p$, its fixed field $U$ is the unique unramified
degree-$p$ intermediate field. Every other degree-$p$ intermediate field $E$ is totally ramified over $F$.
In $K=EU$, a uniformizer of $E$ is still a uniformizer of $K$.
A generator of $\Gal(K/U)$ restricts to a generator of $\Gal(E/F)$, and
$v_K|_E=v_E$.
Thus its uniformizer difference has the same valuation in $E$ and $K$.
\end{proof}

\subsection{Norm subgroups and norm characters}

\begin{proposition}[Distinct norm subgroups]
\label{D:norm-separation}
Let $K/F$ have Galois group $C_\ell^2$, and let $L_0,L$ be distinct
degree-$\ell$ intermediate fields. The norm subgroups
$\Nm_{L_0/F}(L_0^\times)$ and $\Nm_{L/F}(L^\times)$ are distinct if
$L_0/F$ is the unramified degree-$\ell$ subextension, or, when $K/F$ is totally ramified,
$L_0$ is the field selected in Lemma~\ref{D:break-ledger}.
In the totally ramified one-break case the norm subgroups of all
degree-$\ell$ intermediate fields over $F$ are pairwise distinct.
\end{proposition}
\begin{proof}
By \eqref{U:norm-character-group}, every subgroup
$\Nm_{L/F}(L^\times)$ has index $\ell$. If $L_0/F$ is unramified,
\eqref{U:unit-norm-surjective} shows that its nontrivial norm characters
have conductor zero, whereas \eqref{U:norm-character-conductor} gives
positive conductor for every ramified degree-$\ell$ intermediate field.
Equal norm subgroups would give equal character groups, contrary to
their conductors.

In the totally ramified two-break case,
Lemma~\ref{D:break-ledger} and \eqref{U:norm-character-conductor}
give distinct norm-character conductors $t+1<s+1$.

It remains to distinguish these subgroups when every extension $K/L$ and $L/F$ has break $t$.
Here $\ell=p$. Fix $\pi\in K$ of valuation one and put
\[
 c_g=\overline{(g\pi/\pi-1)/\pi^t}\in k,
       \qquad k=k_F=k_K.
\]
Lemma~\ref{D:inertia} makes $g\mapsto c_g$ an injective additive
map, with two-dimensional $\F_p$-image $V$. For a line
$J=\langle\sigma\rangle$, set $L_J=K^J$, $c=c_\sigma$,
$\Pi_J=\Nm_{K/L_J}(\pi)$, and $\varpi=\Nm_{K/F}(\pi)$.
For $\tau\notin J$ write $e=c_\tau$. The critical norm map is $X\mapsto X^p-c^{p-1}X$
by \cite[Proposition~3.4]{Ueda}. Since
$\Nm_{K/L_J}(\tau(\pi)/\pi)=\tau(\Pi_J)/\Pi_J$, its value at $e$ gives
\[
 \overline{(\tau(\Pi_J)/\Pi_J-1)/\Pi_J^t}
       =e^p-c^{p-1}e=:d_J\ne0.
\]
Apply \cite[Proposition~3.4]{Ueda} to $L_J/F$ with the
compatible uniformizers $\Pi_J,\varpi$. Its critical norm image is
$d_J^p\ker\Tr_{k/\F_p}$,
whose annihilator is the line $d_J^{-p}\F_p$ in the residue trace
pairing. By \eqref{U:norm-character-group} and \eqref{U:norm-character-conductor},
the restrictions of $S(L_J/F)$ to $U_F^t/U_F^{t+1}$ form a group of
order $p$. They annihilate the critical norm image, so they form
exactly this line.

Choose a basis $c_1,c_2$ of $V$ and put
$M_V=c_1c_2^p-c_1^pc_2\ne0$. For any ordered basis $c,e$ of $V$,
\[
 c e^p-c^pe=aM_V\quad\text{for some }a\in\F_p^\times,
 \qquad d_J=aM_V/c.
\]
Thus the annihilator line is $(c^p/M_V^p)\F_p$.
Distinct lines $J$ give distinct lines $c\F_p$, and Frobenius is
bijective, so these annihilators are distinct. The norm-character
groups, and hence the subgroups $\Nm_{L_J/F}(L_J^\times)$, are distinct.
\end{proof}

\begin{lemma}[Norm pullback of norm characters]
\label{D:crossed-norm}
Let $K/F$ have Galois group $C_\ell^2$, and let $L_i,L_j$ be
 distinct degree-$\ell$ intermediate fields. Suppose
$\Nm_{L_i/F}(L_i^\times)\ne \Nm_{L_j/F}(L_j^\times)$. Then
\[
 S(L_j/F)\longrightarrow S(K/L_i),\qquad
       \omega\longmapsto\omega\circ \Nm_{L_i/F}
\]
is an isomorphism. In particular it applies to the pairs $(L_0,L)$
in Proposition~\ref{D:norm-separation}.
\end{lemma}
\begin{proof}
Norm transitivity shows that every displayed pullback is trivial on
$\Nm_{K/L_i}(K^\times)$. A nontrivial $\omega$ has kernel
$\Nm_{L_j/F}(L_j^\times)$, since that quotient has prime order $\ell$.
If $\omega\circ \Nm_{L_i/F}$ were trivial, then
\[
 \Nm_{L_i/F}(L_i^\times)\subseteq\ker\omega
 =\Nm_{L_j/F}(L_j^\times).
\]
Both subgroups have index $\ell$ in $F^\times$, so they would be equal,
contrary to the hypothesis.
Both character groups have order $\ell$ by
\eqref{U:norm-character-group}, so the homomorphism is an
isomorphism.
\end{proof}

\subsection{Quadratic norm symbols and characters}\label{D:quadratic-symbols}
For a nonarchimedean local field $F$ of characteristic different from
 two and $A,B\in F^\times$, let $(A,B)$ be the value at $B$ of the
norm character of $F(\sqrt A)/F$, with value $1$ in the split case.

\begin{lemma}[Elementary identities for the quadratic norm symbol]\label{D:quadratic-symbol-identities}
The symbol is symmetric and multiplicative in both variables.
For $A\ne1$, one has $(A,1-A)=1$. If $A+B\ne0$, then
\begin{equation}\label{D:MX:symbol-transform}
 (A,B)=(A+B,-AB).
\end{equation}
For nonsquare $A$, the character $B\mapsto(A,B)$ is nontrivial.
\end{lemma}
\begin{proof}
The existence of a nonzero solution of $X^2=AY^2+BZ^2$ is equivalent
to $(A,B)=1$. For nonsquare $A$, every such solution has $Z\ne0$
and gives $B=\Nm_{F(\sqrt A)/F}(X/Z+(Y/Z)\sqrt A)$;
the split case is immediate. Symmetry of the conic proves symmetry
of the symbol. The norm-character property gives multiplicativity
in the second variable, and symmetry gives it in the first.
The element $1-A$ is the nonzero norm of $1+\sqrt A$ when $A$ is
nonsquare; the split case again gives the Steinberg identity.

When $A+B\ne0$, the invertible substitution $u=Ax+By$, $v=x-y$
in $z^2=Ax^2+By^2$ gives $(A+B)z^2=u^2+ABv^2$.
The equivalent conics prove \eqref{D:MX:symbol-transform}.
Finally, the nonsquare extension has norm quotient of order two by
\eqref{U:norm-character-group}, proving nontriviality.
\end{proof}

\begin{lemma}[Quadratic norm characters]\label{D:quadratic-norms}
Let $K/F$ be biquadratic with quadratic intermediate fields
$L_1,L_2,L_3$, and let $\omega_i$ be the nontrivial character in
$S(L_i/F)$. The characters $\omega_i$ are pairwise distinct and
satisfy $\omega_3=\omega_1\omega_2$.
Consequently Lemma~\ref{D:crossed-norm} applies to every pair
$L_i\ne L_j$.
\end{lemma}
\begin{proof}
If $\operatorname{char}F\ne2$, Lemma~\ref{D:quadratic-symbol-identities}
shows that the map from square classes to quadratic norm characters
is an injective homomorphism. The three fields have classes $a,b,ab$,
which proves the assertion.

If $\operatorname{char}F=2$, formula~\eqref{D:EQ:symbol} is additive
in the Artin--Schreier class $f$.
Lemma~\ref{D:EQ:ASsymbol} also shows that the character is nontrivial
for every nonzero class, whether ramified or unramified. The classes of the three fields are $f,g,f+g$, giving
pairwise distinctness and the product identity.
\end{proof}

\begin{lemma}[Triviality of the quadratic symbol on deep units]\label{D:MX:unit-symbol}
Let $F$ be a nonarchimedean local field of characteristic zero and
residue characteristic two. If $z,x\in F^\times$ have positive valuation and
$v_F(z)+v_F(x)>2v_F(2)$, then
$(1+z,1+x)=1$.
\end{lemma}
\begin{proof}
Write $e=v_F(2)$.
Put $m=v_F(z)$ and $n=v_F(x)$.
If $1+z$ is a square, the assertion is immediate; this includes
$m>2e$ by Lemma~\ref{D:newton}, applied to $X^2-(1+z)$ at $X=1$.

Otherwise fix a uniformizer $\varpi$ of $F$ and put
$r=\floor{m/2}\le e$. The integral element
$(\sqrt{1+z}-1)/\varpi^r$ generates an order of discriminant valuation
$2e-2r$. By Lemma~\ref{C:order-discriminant}, the field
discriminant divides the order discriminant, so the
quadratic norm character has conductor at most $2e-2r$; the unramified
case also satisfies this bound. If $m$ is even, $n\ge2e-2r+1$;
if $m$ is odd, $n\ge2e-2r$. In both cases the character kills $1+x$.
\end{proof}

\section*{Acknowledgments}
I am grateful to Chung Pang Mok for discussions on a local proof of the local
Langlands correspondence, which provided the original motivation for the First
and Second Main Lemma projects, and for his continued interest in this work.

OpenAI ChatGPT was used extensively in the development of this proof.
A substantial part of the argument consists of long local calculations, which
are in principle accessible by standard methods, but would have required a very
large amount of time. ChatGPT was therefore used to accelerate these technical
calculations. The author checked the final mathematical arguments and assumes
responsibility for the results.

\end{document}